\documentclass{amsart}
\usepackage{amsmath, amssymb, graphicx, epsfig, verbatim, psfrag, color,amscd,stmaryrd,leftidx} 
\usepackage{tikz}
\usepackage{amsthm}
\usepackage{enumitem}
\usetikzlibrary{arrows}

\newcommand{\lsupv}[2]{{}^{#1}\mskip-.2\thinmuskip\mathord{#2}}

\newcommand\state{\mathbf{K}}

\usetikzlibrary{matrix,arrows,positioning}
\tikzset{cdlabel/.style={above,sloped,
    execute at begin node=$\scriptstyle,execute at end node=$}}
\tikzset{algarrow/.style={->, thick}}
\tikzset{blgarrow/.style={->, thick}}
\tikzset{clgarrow/.style={->, thick}}
\tikzset{tensoralgarrow/.style={double, double equal sign distance, -implies}}
\tikzset{tensorblgarrow/.style={double, double equal sign distance, -implies}}
\tikzset{tensorclgarrow/.style={double, double equal sign distance, -implies}}
\tikzset{modarrow/.style={->, dashed}}
\tikzset{othmodarrow/.style={->, thick}}
\tikzset{Amodar/.style={->, dashed}}
\tikzset{Dmodar/.style={->, dashed}}
\newcommand\op{\mathrm{op}}
\newcommand\opp{\op}

\newcommand\MasGr{M}
\newcommand\TerMin{t\GenMin}
\newcommand\TerMinM{\TerMin^-}
\newcommand\TerMina{\widehat{\TerMin}}
\newcommand\GenMax{\Omega}
\newcommand\Min{\GenMin}
\newcommand\GenMin{\mho}
\newcommand\KHm{H^-}
\newcommand\KHa{\widehat H}
\newcommand\KCm{C^-}
\newcommand\KC{{\mathcal C}}
\newcommand\KCa{\widehat{C}}
\newcommand\Diag{\mathcal D}
\newcommand\States{\mathfrak{S}}
\newcommand\HFKm{\mathrm{HFK}^-}

\def\endproof{\relax\ifmmode\expandafter\endproofmath\else
  \unskip\nobreak\hfil\penalty50\hskip.75em\hbox{}\nobreak\hfil\bull
  {\parfillskip=0pt \finalhyphendemerits=0 \bigbreak}\fi}
\def\endproofmath$${\eqno\bull$$\bigbreak}
\def\bull{\vbox{\hrule\hbox{\vrule\kern3pt\vbox{\kern6pt}\kern3pt\vrule}\hrule}}

\def\mathcenter#1{%
  \vcenter{\hbox{$#1$}}%
}
\newcommand\CanonDD{\mathcal K}
\newcommand\North{\mathbf N}
\newcommand\South{\mathbf S}
\newcommand\East{\mathbf E}
\newcommand\West{\mathbf W}

\newcommand\Pos{\mathcal P}
\newcommand\Neg{\mathcal N}
\newcommand\Upwards{\mathcal{S}}
\newcommand\HFKa{\widehat{\mathrm{HFK}}}
\newcommand\ground{\mathbf k}
\newcommand\groundj{\mathbf j}
\newcommand\groundi{\mathbf i}
\newtheorem{thm}{Theorem}[section]

\newtheorem{cor}[thm]{Corollary}

\newtheorem{lemma}[thm]{Lemma}
\newtheorem{prop}[thm]{Proposition}
\newtheorem{defn}[thm]{Definition}

\newtheorem{example}[thm]{Example}
\newtheorem{rem}[thm]{Remark}
\newtheorem{remark}[thm]{Remark}

\numberwithin{equation}{section}

\newcommand\OneHalf{\frac{1}{2}}

\newcommand\IdempRing{\mathbf{I}}
\newcommand\Idemp[1]{\mathbf{I}_{#1}}
\newcommand\DT{\boxtimes}
\newcommand\x{\mathbf x}
\newcommand\w{\mathbf w}
\newcommand\y{\mathbf y}

\newcommand\lsup[2]{\leftidx{^{#1}}{#2}{}}
\newcommand\lsub[2]{\leftidx{_{#1}}{#2}}

\newcommand{\AlgA}{{\mathcal A}}
\newcommand{\AlgB}{{\mathcal B}}
\newcommand{\Ideal}{\mathcal I}
\newcommand{\Mor}{\mathrm{Mor}}

\newcommand\z{\mathbf z}

\newcommand{\de}{\delta}

 \newcommand{\Z}{\mathbb Z}  \newcommand{\Q}{\mathbb Q} \newcommand{\R}{\mathbb R}

\newcommand\XX{\mathbf X}
\newcommand\YY{\mathbf Y}
\newcommand\ZZ{\mathbf Z}
\newcommand\Max{\Omega}
\newcommand\Crit{\mathcal E}

\newcommand\Field{\mathbb F}
\DeclareMathOperator{\Hom}{Hom}

\DeclareMathOperator{\Id}{Id}

\newcommand\VRot{\mathcal R}
\newcommand\Alg{\mathcal A}
\newcommand\Blg{\mathcal B}
\newcommand\Clg{\mathcal C}

\newcommand\Dlg{\mathcal D}
\newcommand\Tensor{\mathcal T}
\newcommand\Ainf{{\mathcal A}_{\infty}}
\newcommand\Ainfty\Ainf
\newcommand\Zmod[1]{{\mathbb Z}/{#1}{\mathbb Z}}
\newcommand\Mod{\mathfrak{Mod}}
\newcommand\AlgBZ{\AlgB_0}
\newcommand\Opposite{o}

\newcommand\sgen{\mathbf 1}
\newcommand\MasFilt{\mathfrak m}

\newcommand\NGradingSet{T}
\newcommand\MGradingSet{S}
\newcommand\gr{\mathbf{gr}}

\newcommand\MinGen{\mathbf T}

\makeatletter
\renewenvironment{proof}[1][\proofname]{\par
\pushQED{\qed}%
\normalfont \topsep6\p@\@plus6\p@\relax
\trivlist
\item\relax
{\bf#1\@addpunct{.}}\hspace\labelsep\ignorespaces
}{%
\popQED\endtrivlist\@endpefalse
}
\makeatother

\begin{document}
\title{Kauffman states, bordered algebras, and  a bigraded knot invariant}

\author[Peter S. Ozsv\'ath]{Peter Ozsv\'ath}
\thanks {PSO was partially supported by NSF grants DMS-1258274, DMS-1405114, and DMS-1708284}
\address {Department of Mathematics, Princeton University\\ Princeton, New Jersey 08544} 
\email {petero@math.princeton.edu}

\author[Zolt{\'a}n Szab{\'o}]{Zolt{\'a}n Szab{\'o}}
\thanks{ZSz was partially supported by NSF grants DMS-1006006, DMS-1309152, and DMS-1606571}
\address{Department of Mathematics, Princeton University\\ Princeton, New Jersey 08544}
\email {szabo@math.princeton.edu}

\begin{abstract}We define and study a bigraded knot invariant whose Euler
   characteristic is the Alexander polynomial, closely connected to
   knot Floer homology.  The invariant is the homology of a chain complex
   whose generators correspond to Kauffman states
   for a knot diagram. The definition uses 
   decompositions of knot diagrams: to a collection of points on the line,
   we associate a differential graded algebra; to a partial knot
   diagram, we associate modules over the algebra. The knot invariant
   is obtained from these modules by an appropriate tensor product.
\end{abstract}

\maketitle

\section{Introduction}
\label{sec:Intro}

The Alexander polynomial can be given a state sum
formulation as a count of certain Kauffman states,
each of which contributes a monomial in a formal variable
$t$~\cite{Kauffman}. In~\cite{AltKnots}, this description was lifted to knot Floer
homology~\cite{OSKnots,RasmussenThesis}: knot Floer homology is given
as the homology of a chain complex whose generators correspond
to Kauffman states.  The differentials in the complex, though,
  were not understood explicitly; they were given as counts of
  pseudo-holomorphic curves. A much larger model for knot Floer
homology was described in~\cite{MOS}, where the generators correspond
to certain states in a grid diagram, and whose differentials
  count certain embedded rectangles in the torus. The grid diagram
can be used to compute invariants for small
knots~\cite{BaldwinGillam,Droz}, but computations are limited by the
size of the chain complex (which has $n!$ many generators for a grid
diagram of size $n$).

The aim of this article is to construct and study an invariant of knots,
$\KHm(K)$, with the following properties.
\begin{enumerate}[label=(H-\arabic*),ref=(H-arabic*)]
  \item \label{item:ModuleStructure}
    Letting $\Field=\Zmod{2}$, $\KHm(K)$ is a bigraded module over the polynomial algebra $\Field [U]$.
    That is, there is a vector space splitting $\KHm(K)\cong \oplus_{d,s}\KHm_d(K,s)$, and an endomorphism $U$ of $\KHm(K)$
    with $U\colon \KHm_d(K,s)\to \KHm_{d-2}(K,s-1)$.
  \item If $\Diag$ is a diagram for $K$,
    with a marked edge, then 
    $\KHm(K)$ is obtained as the homology of a chain complex $\KCm(\Diag)$
    associated to the diagram.
  \item The complex $\KCm(\Diag)$ is a bigraded chain complex over $\Field[U]$,
    which is freely generated by the Kauffman states of $\Diag$.
  \item 
    \label{item:Euler}
    The graded Euler characteristic of
    $\KHm(K)$ is related to the symmetrized Alexander polynomial $\Delta_K(t)$ of the knot $K$, as follows:
    there is an identification of  Laurent series in $\Z[t,t^{-1}\rrbracket$
    \begin{equation}
      \label{eq:GradedEulerHFm}
      \sum_{d,s} (-1)^d \dim \KHm_d(K,s) t^s   = \frac{\Delta_K(t)}{(1-t^{-1})}.
    \end{equation}
\end{enumerate}

From its description, this invariant comes equipped with a great deal
of algebraic structure, similar to the Khovanov and Khovanov-Rozansky
categorifications of the Jones polynomial and its generalizations. The
structure also makes computations of the invariant for large examples
feasible.
In this paper, we also give an algebraic proof of invariance, hence
giving a self-contained treatment of this invariant. 

Building on the present work,
  in~\cite{HFKa2}, we generalize the constructions to define an
  invariant with more algebraic structure. In~\cite{HolKnot}, we
  relate the  constructions here and their generalizations
  from~\cite{HFKa2} with pseudo-holomorphic curve counting, to give an
  identification between these algebraically-defined invariants and a
  suitable variant of knot Floer homology~\cite{HolKnot}. See~\cite{OverviewHFK} for an expository paper overviewing this materal.
  A generalization
  of these constructions to links is given in~\cite{Links}.

\subsection{Decomposing knot diagrams}

The knot invariant $\KHm(K)$ is constructed by decomposing
a knot projection for $K$ into elementary pieces, and using those
pieces to put together a chain complex whose homology is $\KHm(K)$.

In a little more detail, a {\em decorated knot diagram} $\Diag$ for
$K$ is an oriented, generic knot projection of $K$ onto ${\mathbb R}^2$, 
together with a choice of a
distinguished edge, which meets the infinite region.
The projection gives a planar graph $G$ whose vertices
correspond to the double-points of the projection of $K$. Since $G$ is
four-valent, there are four distinct quadrants (bounded by edges)
emanating from each vertex, each of which is a corner of the closure
of some region of ${\mathbb R}^2\setminus G$.  Let $m$ denote the number of vertices of
$G$. Clearly, $G$ divides ${\mathbb R}^2$ into $m+2$ regions, one of which is the unbounded one.

\begin{defn}
A {\em Kauffman state} (cf.~\cite{Kauffman}; see also Figure~\ref{fig:SampleState}) for a decorated knot
projection of $K$ is a map $\state$ that associates to each vertex of $G$ one
of the four in-coming quadrants, subject to the following constraints:
\begin{itemize}
\item The quadrants assigned by $\state$ to distinct vertices are subsets of
  distinct bounded regions in ${\mathbb R}^2\setminus G$.
\item The quadrants of the
  bounded region that meets the distinguished edge 
  are not assigned by $\state$ to any of the vertices in $G$.
\end{itemize}
\end{defn}
\begin{figure}[ht]
\includegraphics{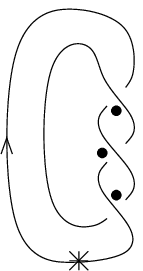}
\caption{\label{fig:SampleState} 
  {\bf{Decorated knot projection for the left-handed trefoil.}}
  The distinguished edge is marked with a star. We have illustrated
   one
  of the three Kauffman states for this projection.}
\end{figure}

Each Kauffman state sets up a one-to-one
correspondence between vertices of $G$ and the
connected components of
$\R^2\setminus G$ that do not meet the distinguished edge.

In~\cite{Kauffman}, Kauffman describes the Alexander polynomial of a
knot as a sum of monomials in $t$ associated to every Kauffman
state. We recall this description (with slight modifications to suit
our purposes). 

\begin{defn}
  Label the four quadrants around each crossing with $0$, and $\pm
  \OneHalf$, according to the orientations as specified in the second
  column of Figure~\ref{fig:LocalCrossing}. The {\em Alexander function}
  of a Kauffman state $\state$, $A(\state)$, is the sum, over each crossing,
  of the contribution of the quadrant occupied by the state.  The
  {\em Maslov function} of a Kauffman state $\state$ is obtained similarly,
  with local contributions as specified in the
  third column of Figure~\ref{fig:LocalCrossing}.
\end{defn}

 \begin{figure}[h]
 \centering
 \input{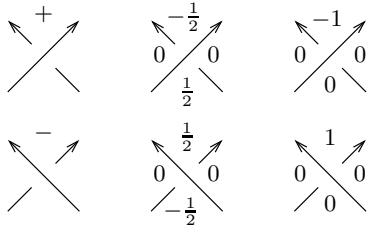}
 \caption{{\bf Sign conventions for crossings, and local Alexander and
     Maslov contributions.}  The first column illustrates the chirality
   of a crossing; the second the Alexander contribution of each
   quadrant; the third the Maslov contribution.}
 \label{fig:LocalCrossing}
 \end{figure}

Let $\States=\States(\Diag)$ denote the set of Kauffman states.
Kauffman shows that the  Alexander polynomial is computed by
$\Delta_K(t) = \sum_{\state\in\States} (-1)^{M(\state)} t^{A(\state)}$.
(Kauffman's description is slightly different; he
does not define the integer $M(\state)$, only its parity,
which suffices to compute the Alexander
polynomial.)

In~\cite{AltKnots}, we gave a description of knot Floer homology as
the homology groups of a chain complex whose generators are Kauffman
states, with bigradings given by the $M$ and $A$ functions defined
above, and whose differential counts pseudo-holomorphic disks. In
certain special cases (for example, for alternating knot diagrams,
after a possible change of basis),
these differentials could be computed explicitly; but in general, their
computation
remained elusive (compare~\cite{MOS}). 

In the present paper, we define $\KHm(K)$, which is the homology of a
chain complex $\KCm(\Diag)$ associated to a diagram. Its generators are
Kauffman states, and its differential is described algebraically.
The construction will involve decompositions of the knot diagram $\Diag$, 
as follows.

Given a generic, oriented knot projection in the $xy$-plane, 
we will consider the intersection of
$\Diag$ with half-planes $y\geq t$ and $y\leq t$, for generic values
of $t$. 
Intersections with $y\geq t$ are called {\em upper partial knot
  diagrams}, and intersections with $y\leq t$ are called {\em lower partial
  knot diagrams}.  
\begin{figure}[ht]
\input{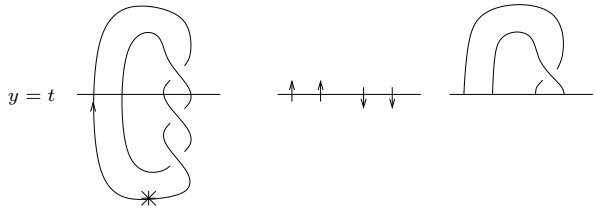}
\caption{\label{fig:SlicingDiagram} {\bf{Slicing the trefoil
      diagram.}}   The $y=t$ slice for the diagram on the left,
  gives the line with $4$ marked points (and orientations) in the
  middle.  The $y\geq t$ upper diagram is on the right.}
\end{figure}
The decorated edge of $\Diag$ will contain one of the
points with minimal $y$ value. See Figure~\ref{fig:SlicingDiagram}.

Our knot invariant will be constructed out of
invariants for upper and lower knot diagrams. In the spirit of
bordered Floer homology~\cite{InvPair}, we will associate the following algebraic 
objects to this picture:
\begin{itemize}
  \item  an algebra associated to
    the intersection of the knot diagram $\Diag$ with generic lines $y=t$.
  \item a ``type $D$ structure'' (in the sense of~\cite{InvPair}, recalled here in
    Section~\ref{subsec:TypeD})
    associated 
    to each (generic) upper diagram.
  \item a right ${\mathcal A}_{\infty}$-module associated to 
    each (generic) lower diagram.
\end{itemize}
Given this data, the complex $\KCm(\Diag)$ is obtained, for generic
$t$, as the tensor product (in the sense of~\cite{InvPair}, recalled
here in Section~\ref{subsec:Tensor}) of the invariant associated to
the $y\geq t$ upper diagram with the $y\leq t$ lower diagram, where
the pairing is taken over the algebra associated to $y=t$ slice.

In more detail, the intersection of the knot diagram $\Diag$ with a
generic horizontal line $y=t$ can be encoded as a line equipped with
$2n$ points $p_1,\dots,p_{2n}$ (i.e. where the knot meets the slice),
half of which are oriented upwards, and half of which are oriented
downwards. Let $\Upwards$ denote the subset of those points oriented
upwards. 
We will define an algebra associated to this configuration.
The points subdivide
the line into $2n-1$ bounded intervals, and two unbounded ones.  There are certain
distinguished {\em basic idempotents} in the algebra which correspond to
choices of $n$ of those bounded intervals. These algebras are
constructed in Section~\ref{sec:Algebras}.

We will describe certain natural modules over this algebra in terms of the following:

\begin{defn}
  \label{def:PartialKauffman}
  An {\em upper Kauffman state} for an upper diagram $y\geq t$ is a
  pair $(\state,I)$ where
  \begin{itemize}
  \item $\state$ associates to each crossing in the upper diagram
    one of the four adjacent quadrants.
  \item $I$ is a basic idempotent for the $y=t$ algebra.
  \end{itemize}
  Moreover, these data are required to satisfy the compatibility conditions:
  \begin{itemize}
  \item The quadrants assigned by $\state$ to distinct vertices are subsets
    of distinct bounded regions in the upper diagram $y\geq t$.
  \item The unbounded region meets none of the intervals in the idempotent $I$.
  \item Each bounded region in the upper diagram $y\geq t$ either:
    \begin{itemize}
      \item  contains a quadrant associated to some vertex by $\state$,
        and  meets none of the intervals in the idempotent $I$; or
      \item  meets exactly one of the intervals in the idempotent $I$,
        and  is not associated to any vertex by $\state$.
    \end{itemize}
  \end{itemize}
\end{defn}

It is easy to see that any Kauffman state can be restricted 
canonically to give an upper Kauffman state on any of its  upper diagrams;
whereas upper Kauffman states on an upper diagram might 
not extend to give a global Kauffman state.
\begin{figure}[ht]
\input{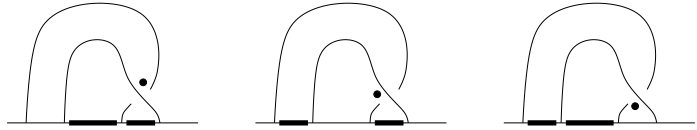}
\caption{\label{fig:PartialKauffman} {\bf{Upper Kauffman states.}}
  Here are all three upper Kauffman states for the upper diagram
  obtained from the trefoil diagram in
  Figure~\ref{fig:SlicingDiagram}, with  idempotent $I$ represented
  by the pairs of darkened intervals.}
\end{figure}

The type $D$ structure of an upper diagram is generated by its
upper Kauffman states.

We will also describe bimodules for enlarging the upper knot diagram.
Specifically, we will associate bimodules to crossings, caps, and cups, 
in Sections~\ref{sec:CrossingDA},~\ref{sec:Maximum},
and~\ref{sec:Minimum} respectively. These bimodules are generated
by certain localized Kauffman states, and the type $D$ structure of an
upper diagram is built as a tensor product of the bimodules associated
to these various pieces. Tensoring together all the pieces gives the
desired chain complex $\KCm(\Diag)$ whose generators are Kauffman
states.

The topological
invariance of $\KHm(K)$ is established in 
Theorem~\ref{thm:KnotInvariance}.
The proof proceeds by showing that
$\KHm(\Diag)$ is invariant under planar isotopies of the knot diagram,
and then locally verifying invariance under the three Reidemeister moves.
In fact, invariance
under Reidemeister moves $2$ and $3$ are part of the
``braid relations'' that the crossing bimodules satisfy (see Section~\ref{sec:Braid}), while Reidemeister $1$ invariance is an easy computation.

The complex $\KCm(\Diag)$ is a chain complex over the polynomial ring
in one generator $U$. We can set $U=0$ to get another chain complex
$\KCa(\Diag)$ whose homology $\KHa(K)$ is also a knot invariant.

Knot Floer homology~\cite{OSKnots} is an invariant with the above
properties, according to~\cite{AltKnots}.  The differentials appearing
in the original definition of knot Floer homology involve analytical
choices; algebraic constructions over a larger base ring were given
in~\cite{CubeRes} and~\cite{BaldwinLevine}, and a chain complex with
many more generators was given in~\cite{MOS}.  The invariant $\KHa(K)$
corresponds to the version of knot Floer homology denoted $\HFKa(K)$
and $\KHm(K)$ corresponds to another version of knot Floer homology,
denoted $\HFKm(K)$.  In~\cite{HolKnot}, we verify that $\KHa(K)$
and $\KHm(K)$ coincides with knot Floer homology $\HFKa(K)$ and
$\HFKm(K)$.

The methods of this paper are conceptually similar to the computation
of Heegaard Floer homology groups of three-manifolds by factoring
mapping classes~\cite{HFa}; but the present constructions are
ultimately algebraic in nature, as is the invariance proof we give
here; compare~\cite{Bohua}.  Although the bimodules we write
  down here might seem {\em ad hoc} in nature; the holomorphic theory
  did give us guiding principles for what to look for; see
  Section~\ref{subsec:Motivation}.
  (See also~\cite{HolKnot})

\subsection{Organization}
In Section~\ref{sec:Prelims} we discuss the algebraic preliminaries,
based on the algebra from of~\cite{Bimodules}, using throughout the
notions of bimodules of various types ($DD$, $DA$, and $AA$).  In
Section~\ref{sec:Algebras} we describe the algebras associated to knot
diagram slices, along with a canonical (invertible) dualizing bimodule
of type $DD$. In Section~\ref{sec:CrossingDD}, we associate natural
type $DD$ bimodules to crossings, which are simple to describe.  In
Section~\ref{sec:CrossingDA} we construct the corresponding type $DA$
bimodules, which are most useful to work with. These bimodules induce
a braid group action on the category of modules over our algebras, as
verified in Section~\ref{sec:Braid}. In Section~\ref{sec:Crit}, we
construct the type $DD$ bimodules associated to a critical point in
the knot diagram, and verify the ``trident relation'', which describes
how these bimodules interact with nearby crossing bimodules.  The
critical point bimodules have two kinds of corresponding type $DA$
bimodules: the bimodule associated to a maximum, constructed in
Section~\ref{sec:Maximum}, and the bimodule associated to a minimum,
constructed in Section~\ref{sec:Minimum}. The theory is equipped with
several symmetries, collected in Section~\ref{sec:Symmetries}.  In
Section~\ref{sec:ConstructionAndInvariance}, we construct the knot
invariant from the constitutent bimodules, and verify its invariance
properties. In Section~\ref{sec:Properties} we verify a few basic
properties of this invariant.

A bordered theory for tangles in the
grid context was developed by Ina Petkova and Vera
V{\'e}rtesi~\cite{PetkovaVertesi}; compare
also~\cite{GentleIntroduction}.  In a different direction, Rumen
Zarev~\cite{Zarev} constructed a ``bordered sutured'' invariant that
can also be used to study knot Floer homology.  Explicit computations
of $\KHa(K)$, for three-stranded pretzel knots, were done by Andrew
Manion. A version of this construction for singular
knots is also studied by Manion~\cite{ManionSingular}.

In~\cite{HFKa2}, we give a slight variation on the present
construction, together with a sign refinement.  The techniques from
that paper lead to efficient computations of $\KHm(K)$ for large
knots~\cite{Program}.  Similar to fast computations for Khovanov
homology~\cite{BarNatan} for non-alternating projections, there are
various cancelling differentials in the invariants associated to
partial knot diagrams that allow for fast computer
calculation. Details will be explained in~\cite{HFKa2}.  As an
illustration, we took the Gauss codes (which specify a knot up to reflection) 
for the knot $K_3$ from~\cite{FreedmanGompfMorrisonWalker},
to obtain a $91$-crossing presentation of $K$.
The  Poincar{\'e}
polynomial of $\KHa(K)$,
$P_K(m,t)=\sum_{d,s} \dim\KHa_d(K,s) m^d t^s$,
is given by
\begin{align*}
2 &t^4 m^6 + t^3(7m^5+3m^3+m) +t^2(10m^4+m^3+16m^2 + m+3)
\\
&+t(7m^3+2m^2+39m + 2 + 3 m^{-1}) +(4m^2+2m+53+2m^{-1}+2m^{-2}) \\
&+t^{-1}(7m+2+39m^{-1}+2m^{-2}+3m^{-3}) +t^{-2}(10+m^{-1}+16m^{-2}+m^{-3}+3m^{-4})\\
&+t^{-3}(7m^{-1}+3m^{-3}+m^{-5})+2t^{-4}m^{-2}.
\end{align*}

{\bf Acknowledgements} We wish to thank Adam Levine, Robert Lipshitz,
Andy Manion, B{\'e}la R{\'a}cz, Dylan Thurston, Rumen Zarev, and Bohua
Zhan for interesting discussions. We would like to thank Robert
Lipshitz, Andy Manion, and Ina Petkova, for their input on an early
draft of this paper; and we would like to thank the referee
for various suggestions.

\newcommand\Barop{\mathrm{Bar}}
\newcommand\DGradingSet\MGradingSet
\section{Algebraic preliminaries}
\label{sec:Prelims}

We recall some algebraic preliminaries from bordered Floer homology.
Further background on $\Ainf$ algebras can be found in~\cite{Keller}.
Most of this material (except  Section~\ref{subsec:Duality}) can be
found, with more detail, in~\cite{Bimodules}.

\subsection{Algebras}
In this paper, we will be concerned with differential graded 
algebras $\AlgA$ (DG algebras) in characteristic $2$.

The DG algebra $\AlgA$  is an abelian group equipped with a differential
$\mu_1\colon \AlgA\to \AlgA$,
and a 
multiplication map
$\mu_2\colon \AlgA\otimes\AlgA\to \AlgA$,
satisfying the usual compatibility conditions
\begin{align*}
 \mu_1^2 &= 0; \\
 \mu_2\circ (\mu_1\otimes \Id + \Id\otimes \mu_1)&=\mu_1\circ \mu_2; \\
\mu_2 \circ (\mu_2\otimes \Id)&=\mu_2\circ
(\Id\otimes\mu_2).
\end{align*}
(The latter two are the Leibniz rule and associativity rule respectively.)
It is customary to abbreviate $\mu_1(a)$ by 
$da$ and $\mu_2(a\otimes b)$ by $a\cdot b$.)

Our algebras are strictly unital; i.e. they are equipped with a
distinguished multiplicative unit $1$ which is a cycle.  We will typically think
of our algebras as defined over a ground ring $\ground$, which in turn
is a direct sum of finitely many copies of $\Field=\Zmod{2}$, equipped
with a vanishing differential. This means that there is a
distinguished subalgebra of $\Alg$, identified with $\ground$, whose
unit $1$ is also the unit in $\AlgA$. Moreover,
$\Alg$ as a bimodule over $\ground$, $\mu_1$ as a bimodule
homomorphism, and $\mu_2\colon \Alg\otimes_{\ground}\Alg\to \Alg$ is a
bimodule homomorphism. (Our algebras will be typically not
finitely generated over the ground ring $\ground$.)

\subsection{Gradings}
\label{sec:PrelimGradings}

Our algebras will be equipped with gradings, a {\em Maslov grading},
which takes values in $\Z$, and an {\em Alexander multi-grading} which 
takes values in some Abelian group $\Lambda=\Lambda_{\Alg}$,
compatible with the algebra actions as follows.

The algebra $\Alg$ is equipped with a direct sum splitting
$\Alg= \bigoplus_{(d;\ell)\in \Z\oplus \Lambda}\Alg_{d;\ell}$.
A non-zero element $a\in\Alg_{d;\ell}$ is called {\em homogeneous with
  grading $(d;\ell)$}; or simply {\em homogeneous with respect to the
  grading by $\Z\oplus\Lambda$}, when we do not wish to specify its
actual grading. Similarly, if $a\in\bigoplus_{d\in\Z} \Alg_{d;\ell}$
for some fixed $\ell\in\Lambda$, we say that $a$ is $\Lambda$-homogeneous
(with grading $\ell$). We require
\[
  \mu_1 \colon \Alg_{d;\ell} \to \Alg_{d-1;\ell} \qquad
  \mu_2\colon \Alg_{d_1;\ell_1}\otimes \Alg_{d_2;\ell_2}  \to 
  \Alg_{d_1+d_2;\ell_1+\ell_2}.
\]
In our present applications, the group $\Lambda_\AlgA$ is  $(\OneHalf {\mathbb Z})^m\subset \Q^m$. 

The algebras considered here will satisfy the following further condition:

\begin{defn}
  \label{def:PositiveAlexanderGrading}
  We say that the Alexander multi-grading on $\Alg$ is {\em positive over $\ground$} if the following two conditions hold:
  \begin{itemize}
  \item $\ground$ is the set of algebra elements in $\Alg$ with Alexander multi-grading $0$.
  \item 
    For
    $a_1,\dots,a_\ell$ with Alexander multi-gradings $\lambda_1,\dots,\lambda_\ell$, if
    $\sum_{i=1}^\ell \lambda_i=0$, then each $\lambda_i=0$.
  \end{itemize}
\end{defn}

In fact, in this paper $\ground$ will consist of homogeneous elements
with bigrading $(0;0)$.

\subsection{Modules}
We will consider several kinds of modules over our algebras.
A {\em right differential module} over $\AlgA$ is a right
$\ground$-module $M$, equipped  with maps $m_1\colon M \to M$
and 
$m_2\colon M \otimes_{\ground} \AlgA\to M$
satisfying 
\begin{align*}
  m_1^2 &= 0;\\
  m_2\circ (m_1\otimes \Id + \Id\otimes \mu_1)&=m_1\circ m_2; \\
  m_2 \circ (m_2\otimes \Id)&=m_2\circ
(\Id\otimes\mu_2).
\end{align*}

We will consider modules that are {\em strictly unital}, meaning that
$m_2(x,1)=x$ for all $x\in M$. 
In this case,   $m_2$ is a right $\ground$-module map
$m_2\colon M\otimes_{\ground} \Alg\to M$.

Weakening associativity, one naturally arrives at the notion of an
$\Ainfty$ module $M$.
A {\em right $\Ainfty$ module} over $\AlgA$ is a right
$\ground$-module $M$, equipped with  maps
\[ m_i\colon M\otimes_{\ground}
\overbrace{\Alg\otimes_{\ground}\dots\otimes_{\ground}\AlgA}^{i-1}
\to M\]
for $i\geq 1$,
satisfying the following strict unitality conditions:
\begin{enumerate}
  \item 
    $m_{2}(x\otimes 1)=x$ for all $x\in M$
  \item 
    $m_{i}(x\otimes a_1\otimes \dots \otimes a_{i-1})=0$
    if $i>2$ and there is some $1\leq j \leq i-1$ with $a_j=1$;
  \end{enumerate}
and a 
compatibility condition which is perhaps best phrased in
terms of the bar construction. (See Equation~\eqref{eq:TypeACondition} below.)
Define
$\Tensor^*(\Alg)=\bigoplus_{i=0}^{\infty} 
\overbrace{\Alg\otimes_{\ground}\dots\otimes_{\ground}
    \Alg}^i$,
with the convention that the $0^{th}$ tensor product of $\Alg$
is $\ground$, which is a chain complex, with a differential
induced by $\mu_1$ and $\mu_2$; i.e.
\begin{align*}
  {\underline d}(a_1\otimes\dots\otimes a_i)
  & = 
  \sum_{j=1}^i a_1\otimes \dots\otimes a_{j-1}\otimes \mu_1(a_j) \otimes a_{j+1} \otimes \dots\otimes  a_i \\
  &+  \sum_{j=1}^{i-1} a_1\otimes\dots\otimes a_{j-1}\otimes \mu_2(a_j\otimes a_{j+1}) \otimes\dots
  \otimes a_i,
\end{align*}
with the understanding that ${\underline d}$ vanishes on $\Tensor^0(\Alg)$.
Consider
${\mathcal T}^*(M)=M\otimes_{\ground} \Tensor^*(\Alg)$.
The maps $m_i$ ($i\geq 1$)
induce a map 
${\underline m}\colon {\mathcal T}^*(M)\to {\mathcal T}^*(M)$
by the formula
\[
  {\underline m}(x\otimes a_1\otimes \dots \otimes a_i)
  = \sum_{j=0}^i m_{j+1}(x\otimes a_1 \otimes \dots \otimes a_j) \otimes a_{j+1}\otimes \dots\otimes a_i \]
The compatibility condition is equivalent to the condition that
\begin{equation}
  \label{eq:TypeACondition}
  {\underline m}({\underline m}(x\otimes {\underline a}))
  + {\underline m}(x\otimes {\underline d}(a))=0;
\end{equation}
i.e. the map ${\underline d}_M\colon {\mathcal T}^*(M) \to {\mathcal T}^*(M)$ 
given by
${\underline d}_M (x\otimes {\underline a})={\underline m}(x\otimes {\underline a})+ x\otimes {\underline d}(\underline a)$
is a differential.

Left differential modules and
$\Ainfty$ modules are defined analogously (though in adherence with
the conventions laid down in~\cite{InvPair}, our $\Ainfty$ modules will
typically be right modules). A differential module is an $\Ainfty$
module with $m_i=0$ for all $i\geq 3$.

Gradings are as follows. The modules $M$ we consider will typically have a Maslov
grading by $\Z$, and a further Alexander multi-grading set $\MGradingSet$, which is a
set $\MGradingSet$ with an action by $\Lambda_{\Alg}$. That is, there is a direct sum splitting
(as $\ground$-modules) $M = \bigoplus_{d\in\Z, s\in \MGradingSet} M_{d;s}$
and we assume that each summand $M_{d;s}$ is a finitely generated $\ground$-module.
The actions $m_i$ will be graded as follows. 
\[ m_i\colon M_{d_0;s} \otimes \Alg_{d_1;\ell_1}\otimes\dots \Alg_{d_{i-1};\ell_{i-1}}
\to M_{i-2+ \sum_{j=0}^{i-1}{d_j};s+\sum_{j=1}^{i-1}{\ell_j}}.\]

We will typically record the algebra $\Alg$ as a subscript for the
$\Ainfty$-module $M$, writing $M_{\Alg}$ if $M$ is a right
$\Ainfty$-module, and $\lsub{\Alg} M$ if $M$ is a left
$\Ainfty$-module.

Given two $\Ainfty$ modules $M_{\Alg}$ and $N_{\Alg}$, a {\em
  morphism} from $M_{\Alg}$ to $N_{\Alg}$ is a sequence of $\ground$-module maps 
$\{\phi_i\colon M_{\AlgA}\otimes_{\ground}\overbrace{\Alg\otimes_{\ground} \dots\otimes_{\ground}
  \Alg}^{i-1} \to N_{\Alg}\}_{i\geq 1}$. A morphism naturally induces a map
${\underline \phi}\colon M_{\AlgA}\otimes \Tensor^*(\Alg) \to N_{\AlgA}\otimes \Tensor^*(\Alg)$
by
\[ {\underline \phi}(x\otimes a_1 \otimes \cdots \otimes a_n) =
\sum_{i=0}^{n} \phi_{i+1}(x \otimes a_1\otimes\dots \otimes a_i)\otimes
a_{i+1}\otimes \dots\otimes a_n.\] When this induced map is a chain
map, we say that the morphism is a {\em homomorphism}.  More generally,
the space of morphisms can be given a differential, so that
$d_{\Mor}(\underline \phi)= {\underline d}_N \circ {\underline \phi} + {\underline \phi}\circ {\underline d}_{M}$.

Let $\Mod_{\Alg}$ resp. $\lsub{\Alg}\Mod$ denote the category of right
resp. left $\Ainfty$ modules over $\Alg$. This is a differential category
(so the morphism spaces are chain complexes). Specifically, given
$M_{\Alg}, N_{\Alg}\in \Mod_{\Alg}$, let $\Mor_{\Alg}(M,N)$ denote
the chain complex whose elements are 
maps $\{ \phi_i \colon M\otimes \Alg^{i-1}\to N\}_{i\geq 1}$, with differential given by
\begin{align*}
  (d\phi)_{i}&(x,a_1,\dots,a_{i-1})=
    \sum_{j=1}^{i} \phi_{i-j+1}(m_j^M(x,a_1,\dots,a_{j-1}),a_{j},\dots,a_{i-1}) \\
     &+  \sum_{j=1}^{i} m^N_{i-j+1}(\phi_j(x,a_1,\dots,a_{j-1}),a_{j},\dots,a_{i-1}) \\
   & + \sum_{j=1}^{i-1} \phi_i(x,a_1,\dots,\mu_1(a_j),\dots,a_{i-1}) \\
    &+
   \sum_{j=1}^{i-2} \phi_{i-1}(x,a_1,\dots,\mu_2(a_j, a_{j+1}),\dots,a_{i-1}).
\end{align*}

There is a  convenient graphical representation of formulas such as the
one above. We represent elements of $M$ by dashed arrows, elements of
$\Tensor^*(\Alg)$ by doubled arrows, and various maps between them
by labelled nodes. For instance, the map  $m\colon M\otimes
\Tensor^*(\Alg)\to M$, the morphism $\phi\colon M \otimes \Tensor^*(\Alg)\to N$,
and the differential ${\underline d}\colon \Tensor^*(\Alg)\to\Tensor^*(\Alg)$ are represented by the pictures
\[ 
\begin{tikzpicture}[scale=.8]
  \node at (0,1) (inM) {};
  \node at (1,1) (inAlg) {};
  \node at (0,0) (m) {$m$};
  \node at (0,-1) (out) {};
  \draw[modarrow] (inM) to (m); 
  \draw[tensoralgarrow] (inAlg) to (m);
  \draw[modarrow] (m) to (out);
\end{tikzpicture} 
\qquad
\qquad
\begin{tikzpicture}[scale=.8]
  \node at (0,1) (inM) {};
  \node at (1,1) (inAlg) {};
  \node at (0,0) (phi) {$\phi$};
  \node at (0,-1) (out) {};
  \draw[modarrow] (inM) to (phi); 
  \draw[tensoralgarrow] (inAlg) to (phi);
  \draw[modarrow] (phi) to (out);
\end{tikzpicture} 
\qquad
\qquad
\begin{tikzpicture}[scale=.8]
  \node at (0,1) (in) {};
  \node at (0,0) (d) {${\underline d}$};
  \node at (0,-1) (out) {};
  \draw[tensoralgarrow] (in) to (d);
  \draw[tensoralgarrow] (d) to (out);
\end{tikzpicture} 
\]
Using the map 
$\Delta\colon \Tensor^*(\Alg)\to\Tensor^*(\Alg)\otimes\Tensor^*(\Alg)$ defined by
$\Delta(a_1\otimes\dots\otimes a_j)=\sum_{i=0}^j (a_1\otimes\dots\otimes a_i)\otimes (a_{i+1}\otimes\dots\otimes a_j)$,
the differential of $\phi$ can be written:
\[ 
\mathcenter{
\begin{tikzpicture}[scale=.7]
  \node at (0,1) (inM) {};
  \node at (1,1) (inAlg) {};
  \node at (0,0) (dphi) {$d\phi$};
  \node at (0,-1) (out) {};
  \draw[modarrow] (inM) to (dphi); 
  \draw[tensoralgarrow] (inAlg) to (dphi);
  \draw[modarrow] (dphi) to (out);
\end{tikzpicture} 
}
= 
\mathcenter{
\begin{tikzpicture}[scale=.7]
  \node at(0,2) (inM) {};
  \node at (1.5,2) (inAlg) {};
  \node at (1.2,.5) (Delta) {$\Delta$};
  \node at (0,0) (m) {$m$};
  \node at (0,-1) (phi) {$\phi$};
  \node at (0,-2) (Mout) {};
  \draw[modarrow] (inM) to (m);
  \draw[tensoralgarrow] (inAlg) to (Delta);
    \draw[modarrow] (m) to (phi);
  \draw[tensoralgarrow] (Delta) to (phi);
  \draw[tensoralgarrow] (Delta) to (m);
  \draw[modarrow] (phi) to (Mout);
\end{tikzpicture}}
+
\mathcenter{
\begin{tikzpicture}[scale=.7]
  \node at(0,2) (inM) {};
  \node at (1.5,2) (inAlg) {};
  \node at (1.2,.5) (Delta) {$\Delta$};
  \node at (0,0) (phi) {$\phi$};
  \node at (0,-1) (m) {$m$};
  \node at (0,-2) (Mout) {};
  \draw[modarrow] (inM) to (phi);
  \draw[tensoralgarrow] (inAlg) to (Delta);
    \draw[modarrow] (phi) to (m);
  \draw[tensoralgarrow] (Delta) to (m);
  \draw[tensoralgarrow] (Delta) to (phi);
  \draw[modarrow] (m) to (Mout);
\end{tikzpicture}}
+ 
\mathcenter{
\begin{tikzpicture}[scale=.7]
  \node at(0,2) (inM) {};
  \node at (1.5,2) (inAlg) {};
  \node at (1.2,.5) (d) {${\underline d}$};
  \node at (0,0) (phi) {$\phi$};
  \node at (0,-2) (Mout) {};
  \draw[modarrow] (inM) to (phi);
  \draw[tensoralgarrow] (inAlg) to (d);
  \draw[tensoralgarrow] (d) to (phi);
  \draw[modarrow] (phi) to (Mout);
\end{tikzpicture}}.\]

Let $M_{\Alg}$ be a right $\Ainfty$-module over $\ground$, with a ($\Z$-valued) Maslov grading and an Alexander multi-grading 
with values in
$\MGradingSet$. We can form
the {\em opposite module} $\lsub{\Alg}{\overline M}$, which is a space of
maps from $M$ to $\Field$, also equipped with a Maslov grading and a grading by $\MGradingSet$, 
\begin{equation}
  \label{eq:OppositeTypeA}
  \lsub{\Alg}{\overline M} = \bigoplus_{d\in\Z, s\in \MGradingSet} \Hom_{\Field}(M_{-d;-s},\Field),
\end{equation}
with action specified as
follows. For fixed $\phi\in\Hom_{\Field}(M,\Field)$ and
$a_1\dots,a_{i-1}\in\Alg$, let ${\overline m}_i(a_1\otimes \dots
a_{i-1}\otimes \phi)$ be the homomorphism from $M$ to $\Field$ whose evaluation on $x$ is given by
$\phi(m_i(x\otimes a_{i-1}\otimes \dots\otimes a_{1}))$.
(Note that the tensor factors appearing here are over $\Zmod{2}$;
when $M$ is strictly unital, we can think of the tensor factors as taken 
over $\ground$; with the understanding that in 
$a_{i-1}\otimes\dots\otimes a_{1}$, the bimodule actions of $\ground$ on 
the $a_j$ are also opposites.) 

\subsection{Type $D$ structures}
\label{subsec:TypeD}

A {\em left type $D$ structure} over $\Alg$ is a  left
$\ground$-module $X$, equipped with a $\ground$-linear map
$\delta^1\colon X \to \Alg\otimes_{\ground} X$,
satisfying the compatibility condition
$(\mu_2\otimes \Id_{X})\circ (\Id_{\Alg}\otimes \delta^1)\circ \delta^1 +
(\mu_1\otimes \Id)\circ \delta^1=0$.  A right type $D$ structure is
defined analogously.

The maps will be drawn 
\[\begin{tikzpicture}[scale=.7]
  \node at (0,1) (inD) {};
  \node at (0,0) (d) {$\delta^1$};
  \node at (-1,-1) (outA) {};
  \node at (0,-1) (outD) {};
  \draw[modarrow] (inD) to (d); 
  \draw[algarrow] (d) to (outA);
  \draw[modarrow] (d) to (outD);
\end{tikzpicture}\]
and the structure relation is drawn
\[ 
\mathcenter{\begin{tikzpicture}[scale=.8]
  \node at (0,1) (inD) {};
  \node at (0,0) (d) {$\delta^1$};
  \node at (-.5,-1) (mu) {$\mu_1$};
  \node at (-1,-2) (outA) {};
  \node at (0,-2) (outD) {};
  \draw[modarrow] (inD) to (d); 
  \draw[algarrow] (d) to (mu);
  \draw[modarrow] (d) to (outD);
  \draw[algarrow] (mu) to (outA);
\end{tikzpicture}
}+
\mathcenter{
\begin{tikzpicture}[scale=.7]
  \node at (0,1.5) (inD) {};
  \node at (0,.5) (d1) {$\delta^1$};
  \node at (0,-1) (d2) {$\delta^1$};
  \node at (-1.5,-1.25) (mu) {$\mu_2$};
  \node at (0,-2.5) (outD) {};
  \node at (-2,-2.5) (outA) {};
  \draw[modarrow] (inD) to (d1); 
  \draw[algarrow] (d1) to (mu);
  \draw[modarrow] (d1) to (d2);
  \draw[algarrow] (d2) to (mu);
  \draw[algarrow] (mu) to (outA);
  \draw[modarrow] (d2) to (outD);
\end{tikzpicture}}
=0
\]

As in the case of modules, our type $D$ structures will have a
grading by $\Z$ (the Maslov grading), and an  Alexander grading with set
set $\DGradingSet=\DGradingSet_X$, a set with an action by
$\Lambda_{\Alg}$. (In our case, this will typically be a quotient of
$\Lambda_{\Alg}$.) Thus, $X$ is given a direct sum
splitting
$X = \bigoplus_{d\in\Z;s\in \DGradingSet} X_{d;s}$,
where each $X_{d;s}$ is a $\ground$-module.
(In fact, in the cases of relevance to us, $X$ will be finitely generated as a 
$\ground$-module.)
The actions respect these gradings, in the sense that
\[ \delta^1\colon X_{d;s} \to 
\bigoplus_{d_0+d_1=d-1; \ell_0 + s_1=s}
\Alg_{d_0;\ell_0}\otimes X_{d_1;s_1}.\]
We abbreviate this, writing
$\delta^1\colon X_{d;s}\to (\Alg\otimes X)_{d-1;s}$.

A left type $D$ structure $X$ induces a left differential module $\Alg\DT X$ in
a natural way. As a left $\Alg$-module, the space is $\Alg\otimes_{\ground} X$; i.e.
given $a\in\Alg$ and $b\otimes \x\in \Alg\otimes_{\ground} X$, we define
$m_2(a,b\otimes x)=\mu_2(a,b)\otimes x$.
The operator
$m_1=\mu_1\otimes \Id_{X} + (\mu_2\otimes \Id_X) \circ
(\Id_{\Alg}\otimes \delta^1_X)$
which can be graphically represented by
\[
{\mathcenter{\begin{tikzpicture}[scale=.7]
  \node at (0,1) (inA) {};
  \node at (1,1) (inD) {};
  \node at (0,-1) (outA) {};
  \node at (1,-1) (outD) {};
  \node at (0,0) (diffA) {$\mu_1$};
  \draw[algarrow] (inA) to (diffA);
  \draw[algarrow] (diffA) to (outA);
  \draw[modarrow] (inD) to (outD);
\end{tikzpicture}}}+
\mathcenter{\begin{tikzpicture}[scale=.7]
  \node at (0,1) (inA) {};
  \node at (1,1) (inD) {};
  \node at (0,-2) (outA) {};
  \node at (1,-2) (outD) {};
  \node at (1,0) (diffD) {$\delta^1$};
  \node at (0,-1) (diffA) {$\mu_2$};
  \draw[algarrow] (inA) to (diffA);
  \draw[algarrow] (diffA) to (outA);
  \draw[modarrow] (inD) to (diffD);
  \draw[modarrow] (diffD) to (outD);
  \draw[algarrow] (diffD) to (diffA);
\end{tikzpicture}}
\]
induces differential on  $\Alg\otimes_{\ground}X=\Alg\DT X$; i.e. $m_1^2=0$.
It is easy to see that
the differential and the $\Alg$-action satisfy a Leibniz rule; i.e. $\Alg\DT X$ is a
differential $\Alg$-module.
(This is the special case of a more
general construction: it is the tensor product of $\Alg$, viewed as a
bimodule over itself, with the type $D$ structure $X$.)

Conversely,
let $M$ be a left DG-module over $\Alg$ that splits as a
left $\Alg$-module as a direct sum of modules that are
isomorphic to left ideals in $\Alg$ by  elements of $\ground$; i.e.
\begin{equation} 
  \label{eq:TypeDModuleSplitting}
  M = \bigoplus_{g\in G} \Alg \cdot \YY(g)
\end{equation}
for some finite set $G$, and a map $\YY\colon G\to \ground$.
There is a $\ground$-submodule of $M$, 
$X=\bigoplus_{g\in G} \ground\cdot \YY(g)$.
Restricting the differential $m_1$ to $X$ gives a map
$\delta^1\colon X \to \Alg\otimes X=M$;
the hypothesis that $m_1^2=0$ is equivalent to the condition
that $\delta^1$ determines a type $D$ structure.

We will typically record the algebra as a superscript for a type $D$
structure, writing $\lsup{\Alg}X$ to denote a left type $D$ structure
$X$ over $\Alg$.  There is an analogous notion of right type $D$
structures; for a right type $D$ structure, we record the algebra (as
a superscript) on the right, e.g.  writing $X^{\Alg}$.

Analogously, we let $\lsup{\Alg}\Mod$ and $\Mod^{\Alg}$ denote the
category of left resp. right type $D$ structures. For this category,
$\Mor(\lsup{\Alg}P,~\lsup{\Alg}Q)$ is defined to be the chain complex of
maps
$h^1\colon  P \to \Alg \otimes_{\ground} Q$,
where the differential is specified by
\[ d (h^1) = (\mu_1^{\Alg} \otimes \Id_{Q}) \circ h^1 + (\mu_2^{\Alg}\otimes \Id_{Q})
\circ (\Id_{\Alg}\otimes h^1)\circ \delta^1_{P} + 
 (\mu_2^{\Alg}\otimes \Id_{Q})
\circ (\Id_{\Alg}\otimes \delta^1_Q)\circ h^1. \]
There is also a composition map, which is a chain map $\Mor(\lsup{\Alg}P,\lsup{\Alg}Q)\otimes \Mor(\lsup{\Alg}Q,\lsup{\Alg}R)\to
\Mor(\lsup{\Alg}P,~\lsup{\Alg}R)$, defined by taking $f^1\otimes g^1$ to
\begin{equation}
  \label{eq:ComposeD}
  (f\circ g)^1= (\mu_2\otimes \Id_Z)\circ (\Id_{\Alg}\otimes g^1)\circ f^1.
\end{equation}

If $\lsup{\Alg}X$ is a left type $D$ structure, we can form the {\em
  opposite} type $D$ structure as follows. 
As a $\ground$-module, ${\overline X}^{\Alg}$  is a  space of vector
space maps from $M$ to $\Field$:
\[ {\overline X}^{\Alg} = \bigoplus_{d\in\Z; s\in \DGradingSet}\Hom_{\Field}(X_{-d;-s},\Field).\]
This inherits naturally the structure
of a right $\ground$-module: the action of $\iota\in\ground$ on
$\phi\colon M \to \Field$ is the map that sends $x$ to
$\phi(\iota\cdot x)$.
The requisite map
${\overline\delta}^1\colon {\overline X}^{\Alg}\to {\overline X}^{\Alg}\otimes \Alg$
is adjoint to the map $\delta^1$ for $\lsup{\Alg}X$; i.e.
given $\phi\in {\overline X}^{\Alg}$, define ${\overline \delta}^1(\phi)\in 
\Hom(X,\Field) 
\otimes_{\ground} \Alg
\subset \Hom(X,\Alg)$ to be the map that sends
$x$ to $\langle \delta^1 x,\phi\rangle$, where 
\[ \langle \cdot,\cdot\rangle\colon \Alg\otimes X \otimes \Hom(X,\Field) \rightarrow \Alg \]
is induced by the evaluation  map.

\subsection{Tensor products}
\label{subsec:Tensor}

We recall the pairing between $\Ainfty$ modules and type $D$
structures from~\cite{InvPair}, which in fact can be thought of as a
model for the derived tensor product. (See for
example~\cite[Proposition~2.3.18]{Bimodules})

Fix first a type $D$ structure $\lsup{\Alg}X$.  There are
maps for integers $j\geq 0$ with
\begin{equation}
  \label{eq:DefIteratedDelta}
  \delta^{j}\colon X \mapsto \overbrace{\Alg\otimes_\ground \dots\otimes_\ground \Alg}^j \otimes_{\ground} X,
\end{equation}
with the following inductive definition:
\begin{itemize}
  \item $\delta^0$ is the identity map; 
  \item $\delta^1$ is as specified by the type $D$ structure;
  \item 
    and finally,
    $\delta^j= (\Id_{\Alg^{\otimes{j-1}}} \otimes \delta^1)\circ \delta^{j-1}$.
\end{itemize}
The sum $\sum_{j=0}^{\infty} \delta^j$ is notated
\[ 
\mathcenter{
\begin{tikzpicture}[scale=.8]
  \node at (0,1) (inD) {};
  \node at (0,0) (d) {$\delta$};
  \node at (-1,-1) (outA) {};
  \node at (0,-1) (outD) {};
  \draw[modarrow] (inD) to (d); 
  \draw[tensoralgarrow] (d) to (outA);
  \draw[modarrow] (d) to (outD);
\end{tikzpicture}}
=
\mathcenter{
\begin{tikzpicture}[scale=.8]
  \node at (0,1) (inD) {};
  \node at (0,-1) (outD) {};
  \draw[modarrow] (inD) to (outD);
\end{tikzpicture}}
+
\mathcenter{
\begin{tikzpicture}[scale=.8]
  \node at (0,1) (inD) {};
  \node at (0,0) (d) {$\delta^1$};
  \node at (-1,-1) (outA) {};
  \node at (0,-1) (outD) {};
  \draw[modarrow] (inD) to (d); 
  \draw[algarrow] (d) to (outA);
  \draw[modarrow] (d) to (outD);
\end{tikzpicture}}+
\mathcenter{
\begin{tikzpicture}[scale=.8]
  \node at (0,1) (inD) {};
  \node at (0,0) (d1) {$\delta^1$};
  \node at (0,-1) (d2) {$\delta^1$};
  \node at (-1.5,-2) (outA1) {};
  \node at (-1,-2) (outA2) {};
  \node at (0,-2) (outD) {};
  \draw[modarrow] (inD) to (d1); 
  \draw[algarrow] (d1) to (outA1);
  \draw[modarrow] (d1) to (d2);
  \draw[algarrow] (d2) to (outA2);
  \draw[modarrow] (d2) to (outD);
\end{tikzpicture}}
+\dots.
\]
(Note that in general, the image of $\delta$ is contained in 
$\left(\prod_{m=0}^{\infty}\Tensor^*(\Alg)\right)\otimes X$.)

Recall that if $M$ is a right $\Ainfty$-module and $X$ is a left type
$D$ structure, then under suitable circumstances, we can form the
tensor product $M\DT X$. This is a chain complex whose underlying
vector space is $M\otimes_{\ground} X$, and with  differential 
\begin{align*} 
\partial (p\otimes x)&=
\sum_{j=0}^{\infty} (m_{j+1} \otimes \Id_{X})\circ (p \otimes \delta^j(x)) \\
&=
{\mathcenter{\begin{tikzpicture}[scale=.8]
  \node at (0,1) (inA) {};
  \node at (.75,1) (inD) {};
  \node at (0,-1) (outA) {};
  \node at (.75,-1) (outD) {};
  \node at (0,0) (diffA) {$m_1$};
  \draw[algarrow] (inA) to (diffA);
  \draw[algarrow] (diffA) to (outA);
  \draw[modarrow] (inD) to (outD);
\end{tikzpicture}}}+
\mathcenter{\begin{tikzpicture}[scale=.8]
  \node at (0,1) (inA) {};
  \node at (1,1) (inD) {};
  \node at (0,-2) (outA) {};
  \node at (1,-2) (outD) {};
  \node at (1,0) (diffD) {$\delta^1$};
  \node at (0,-1) (diffA) {$m_2$};
  \draw[algarrow] (inA) to (diffA);
  \draw[algarrow] (diffA) to (outA);
  \draw[modarrow] (inD) to (diffD);
  \draw[modarrow] (diffD) to (outD);
  \draw[algarrow] (diffD) to (diffA);
\end{tikzpicture}}+
\mathcenter{\begin{tikzpicture}[scale=.8]
  \node at (0,1) (inA) {};
  \node at (1,1) (inD) {};
  \node at (0,-3) (outA) {};
  \node at (1,-3) (outD) {};
  \node at (1,0) (diffD1) {$\delta^1$};
  \node at (1,-1) (diffD2) {$\delta^1$};
  \node at (0,-2) (diffA) {$m_3$};
  \draw[algarrow] (inA) to (diffA);
  \draw[algarrow] (diffA) to (outA);
  \draw[modarrow] (inD) to (diffD1);
  \draw[modarrow] (diffD1) to (diffD2);
  \draw[modarrow] (diffD2) to (outD);
  \draw[algarrow] (diffD1) to (diffA);
  \draw[algarrow] (diffD2) to (diffA);
\end{tikzpicture}}
+~\dots = 
\mathcenter{\begin{tikzpicture}[scale=.8]
  \node at (0,1) (inA) {};
  \node at (1,1) (inD) {};
  \node at (0,-2) (outA) {};
  \node at (1,-2) (outD) {};
  \node at (1,0) (diffD) {$\delta$};
  \node at (0,-1) (diffA) {$m$};
  \draw[algarrow] (inA) to (diffA);
  \draw[algarrow] (diffA) to (outA);
  \draw[modarrow] (inD) to (diffD);
  \draw[modarrow] (diffD) to (outD);
  \draw[tensoralgarrow] (diffD) to (diffA);
\end{tikzpicture}}
\end{align*}

In general, the sum appearing in the definition of $\partial$ has infinitely many terms. The suitable
circumstances needed to define $\partial$ are those where the above
sum is finite. For instance, if the module $M$ has the property that
for all $p\in M$, $m_{j+1}(p,a_1,\dots,a_{j})=0$ for all sufficiently
large $j$ and any sequence $a_1,\dots,a_j$, then the sum is guaranteed
to be finite. Such an $\Ainfty$ module is called a {\em bounded
  $\Ainfty$ module}.  Similarly, finiteness is guaranteed if $X$ has
the property that for all $x\in X$, there is a $j$ with the property
that $\delta^j x=0$ for all sufficiently large $j$. Such a type $D$
structure is called a {\em bounded type $D$ structure}. To
recapitulate, $M\DT X$ exists if either $M_{\Alg}$ or $\lsup{\Alg}X$
is bounded.

Let $M$ be a right $\Ainfty$ module with Alexander grading set $\MGradingSet$ and
$X$ a left type $D$ structure with Alexander grading set $T$.  Then the tensor product
$M\DT X$ is naturally graded by  the product of $\Z$ (the Maslov
grading) and the Alexander grading set $\MGradingSet\times_{\Lambda} \NGradingSet= 
(\MGradingSet\times \NGradingSet)/\Lambda_{\Alg}$.

\subsection{Bimodules}

If $\Alg$ and $\Blg$ are two differential graded algebras over ground rings
$\groundj$ and $\ground$ respectively, a
{\em left/left type $DD$ bimodule} is a type $D$ structure over the
tensor product $\Alg\otimes\Blg$. 
In particular, it is a left module over $\groundj\otimes_{\Field}\ground$.
A left/right type $DD$ bimodule is a left/left type $DD$ bimodule over $\Alg\otimes\Blg^\op$.

A left/right {\em type~$DA$ bimodule} is a
$\groundj-\ground$ bimodule equipped with maps for $i\geq 1$:
\[\delta^1_i\colon X\otimes_{\ground}
\overbrace{\Blg\otimes_{\ground}\dots\otimes_{\ground}\Blg}^{i-1}
\to \Alg\otimes_{\groundj} X,\]
satisfying the structure equation
\begin{align*}
  0 &= 
  (\mu_1^{\Alg}\otimes \Id_{X})\circ \delta^1_{i}(x\otimes a_1 \otimes \dots\otimes a_{i-1})  \\
  &+\sum_{j=1}^{i-1} \delta^1_{i}(x\otimes  a_1\otimes\dots\otimes a_{j-1}\otimes \mu_1^{\Blg}(a_j)\otimes
  a_{j+1}\otimes\dots\otimes a_{i-1}) \\
  &+
  \sum_{j=1}^{i-2} \delta^1_i(x\otimes  a_1\otimes\dots\otimes a_{j-1}\otimes  \mu_2^{\Blg}(a_j\otimes a_{j+1})
  \otimes
  a_{j+2}\otimes\dots\otimes a_{i-1}) \\
  & + 
  \sum_{j=1}^{i}
  (\mu_2^{\Alg} \otimes \Id_{X}) \circ (\Id_{\Alg}\otimes \delta^1_{i-j+1})\circ (\delta^1_{j}(x\otimes a_{1}\otimes\dots\otimes
  a_{j-1})\otimes a_{j}\otimes\dots\otimes a_{i-1}) \\
  &=
\mathcenter{\begin{tikzpicture}[scale=.8]
  \node at (0,1) (inD) {};
  \node at (0,0) (d) {$\delta^1$};
  \node at (-1,-1) (mu) {$\mu_1$};
  \node at (-2,-2) (outA) {};
  \node at (1,1) (inB) {};
  \node at (0,-2) (outD) {};
  \draw[modarrow] (inD) to (d); 
  \draw[tensoralgarrow] (inB) to (d);
  \draw[algarrow] (d) to (mu);
  \draw[modarrow] (d) to (outD);
  \draw[algarrow] (mu) to (outA);
\end{tikzpicture}
}+
\mathcenter{\begin{tikzpicture}[scale=.8]
  \node at (0,2) (inD) {};
  \node at (0,0) (d) {$\delta^1$};
  \node at (1,1) (mu) {${\underline d}$};
  \node at (-1,-1) (outA) {};
  \node at (2,2) (inB) {};
  \node at (0,-1) (outD) {};
  \draw[modarrow] (inD) to (d); 
  \draw[tensoralgarrow] (inB) to (mu);
  \draw[tensoralgarrow] (mu) to (d);
  \draw[algarrow] (d) to (outA);
  \draw[modarrow] (d) to (outD);
\end{tikzpicture}}
+
\mathcenter{
\begin{tikzpicture}[scale=.8]
  \node at (0,2) (inD) {};
  \node at (0,0) (d1) {$\delta^1$};
  \node at (0,-1) (d2) {$\delta^1$};
  \node at (-1,-2) (mu) {$\mu_2$};
  \node at (2,2) (inB) {};
  \node at (1,1) (Delta) {$\Delta$};
  \node at (0,-3) (outD) {};
  \node at (-2,-3) (outA) {};
  \draw[modarrow] (inD) to (d1); 
  \draw[tensoralgarrow] (inB) to (Delta);
  \draw[tensoralgarrow] (Delta) to (d1);
  \draw[tensoralgarrow] (Delta) to (d2);
  \draw[algarrow] (d1) to (mu);
  \draw[modarrow] (d1) to (d2);
  \draw[algarrow] (d2) to (mu);
  \draw[algarrow] (mu) to (outA);
  \draw[modarrow] (d2) to (outD);
\end{tikzpicture}}
\end{align*}

Here, we think of $\Blg$ as the algebra of inputs, and $\Alg$ as
the output algebra.

\begin{example}
  \label{ex:IdBimodule}
  Fix an algebra $\Alg$ over $\ground$. 
  The identity bimodule $\lsup{\Alg}\Id_{\Alg}$
  is the type $DA$ bimodule whose underlying
  $\ground$-$\ground$ bimodule is $\ground$,
  so that the maps  $\delta^1_j$ have the form
  \[\delta^1_j\colon \overbrace{\Alg\otimes_{\ground} \dots\otimes_{\ground}\Alg}^{j-1} \to \Alg\]
  and whose operations are given by $\delta^1_2(a)=a$, and
  $\delta^1_j=0$ for $j=1$ and $j>2$.
\end{example}

\begin{example}
  \label{ex:MorphismBimodule}
  Let $\phi\colon \Blg\to\Alg$ be a homomorphism of DG algebras over a base ring $\ground$.
  This can be viewed as a bimodule $\lsup{\Alg}[\phi]_{\Blg}$ with a single generator over $\ground$,
  which we denote $\sgen$, with $\delta^1_2(\sgen,b)=\phi(b)\otimes \sgen$ and $\delta^1_j=0$ for $j\neq 2$.
\end{example}

Our $DA$ bimodules $\lsup{\Blg}X_{\Alg}$ will be bigraded as in Section~\ref{sec:PrelimGradings}.
Specifically, there will be some set $\MGradingSet$ with action by $\Lambda_{\Alg}\oplus\Lambda_{\Blg}$, and
$X=\bigoplus_{(d;s)\in\Z\oplus \MGradingSet} X_{(d;s)}$.
The actions respect  these gradings as follows:
\[ \delta^1_i\colon X_{d_1;s}\otimes_{\ground} \Blg_{d_2;\ell_2}
\otimes_{\ground}\dots\otimes_{\ground} \Blg_{d_i;\ell_i}\to
(\Alg \otimes X)_{i-2+\sum d_j;s+\sum_{j=2}^{i}\ell_j}.\]

A left/right {\em type~AA bimodule over $\Alg$ and $\Blg$}   is a
left/right bimodule over $\IdempRing(\Alg)$ and $\IdempRing(\Blg)$,
equipped with maps indexed by integers $i,j\geq 0$,
\[ m_{i|1|j}\colon
\overbrace{\Alg\otimes_{\IdempRing(\Alg)}\otimes\dots\otimes\Alg}^i
\otimes_{\IdempRing(\Alg)}\otimes X \otimes_{\IdempRing(\Blg)}
\overbrace{\Blg\otimes_{\IdempRing(\Blg)}\otimes\dots\otimes\Blg}^j \to
X.\] defined for all non-negative integers $i$ and $j$, satisfying a structure
equation which we will state shortly.  The maps can $m_{i|1|j}$ can be
assembled to form a map
\begin{align*}
m \colon \Tensor^*(\Alg)\otimes_{\IdempRing(\Alg)}\otimes M \otimes
_{\IdempRing(\Blg)} \Tensor^*(\Blg) \to 
M 
\end{align*}
that is represented by the diagram
\[
    \begin{tikzpicture}[scale=.8,baseline=(x.base)]
        \node at (-1,2) (out) {};
        \node at (-1,4) (m) {$m$};
        \node at (-1,6) (x) {${\mathbf x}$};
        \node at (0,6) (alg) {${\underline b}$};
        \node at (-2,6) (lalg) {${\underline a}$};
        \draw[tensorblgarrow, bend left=15] (alg) to (m);
        \draw[tensoralgarrow, bend right=15] (lalg) to (m);
        \draw[modarrow] (x) to (m);
        \draw[modarrow](m) to (out);
      \end{tikzpicture}
\]
and satisfies the structure equation 
(for all ${\underline a}\in \Tensor^*(\Alg)$ and ${\underline b}\in \Tensor^*(\Blg)$):
\[
    \begin{tikzpicture}[scale=.8,baseline=(x.base)]
        \node at (-1,2) (out) {};
        \node at (-1,4) (m) {$m$};
        \node at (0,5) (d) {${\underline d}$};
        \node at (-1,6) (x) {${\mathbf x}$};
        \node at (0,6) (alg) {${\underline b}$};
        \node at (-2,6) (lalg) {${\underline a}$};
        \draw[tensorblgarrow] (alg) to (d);
        \draw[tensorblgarrow, bend left=15] (d) to (m);
        \draw[tensoralgarrow, bend right=15] (lalg) to (m);
        \draw[modarrow] (x) to (m);
        \draw[modarrow](m) to (out);
      \end{tikzpicture}
      +
    \begin{tikzpicture}[scale=.8,baseline=(x.base)]
        \node at (-1,2) (out) {};
        \node at (-1,4) (m) {$m$};
        \node at (-2,5) (d) {${\underline d}$};
        \node at (-1,6) (x) {${\mathbf x}$};
        \node at (0,6) (alg) {${\underline b}$};
        \node at (-2,6) (lalg) {${\underline a}$};
        \draw[tensorblgarrow, bend left=15] (alg) to (m);
        \draw[tensoralgarrow] (lalg) to (d);
        \draw[tensoralgarrow, bend right=15] (d) to (m);
        \draw[modarrow] (x) to (m);
        \draw[modarrow](m) to (out);
      \end{tikzpicture}+
    \begin{tikzpicture}[scale=.8,baseline=(x.base)]
        \node at (-1,2) (out) {};
        \node at (-1,4) (m1) {$m$};
        \node at (-1,3) (m2) {$m$};
        \node at (0,5) (d) {$\Delta$};
        \node at (-2,5) (dl) {$\Delta$};
        \node at (-1,6) (x) {${\mathbf x}$};
        \node at (0,6) (alg) {${\underline b}$};
        \node at (-2,6) (lalg) {${\underline a}$};
        \draw[tensorblgarrow] (alg) to (d);
        \draw[tensorblgarrow, bend left=15] (d) to (m2);
        \draw[tensorblgarrow, bend left=15] (d) to (m1);
        \draw[tensoralgarrow] (lalg) to (dl);
        \draw[tensoralgarrow, bend right=15] (dl) to (m2);
        \draw[tensoralgarrow, bend right=15] (dl) to (m1);
        \draw[modarrow] (x) to (m1);
        \draw[modarrow] (m1) to (m2);
        \draw[modarrow](m2) to (out);
      \end{tikzpicture}=0
\]

As a basic example, if $\Alg$ is a DG algebra, we can view it as a
bimodule over itself; in this case, we write
$\lsub{\Alg}\Alg_{\Alg}$. The operation $m_{0|1|0}$ is the
differential on the algebra, $m_{1|1|0}(a\otimes b)= a \cdot b$, and
$m_{0|1|1}(b\otimes c)=b\cdot c$, for any $a,b,c\in \Alg$ (except now
$b$ is viewed as an element in the bimodule). All other operations
vanish.

Bimodules have opposites, defined by the straightforward generalization of 
Equation~\eqref{eq:OppositeTypeA}.
For example, the 
bimodule $\lsub{\Alg}{\overline A}_{\Alg}$, the opposite bimodule of
$\lsub{\Alg}{\Alg}_{\Alg}$, consists of 
maps from $\Alg\to\Field$; more precisely given $(d,\ell)\in\Z\oplus \Lambda$,
\begin{equation}
  \label{eq:OppositeBimodule}
  {\overline{\Alg}}_{(d,\ell)} = \Hom_{\Field}(\Alg_{(-d,-\ell)},\Field).
\end{equation}
This has operations
\begin{align*}
  m_{0|1|0}(x\mapsto \psi(x))&= (x\mapsto \psi(dx)) \\
  m_{1|1|0}(a\otimes (x\mapsto \psi(x))) &= (x\mapsto \psi(x\cdot a))  \\
  m_{0|1|1}((x\mapsto \psi(x))\otimes b) &= (x\mapsto \psi(b\cdot x)) 
\end{align*}

A right/right type $AA$ bimodule over
$\Alg$ and $\Blg$ is a  left/right bimodules
over $\Alg^{op}$ (the ``opposite algebra'') and $\Blg$.

A {\em morphism} between type $DA$ bimodules $h^1\colon \lsup{\Alg}X_{\Blg}\to \lsup{\Alg}Y_{\Blg}$ is a sequence of maps
$\{h^1_j\colon X \otimes_{\ground} \overbrace{\Blg\otimes_{\ground}\dots\otimes_{\ground}\Blg}^{j-1} \to \Alg\otimes_{\groundj} Y\}_{j=1}^{\infty}$,
abbreviated
\[\mathcenter{\begin{tikzpicture}[scale=.8]
  \node at (0,1) (inD) {};
  \node at (0,0) (h) {$h^1$};
  \node at (-1,-1) (outA) {};
  \node at (1,1) (inB) {};
  \node at (0,-1) (outD) {};
  \draw[modarrow] (inD) to (h); 
  \draw[tensoralgarrow] (inB) to (h);
  \draw[algarrow] (h) to (outA);
  \draw[modarrow] (h) to (outD);
\end{tikzpicture}}\]
The differential of $h^1$ is the morphism $dh^1$ represented as the sum
\[\mathcenter{\begin{tikzpicture}[scale=.8]
  \node at (0,1) (inD) {};
  \node at (0,0) (d) {$h^1$};
  \node at (-1,-1) (mu) {$\mu_1$};
  \node at (-2,-2) (outA) {};
  \node at (1,1) (inB) {};
  \node at (0,-2) (outD) {};
  \draw[modarrow] (inD) to (d); 
  \draw[tensoralgarrow] (inB) to (d);
  \draw[algarrow] (d) to (mu);
  \draw[modarrow] (d) to (outD);
  \draw[algarrow] (mu) to (outA);
\end{tikzpicture}
}+
\mathcenter{\begin{tikzpicture}[scale=.8]
  \node at (0,2) (inD) {};
  \node at (0,0) (d) {$h^1$};
  \node at (1,1) (mu) {${\underline d}$};
  \node at (-1,-1) (outA) {};
  \node at (2,2) (inB) {};
  \node at (0,-1) (outD) {};
  \draw[modarrow] (inD) to (d); 
  \draw[tensoralgarrow] (inB) to (mu);
  \draw[tensoralgarrow] (mu) to (d);
  \draw[algarrow] (d) to (outA);
  \draw[modarrow] (d) to (outD);
\end{tikzpicture}}+
\mathcenter{
\begin{tikzpicture}[scale=.8]
  \node at (0,2) (inD) {};
  \node at (0,0) (d1) {$\delta^1$};
  \node at (0,-1) (d2) {$h^1$};
  \node at (-1,-2) (mu) {$\mu_2$};
  \node at (2,2) (inB) {};
  \node at (1,1) (Delta) {$\Delta$};
  \node at (0,-3) (outD) {};
  \node at (-2,-3) (outA) {};
  \draw[modarrow] (inD) to (d1); 
  \draw[tensoralgarrow] (inB) to (Delta);
  \draw[tensoralgarrow] (Delta) to (d1);
  \draw[tensoralgarrow] (Delta) to (d2);
  \draw[algarrow] (d1) to (mu);
  \draw[modarrow] (d1) to (d2);
  \draw[algarrow] (d2) to (mu);
  \draw[algarrow] (mu) to (outA);
  \draw[modarrow] (d2) to (outD);
\end{tikzpicture}}
+
\mathcenter{
\begin{tikzpicture}[scale=.8]
  \node at (0,2) (inD) {};
  \node at (0,0) (d1) {$h^1$};
  \node at (0,-1) (d2) {$\delta^1$};
  \node at (-1,-2) (mu) {$\mu_2$};
  \node at (2,2) (inB) {};
  \node at (1,1) (Delta) {$\Delta$};
  \node at (0,-3) (outD) {};
  \node at (-2,-3) (outA) {};
  \draw[modarrow] (inD) to (d1); 
  \draw[tensoralgarrow] (inB) to (Delta);
  \draw[tensoralgarrow] (Delta) to (d1);
  \draw[tensoralgarrow] (Delta) to (d2);
  \draw[algarrow] (d1) to (mu);
  \draw[modarrow] (d1) to (d2);
  \draw[algarrow] (d2) to (mu);
  \draw[algarrow] (mu) to (outA);
  \draw[modarrow] (d2) to (outD);
\end{tikzpicture}}
\]
A {\em homomorphism} is a morphism whose differential is zero. Two homomorphisms are {\em homotopic}
if their difference is the differential of another morphism.

Morphisms can be composed; given morphisms 
$f^1\colon \lsup{\Alg}X_{\Blg}\to \lsup{\Alg}Y_{\Blg}$ and
$g^1\colon \lsup{\Alg}Y_{\Blg}\to \lsup{\Alg}Z_{\Blg}$, the composition is defined by the picture
\[\mathcenter{
\begin{tikzpicture}[scale=.8]
  \node at (0,2) (inD) {};
  \node at (0,0.5) (d1) {$f^1$};
  \node at (0,-.7) (d2) {$g^1$};
  \node at (-1.2,-1) (mu) {$\mu_2$};
  \node at (2.4,2) (inB) {};
  \node at (1.2,1) (Delta) {$\Delta$};
  \node at (0,-1.5) (outD) {};
  \node at (-2,-1.5) (outA) {};
  \draw[modarrow] (inD) to (d1); 
  \draw[tensoralgarrow] (inB) to (Delta);
  \draw[tensoralgarrow] (Delta) to (d1);
  \draw[tensoralgarrow] (Delta) to (d2);
  \draw[algarrow] (d1) to (mu);
  \draw[modarrow] (d1) to (d2);
  \draw[algarrow] (d2) to (mu);
  \draw[algarrow] (mu) to (outA);
  \draw[modarrow] (d2) to (outD);
\end{tikzpicture}}
\]

Morphisms, homomorphisms, and homotopies can be defined for bimodules
of other types in a straightforward way;
see~\cite[Section~2.2.4]{Bimodules}.

\subsection{Tensor products of bimodules}

We recall here the tensor products of various types of bimodules; see~\cite[Section~2.3.2]{Bimodules} for details.  Let $\Alg$, $\Blg$, $\Clg$ be
three differential graded algebras over base rings $\groundi$, $\groundj$, and $\ground$ respectively, 
and and let $\lsup{\Alg} X_{\Blg}$ and
$\lsup{\Blg} Y_{\Clg}$ be type $DA$ bimodules, their tensor product $X\DT Y$
is
a type $DA$ bimodule structure on the vector space
$X\otimes Y$ (where the tensor product is taken over the ground ring of $\Blg$), with structure maps $\delta^1\colon X\otimes_{\groundj} Y \to \Alg\otimes_{\groundi} X \otimes_{\groundj} Y$ which can
be represented as
\[\mathcenter{\begin{tikzpicture}
  \node at (0,1) (inX) {};
  \node at (2,1) (inC) {};
  \node at (1,1) (inY) {};
  \node at (0,-2) (outX) {};
  \node at (1,-2) (outY) {};
  \node at (-1,-2) (outA) {};
  \node at (1,0) (diffY) {$\delta_Y$};
  \node at (0,-1) (diffX) {$\delta^1_X$};
  \draw[modarrow] (inX) to (diffX);
  \draw[modarrow] (diffX) to (outX);
  \draw[modarrow] (inY) to (diffY);
  \draw[modarrow] (diffY) to (outY);
  \draw[tensoralgarrow] (diffY) to (diffX);
  \draw[tensoralgarrow] (inC) to (diffY);
  \draw[algarrow] (diffX) to (outA);
\end{tikzpicture}}\]
Here the map $\delta_Y$ is obtained by iterating $\delta^1_Y$ (as in Equation~\eqref{eq:DefIteratedDelta}).

The following is immediate from the definition: 

\begin{lemma}
  \label{lem:AssociateDA}
  Let $\Alg$, $\Blg$, $\Clg$, and $\Dlg$ be four differential graded algebras,
  and fix type $DA$ bimodules $\lsup{\Alg} X_{\Blg}$ and $\lsup{\Blg} Y_{\Clg}$
  and $\lsup{\Clg} Z_{\Dlg}$. Then, there is an isomorphism
  \[ (\lsup{\Alg} X_{\Blg}\DT~ \lsup{\Blg}Y_{\Clg})\DT~ \lsup{\Clg}Z_{\Dlg}
  \cong 
  \lsup{\Alg}\!\!X_{\Blg}\DT (\lsup{\Blg} Y_{\Clg}\DT~ \lsup{\Clg}Z_{\Dlg}).
  \]
\end{lemma} 

% \begin{proof}
%  Associating in either order, 
%  the map $\delta^1\colon X \otimes Y \otimes Z\to \Alg\otimes X \otimes Y \otimes Z$ has the description
%  \[\mathcenter{\begin{tikzpicture}
%      \node at (0,2) (inX) {};
%      \node at (3,2) (inC) {};
%      \node at (1,2) (inY) {};
%      \node at (2,2) (inZ) {};
%      \node at (0,-2) (outX) {};
%      \node at (1,-2) (outY) {};
%      \node at (2,-2) (outZ) {};
%      \node at (-1,-2) (outA) {};
%      \node at (2,1) (d3) {$\delta_Z$};
%      \node at (1,0) (d2) {$\delta_Y$};
%      \node at (0,-1) (d1) {$\delta^1_X$};
%      \draw[tensoralgarrow] (inC) to (d3);
%      \draw[modarrow] (inZ) to (d3);
%      \draw[modarrow] (d3) to (outZ);
%      \draw[modarrow] (inY) to (d2);
%      \draw[modarrow] (d2) to (outY);
%      \draw[modarrow] (inX) to (d1);
%      \draw[modarrow] (d1) to (outX);
%      \draw[tensoralgarrow] (d3) to (d2);
%      \draw[tensoralgarrow] (d2) to (d1);
%      \draw[algarrow] (d1) to (outA);
%    \end{tikzpicture}}\]
%\end{proof}

Given a type $DA$ bimodule  $\lsup{\Alg}X_{\Blg}$  and
a type $DD$ bimodule
$\lsup{\Blg}Y^{\Clg}$,
their tensor product $\lsup{\Alg}X_{\Blg}\DT \lsup{\Blg}Y^{\Clg}$,
when it makes sense, is a type $DD$ bimodule over $\Alg$ and $\Clg$. 
The type $DD$ structure map is described by
\begin{equation}
  \label{eq:DefDD}
  \mathcenter{\begin{tikzpicture}
  \node at (0,1) (inX) {};
  \node at (2,-1) (mult) {$\Pi$};
  \node at (2,-2) (outC) {};
  \node at (1,1) (inY) {};
  \node at (0,-2) (outX) {};
  \node at (1,-2) (outY) {};
  \node at (-1,-2) (outA) {};
  \node at (1,0) (diffY) {$\delta_Y$};
  \node at (0,-1) (diffX) {$\delta^1_X$};
  \draw[modarrow] (inX) to (diffX);
  \draw[modarrow] (diffX) to (outX);
  \draw[modarrow] (inY) to (diffY);
  \draw[modarrow] (diffY) to (outY);
  \draw[tensoralgarrow] (diffY) to (diffX);
  \draw[tensoralgarrow] (diffY) to (mult);
  \draw[algarrow] (mult) to (outC);
  \draw[algarrow] (diffX) to (outA);
\end{tikzpicture}}
\end{equation}
where $\Pi(b_1\otimes \dots\otimes b_j)=b_1\cdots b_j$.

Of course, the sum implicit in the above description is not always
finite; we describe a case where it is. (See
also~\cite[Section~2.2.4]{Bimodules}.)
 Consider the map
\[ \delta^j_Y\colon Y \to (\Blg^{\otimes j})\otimes Y\otimes (\Clg^{\otimes j})\] 
obtained by iterating $\delta^1$ 
(i.e. so that the map $\delta_Y$ appearing in Equation~\eqref{eq:DefDD}
is given as $\delta_Y=\sum_{j=0}^{\infty} \delta^j_Y$).

\begin{defn}
  For
fixed integer $j\geq 1$, {\em a length $j$ $\Blg$-sequence out of $\y\in X$}
is any sequence of algebra elements $(b_1,\dots,b_j)$ in $\Blg$ with
the property that, for a suitable choice of $\z\in X$ and sequence
$(c_1,\dots,c_j)$ in $\Clg$, $(b_1\otimes\dots\otimes b_j)\otimes
\z\otimes (c_1\otimes\dots\otimes c_j)$ appears with non-zero
multiplicity in $\delta^j(\y)$. 
\end{defn}

\begin{defn}
  \label{def:Ycompatible}
  Fix DG algebras $\Alg$, $\Blg$, and $\Clg$, and bimodules
  $\lsup{\Alg}X_{\Blg}$ and $\lsup{\Blg}Y^{\Clg}$. We say that $X$ is
  {\em {$Y$-compatible over $\Blg$}}, or when it is unambiguous,
  simply {\em{$Y$-compatible}} if for any non-zero $\x\otimes \y\in
  X\otimes Y$, if $j$ is a sufficiently large integer, then for any
  length $j$ $\Blg$-sequence out of $\y$ $(b_1,\dots,b_j)$, we have
  that $\delta^1_{j+1}(\x,b_1,\dots b_j)=0$.  Similarly, a morphism
  $\phi\in\Mor(X,X')$ is called {\em $Y$-compatible} if for any
  $\x\otimes\y\in X\otimes Y$,  if $j$ is sufficiently large,
  then for any length $j$ $\Blg$-sequence out of $\y$
  $(b_1,\dots,b_j)$, we have that $\phi^1_{j+1}(\x,b_1,\dots,b_j)=0$.
  If $X$ and $X'$ are $Y$-compatible, we say that they are {\em
    {$Y$-compatibly homotopy equivalent}} if there are $Y$-compatible
  morphisms $\phi\colon X\to X'$, $\psi\colon X'\to X$, $h\colon X\to
  X$, and $h'\colon X'\to X'$, so that $d \phi=0$, $d\psi=0$, $\psi
  \circ \phi = \Id_{X} + d h$ and $\phi\circ \psi=\Id_{X'} + dh'$.
\end{defn}

\begin{prop}
  Fix $\lsup{\Alg}X_{\Blg}$ and $\lsup{\Blg}Y^{\Clg}$. If $X$ is
  $Y$-compatible, we can form the type $DD$ bimodule $X\DT Y$, as
  defined in Equation~\eqref{eq:DefDD}.  Moreover, if $\lsup{\Alg}X'_{\Blg}$
  is also $Y$-compatible, and it is $Y$-compatibly homotopy-equivalent to 
  $X$, then $X\DT Y$ and $X'\DT Y$ are homotopy equivalent type $DD$ bimodules.
\end{prop}

\begin{proof}
  It is straightforward to see that the 
  $Y$-compatibility ensures that all the infinite sums appearing in the needed maps
  are all finite.
\end{proof}

For a very simple special case, suppose that $\lsup{\Alg}X_{\Blg}$ has
the property that for all sufficiently large $j$, $\delta^1_j=0$; then $X$
is $Y$-compatible for all $\lsup{\Blg}Y^{\Clg}$.

\begin{lemma}
  \label{lem:BarResolution}
  Let $\Alg$ and $\Blg$ be DG modules, and let $\lsup{\Alg}M_{\Blg}$ be
  a type $DA$ bimodule with the property that $\delta^1_j=0$ for all sufficiently large $j$.
  Then, $M$ is homotopy equivalent to a type $DA$ bimodule with $\delta^1_j=0$ for $j>2$.
\end{lemma}

\begin{proof}
  Let $M'$ be the bar resolution of $M$; see~\cite{Keller} or~\cite{Bimodules},
 $\lsup{\Alg}M'_{\Blg}=
  \lsup{\Alg}M_{\Blg}\DT \lsupv{\Blg}\Barop\!\lsup{\Blg}\DT_{\Blg}{\Blg}_{\Blg}$.
  The desired homotopy equivalence is obtained from the homotopy equivalence
  $\lsup{\Blg}\Id_{\Blg}\to \lsupv{\Blg}\Barop\lsup{\Blg}\DT_{\Blg}{\Blg}_{\Blg}$
  by tensor product with the identity map on $M$. 
  (Boundedness of $M'$ ensures that this is a homotopy equivalence; compare~\cite[Lemma~2.3.19]{Bimodules}.)
\end{proof}

We have the following version of associativity (compare~\cite[Proposition~2.3.15]{Bimodules}):

\begin{lemma}
  \label{lem:AssociateDD}
  Let $\Alg$, $\Blg$, $\Clg$, and $\Dlg$ be four differential graded algebras;
  $\lsup{\Alg}X_{\Blg}$ and $\lsub{\Clg}Z^{\Dlg}$
  are bimodules of type $DA$; and
  $\lsup{\Blg}Y^{\Clg}$ is of type $DD$. Suppose 
  that $\lsup{\Alg}X_{\Blg}$ and $\lsub{\Clg}Z^{\Dlg}$ are both
  bounded, in the sense that for sufficiently large $j$, $\delta^1_j=0$. Then,
  \[ (\lsup{\Alg} X_{\Blg}\DT~ \lsup{\Blg}Y^{\Clg})\DT~ \lsub{\Clg}Z^{\Dlg}
  \simeq 
  \lsup{\Alg} X_{\Blg}\DT~(\lsup{\Blg}Y^{\Clg}\DT~ \lsub{\Clg}Z^{\Dlg}).
  \]
\end{lemma}

\begin{proof}
  This is clear if $\delta^1_j=0$ for all $j>2$ on
  $\lsup{\Alg}X_{\Blg}$. We can reduce to this case by
  Lemma~\ref{lem:BarResolution}.
\end{proof}

\subsection{Taking homology of bimodules}

By the ``homological perturbation lemma'', $\Ainfty$ structures
persist after taking homology.  (See for example~\cite{Keller} for the
case of algebras.) We will make use of the analogous result for bimodules.

We start with a standard result from homological algebra:

\begin{lemma}
  \label{lem:TakeHomology}
  Let $\ground\cong \Field^k$ for some $k$.
  Let $Y$ be an $\MGradingSet$-graded chain complex over $\ground$, and $Z=H(Y)$. Then, $Z$ is also an $\MGradingSet$-graded chain complex
  (with trivial differential), and there are $\MGradingSet$-graded chain maps $f\colon Z \to Y$ and $g\colon Y\to Z$ and an 
  $\MGradingSet$-graded map
  $T\colon Y\to Y$ satisfying the identities
  \begin{align*}
    T\circ T = 0 \qquad &T\circ f=0 \qquad g\circ T=0  \\
    g \circ f = \Id_Z \qquad  &f\circ g = \Id_Y + \partial\circ T + T\circ \partial.
    \end{align*}
\end{lemma}

We will use two versions of the homological perturbation lemma:

\begin{lemma}
  \label{lem:HomologicalPerturbation}
  Let $\lsub{\Alg}Y_{\Blg}$ be a strictly unital $\Ainfty$ bimodule with grading set
  $\MGradingSet$, let $Z$ denote its homology. Let $f\colon Z\to Y$ be a
  homotopy equivalence of $\MGradingSet$-graded complexes of $\ground$-modules
  as in Lemma~\ref{lem:TakeHomology}. Then there is an $\Ainfty$
  bimodule structure on $Z$, denoted $\lsub{\Alg}Z_{\Blg}$, and an
  $\Ainfty$ homotopy equivalence $\phi\colon \lsub{\Alg}Z_{\Blg}\to
  \lsub{\Alg}Y_{\Blg}$ with $\phi_{0|1|0}=f$. 
\end{lemma}

\begin{proof}
  By hypothesis, we have maps $f\colon Z \to Y$ and $g\colon Y \to Z$
  so that $f\circ g = \Id + \partial \circ T + T \circ \partial$.
  The differential on $Z$ vanishes; i.e.  $m_{0|1|0}=0$.
  Operations $m_{i|1|j}$ with $i+j>0$ are described by
  \[    m^Z(\underline{a}\otimes \x\otimes{\underline{b}})=
    \begin{tikzpicture}[scale=.8,baseline=(x.base)]
        \node at (-1,2) (terminal) {};
        \node at (-1,3) (g) {$g$};
        \node at (-1,4) (m1) {$m$};
        \node at (-1,5) (f) {$f$};
%        \node at (0,5) (mu) {$\Delta$};
%        \node at (-2,5) (lmu) {$\Delta$};
        \node at (-1,6) (x) {${\mathbf x}$};
        \node at (0,6) (alg) {${\underline b}$};
        \node at (-2,6) (lalg) {${\underline a}$};
        \draw[modarrow] (g) to (terminal);
        \draw[othmodarrow] (m1) to (g);
        \draw[modarrow] (x) to (f);
        \draw[othmodarrow] (f) to (m1);
%        \draw[tensorblgarrow] (alg) to (mu);
        \draw[tensorblgarrow, bend left=15] (alg) to (m1);
%        \draw[tensoralgarrow] (lalg) to (lmu);
        \draw[tensoralgarrow, bend right=15] (lalg) to (m1);
      \end{tikzpicture}
      +
    \begin{tikzpicture}[scale=.8,baseline=(x.base)]
        \node at (-1,1) (terminal) {};
        \node at (-1,2) (g) {$g$};
        \node at (-1,3) (m2) {$m$};
        \node at (-1,4) (T1) {$T$};
        \node at (-1,5) (m1) {$m$};
        \node at (-1,6) (f) {$f$};
        \node at (-2,6) (lmu) {$\Delta$};
        \node at (0,6) (mu) {$\Delta$};
        \node at (-1,7) (x) {${\mathbf x}$};
        \node at (0,7) (alg) {${\underline b}$};
        \node at (-2,7) (lalg) {${\underline a}$};
        \draw[othmodarrow] (m2) to (g);
        \draw[modarrow] (g) to (terminal);
        \draw[othmodarrow] (T1) to (m2);
        \draw[othmodarrow] (m1) to (T1);
        \draw[modarrow] (x) to (f);
        \draw[othmodarrow] (f) to (m1);
        \draw[tensorblgarrow] (alg) to (mu);
        \draw[tensorblgarrow, bend left=15] (mu) to (m2);
        \draw[tensorblgarrow, bend left=15] (mu) to (m1);
        \draw[tensoralgarrow] (lalg) to (lmu);
        \draw[tensoralgarrow, bend right=15] (lmu) to (m2);
        \draw[tensoralgarrow, bend right=15] (lmu) to (m1);
      \end{tikzpicture}
      +
          \begin{tikzpicture}[scale=.8,baseline=(x.base)]
        \node at (-1,-1) (terminal) {};
        \node at (-1,0) (g) {$g$};
        \node at (-1,1) (m3) {$m$};
        \node at (-1,2) (T2) {$T$};
        \node at (-1,3) (m2) {$m$};
        \node at (-1,4) (T1) {$T$};
        \node at (-1,5) (m1) {$m$};
        \node at (-1,6) (f) {$f$};
        \node at (-2,6) (lmu) {$\Delta$};
        \node at (0,6) (mu) {$\Delta$};
        \node at (-1,7) (x) {${\mathbf x}$};
        \node at (0,7) (alg) {${\underline b}$};
        \node at (-2,7) (lalg) {${\underline a}$};
        \draw[othmodarrow] (m3) to (g);
        \draw[modarrow] (g) to (terminal);
        \draw[othmodarrow] (T2) to (m3);
        \draw[othmodarrow] (m2) to (T2);
        \draw[othmodarrow] (T1) to (m2);
        \draw[othmodarrow] (m1) to (T1);
        \draw[modarrow] (x) to (f);
        \draw[othmodarrow] (f) to (m1);
        \draw[tensorblgarrow] (alg) to (mu);
        \draw[tensorblgarrow, bend left=15] (mu) to (m3);
        \draw[tensorblgarrow, bend left=15] (mu) to (m2);
        \draw[tensorblgarrow, bend left=15] (mu) to (m1);
        \draw[tensorblgarrow, bend left=15] (mu) to (m3);
        \draw[tensoralgarrow] (lalg) to (lmu);
        \draw[tensoralgarrow, bend right=15] (lmu) to (m3);
        \draw[tensoralgarrow, bend right=15] (lmu) to (m2);
        \draw[tensoralgarrow, bend right=15] (lmu) to  (m1);
      \end{tikzpicture}
      +\;
      \cdots \]
      where  $a_1\otimes\dots\otimes a_i\in {\underline a}\in\Tensor^*(\Alg)$ and 
      $b_1\otimes\dots\otimes b_j={\underline b}\in\Tensor^*(\Blg)$;
      and each node laebelled by $m$ contains an operation
      $m_{a|1|b}$ with $a+b>0$;
      where the splitting operator $\Delta$ is generalized to have arbitrarily many outputs.
      With those operations, we have an $\Ainfty$-bimodule homomorphism
      defined by 
      $\phi_{0|1|0}=f$ and $\phi_{i|1|j}$ with $i+j>0$ specified by
    \[\phi(\underline{a}\otimes\x\otimes{\underline{b}})=
%    \begin{tikzpicture}[scale=.8,baseline=(x.base)]
%        \node at (-1,2) (terminal) {};
%        \node at (-1,3) (f) {$f$};
%        \node at (-1,4) (x) {${\mathbf x}$};
%        % 
%        \draw[modarrow] (x) to (f);
%        \draw[algarrow] (f) to (terminal);
%      \end{tikzpicture}
%    +
    \begin{tikzpicture}[scale=.8,baseline=(x.base)]
        \node at (-1,2) (terminal) {};
        \node at (-1,3) (T1) {$T$};
        \node at (-1,4) (m1) {$m$};
        \node at (-1,5) (f) {$f$};
        \node at (-1,6) (x) {${\mathbf x}$};
        \node at (0,6) (blg) {${\underline b}$};
        \node at (-2,6) (alg) {${\underline a}$};
        \draw[algarrow] (T1) to (terminal);
        \draw[algarrow] (m1) to (T1);
        \draw[modarrow] (x) to (f);
        \draw[algarrow] (f) to (m1);
        \draw[tensorblgarrow, bend left=15] (blg) to (m1); 
        \draw[tensoralgarrow, bend right=15] (alg) to (m1); 
      \end{tikzpicture}
    +
    \begin{tikzpicture}[scale=.8,baseline=(x.base)]
        \node at (-1,1) (terminal) {};
        \node at (-1,2) (T2) {$T$};
        \node at (-1,3) (m2) {$m$};
        \node at (-1,4) (T1) {$T$};
        \node at (-1,5) (m1) {$m$};
        \node at (-1,6) (f) {$f$};
        \node at (0,6) (mu) {$\Delta$};
        \node at (-2,6) (lmu) {$\Delta$};
        \node at (-1,7) (x) {${\mathbf x}$};
        \node at (0,7) (blg) {${\underline b}$};
        \node at (-2,7) (alg) {${\underline a}$};
        \draw[algarrow] (m2) to (T2);
        \draw[algarrow] (T2) to (terminal);
        \draw[algarrow] (T1) to (m2);
        \draw[algarrow] (m1) to (T1);
        \draw[modarrow] (x) to (f);
        \draw[algarrow] (f) to (m1);
        \draw[tensorblgarrow] (blg) to (mu);
        \draw[tensorblgarrow, bend left=15] (mu) to (m2);
        \draw[tensorblgarrow, bend left=15] (mu) to (m1);
        \draw[tensoralgarrow] (alg) to (lmu);
        \draw[tensoralgarrow, bend right=15] (lmu) to (m1);
        \draw[tensoralgarrow, bend right=15] (lmu) to (m2);
      \end{tikzpicture}
    + \;\cdots
    \]
    (again where each $m$-labelled node must have at least two
    inputs).  With these definitions, it is straightforward to verify
    that $\phi$ is a homotopy equivalence of $\Ainfty$ bimodules;
    compare~\cite{Keller}. Strict unitality of both $m^Z$ and $\phi$ 
    follows immediately from these explicit
    descriptions of $m^Z$ and $\phi$, the conditions that $T^2=f\circ
    T=T\circ g=0$, and the strict unitality of $Y$.
\end{proof}

We will also use a variant for $DA$ bimodules:

\begin{lemma}
  \label{lem:HomologicalPerturbation2}
  Let $\lsup{\Alg}Y_{\Blg}$ be a strictly unital type $DA$ bimodule
  with grading set
  $\MGradingSet$, and let $\lsup{\Alg}Z$ be type $D$ structure over $\Alg$. 
  Suppose that there are type $D$ structure homomorphisms $f\colon \lsup{\Alg}Z\to \lsup{\Alg}Y$ 
  (i.e. as the notation suggests, we are forgetting here about the right $\Blg$-action)
  and $g\colon \lsup{\Alg}Y\to \lsup{\Alg}Z$
  and a type $D$ structure morphism $T\colon \lsup{\Alg}Y \to~ \lsup{\Alg}\!Y$ so that
  \[\begin{array}{lll}
    f\circ g = \Id_{Z}, & 
    g\circ f= \Id_{Y} + d T, &
    T\circ T =0.
    \end{array}\]
    (Here, $T\circ T$ denotes the composite of type $D$ structures; see Equation~\eqref{eq:ComposeD}.)
    Then $\lsup{\Alg}Z$ can be turned into
  a strictly unital type $DA$ bimodule, denoted $\lsup{\Alg}Z_{\Blg}$; and there is an 
  $\Ainfty$ homotopy equivalence 
  $\phi\colon \lsup{\Alg}Z_{\Blg}\to  \lsup{\Alg}Y_{\Blg}$ with $\phi^1_{1}=f$. 
\end{lemma}

\begin{proof}
  $\lsup{\Alg}Z$ is already equipped with an action $\delta^1_1$. For $j>1$, operations
  $\delta^1_j$ on ${\underline b}=b_1\otimes \dots \otimes b_{j-1}$ are 
  specified by in Figure~\ref{fig:Delta1}
\begin{figure}
\[    \delta^1(\x\otimes{\underline{b}})=
    \begin{tikzpicture}[scale=.8,baseline=(x.base)]
        \node at (-1,1) (terminal) {};
        \node at (-1,3) (g) {$g$};
        \node at (-1,4) (m1) {$m$};
        \node at (-1,5) (f) {$f$};
%        \node at (0,5) (mu) {$\Delta$};
        \node at (-1,6) (x) {${\mathbf x}$};
        \node at (0,6) (alg) {${\underline b}$};
        \node at (-2,2) (pi) {$\Pi$};
        \node at (-2,1) (aout) {};
        \draw[modarrow] (g) to (terminal);
        \draw[othmodarrow] (m1) to (g);
        \draw[modarrow] (x) to (f);
        \draw[othmodarrow] (f) to (m1);
%        \draw[tensorblgarrow] (alg) to (mu);
        \draw[tensorblgarrow, bend left=15] (alg) to (m1);
        \draw[algarrow, bend right=15] (f) to (pi);
        \draw[algarrow, bend right=15] (m1) to (pi);
        \draw[algarrow, bend right=15] (g) to (pi);
        \draw[algarrow] (pi) to (aout);
%        \draw[tensoralgarrow, bend right=15] (lmu) to (m1);
      \end{tikzpicture}
      +
    \begin{tikzpicture}[scale=.8,baseline=(x.base)]
        \node at (-1,0) (terminal) {};
        \node at (-1,2) (g) {$g$};
        \node at (-1,3) (m2) {$m$};
        \node at (-1,4) (T1) {$T$};
        \node at (-1,5) (m1) {$m$};
        \node at (-1,6) (f) {$f$};
        \node at (0,6) (mu) {$\Delta$};
        \node at (-1,7) (x) {${\mathbf x}$};
        \node at (0,7) (alg) {${\underline b}$};
        \node at (-2,0) (aout) {};
        \node at (-2,1) (pi) {$\Pi$};
        \draw[othmodarrow] (m2) to (g);
        \draw[modarrow] (g) to (terminal);
        \draw[othmodarrow] (T1) to (m2);
        \draw[othmodarrow] (m1) to (T1);
        \draw[modarrow] (x) to (f);
        \draw[othmodarrow] (f) to (m1);
        \draw[tensorblgarrow] (alg) to (mu);
        \draw[tensorblgarrow, bend left=15] (mu) to (m2);
        \draw[tensorblgarrow, bend left=15] (mu) to (m1);
        \draw[algarrow, bend right=15] (m2) to (pi);
        \draw[algarrow, bend right=15] (T1) to (pi);
        \draw[algarrow, bend right=15] (m1) to (pi);
        \draw[algarrow, bend right=15] (f) to (pi);
        \draw[algarrow, bend right=15] (g) to (pi);
        \draw[algarrow] (pi) to (aout);
      \end{tikzpicture}
      +
          \begin{tikzpicture}[scale=.8,baseline=(x.base)]
        \node at (-1,-1.5) (terminal) {};
        \node at (-1,0) (g) {$g$};
        \node at (-1,1) (m3) {$m$};
        \node at (-1,2) (T2) {$T$};
        \node at (-1,3) (m2) {$m$};
        \node at (-1,4) (T1) {$T$};
        \node at (-1,5) (m1) {$m$};
        \node at (-1,6) (f) {$f$};
        \node at (0,6) (mu) {$\Delta$};
        \node at (-1,7) (x) {${\mathbf x}$};
        \node at (0,7) (alg) {${\underline b}$};
        \node at (-2,-1.5) (aout) {};
        \node at (-2,-.5) (pi) {$\Pi$};
        \draw[othmodarrow] (m3) to (g);
        \draw[modarrow] (g) to (terminal);
        \draw[othmodarrow] (T2) to (m3);
        \draw[othmodarrow] (m2) to (T2);
        \draw[othmodarrow] (T1) to (m2);
        \draw[othmodarrow] (m1) to (T1);
        \draw[modarrow] (x) to (f);
        \draw[othmodarrow] (f) to (m1);
        \draw[tensorblgarrow] (alg) to (mu);
        \draw[tensorblgarrow, bend left=15] (mu) to (m3);
        \draw[tensorblgarrow, bend left=15] (mu) to (m2);
        \draw[tensorblgarrow, bend left=15] (mu) to (m1);
        \draw[tensorblgarrow, bend left=15] (mu) to (m3);
        \draw[algarrow, bend right=15] (m3) to (pi);
        \draw[algarrow, bend right=15] (T2) to (pi);
        \draw[algarrow, bend right=15] (m2) to (pi);
        \draw[algarrow, bend right=15] (T1) to (pi);
        \draw[algarrow, bend right=15] (m1) to (pi);
        \draw[algarrow, bend right=15] (f) to (pi);
        \draw[algarrow, bend right=15] (g) to (pi);
        \draw[algarrow] (pi) to (aout);
      \end{tikzpicture}
      +\;
      \cdots \]
      \caption{\label{fig:Delta1}{\bf $\delta^1$ action on $Z$}}
\end{figure}
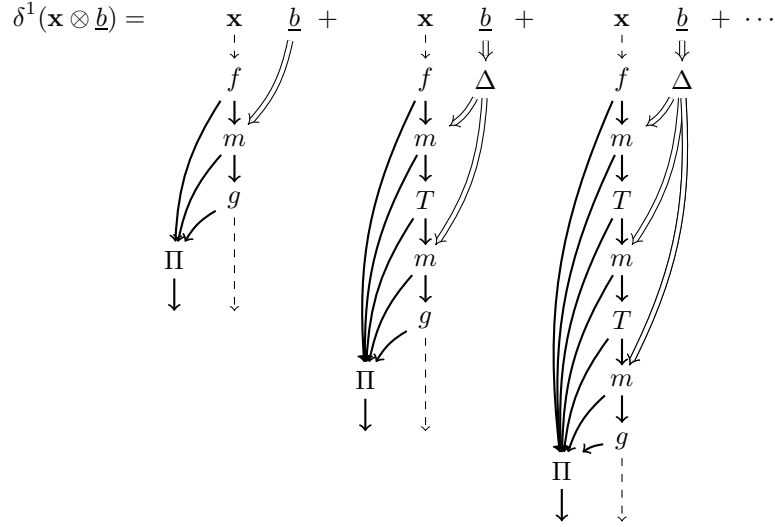
      similarly, define $\phi^1_1=f$, and for $j>1$,  $\phi^1$ 
      on ${\underline b}=b_1\otimes\dots\otimes b_{j-1}$ is as shown
      in Figure~\ref{fig:phi1}
\begin{figure}
   \[   \phi^1(\x\otimes{\underline{b}})=
    \begin{tikzpicture}[scale=.8,baseline=(x.base)]
        \node at (-1,2) (terminal) {};
%        \node at (-1,3) (m2) {$m$};
        \node at (-1,2) (g) {};
        \node at (-2,3) (pi) {$\Pi$};
        \node at (-1,4) (T1) {$T$};
        \node at (-1,5) (m1) {$m$};
        \node at (-1,6) (f) {$f$};
        \node at (-1,7) (x) {${\mathbf x}$};
        \node at (0,7) (alg) {${\underline b}$};
        \node at (-2,2) (aout) {};
        \draw[algarrow, bend right=15] (m1) to (pi);
        \draw[algarrow, bend right=15] (T1) to (pi);
        \draw[algarrow, bend right=15] (f) to (pi);
        \draw[othmodarrow] (T1) to (terminal);
        \draw[othmodarrow] (m1) to (T1);
        \draw[modarrow] (x) to (f);
        \draw[othmodarrow] (f) to (m1);
        \draw[tensorblgarrow, bend left=15] (alg) to (m1);
        \draw[algarrow] (pi) to (aout);
      \end{tikzpicture}
      +
          \begin{tikzpicture}[scale=.8,baseline=(x.base)]
        \node at (-1,-1) (terminal) {};
        \node at (-1,.5) (out) {};
        \node at (-1,2) (T2) {$T$};
        \node at (-1,3) (m2) {$m$};
        \node at (-1,4) (T1) {$T$};
        \node at (-1,5) (m1) {$m$};
        \node at (-1,6) (f) {$f$};
        \node at (0,6) (mu) {$\Delta$};
        \node at (-1,7) (x) {${\mathbf x}$};
        \node at (0,7) (alg) {${\underline b}$};
        \node at (-2,.5) (aout) {};
        \node at (-2,1.5) (pi) {$\Pi$};
        \draw[othmodarrow] (T2) to (out);
        \draw[othmodarrow] (m2) to (T2);
        \draw[othmodarrow] (T1) to (m2);
        \draw[othmodarrow] (m1) to (T1);
        \draw[modarrow] (x) to (f);
        \draw[othmodarrow] (f) to (m1);
        \draw[tensorblgarrow] (alg) to (mu);
        \draw[tensorblgarrow, bend left=15] (mu) to (m2);
        \draw[tensorblgarrow, bend left=15] (mu) to (m1);
        \draw[algarrow, bend right=15] (m2) to (pi);
        \draw[algarrow, bend right=15] (T2) to (pi);
        \draw[algarrow, bend right=15] (m1) to (pi);
        \draw[algarrow, bend right=15] (T1) to (pi);
        \draw[algarrow] (pi) to (aout);
      \end{tikzpicture}
      +\;
      \cdots \]
      \caption{\label{fig:phi1}{\bf $\phi^1$ morphism}}
\end{figure}
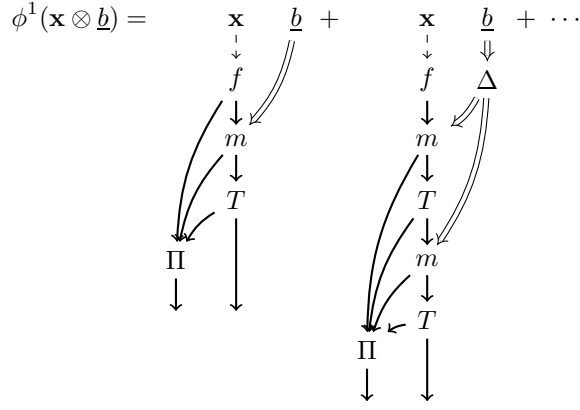
\end{proof}

\subsection{Koszul duality}
\label{subsec:Duality}

Let $\Alg$ and $\Blg$ be two differential graded algebras 
$\Lambda$.
Let $\lsup{\Alg} X^{\Blg}$ be a type $DD$ bimodule and
$\lsub{\Blg} Y_{\Alg}$ be a type $AA$ bimodule. We say that 
these two bimodules are {\em quasi-inverses} if there are 
homotopy equivalences
\[
  \lsup{\Alg} X^{\Blg} \DT~ \lsub{\Blg} Y_{\Alg} 
  \simeq \lsup{\Alg}\Id_{\Alg} \qquad{\text{and}}\qquad
  \lsub{\Blg} Y_{\Alg} \DT~\lsup{\Alg}\!\! X^{\Blg} 
  \simeq \lsub{\Blg}\Id^{\Blg};
\]
and $\lsup{\Alg} X^{\Blg}$
resp. $\lsub{\Blg}Y_{\Alg}$ as a {\em quasi-invertible} type $DD$
resp. type $AA$ bimodule.

To see that $X$ and $Y$ are quasi-inverses, it suffices to exhibit
homomorphisms
$\phi^1\colon\lsup{\Alg}{X^{\Blg}} \DT \lsub{\Blg}{Y_{\Alg}}\to\lsup{\Alg}\Id_{\Alg} $ and
$\psi^1\colon \lsub{\Blg} Y_{\Alg} \DT~\lsup{\Alg}{X^{\Blg}} \to
\lsub{\Blg}\Id^{\Blg}$
so that the maps 
\begin{align*}
  \Id_{\Alg}\DT \phi^1\colon
\lsub{\Alg}\Alg_{\Alg}\DT ~\lsup{\Alg} X^{\Blg} \DT~ \lsub{\Blg}
Y_{\Alg} & \to \lsub{\Alg}\Alg_{\Alg} \\
\psi^1\DT \Id_{\Blg}\colon
\lsub{\Blg}Y_{\Alg}\DT \lsup{\Alg}X^{\Blg}\DT \lsub{\Blg}\Blg_{\Blg}&\to
\lsub{\Blg}\Blg_{\Blg} 
\end{align*}
induce isomorphisms on homology, according to~\cite[Corollary~{2.4.4}]{Bimodules}.

\begin{defn}
  \label{def:KoszulDual}
  Let $\ground$ be a ring that is a finite direct sum of $\Field$.
  Let $\Alg$ and $\Blg$ be two DG algebras, both of which are positively graded over $\ground$
  (in the sense of Definition~\ref{def:PositiveAlexanderGrading}),
  with Alexander grading specified by the same Abelian group $\Lambda=\Lambda_{\Alg}=\Lambda_{\Blg}$.
  A bimodule
$\lsup{\Alg}X^{\Blg}$ is called a {\em Koszul dualizing bimodule}
if it satisfies the following properties:
\begin{enumerate}[label=(K-\arabic*),ref=(K-\arabic*)]
\item
  \label{Koszul:Grading}
  $\lsup{\Alg}X^{\Blg}$ is graded by $\Z\oplus \Lambda$, where
    the action of 
    $\Lambda_\Alg\oplus\Lambda_\Blg$ is specified by 
    the map $\Lambda_\Alg\oplus\Lambda_\Blg\to\Lambda$ given by $(a,b)\mapsto
  (b-a)$.
\item
  \label{Koszul:Rk1}
  $\lsup{\Alg}X^{\Blg}$ has rank one; i.e. there is an isomorphism
  $\lsup{\Alg}X^{\Blg}\cong \ground$
  as bigraded $\ground$-$\ground$ bimodules. (In particular, $\lsup{\Alg}X^{\Blg}$ is supported in bigrading $(0;0)$).
\item 
  \label{Kosz:Inv}
  $\lsup{\Alg}X^{\Blg}$ is quasi-invertible.
\end{enumerate}
If such a bimodule exists, $\Alg$ and $\Blg$ are called {\em Koszul dual} to one another.
\end{defn}

For a Koszul dualizing bimodule, it follows that
  \begin{equation}
    \label{Koszul:Pos}
    \delta^1(\lsup{\Alg}X^{\Blg})\subset \bigoplus_{\lambda\in\Lambda\setminus 0} \Alg_\lambda\otimes \lsup{\Alg}X^{\Blg} \otimes \Blg_\lambda.
  \end{equation}
  (Indeed, this conclusion holds even for $\lsup{\Alg}X^{\Blg}$ as in the definition that do not satisfy Property~\ref{Kosz:Inv}.)

\begin{defn}
  If $\lsup{\Alg}X^{\Blg}$ is a type $DD$ bimodule, we can form 
  the {\em candidate quasi-inverse module}
  \begin{equation}
    \label{eq:CandidateInverseModule}
    \lsup{\Alg}\Mor(\lsup{\Alg}X^{\Blg}\DT \lsub{\Blg}\Blg_{\Blg},\lsup{\Alg}\Id_{\Alg})
  \cong \lsub{\Blg}{\overline B}_{\Blg}\DT \lsup{\Blg}{\overline
    X}^{\Alg} \DT \lsub{\Alg}\Alg_{\Alg}.
  \end{equation}
\end{defn}

\begin{lemma}
  \label{lem:CanonicalSubcomplex}
  For a type $DD$ bimodule $\lsup{\Alg}X^{\Blg}$ satisfying properties~\ref{Koszul:Grading} and
  \ref{Koszul:Rk1}, the
  candidate quasi-inverse module has a subcomplex ${\overline \ground}\DT \lsup{\Blg}{\overline X}^{\Alg}\DT {\ground}$ which is isomorphic to
  $\ground$.
\end{lemma}

\begin{proof}
  For any $DD$ bimodule, ${\overline \ground}\DT \lsup{\Blg}{\overline
    X}^{\Alg}\DT \Alg$ is a subcomplex of the candidate quasi-inverse
  module.  
  Comditions~\ref{Koszul:Grading} and~\ref{Koszul:Rk1} ensures that 
  Equation~\eqref{Koszul:Pos} holds, so the induced
  differential on ${\overline \ground}\DT \lsup{\Blg}{\overline
    X}^{\Alg}\DT \Alg$ is simply the differential on $\Alg$; and therefore
  ${\overline \ground}\DT \lsup{\Blg}{\overline X}^{\Alg}\DT \ground$
  is a subcomplex. Condition~\ref{Koszul:Rk1} now gives the desired
  identification of the subcomplex with $\ground$.
\end{proof}

\begin{lemma}
  \label{lem:CandidateIsInverse}
  Fix $\ground$ as above in Definition~\ref{def:KoszulDual}, and let
  $\lsup{\Alg}X^{\Blg}$ be a type $DD$ bimodule satisfying
  Conditions~\ref{Koszul:Grading}-\ref{Koszul:Pos} of
  Definition~\ref{def:KoszulDual}. Suppose that the inclusion map
  from $\ground$ to the candidate quasi-inverse module coming from
  Lemma~\ref{lem:CanonicalSubcomplex} induces an isomorphism on
  homology. Then the candidate quasi-inverse module is a quasi-inverse
  of $\lsup{\Alg}X^{\Blg}$.
\end{lemma}

\begin{proof}
  Let $\lsub{\Blg}Y_{\Alg}$ be the candidate quasi-inverse.  Our goal
  is to show, under the stated hypotheses, that $\lsub{\Blg}Y_{\Alg}$
  is a quasi-inverse to $\lsup{\Alg}X^{\Blg}$. 

  To this end, consider a $DA$ bimodule by tensoring $X$ with $Y$:
  \[ \lsup{\Alg}X^{\Blg}\DT \lsub{\Blg}Y_{\Alg}
  =\lsup{\Alg}X^{\Blg}\DT \lsub{\Blg}{\overline B}_{\Blg}\DT \lsup{\Blg}{\overline
    X}^{\Alg} \DT \lsub{\Alg}\Alg_{\Alg}.\]
  The differential on this bimodule is given pictorially by:
  \[ 
  \mathcenter{\begin{tikzpicture}
      \node at (0,-1) (x) {$X$};
      \node at (1,-1) (bbar) {${\overline\Blg}$};
      \node at (1.5,-1) (xbar) {${\overline X}$};
      \node at (2,-1) (a) {$\Alg$};
      \node at (0,-5) (xend) {$~$};
      \node at (1,-5) (bbarend) {$~$};
      \node at (1.5,-5) (xbarend) {$~$};
      \node at (2,-5) (aend) {$~$};
      \node at (0,-2) (xd) {$\delta^1_{X}$};
      \node at (1,-4) (xb) {$m^{\overline B}_{1|1|0}$};
      \node at (-1,-5) (output) {$\Alg$};
      \draw[->](x) to (xd);
      \draw[->](xd) to (xb);
      \draw[->](xd) to (output);
      \draw[->](xd) to (xend);
      \draw[->](bbar) to (xb);
      \draw[->](xb) to (bbarend);
      \draw[->] (xbar) to (xbarend);
      \draw[->](a) to (aend);
    \end{tikzpicture}}+
  \mathcenter{\begin{tikzpicture}
      \node at (0.5,-1) (x) {$X$};
      \node at (1,-1) (bbar) {${\overline\Blg}$};
      \node at (1.5,-1) (xbar) {${\overline X}$};
      \node at (2,-1) (a) {$\Alg$};
      \node at (0.5,-5) (xend) {$~$};
      \node at (1,-5) (bbarend) {$~$};
      \node at (1.5,-5) (xbarend) {$~$};
      \node at (2,-5) (aend) {$~$};
      \node at (1,-3) (bd) {$m_{0|1|0}^{\overline B}$};
      \draw[->](x) to (xend);
      \draw[->](bbar) to (bd);
      \draw[->] (bd) to (bbarend);
      \draw[->] (xbar) to (xbarend);
      \draw[->](a) to (aend);
    \end{tikzpicture}}+
  \mathcenter{\begin{tikzpicture}
      \node at (0.5,-1) (x) {$X$};
      \node at (1,-1) (bbar) {${\overline\Blg}$};
      \node at (2,-1) (xbar) {${\overline X}$};
      \node at (3,-1) (a) {$\Alg$};
      \node at (0.5,-5) (xend) {$~$};
      \node at (1,-5) (bbarend) {$~$};
      \node at (2,-5) (xbarend) {$~$};
      \node at (3,-5) (aend) {$~$};
      \node at (2,-2) (xd) {$\delta^1_{\overline X}$};
      \node at (1,-4) (k) {$m^{\overline B}_{0|1|1}$};
      \node at (3,-4) (mu) {$\mu_2^{\Alg}$};
      \draw[->](x) to (xend);
      \draw[->](bbar) to (k);
      \draw[->](k) to (bbarend);
      \draw[->](xbar) to (xd);
      \draw[->](xd) to (xbarend);
      \draw[->](xd) to (k);
      \draw[->](xd) to (mu);
      \draw[->](mu) to (aend);
      \draw[->](a) to (mu);
    \end{tikzpicture}}+ 
  \mathcenter{\begin{tikzpicture}
      \node at (0.5,-1) (x) {$X$};
      \node at (1,-1) (bbar) {${\overline\Blg}$};
      \node at (1.5,-1) (xbar) {${\overline X}$};
      \node at (2,-1) (a) {$\Alg$};
      \node at (0.5,-5) (xend) {$~$};
      \node at (1,-5) (bbarend) {$~$};
      \node at (1.5,-5) (xbarend) {$~$};
      \node at (2,-5) (aend) {$~$};
      \node at (2,-3) (ad) {$\mu_1^{\Alg}$};
      \draw[->](x) to (xend);
      \draw[->](bbar) to (bbarend);
      \draw[->] (xbar) to (xbarend);
      \draw[->](a) to (ad);
      \draw[->] (ad) to (aend);
    \end{tikzpicture}}
  \]
  (Arrows that output $1\in\Alg$, which would appear
    in the second, third, and fourth terms, are suppressed.)
  There is a natural map
  $h\colon \lsup{\Alg}X^{\Blg}\DT \lsub{\Blg}Y_{\Alg} \to \lsup{\Alg}\Id_{\Alg},$
  defined by
  $x\otimes \psi\to \psi(x\otimes {\mathbf 1}),$ where here
  ${\mathbf 1}\in \Blg$ is the unit; 
  i.e. 
  given $\x\in X$ and 
  $\psi\in Y=\lsup{\Alg}\Mor(\lsup{\Alg}X^{\Blg}\DT \lsub{\Blg}\Blg_{\Blg},\lsup{\Alg}\Id_{\Alg})$,
  $x\otimes {\mathbf 1}\in \lsup{\Alg}X^{\Blg}\DT \lsub{\Blg}\Blg_{\Blg}$, so 
  $\psi(x\otimes {\mathbf 1})\in \lsup{\Alg}\Id_{\Alg}$.
  This natural map can be viewed as a $DA$ morphism. Recall that
  a $DA$ morphism 
  $\lsup{\Alg}X^{\Blg}\DT \lsub{\Blg}Y_{\Alg} \to \lsup{\Alg}\Id_{\Alg}$
  is specified by a sequence of maps indexed by $j\geq 1$:
  \[ h^1_j\colon \lsup{\Alg}X^{\Blg}\DT \lsub{\Blg}Y_{\Alg}\otimes
  \overbrace{\Alg\otimes\dots\otimes\Alg}^{j-1} \to \Alg.\] The
  morphism under question has $h^1_j=0$ for $j>1$, and it was
  specified above in the case where $j=1$.
  Pictorially, $h$ is represented as follows:
  \[ 
  \mathcenter{\begin{tikzpicture}
      \node at (0.5,-1) (x) {$X$};
      \node at (1,-1) (bbar) {${\overline\Blg}$};
      \node at (1.5,-1) (xbar) {${\overline X}$};
      \node at (2,-1) (a) {$\Alg$};
      \node at (1,-2) (bbarend) {$\Field$};
      \node at (1,-3) (k) {$K$};
      \node at (0,-5) (Aend) {$~$};
      \draw[->,bend right=15](x) to (k);
      \draw[->,bend left=15](xbar) to (k);
      \draw[->](bbar) to node[right]{\tiny{${\overline 1}$}} (bbarend);
      \draw[->] (a) .. controls (2,-4) and (0,-4) ..  (Aend);
    \end{tikzpicture}}
  \]
  Where here $K$ denotes the Kronecker pairing, and ${\overline 1}$ is
  dual to the map $1\colon \Field\to \Blg$.
  
  The verification that $h$ is a $DA$ bimodule homomorphism is
  straightforward.  

  So far, we have not used the hypothesis on the
  homology of $\lsub{\Blg}Y_{\Alg}$ (which is needed to verify that
  $h$ is a homotopy equivalence).  From the hypothesis, there is a rank
  one bimodule $\lsub{\Blg}Z_{\Alg}$; this is the bimodule structure
  induced on $\ground$ from its quasi-isomorphism with
  $\lsub{\Blg}Y_{\Alg}$ as in Lemma~\ref{lem:HomologicalPerturbation}.
  Let $\phi \colon \lsub{\Blg}Z_{\Alg}\to \lsub{\Blg}Y_{\Alg}$
  be the induced bimodule quasi-isomorphism.
  In more detail, denote the isomorphism $\ground\cong X$ by $\iota\mapsto x_\iota$;
  taking its dual and using an isomorphism $\ground\cong {\overline \ground}$
  induced from an $\Field$-basis for $\ground$ (which in our case is
  supplied by the basic idempotents), we get an isomorphism
  $\ground\to {\overline{X}}$ denoted $\iota\mapsto {\overline x}_{\iota}$.
  Now, $\phi_{0|1|0}(\iota)={\overline 1}\otimes {\overline x}_{\iota}\otimes 1$.

  We claim that on the bimodule $\lsub{\Blg}Z_{\Alg}$, for all $k> 0$,
  \begin{equation}
    \label{eq:SkewActionsVanish}
    m_{k|1|0}=0.
  \end{equation}
  So see this, observe that, the image of
  $m_{k|1|0}(b_1,\dots,b_{k-1},x)$ is supported in the portion in
  grading $\lambda(x)+\lambda(b_1)+\dots+\lambda(b_{k-1})$ which,
  since $\lambda(x)=0$ (because of Condition~\ref{Koszul:Rk1}), means
  that $\lambda(b_1)+\dots+\lambda(b_{k-1})=0$.  By positivity of the
  grading on $\Alg$, we conclude that each $\lambda(b_i)=0$, and hence
  that each $b_i\in \ground$; so these algebra elements have Maslov
  grading $0$.  Since $Z$ is supported in Maslov grading $0$, as well,
  we conclude from the grading conventions that $k=1$, as required by
  Equation~\eqref{eq:SkewActionsVanish}.  Similar reasoning shows
that for the bimodule homomorphism, for all $k> 0$,
  \begin{equation}
    \label{eq:SkewMapsVanish}
    \phi_{k|1|0}=0.
  \end{equation}

  To verify $h$ is a $DA$ bimodule quasi-isomorphism, consider the induced
  $AA$-bimodule homomorphism
  \[ (\Id_{\Alg}\DT h)
  \colon \lsub{\Alg}\Alg_{\Alg}\DT \lsup{\Alg}X^{\Blg}\DT \lsub{\Blg}Y_{\Blg}
  =\lsub{\Alg}\Alg_{\Alg}\DT \lsup{\Alg}X^{\Blg}\DT
  \lsub{\Blg}{\overline \Blg}_{\Blg}\DT \lsub{\Blg}X^{\Alg}\DT
  \lsub{\Alg}\Alg_{\Alg}
  \to \lsub{\Alg}\Alg_{\Alg}.\]
  To verify that it is a quasi-isomorphism,
  we form the following composite
  \[ \lsub{\Alg}\Alg_{\Alg} \DT \lsup{\Alg}X^{\Blg}\DT
  \lsub{\Blg}Z_{\Alg} \overset{\Id_{\Alg\DT X}\DT \phi}{\longrightarrow}
  \Alg\DT X\DT {\overline\Blg}\DT {\overline X} \DT{\Alg}
  \overset{\Id_{\Alg}\DT h}{\longrightarrow} \lsub{\Alg}\Alg_{\Alg},\]
  which we denote $\kappa$.  From
  Equation~\eqref{eq:SkewActionsVanish}, we conclude that the chain
  complex on the left is identified with $\Alg$.  From
  Equation~\eqref{eq:SkewMapsVanish}, we conclude that that
  $\kappa_{0|1|0}$ is the chain map $(\Id_{\Alg} \otimes h^1_{1})\circ
  (\Id_{\Alg\DT X}\otimes \phi_{0|1|0})$ that sends $a\otimes x_{\iota} \otimes
  \iota$ to $a$; in particular, it is an isomorphism of chain complexes.
\end{proof}

\begin{remark}
  In fact, if $\lsup{\Alg}X^{\Blg}$ is quasi-invertible, the
  quasi-inverse is always given by the above bimodules by the argument
  from~\cite[Proposition~9.2]{Bimodules}; i.e. 
  there are quasi-isomorphisms:
  \begin{align*}
    \lsub{\Blg}Y_{\Alg}&\simeq 
    \Mor_{\Blg}(\lsub{\Blg}\Blg_{\Blg},\lsub{\Blg}Y_{\Alg}) \\
    &\simeq \Mor^{\Alg}(\lsup{\Alg}X^{\Blg}\DT \lsub{\Blg}\Blg_{\Blg},
    \lsup{\Alg}X^{\Blg}\DT \lsub{\Blg} Y_\Blg) \\
    &\simeq \Mor^{\Alg}(\lsup{\Alg}X^{\Blg}\DT \lsub{\Blg}\Blg_{\Blg},
    \lsup{\Alg}\Id_{\Alg});
  \end{align*}
  the first of these is true for arbitrary $\lsub{\Blg}Y_{\Alg}$,
  the second uses the fact that $X\DT $ induces an equivalence of
  categories, and the third uses the fact that $X$ and $Y$ are quasi-inverses.
\end{remark}

\newcommand\Alex{\mathrm{Alex}}
\newcommand\Maslov{\mathfrak{m}}
\newcommand\Filt{\mathcal F}
\section{The algebras}
\label{sec:Algebras}

We describe differential graded algebras used in the construction of
our knot invariant.  In fact, we will find it convenient to work with
a more general construction $\AlgB(m,k,\Upwards)$, where $0\leq k\leq
m+1$, and $\Upwards$ is an arbitrary subset of $\{1,...,m\}$.  The
integer $m$ is called the {\em index}.  The integer $k$ is the {\em
  number of occupied positions}; together with the index, it
determines the base ring $\ground=\IdempRing(\AlgB)$.  When
$\Upwards_1\subset\Upwards_2\subset\{1,\dots,m\}$, then
$\AlgB(m,k,\Upwards_1)$ will be a differential subalgebra of
$\AlgB(m,k,\Upwards_2)$.

When $\Upwards=\emptyset$, the differential on the algebra
$\AlgB(m,k,\emptyset)=\AlgB(m,k)$ vanishes. In fact, this 
algebra is the quotient of a larger algebra
$\AlgBZ(m,k)$, which we define first.

When constructing the knot invariant for a knot with bridge number
$n$, we will have $m=2n$, $k=n$, and $\Upwards$ will correspond to
those ($n$) strands that are oriented upwards.  In fact, for the
purposes of the knot invariant, it would suffice to work in a summand
corresponding to certain idempotents; see
Remark~\ref{rem:CanUseSubalgebra}. We have chosen to use the larger
algebra in our constructions, as it satisfies a duality described in
Section~\ref{subsec:OurDuality}.

\subsection{The algebra $\AlgB_0(m,k)$}
\label{sec:BlgZ}

We define
$\AlgBZ(m,k)$, which is a graded algebra over $\Field[U_1,\dots,U_m]$,
whose (Alexander multi-grading) set is $(\frac{1}{2}\Z)^m$. (See Section~\ref{sec:PrelimGradings}.)
Basic idempotents in $\AlgB_0(m,k)$ correspond to {\em idempotent states}, or
{\em I-states} for short,  $\x=(x_1,\dots,x_k)$,
which
are increasing sequences of integers 
\begin{equation}
  \label{eq:Unrestricted}
  0\leq x_1<...<x_k\leq m. 
\end{equation}
The basic idempotent corresponding to ${\bf x}$ will be denoted by
$\Idemp{\bf x}$. 
The elements $\Idemp{\x}$ are generators of a ring
$\ground=\IdempRing(m,k)$ satisfying:
\[ \Idemp{\x}\cdot\Idemp{\y}=\left\{\begin{array}{ll}
\Idemp{\x} &{\text{if $\x=\y$}} \\
0 &{\text{if $\x\neq\y$}.}
\end{array}
\right.\]

\begin{rem}
  Recall from Section~\ref{sec:Intro} that these idempotents can be
  interpreted as marking intervals which are intersections of the
  regions in a knot diagram with the $y=t$ slice.  In fact, the
    basic idempotents referred to there do not allow for the unbounded
    intervals, corresponding to replacing Equation~\eqref{eq:Unrestricted}
    by the condition
    \begin{equation}
      \label{eq:Restricted}
    1\leq x_1<\dots<x_k\leq m-1.
  \end{equation} 
  Our basic idempotents now will be as in Equation~\eqref{eq:Unrestricted};
  we return to the restricted case in Section~\ref{sec:Properties}.
\end{rem}

The unit ${\bf 1}$ in the algebra is given by
the sum of the basic idempotents.

Given an I-state $\x$, define its {\em weight}
$v^{\x}\in \Z^{m}$ by
\begin{equation}
  \label{eq:DefOfV}
  v^{\x}_{i} = \# \{x\in \x\big| x\geq i\}.
\end{equation}
Given two I-states $\x$ and $\y$, define their {\em minimal relative weight vector} $w^{\x,\y} \in (\frac{1}{2} \Z)^m$ to be  given by
\[ w^{\x,\y}_{i} = \frac{1}{2}\big|v^{\x}_i-v^{\y}_i\big|.\]

$\AlgB_0(m,k)$ is defined so that  there is an identification of $\Field[U_1,\dots,U_m]$-modules
$\Idemp{\x} \cdot \AlgB_0(m,k) \cdot \Idemp{\y} \cong \Field[U_1,\dots,U_m]$; denote the identification by 
\[ \phi^{\x,\y}\colon \Field[U_1,\dots,U_m] \rightarrow 
\Idemp{\x} \cdot \AlgB_0(m,k) \cdot \Idemp{\y} \]
A grading by $(\OneHalf \Z)^{m}$ on $\Idemp{\x} \cdot \AlgB_0(m,k) \cdot \Idemp{\y}$  is specified by
\begin{equation}
  \label{eq:WeightFunction}
  w(\phi(U_1^{t_1}\dots U_m^{t_m}))= w^{\x,\y}+(t_1,\dots,t_m),
\end{equation}
for non-negative integers $t_1,\dots,t_m$.

Multiplication 
\[ 
\Big(\Idemp{\x} \cdot \AlgB_0(m,k) \cdot \Idemp{\y} \Big)
*
\Big(\Idemp{\y} \cdot \AlgB_0(m,k) \cdot \Idemp{\z} \Big)
\to \Big(\Idemp{\x} \cdot \AlgB_0(m,k) \cdot \Idemp{\z} \Big)
\] 
is the unique non-trivial, grading-preserving
$\Field[U_1,\dots,U_m]$-equivariant map.
Explicitly,
given I-states $\x$, $\y$,
and $\z$, if we define
$g^{\x,\y,\z}= U^{t_1}\cdots U^{t_m}$,
where 
$t_i=w^{\x,\y}_i + w^{\y,\z}_i - w^{\x,\z}_i$,
then for $a, b\in \Field[U_1,\dots,U_m]$, 
\[ \phi^{\x,\y}(a)* \phi^{\y,\z}(b) = 
\phi^{\x,\z}(a \cdot b \cdot g^{\x,\y,\z}).
\]

Suppose that $\x$ is an I-state with $j-1\in \x$ but $j\not\in\x$. 
Then, we can form a new I-state $\y=\x\cup \{j\}\setminus\{j-1\}$, and let  
$R_j^\x = \phi^{\x,\y}(1)$. We define
\[R_{j} = \sum_{\{\x\big|j-1\in \x, j\not\in\x\}} R_j^\x,\]
so that
\[ \Idemp{\x} \cdot R_j = \left\{\begin{array}{ll}
R_j^\x &{\text{if $j-1\in \x$ and $j\not\in \x$}} \\
0 &{\text{otherwise.}}
\end{array}\right.\]
Similarly, let $L_j^\y = \phi^{\y,\x}(1)$, and
$L_j = \sum_{\{\y\big| j\in \y, j-1\not\in\y\}} L_j^\y$.
Less formally $R_j$  moves
one of the coordinates of an I-state from the 
$(j-1)^{st}$ position to the $j^{th}$, and $L_j$ changes it 
back from $j^{th}$ to $(j-1)^{st}$. 
The elements $L_j$ and $R_j$ are called the {\em left shifts} and {\em right shifts} respectively.
See Figure~\ref{fig:LRrelation}.

For $i=1,\dots,m$,  $U_i$ induces an algebra element in $\AlgBZ$,
defined by $\sum_{\x}\phi^{\x,\x}(U_i)$. 
For notational simplicity, we also denote this induced element by $U_i$.

\begin{figure}[ht]
\input{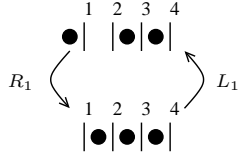}
\caption{\label{fig:LRrelation} {\bf{Picture of algebra elements in ${\mathcal B}_0(4,3)$.}}  
Idempotents $\Idemp{\{0,2,3\}}$  and $\Idemp{\{1,2,3\}}$ are pictured here;
the algebra elements $L_1$ and $R_1$ which connect them are indicated
(e.g. $\Idemp{\{0,2,3\}}\cdot R_1\cdot \Idemp{\{1,2,3\}}$  is non-zero). Here, 
$\Idemp{\{0,2,3\}}\cdot R_1\cdot L_1 = \Idemp{\{0,2,3\}} \cdot U_1$.}
\end{figure}

\begin{prop}
  The weight function from Equation~\eqref{eq:WeightFunction} descends
  to a grading $w=(w_1,\dots,w_m)$ on $\AlgBZ(m,k)$ with values in
  $(\OneHalf \Z)^m$.
\end{prop}

\begin{proof}
  It suffices to show that 
  \[ w(\phi^{\x,\y}(1) * \phi^{\y,\z}(1))=w(\phi^{\x,\y}(1)) + w(\phi^{\y,\z}(1)).\]
  which follows from the definition of $g^{\x,\y,\z}$.
\end{proof}

We have no further need to distinguish the product on the algebra $\AlgBZ(m,k)$
from other products; so we will abbreviate $a* b$ by $a \cdot b$.

In the notation from Section~\ref{sec:PrelimGradings},
$\AlgBZ(m,k)$ is graded by the Abelian group $\Lambda=(\OneHalf\Z)^m$.
Recall that an element $a\in\AlgBZ(m,k)$ that is supported in some fixed grading
$\lambda\in\Lambda$ is called {\em homogeneous}.

\begin{prop}
  \label{prop:AlgebraGenerators}
  The algebra $\AlgBZ(m,k)$ is generated over $\Field$
  by the elements
  $L_i$, $R_i$, $U_i$, and the idempotents $\Idemp{\x}$. 
\end{prop}

\begin{proof}
  Let $a$ be a homogeneous algebra elements with 
  $\Idemp{\x}\cdot a\cdot  \Idemp{\y} =a$.  If
  $\x=\y$, then $a$ factors as a product of $U_i$.  Otherwise, suppose
  that some $i=x_t>y_t$; and indeed choose $t$ minimal with this
  property.  Then, it is easy to see that we can factor $a=L_i\cdot
  b$.  Otherwise,  there is some $y_t>x_t=i$, and we can choose $t$
  maximal with this property.  Then, we can factor $a=R_{i+1}\cdot b$.  In
  both cases, the total weight of $b$ is smaller than that of $a$, so
  the result follows by induction on weight. 
\end{proof}

\subsection{The $\AlgB(m,k)$}

The algebra $\AlgB(m,k)$ is the quotient of $\AlgBZ(m,k)$ by the relations
\begin{align}
 L_{i+1}\cdot L_i &= 0 
  \label{eq:MoveTooFarLs} \\
 R_i\cdot R_{i+1}&=0;
  \label{eq:MoveTooFarRs}
\end{align}
and also, if 
$\{x_1,...,x_k\}\cap \{j-1,j\}=\emptyset$, then 
\begin{equation}
\label{eq:KillUs} 
\Idemp{\x}\cdot U_j=0.
\end{equation}

As we shall see, the relations Equations~\eqref{eq:MoveTooFarLs}
and~\eqref{eq:MoveTooFarRs} guarantee that algebra elements in
$\AlgB(m,k)$ cannot move coordinates in the idempotents states by more
than one unit.
We formulate this quotient operation as follows:
\begin{defn}
  Let ${\mathcal J}$ be the (two-sided) ideal in $\AlgBZ(m,k)$ that is generated by 
  $L_{i+1}\cdot L_i$, $R_i\cdot R_{i+1}$, and $\Idemp{\x} \cdot U_j$, when
  $\{x_1,...,x_k\}\cap \{j-1,j\}=\emptyset$.
  Let $\AlgB(m,k)$ be the quotient of $\AlgBZ(m,k)$ by this two-sided ideal.
\end{defn}

Note that ${\mathcal J}$ is generated by elements that are homogeneous with
respect to the weights. Thus, the weights induce a grading on the quotient algebra $\AlgB(m,k)$.

The ideal ${\mathcal J}$ can be understood concretely. To do so, we set up some notation.

\begin{defn}
\label{def:CloseEnough}
Two I-states $\x$ and $\y$ are said to be {\em far}  if there is some
$i=1,\dots,k $ with
$|x_i-y_i|>1$; otherwise they are called {\em close enough}.
\end{defn}

Given two I-states $\x$ and $\y$, we define an ideal
$\Ideal(\x,\y)\subset \Field[U_1,\dots,U_m]$.  If the I-states $\x$ and
$\y$ are far, let $\Ideal(\x,\y)=\Field[U_1,\dots,U_m]$.  If
$\x$ and $\y$ are close enough, let $\Ideal(\x,\y)$
be the ideal generated by monomials $U_{i+1}\cdots U_j$, taken over all
$0\leq i<j\leq m$ where $i$ and $j$ satisfy the following conditions:
\begin{itemize}
\item $i, j \in \{0,\dots,m\}\setminus (\x\cap\y)$
\item 
for all integers $t$ with $i< t < j$, $t\in\x\cap\y$.
\item
$w^{\x,\y}_t=0$ for all $i+1\leq  t \leq j$.
\end{itemize}

\begin{defn}
  \label{def:GeneratingInterval}
  For $i<j$ as above, we call the interval $[i+1,j]$ a  {\em generating interval for $\x$ and
    $\y$}. (Observe that a generating interval can have $i+1=j$, and 
  this corresponds to $U_j$.)
\end{defn}

\begin{figure}
\input{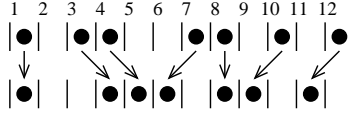}
\caption{\label{fig:Unused} 
  The ideal $\Ideal(\x,\y)$ is generated by $U_1 U_2$, $U_3$,
  $U_6$, $U_8 U_9$, and $U_{11}$.}
\end{figure}

\begin{prop}
  \label{prop:IdentifyJ}
  For all I-states $\x$ and $\y$, $\Idemp{\x}\cdot {\mathcal J}\cdot \Idemp{\y}
  = \phi^{\x,\y}(\Ideal(\x,\y))$.
\end{prop}

\begin{proof}
  Let ${\mathcal I}=\bigoplus_{\x,\y} \phi^{\x,\y}(\Ideal(\x,\y))$.
  First we prove ${\mathcal I}\subseteq {\mathcal J}$; i.e. for any two 
  I-states $\x$ and $\y$
we have $ \phi^{\x,\y}(\Ideal(\x,\y)) \subset {\mathcal J}$.

  If $\x$ and $\y$ are far, so that
  $\Ideal(\x,\y)= \Field[U_1,\dots,U_m]$, then we claim that
   $\phi^{\x,\y}(1)$
  can be factored as a product $a\cdot R_{t}R_{t+1}\cdot b$ or
  $a\cdot L_{t} L_{t-1} \cdot b$ for some $t$. This is obvious, for example
if there is an  $i$ with 
  $x_i> y_i+ 1$, then
$\phi^{\x, \y} (1)$ decomposes as $ a\cdot L_t\cdot L_{t-1}\cdot b$,
where $t=x_i$.

  If $\x$ and $\y$ are not far, we show that for all generating intervals $[i+1,j]$ for $\x$ and $\y$, the
  algebra elements $\phi^{\x,\y}(U_{i+1}\dots U_j)$ are in ${\mathcal
    J}$.

  Consider first the case where $i+1=j$.  Then, $\phi^{\x,\y}(U_j)$ can be factored as $
  a\cdot \Idemp{{\mathbf w}} \cdot U_j\cdot b$, where $j-1, j\not\in {\mathbf
    w}$.   In general, if $i<j$ is a generating interval for $\x$ and $\y$, 
  $\phi^{\x,\y}(U_{i+1}\cdots U_j)$ has a decomposition as 
  \[ (a \cdot L_{i+1}\cdots L_{j-1} )\cdot \Idemp{{\mathbf w}}\cdot  U_{j} \cdot
  R_{j-1}\cdots R_{i+1}\cdot b,\]
  where $j-1,j\not\in {\mathbf w}$. In these cases, $\Idemp{{\mathbf w}}\cdot U_j\in {\mathcal J}$.
  (For example, for the I-states $\x$ and $\y$ in Figure~\ref{fig:Unused}, 
  \[
\begin{array}{lll}
  \phi^{\x,\y}(U_3)&=
    \phi^{\x,\y}(1) \cdot \Idemp{\y}\cdot U_3 \cdot \phi^{\y,\y}(1) 
    & \\
    \phi^{\x,\y}(U_{11})&= \phi^{\x,{\mathbf w}}(1) 
    \cdot\Idemp{\w}\cdot U_{11} \cdot L_{12}
    &{\text{where $\w=\{1,4,5,6,7,8,9,12\}$}} \\
    \phi^{\x,\y}(U_1 U_2) &= \phi^{\x,{\mathbf w}}(1) \cdot\Idemp{\w}\cdot U_2 \cdot \phi^{\w,\y}(1)
    &{\text{where $\w=\{0,3,4,7,8,10,12\}$}}) \\  
  \end{array}
  \]
  This completes the proof that ${\mathcal I}\subseteq {\mathcal J}$.

  To prove that ${\mathcal J}\subseteq {\mathcal
    I}$, we first prove that ${\mathcal I}$ is an ideal; i.e. if $a$ is an
  arbitrary algebra element, then 
  $a\cdot {\mathcal I}\subseteq  {\mathcal I}$;
  and also
  $ {\mathcal I}\cdot a \subseteq {\mathcal I}$.
  In view of Proposition~\ref{prop:AlgebraGenerators}, it suffices
  to verify this in the case where $a= \Idemp{\x} \cdot U_i$, $\Idemp{\x} \cdot R_i$, or
  $\Idemp{\x} \cdot L_i$.
 
  When $a=\Idemp{\x} \cdot U_i$, the result follows readily. 

  By symmetry, we can consider the case where $a= \Idemp{\x} \cdot R_i \cdot \Idemp{\y} \neq 0$. 
  Let $t$ denote the index with $x_t=i-1$ and $y_t=i$.
  We wish to show that $\Idemp{\x} \cdot R_i\cdot
  \phi^{\y,\z}(\Ideal(\y,\z))\subset \phi ^{\x, \z} (\Ideal(\x,\z))$.  There are
  cases:
  \begin{enumerate}
  \item
    \label{case:xzfar}
    If $\x$ and $\z$ are far, the statement is vacuously true.
  \item If $\y$ and $\z$ are far, but $\x$ and $\z$ are
    close enough, then $z_t=i-2$ and $g^{\x,\y,\z}=U_i$. 
Futhermore we have $i\not\in \x$, $i-1 \not\in \z$, and $w_i^{\x,\z}=0$.
It follows that  
$U_i\in \Ideal(\x,\z)$ and
this finishes the argument in this subcase.
  \item 
    \label{case:xyNotFar}
    If $\y$ and $\z$ are not far, and $\x$ and $\z$ are not far,
    as well, then there are two subcases according to the value of 
    $z_t$.
    \begin{enumerate}[label=(\ref{case:xyNotFar}\alph*),ref=(\ref{case:xyNotFar}\alph*)]
      \item
      \label{case:MovingOn}
      If $z_t=i$, the generating
      intervals for $\x,\z$ are contained in the generating intervals
      for $\y,\z$, and the containment is clear.
      \item
        \label{case:MovingBack}
        If $z_t=i-1$, there can be at most one generating interval for
        $(\y,\z)$ which is not also a generating interval for
        $(\x,\z)$, and that is an interval which terminates in
        $i$. Since $g^{\x,\y,\z}=U_i$, it follows that 
        \[\Idemp{\x}\cdot
        R_i\cdot \phi^{\y,\z}(U_{j}U_{j+1}\cdots U_{i-1})=\phi^{\x,\z}(U_{j}U_{j+1}
        \cdots U_{i}).\]
\end{enumerate}
\end{enumerate}      
%\begin{figure}[ht]
%        \input{IdealContainment.pstex_t}
%        \caption{\label{fig:IdealContainment} {\bf{Cases of ${\mathcal J}
%\subseteq {\mathcal I}$.}} This illustrates cases from 
%Proposition~\ref{prop:IdentifyJ}.
%          The top row represents the initial I-state $\x$, the arrow below
%        represents the algebra element $R_i$, connecting to the I-state $\y$.
%      The arrow below illustrates the three different possibilities
%      when $\y$ and $\z$ are not far.}
%      \end{figure}

  The containment $\phi^{\x,\y}(\Ideal(\x,\y))\cdot R_i \cdot \Idemp{\z} 
  \subset \phi ^{\x, \z}( \Ideal(\x,\z))$ follows symmetrically.

  We must now verify that the definining relations for ${\mathcal J}$
  are contained in $\Ideal$.  Clearly, if $\Idemp{\x}\cdot R_{i}\cdot
  R_{i+1}\cdot\Idemp{\y}$ is non-zero in $\AlgBZ$, then $\x$ and $\y$
  are too far, and so $\Ideal(\x,\y)=\Field[U_1,\dots,U_m]$, proving
  the containment.  The same argument works for $L_{i+1}\cdot
  L_{i}$. Finally, if $\x\cap\{i-1,i\}=\emptyset$, then $U_i$ is a
  monomial corresponding to a generating interval for $(\x,\x)$, so
  $U_i\in \Ideal(\x,\x)$.
\end{proof}

The following lemma will be useful later.

\begin{lemma}
  \label{lem:NontrivProd}
  Let $a\in\AlgB(m,k)$ be homogeneous, and suppose that $w_i(a)\in\Z$.
  Then $a\cdot R_i\neq 0$ implies that $a\cdot U_i\neq 0$; similarly,
  $R_i\cdot a \neq 0$ implies that $U_i\cdot a \neq 0$.
\end{lemma}

\begin{proof}
  Suppose that $b=a\cdot R_i\neq 0$, and $b=\Idemp{\x}\cdot b
  \cdot \Idemp{\y}$.  Since $w_i(a)\in\Z$ and $b\neq 0$, it follows
  that $w_i^{\x,\y}=\OneHalf$.  Thus, it follows that $i$ is not in
  any of the generating intervals for $\x$ and $\y$, so $b\cdot
  U_i\neq 0$, so $a\cdot U_i\neq 0$, as well.
\end{proof}

\begin{prop}
  \label{prop:DeterminedByMultiplicities}
  A homogeneous non-zero algebra element $a=\Idemp{\x} \cdot a\cdot
  \Idemp{\y}\in \Blg(m,k)$ is uniquely characterized by its initial
  (or terminal) idempotent and its weight.
\end{prop}

\begin{proof}
  Fix initial and terminal idempotent states $\x$ and $\y$.
  By Proposition~\ref{prop:IdentifyJ}, $\Idemp{\x}\cdot \Blg(m,k)\cdot \Idemp{\y}$
  is the quotient of the polynomial algebra $\Field[U_1,\dots,U_m]$ by
  an ideal which is homogeneous with respect to the weights.  An
  element of this quotient space in turn is uniquely determined by its
  various weights.  

  Note that the weight of $a$ modulo $1$ determines $w^{\x,\y}$.  It
  follows that $\x$ and $w(a)$ determines $\y$.
\end{proof}

\subsection{Defining $\AlgB(m,k,\Upwards)$.}

Let $\Upwards\subset \{1,\dots,m\}$ be a subset. 
Define $\AlgB(m,k,\Upwards)$ to be the differential graded algebra obtained
by adjoining new algebra elements $C_i$ for $i\in \Upwards$ to $\AlgB(m,k)$,
which satisfy the following properties:
\begin{itemize}
  \item The $C_i$ commute with all other algebra elements in
    $\AlgB(m,k,\Upwards)$.
  \item The square of $C_i$ vanishes.
  \item The differential of $C_i$ is $U_i$.
\end{itemize}
This construction can be done since the $U_i$ are in the center of
$\AlgB(m,k,\Upwards)$ and $d U_i=0$. 
More formally, if $\Upwards=\{i_1,...,i_n\}$, then
 $\AlgB(m,k,\Upwards)$ is defined by the formula
\[\AlgB(m,k,\Upwards)=
\frac{\AlgB(m,k)[C_{i_1},\dots,C_{i_n}]}{\{C_{j}^2=0,
  dC_{j} = U_{j}\}_{j\in \Upwards}}.\]

The algebra $\AlgB(m,k,\Upwards)$ is equipped with a distinguished basis as an $\Field$-vector space, with basis vectors corresponding to the following data:
\begin{itemize}
  \item  A pair of idempotents $\x$ and $\y$ that are not far,
  \item a monomial $p$ in $\Field[U_1,\dots,U_m]$ that is not divisible
    by any monomial associated to any generating interval for $\x$ and $\y$,
  \item a (possibly empty) subset $J$ of $\Upwards$.
\end{itemize}
The corresponding algebra element is $\phi^{\x,\y}(p) \cdot
\prod_{j\in J} C_j$. We call these basis vectors {\em pure algebra elements}.

\subsection{Gradings}

% \smargin{Refined grading, with values in $(\frac{1}{2}(\Z\oplus \Z))^{m}$ }
% $(a_i,b_i)_{i=1}^m$
% $w_i = \frac{a_i+b_i}{2}$, 
% \begin{align}
% v_i(R_i) = (1,0) &\qquad v_i(L_i)=(0,1)
% \label{eq:RefinedGradings} \\
% v_i(U_i)=v_i(C_i)=(1,1) &\qquad v_i(R_j)=v_i(L_j)=v_i(C_j)=v_i(U_j)=(0,0) \nonumber
% \end{align}
% for $j\neq i$.

Note first since $\Blg(m,k)$ is obtained as a quotient of
$\AlgBZ(m,k)$ by a $w$-homogeneous ideal, the $w$-grading by
$(\OneHalf \Z)^m$ descends to a grading on $\Blg(m,k)$. The $w$-gradings
extend to $\AlgB(m,k,\Upwards)$, by declaring $w_i(C_j)=1$ if $i=j$ and $0$
otherwise.

Explicitly, for $1\leq i\leq m$, the $i^{th}$ weight $w_i$ of $L_i$
and $R_i$ is $1/2$, and $w_i$ of $U_i$ and $C_i$ is $1$, and $w_i$
vanishes on $L_j$, $R_j$, $U_j$,  and $C_j$  with $j\neq i$.  These functions
each induce gradings on $\AlgB(m,k,\Upwards)$; i.e. if $a$ and $b$ are
homogenous elements with $a\cdot b\neq 0$, then $w_i(a\cdot
b)=w_i(a)+w_i(b)$.

We will consider the following specialization, called the {\em Alexander grading}:
\begin{equation}
  \label{eq:DefAlex}
  \Alex(a)= -\sum_{s\in \Upwards} w_s(a)+ \sum_{t\not\in \Upwards} w_t (a).
\end{equation}

For $i=1,\dots,m$, there is a filtration $\MasFilt_i$  
$\AlgB(m,k,\Upwards)$ with values in $\{0,1\}$, specified by the function on pure algebra
elements $b$ so that $\MasFilt_i(b)=1$ if $b$ is divisible by $C_i$ and $0$
otherwise.  This extends to a filtration on the algebra:
\[ \AlgB(m,k,\Upwards)=\AlgB(m,k,\Upwards)_{\MasFilt_i=0}\oplus
\AlgB(m,k,\Upwards)_{\MasFilt_i=1},\] where $\Blg(m,k,\Upwards)_{\MasFilt_i=c}$
for $c=0$ or $1$ is the vector space spanned by generators
$a\in\Blg(m,k,\Upwards)$ with $\MasFilt_i(a)=c$. 
We call this a filtration,
because the differential on the algebra $d$ preseves
$\Blg(m,k,\Upwards)_{\MasFilt_i=0}$, but not
$\Blg(m,k,\Upwards)_{\MasFilt_i=1}$. 
Elements of $\Blg(m,k,\Upwards)$ that are contained entirely in 
$\AlgB(m,k,\Upwards)_{\MasFilt_i=0}$ or $\AlgB(m,k,\Upwards)_{\MasFilt_i=1}$ are called {\em $\MasFilt_i$-homogeneous}.
Note that $\sum_{i=1}^m \MasFilt_i$
induces a $\Z$-grading on $\AlgB(m,k,\Upwards)$ that drops by one
under the differential. We will be interested in a different normalization of this, 
the {\em Maslov grading}, defined by 
\begin{equation} 
\label{eq:DefMaslov}
\Maslov(a)= \#(~\text{$C_i$ in $a$}) - 2 \sum_{s\in \Upwards} w_s(a) 
= \sum_{i=1}^m \MasFilt_i(a) - 2 \sum_{s\in \Upwards} w_s(a).
\end{equation}

\subsection{Examples}
When $k=0$, $\AlgB(m,0)\cong \Field$; where $1=\Idemp{\emptyset}$. The
elements $L_i$, $R_i$, and $U_i$ are zero.  When $k=m+1$,
$\AlgB(m,m+1)=\Field[U_1,...,U_m]$;
there is only one idempotent $1=\Idemp{\{0,\dots,m\}}$.

Some other examples can be illustrated by the use of path algebras.
Given a directed graph $\Gamma$, the {\em path algebra} is the
$\Field$-vector space generated by sequences of edges $e_1*\dots *e_n$
where the terminal vertex of $e_i$ is the initial vertex of $e_{i+1}$. We include
also trivial paths, based at any vertex. If two paths can be concatenated, then their
product is the concatenation; otherwise it is zero. 

For example, consider Figure~\ref{fig:B11}, which is a graph with two
vertices and four edges.  Two of the edges (closed loops) are labelled
by $U$, but they are not the same, as they have different initial
points.  Identifying $L*R$ and $R*L$ with the corresponding closed
loops labelled by $U$, we obtain the algebra $\AlgB(1,1)$.

\begin{figure}[ht]
  \input{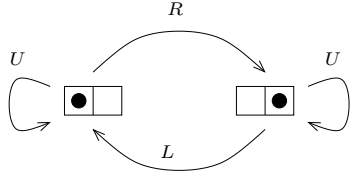}
\caption{\label{fig:B11} {\bf{Picture of $\AlgB(1,1)$.}}  The two
  idempotents $\Idemp{\{0\}}$  and $\Idemp{\{1\}}$  are pictured, with arrows
  corresponding to algebra elements connecting idempotents. The
  algebra $\AlgB(1,1)$ is the quotient of the pictured path algebra, 
  where $R*L$ and $L*R$ are identified with the two closed loops
  labelled by $U$ (that are distinguished by their starting points).}
\end{figure}

\begin{figure}[ht]
\input{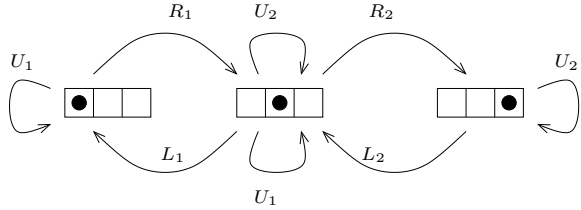}
\caption{\label{fig:B21} {\bf{Picture of $\AlgB(2,1)$.}}  The three
  idempotents $\Idemp{\{0\}}$, $\Idemp{\{1\}}$, and $\Idemp{\{2\}}$  are
  pictured, with arrows corresponding to algebra elements connecting
  idempotents.  The algebra $\AlgB(2,1)$ can be thought of as a
  quotient of the pictured path algebra, divided out by relations $R_1* R_2=0$,
  $L_2* L_1=0$, $R_i * L_i=U_i$, $L_i*R_i=U_i$, $U_1*U_2=U_2*U_1$.}
\end{figure}

\begin{figure}[ht]
\input{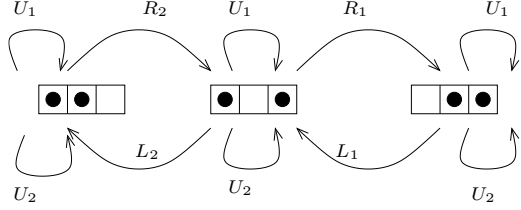}
\caption{\label{fig:B22} {\bf{Picture of $\AlgB(2,2)$.}}  The three
  idempotents $\Idemp{\{0,1\}}$, $\Idemp{\{0,2\}}$, and
  $\Idemp{\{1,2\}}$ are pictured, with arrows corresponding to algebra
  generators.  The algebra $\AlgB(2,2)$ can be thought of as a
  quotient of the pictured path algebra by $R_i * L_i=U_i$, $L_i*
  R_i=U_i$, 
  $U_1* L_2=L_2*U_1$,
  $U_1* R_2=R_2*U_1$, $U_2* L_1=L_1*U_2$, and $U_2* R_1=R_1*U_2$.}
\end{figure}

\subsection{Symmetries in the algebras}
\label{subsec:Symmetry}

Consider the map $\rho\colon \{0,\dots,m\}\to \{0,\dots,m\}$ with
$\rho(i)=m-i$. There is a map
\begin{equation}
  \label{eq:DefVRot}
  \VRot\colon \Blg(m,k,\Upwards)\to \Blg(m,k,\rho'_m(\Upwards))
\end{equation}
where $\rho'_m(i)=m+1-i$, characterized by the following properties:
\begin{itemize}
\item $\VRot(\Idemp{\x})=\Idemp{\rho(\x)}$.
\item If $a\in\Blg(m,k)$ is non-zero and homogeneous with  weights 
$(w_1(a),\dots,w_m(a))$,
then $\VRot(a)=b$ is the non-zero element that is homogeneous with weights given by
$w_{i}(b)=w_{m+1-i}(a)$. 
\end{itemize}
We extend this 
to $\Blg(m,k,\Upwards)\supset \Blg(m,k)$ 
by requiring $\VRot(C_j\cdot a)=C_{m+1-j}\cdot \VRot(a)$.
Note that for all $i=1,\dots,m$ and $j\in\Upwards$,
\[\begin{array}{llll}
 \VRot(L_i)=R_{m+1-i} &\VRot(R_i)=L_{m+1-i} &
 \VRot(U_i)=U_{m+1-i} & \VRot(C_j)=C_{m+1-j}
\end{array}\]
Clearly, this induced map $\VRot$ induces an isomorphism of algebras.

Another symmetry identifies the algebra with its ``opposite'' algebra. Specifically, 
the map
$\Opposite(\Idemp{\x})=\Idemp{\x}$
extends to an isomorphism of rings
\begin{equation}
  \label{eq:OppositeIsomorphism}
  \Opposite\colon \Blg(m,k,\Upwards)\to \Blg(m,k,\Upwards)^{\op},
\end{equation}
with $\Opposite(L_i)=R_i$, $\Opposite(R_i)=L_{i}$, 
$\Opposite(U_i)=U_{i}$, and $\Opposite(C_j)=C_j$.

\subsection {Canonical $DD$ bimodules}
\label{subsec:CanonDD}

Let 
\begin{equation}
  \label{eq:SpecifyCanonicalDDAlgebras}
    \AlgB_1=\AlgB(m,k_1,\Upwards_1), \ \ \ \AlgB_2=\AlgB(m,k_2,\Upwards_2)
\end{equation}
where $k_1+k_2=m+1$ and $\Upwards_2=\{1,...,m\}-\Upwards_1$. 

Note that there is a natural one-to-one correspondence between the
$I$-states for $\AlgB_1$ and $\AlgB_2$: if $\x\subset \{0,\dots,m\}$
is a $k_1$-element subset, then its complement $\x'$ is a
$k_2$-element subset of $\{0,\dots,m\}$. In this case, we say that
$\x$ and $\x'$ are {\em complementary $I$-states}.

We define 
a $DD$ bimodule $\AlgB_1$ and $\AlgB_2$  as follows.
Let $\CanonDD$ be the $\Field$-vector space whose generators $k_\x$ correspond
to $I$-states for $\AlgB(m,k_1,\Upwards_1)$. We give $\CanonDD$ the structure of 
a left module over $\IdempRing(\AlgB_1)\otimes \IdempRing(\AlgB_2)$,  determined 
by
\[(\Idemp{\y} \otimes \Idemp{\w})  \cdot k_\x=
\left\{\begin{array}{ll}
k_\x &{\text{if $\x=\y$ and $\w$ is complementary to $\x$}} \\ 
0 & {\text{otherwise.}}
\end{array}\right.\]
The algebra element
\[
A = \sum_{i=1}^m \left(L_i\otimes R_i + R_i\otimes L_i\right) + \sum_{s\in \Upwards_1}
  C_s\otimes U_s + \sum_{t\in \Upwards_2} U_t \otimes C_t\in
  \AlgB_1\otimes\AlgB_2\]
specifies a map
$\delta^1 \colon \CanonDD \to \AlgB_1\otimes\AlgB_2 \otimes \CanonDD$
by
$\delta^1(v)=A\otimes v$
(where the tensor product is taken over $\IdempRing(\AlgB_1)\otimes \IdempRing(\AlgB_2)$).

These data can be represented graphically, as follows. Take $m$
vertical segments, and orient them arbitrarily.  
The upwards
pointing segments specify $\Upwards_2$ and the downwards pointing ones
specify $\Upwards_1$.  
We draw a horizontal
arc crossing each of the vertical segment as a placeholder, and keep only half of each vertical 
segment above (resp. below) the horizontal arc if the segment is oriented upwards (resp. downwards).
The
horizontal arc is divided into $m+1$ intervals by the vertical arcs.
An element of $k_\x$ is represented as a collection of $m+1$ dark dots 
corresponding to the intervals in the horizontal segment, 
each of distributed either above or below te horizontal segment. The set of intervals 
whose dots are below give $\x$. When illustrating generators of $\Blg_1\otimes\Blg_2\otimes \CanonDD$,
we draw the algebra element from $\Blg_1$ below the diagram for $k_\x$ and the algebra element
from $\Blg_2$ above. See Figure~\ref{fig:TermsInDDid} for an illustration.

\begin{figure}[ht]
\input{TermsInDDid.pstex_t}
\caption{\label{fig:TermsInDDid} {\bf{The canonical $DD$ bimodule.}}
  In the left column, we have the three generators  for the $DD$ bimodule,
  with $\AlgB_1=\AlgB(2,1,\{2\})$ and $\AlgB_2=\AlgB(2,2,\{1\})$. They correspond to
  the $I$-states $\{2\}$, $\{1\}$, and $\{0\}$ respectively. 
  To the right we have 
  the non-zero terms 
  in the differential. To read off the algebra element in $\Blg_1$, reflect the 
  picture vertically. For example, if the three generators
  on the left are denoted $E$, $F$, and $G$, then the first equation
  expresses
  $\delta^1(E)=(L_2\otimes R_2) \otimes F + (C_2\otimes U_2)\otimes E$.}
\end{figure}

\begin{lemma}
  \label{lem:CanonicalIsDD}
  The map $\delta^1$ satisfies the type $DD$ structure relation.
\end{lemma}

\begin{proof}
  This is equivalent to the statement that $dA + A\cdot A =0$,
  thought of as an element of $\AlgB_1\otimes \AlgB_2$.
  We consider the four types of terms $A$: $L_i\otimes R_i$, $R_i\otimes L_i$, $C_i\otimes U_i$ when $i\in \Upwards_1$, and $U_i\otimes C_i$ when $i\in \Upwards_2$.
  In $dA + A\cdot A$, 
  some of the terms cancel since $U_i$ and $C_i$ are in the center;
  other terms cancel when the indices $i$ and $j$ are sufficiently
  far ($|i-j|>1$). When $|i-j|=1$, terms of the first type and the
 second type cancel. When $i=j$, a term of the first and second type
  cancel with differentials of terms of the third or fourth type.
\end{proof}

The two out-going algebras $\Blg_1$ and $\Blg_2$ are graded by $\OneHalf\Z^m$, 
\[
  \gr_{\Blg_1}(a_1)=(w_1(a_1),\dots,w_m(a_1)) \qquad
  \gr_{\Blg_2}(a_2)=(w_1(a_2),\dots,w_m(a_2));
\]
the canonical $DD$ bimodule is also graded by $\OneHalf\Z^m$
where $(v_1,v_2)\in\OneHalf \Z^m \oplus \OneHalf\Z^m$ acts on $w\in \OneHalf\Z^m$
by $(v_1,v_2)\cdot w = (v_1-v_2+w)$.
In fact, the module is supported in grading $0$;
i.e. for
each term $a_1\otimes a_2$ in $A$ specifying $\delta^1$,
$\gr_{\Blg_2}(a_2)=\gr_{\Blg_1}(a_1)$.

\subsection{The canonical $DD$ bimodule is invertible}
\label{subsec:OurDuality}

\subsubsection{A candidate for the inverse module}
\label{subsec:CandidateInverse}

The main result is to prove that the bimodule defined in the previous 
section is invertible. 
Let
\[
Y_{\AlgB_1,\AlgB_2}=\Mor^{\AlgB_1}((\lsub{\AlgB_2}{(\AlgB_2)}_{\AlgB_2})\DT_{\AlgB_2} (\lsup{\AlgB_1,\AlgB_2}\CanonDD),\lsup{\AlgB_1}{\mathbb
  I}_{\AlgB_1}).\]
This is naturally a  $\AlgB_2$-$\AlgB_1$-bimodule (of type $AA$, with both
actions on the right).

As a vector space, $Y$ is spanned by elements of the form
$({\overline a}\big| b)$, where ${\overline a}\in{\overline{\AlgB_2}}$
and $b\in\AlgB_1$, where here ${\overline{\AlgB}}$ is opposite
bimodule to $\AlgB$, thought of as a bimodule (as in
Equation~\eqref{eq:OppositeBimodule}), subject the restriction
that the left idempotent of ${\overline a}$ is complementary to the left
idempotent of $b$. 
Recall that ${\overline {\AlgB_2}}$ is the $\AlgB_2$-module
consisting of maps from $\AlgB_2$ to
$\Field$. This is a left-right $\AlgB_2$-$\AlgB_2$ bimodule by the rule
\[x\cdot {\overline a}\cdot y = (\xi \mapsto {\overline a}(y\cdot
\xi\cdot x)),\]
for $x,y\in\AlgB_2$ and ${\overline a}\in {\overline {\AlgB_2}}$.
We can take these generating vectors so that ${\overline a}$ is dual to
a generating algebra element in $\AlgB_2$; note that the right
idempotent of $a$ is the left idempotent of its dual ${\overline a}$.

The differential on $Y$ has terms
$({\overline a}|b)  \mapsto (L_i\cdot {\overline a} | R_i \cdot b)$,
$({\overline a}|b)  \mapsto (R_i\cdot  {\overline a} | L_i \cdot b)$, 
and furthermore
$({\overline  a}| b)  \mapsto (U_i\cdot {\overline a} | C_i \cdot b)$
if $i\in \Upwards_1$; otherwise 
$({\overline  a}| b)  \mapsto (C_i \cdot {\overline a} | U_i \cdot b)$.
Finally, there are terms
$( {\overline a} | C_i \cdot b) \mapsto ({\overline a}| U_i\cdot b)$
or
$( {\overline{U_i \cdot a}} | b) \mapsto ({\overline{C_i \cdot a}}| b)$.
The action by $\AlgB_2$-$\AlgB_1$ is given by
\[ ({\overline a}|b)\cdot
 (b_2\otimes b_1) =(\xi\mapsto {\overline a}(b_2\cdot \xi)\big| b \cdot
b_1) \]

We draw pictures of this action as follows. We draw a pair
$({\overline a}|b)$ where $a$ and $b$ are pure algebra elements, by drawing first a
graphical representation of $a$ (dual to ${\overline a}$) on top of a
graphical representation for $b$.  In this picture, the right
idempotent of $b$ is on the bottom, and the right idempotent of
${\overline a}$ is the initial state (on the top) of $a$.
See Figure~\ref{fig:AApic}.

\begin{figure}[ht]
  \input{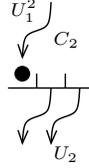}
\caption{\label{fig:AApic}
{\bf{Graphical representative of elements of $Y$.}}
Here we have a picture of $({\overline{L_1 \cdot U_1^2 \cdot C_2}}|L_1\cdot L_2 \cdot U_2)$
in $Y$, thought of as an element of ${\overline{\AlgB_2(2,1,\{1,2\})}}\otimes \AlgB_1(2,2,\emptyset)$.}
\end{figure}

\subsubsection{An example: }

Consider $\AlgB_1={\mathcal B}(1,1,\{1\})$
and $\AlgB_2={\mathcal B}(1,1,\emptyset)$.
The algebra $\AlgB_2$ has generators $L$, $R$, and $U$, and idempotents
$\Idemp{\{0\}}$  and $\Idemp{\{1\}}$.  Moreover,
\[ \Idemp{\{0\}} \cdot R \cdot \Idemp{\{1\}} = R\qquad 
\Idemp{\{1\}} \cdot L \cdot \Idemp{\{0\}} = L,\]
while $\AlgB_1$ has the same generators, and an additional $C$.
The bimodule decomposes into four summands, according to the right
idempotent
of $({\overline a}|b)$. 

In right idempotent $\Idemp{\{1\}}\otimes \Idemp{\{0\}}$, the complex further decomposes
into summands. One of these summands contains the single element
$({\overline{\Idemp{\{1\}}}} | \Idemp{\{0\}}) $.
Another summand is the square
  \begin{equation}
  \begin{tikzpicture}[x=3cm,y=55pt]
      \node at (0,0) (L) {$({\overline {\Idemp{\{1\}}\cdot U}}| \Idemp{\{0\}})$} ;
      \node at (1,.5) (T) {$({\overline{L}}|L)$};
      \node at (1,-.5)(B) {$({\overline{\Idemp{\{1\}}}}| C\cdot \Idemp{\{0\}}) $};
      \node at (2,0) (R) {$({\overline{\Idemp{\{1\}}}}|U \cdot \Idemp{\{0\}})$};
      \draw[->] (L) to (T) ;
      \draw[->] (L) to (B) ;
      \draw[->](B) to (R);
      \draw[->](T) to (R);
    \end{tikzpicture}
    \label{eq:WtOne}
  \end{equation}
See Figure~\ref{fig:AAidEx} for a picture.
(In fact, there are infinitely many different summands.)
\begin{figure}[ht]
\input{AAidEx.pstex_t}
\caption{\label{fig:AAidEx} 
  Terms in the differential of $Y$.
  We have drawn pairs of algebra elements; the top algebra element should be dualized
  to get the corresponding generator.}
\end{figure}

In right idempotent $\Idemp{\{0\}}\otimes \Idemp{\{0\}}$, there is a collection of acyclic complexes.
For example, in the portion with total weight $1/2$, we have:
  \begin{equation}
  \begin{tikzpicture}[x=3cm,y=55pt]
      \node at (0,0) (L) {$({\overline {R}}| \Idemp{\{0\}})$} ;
      \node at (2,0) (R) {$({\overline{\Idemp{\{0\}}}} |L).$};
      \draw[->] (L) to (R) ;
    \end{tikzpicture}
  \end{equation}
In the portion with total weight $3/2$, we have:
  \begin{equation}
  \begin{tikzpicture}[x=3cm,y=55pt]
    \node at (-1,-1) (URn) {$({\overline{R\cdot U}}|\Idemp{\{0\}})$} ;
    \node at (1,0) (RnC) {$({\overline{R}}|C \cdot \Idemp{\{0\}})$};
      \node at (0,-1) (UnL) {$({\overline{\Idemp{\{0\}} \cdot U}}| L)$} ;
      \node at (1,-1) (RnU) {$({\overline{R}}|U \cdot \Idemp{\{0\}})$};
      \node at (2,0)(nLC) {$({\overline{\Idemp{\{0\}}}}| L \cdot C)$};
      \node at (2,-1) (nLU) {$({\overline{\Idemp{\{0\}}}}|L \cdot U).$};
      \draw[->] (URn) to (RnC) ;
      \draw[->] (URn) to (UnL) ;
      \draw[->] (UnL) to (RnU) ;
      \draw[->] (UnL) to (nLC) ;
      \draw[->](nLC) to (nLU);
      \draw[->](RnU) to (nLU);
      \draw[->] (RnC) to (nLC) ;
      \draw[->] (RnC) to (RnU) ;
    \end{tikzpicture}
    \label{eq:WtThreeHalves}
  \end{equation}

\subsubsection{The candidate is the inverse}

We show now that the candidate inverse from Section~\ref{subsec:CandidateInverse} is indeed the inverse for the canonical
type $DD$ bimodule, by verifying the hypotheses of Lemma~\ref{lem:CandidateIsInverse}.
Continuing notation from earlier, the algebras $\Blg_1$ and $\Blg_2$ are as in Equation~\eqref{eq:SpecifyCanonicalDDAlgebras};
$Y$ is as in Section~\ref{subsec:CandidateInverse}, generated by pairs $({\overline a}|b)$,
with $a\in \Blg_2$ and $b\in\Blg_1$.

\begin{prop}
  \label{prop:RankOneHomology}
  The rank of the homology group of 
  $Y$ is $\binom{m+1}{k}$;
  it is generated by the elements of the form
  $({\overline{\Idemp{\x'}}} |\Idemp{\x} )$,
  where $\x$ and $\x'$ are complementary idempotents.
\end{prop}

\begin{lemma}
  \label{lem:Ysplits}
  The candidate complex $Y$ decomposes into a direct sum of complexes
  $C(Z,\x,\y)$,
  indexed by idempotent states $\x$ and $\y$ and $Z\in (\OneHalf \Z)^m$,
  where $C(Z,\x,\y)$ is the vector space generated by pairs $({\overline a}|b)$
  where $a$ and $b$ are pure algebra elements
  with $\Idemp{\x}\cdot a=a$ and $b\cdot\Idemp{\y}=b$, and $w(a)+w(b)=Z$.
\end{lemma}

\begin{proof}
  Since the left idempotent of $a$ is the right idempotent of
  ${\overline a}$, the splitting by idempotents corresponds to the
  splitting of $Y$ according to right idempotents.  The fact that the
  splitting by $Z$ is well defined follows immediately from the
  definition of $\partial$. (More abstractly, it is a formal
  consequence of the fact that $\delta^1$ on the tautological
  $DD$ bimodule is specified by an algebra element $A=\sum a_i\otimes
  b_i$ where the weight of $a_i$ equals the weight of $b_i$.)
\end{proof}

Write the differential $\partial$ on $Y$ as a sum of terms
$\partial= \sum_{i=1}^m \partial_i$, where $\partial_i$ involves terms
on the $i^{th}$ strand. More precisely, there is a differential
$d_i \colon \AlgB(m,k,\Upwards)\to \AlgB(m,k,\Upwards)$ that vanishes unless $i\in \Upwards$, in which 
case $d_i (C_i)=U_i$, but $d_i(C_j)=0$ for all $j\neq i$ and $d_i(L_j)=d_i(U_j)=d_i(R_j)=0$
for all $j=1,\dots,m$.
Define 
\begin{equation}
  \label{eq:ExplicitPartialI}
  \partial_i({\overline a}|b)=
\left\{\begin{array}{ll}
({\overline a}| d_ib) + 
(R_i \cdot {\overline a}| L_i \cdot b)
+ (L_i \cdot {\overline a}| R_i \cdot b)
+ (U_i\cdot {\overline a}|C_i\cdot b) &{\text{if $i\in \Upwards_1$}} \\
({\overline d_i}{\overline a}| b) + 
(R_i \cdot {\overline a}| L_i \cdot b)
+ (L_i \cdot {\overline a}| R_i \cdot b)
+ (C_i\cdot {\overline a}|U_i\cdot b) &{\text{if $i\in \Upwards_2$,}} 
\end{array}
\right.
\end{equation}
where ${\overline d_i}\colon {\overline \Blg}_2\to {\overline \Blg}_2$ is dual to the differential
$d_i\colon \Blg_2\to \Blg_2$.

\begin{lemma}
  \label{lem:PlaceFiltrations}
  The differential $\partial$ on $Y$ can be written as $\partial = \sum_{i=1}^m \partial_i$,
  where $\partial_i^2 = 0$ and $\partial_i \circ \partial_j=\partial_j \circ \partial_i$ for $i\neq j$.
\end{lemma}

\begin{proof}
  This can be seen directly. A little more conceptually, 
  on homogeneous generators $({\overline a}|b)$ where $a$ and $b$ are pure algebra elements, the functions
  $w_j(a)$ for $j=1,\dots,m$  and 
  \begin{equation} {\mathfrak m}_j({\overline a}|b)=
\left\{\begin{array}{ll}
    \MasFilt_j(b) & {\text{if $j\in\Upwards_1$}} \\
    1-\MasFilt_j(a) &{\text{if $j\in\Upwards_2$},}
\end{array}\right.
    \label{eq:DefOfMuJ}
    \end{equation}
  for $j=1,\dots,m$ induce a filtration on $C(Z,\x,\y)$ by
  $(\OneHalf \Z)^m \times \{0,1\}^m$. Restricting to values with $j\neq i$, we have a filtration
  on $C(Z,\x,\y)$ with   $(\OneHalf \Z)^{m-1} \times \{0,1\}^{m-1}$. 
  The associated graded object is clearly $(C(Z,\x,\y),\partial_i)$.
\end{proof}

\begin{prop}
  \label{prop:HomologyTrivialSomewhere}
  Fix idempotents $\x$ and $\y$ and a total weight $Z$;
  suppose moreover that the weight $Z$ is non-zero. 
  Then there is a position $i$ with the property that 
  $H_*(C(Z,\x,\y),\partial_i)=0$
\end{prop}

Proposition~\ref{prop:HomologyTrivialSomewhere} is proved using the following lemma:

\begin{lemma}
  \label{lem:ComplexSplits}
  The chain complex $(C(Z,\x,\y),\partial_i)$ splits into subcomplexes indexed by
  $(\OneHalf \Z)^{m-1}\times\{0,1\}^{m-1}$ that are spanned by
  $({\overline a}|b)$, where $a$ and $b$ are pure algebra elements
  whose weights $w_j(b)$ with $j\neq i$ and 
  values ${\mathfrak m}_j({\overline a}|b)$
  (as defined in Equation~\eqref{eq:DefOfMuJ} above)
  for all $j\neq i$ are specified.
\end{lemma}

\begin{proof}
  This follows from the form of $\partial_i$: it does not involve any of the $C_j$ with $j\neq i$, and it does
  not change the weight of $b$ away from $i$.
\end{proof}

\begin{lemma}
  \label{lem:HomologyOfCwt1}
  For fixed $Z$ and $i\in\{1,\dots,m\}$, $\x$, $\y$ and $(\OneHalf
  \Z)^{m-1}\times \{0,1\}^{m-1}$, consider the corresponding
  subcomplex $C$ of $(C(Z,\x,\y),\partial_i)$ as in
  Lemma~\ref{lem:ComplexSplits}.  Then $H(C)\neq 0$ precisely when
  there is a single element in $C$.  When the homology of $C$ is
  non-zero, there are pure algebra elements $a$ and $b$, and a single
  generator $({\overline a}|b)$ in $C$; and one of the following
  holds:
  \begin{itemize}
  \item $w_i(a)=w_i(b)=0$
  \item $w_i(a)+w_i(b)=\OneHalf$
  \item $w_i(a)+w_i(b)=1$ and either $a=a_0\cdot C_i$ or $b=b_0\cdot C_i$.
  \end{itemize}
  In particular, writing $z_i=w_i(a)+w_i(b)$, if $H(C(Z,\x,\y),\partial_i)\neq 0$, then $z_i\leq 1$.
\end{lemma}

\begin{proof}
  When $z_i=0$, the statement is clear; otherwise, there are cases, according to the local picture of the summand of
  $C(Z,\x,\y)$ near $i$.
  Specifically, if $C(Z,\x,\y)$ contains an element $({\overline a}|b)$ with $a=\Idemp{\x}\cdot a\cdot \Idemp{\w}$,
  we consider separately the two cases:
  \begin{enumerate}[label=(C-\arabic*),ref=(C-\arabic*)]
  \item \label{case:YesYes} $i-1,i\in\w$.
  \item \label{case:YesNo} Exactly one of $i-1$ or $i$ is in $\w$.
  \end{enumerate}
  There is a third case, where neither $i-1,i\in\w$; but that is symmetric to Case~\ref{case:YesYes} by a symmetry
  that exchanges roles of $\Blg_1$ and $\Blg_2$, and then dualizes $C$.

  There are five further subcases of  Case~\ref{case:YesYes}, 
  as follows.
  if $i-1$ and $i$ are the $j^{th}$ and $j+1^{st}$ terms in the sequence $\w$,
  we further subdivide according to the placement of the 
  $j^{th}$ and the $j+1^{st}$ terms in the sequence of $\x$. 
  Since $\x$ and $\w$ are not too far, the pair $\{j,j+1\}$ must be one of
  $\{i-2,i-1\}$, $\{i,i+1\}$,  
  $\{i-2,i\}$,
  $\{i-1,i+1\}$,
  $\{i-1,i\}$,
  $\{i-2,i+1\}$. 

  The cases where $\{j,j+1\}=\{i-2,i\}$ and $\{i-1,i+1\}$ are
  exchanged by reflection of through a vertical axis (compare the map
  $\VRot$ from Equation~\eqref{eq:DefVRot}), as are $\{i-2,i-1\}$ and
  $\{i,i+1\}$. Dropping two symmetric cases, we arrive at the four pictures
  in the first three columns of Figure~\ref{fig:HomologyOfCwt1}.

  Similarly, in Case~\ref{case:YesNo}, 
  if $i-1\in \w$ (which can be arranged after a vertical reflection)
  is the $j^{th}$ term in $\w$ and $i$ is the $k^{th}$ one in the left idempotent of $b$,
  then we further subdivide according to the placement of the $j^{th}$ and $k^{th}$ terms in $\x$ and $\y$ respectively.
  These possibilities 
  (after eliminating symmetric duplicates)
  are represented in the remaining pictures in Figure~\ref{fig:HomologyOfCwt1}.

  \begin{figure}[ht]
    \input{HomologyOfCwt1.pstex_t}
    \caption{\label{fig:HomologyOfCwt1} {\bf{$C(Z,\x,\y)$ cases.}}}
  \end{figure}

  We consider now Cases~\ref{case:YesYes} and~\ref{case:YesNo} separately, and further subdivide them into subcases
  (that are related to the cases in Figure~\ref{fig:HomologyOfCwt1}).
  
  {\bf{Case~\ref{case:YesYes}, with $z_i=1/2$.}}
  (The first column in Figure~\ref{fig:HomologyOfCwt1}.)
  There is now a single generator of this type, of the form $({\overline a}|b)$ with $w_i(a)=1/2$ and $w_i(b)=0$.
  so the homology is one-dimensional.

  {\bf{Case~\ref{case:YesYes}, with $z_i\not\in\Z$ and $z_i>1/2$.}}
  (Again, this is illustrated in first column in Figure~\ref{fig:HomologyOfCwt1}.)
  This has two further subcases, according to whether $i\in\Upwards_1$ 
  or $\Upwards_2$ (i.e. 
  whether $C_i\in\AlgB_1$ or $\AlgB_2$). If $i\in \Upwards_2$, there are two 
  generators of the complex, $({\overline{a U_i^k}}| b)$
  and $({\overline{a U_i^{k-1} C_i}}| b)$ where $k=z_i-\OneHalf\geq 1$, 
  and $a$ and $b$ are fixed pure algebra elements with
  $w_i(a)=\OneHalf$, $w_i(b)=0$.
  To see that
  both types of terms appear, note that under the present hypotheses
  on $a=a\cdot \Idemp{\w}$  and $\w$, if $a\neq 0$ then $U_i\cdot
  a\neq 0$, as well. Clearly,
  \[ \partial_i  ({\overline{a U_i^k}}| b)=({\overline{a U_i^{k-1} C_i}}| b),\]
  so the homology of the corresponding complex vanishes. When $i\not \in \Upwards_2$, the terms are of the form
  $({\overline{a U_i^n}}| b)$ and $({\overline{a U_i^{n-1}}}| C_i b)$,
  and the homology again vanishes.
  
  {\bf{Case~\ref{case:YesYes}, with $z_i\in\Z$.}}  
  When $i\not\in\Upwards_2$, we either have a single element
  $({\overline{a}}|C_i b)$ with $w_i(a)=w_i(b)=0$ (as in the third column of Figure~\ref{fig:HomologyOfCwt1};
  this can also appear in the second column),
  or elements of the form
  $({\overline{a U_i^j}}| b)$ and $({\overline{a U_i^j}}|C_i b)$, with fixed $a$, $b$ so that 
  $w_i(a)=w_i(b)=0$, where $j=z_i$ and the differential is given by
  \[ \partial_i  ({\overline{a U_i^j}}|b)=
  ({\overline{a U_i^{j-1}}}|  C_i\cdot b);\]
  so the homology is trivial.
  The case where
  $i\in\Upwards_2$ works similarly.  We either have a single element element
  $({\overline{a C_i}} |b)$, $w_i(a)=w_i(b)=0$  or elements of the form
  $({\overline{a U_i^j}}| b)$ and $({\overline{a U_i^{j-1} C_i}}|
  b)$. 
  Once again, the differentials cancel out the homology.

%  This is the 
%  second or third picture in Figure~\ref{fig:HomologyOfCwt1}.
%  When $i\in \Upwards_2$, we have either a single
% When $i\not\in \Upwards_2$, again there is either the single
%  element $({\overline a}|b C_i)$ or $({\overline{a U_i^n}}| b)$ and
%  $({\overline{a U_i^{n-1}}}| C_i \cdot b)$.

  {\bf{Case~\ref{case:YesNo}, with $z_i=\OneHalf$.}}  (This is the
  fifth or the eight column of Figure~\ref{fig:HomologyOfCwt1}.)
  There are two possibly non-zero elements in the complex.  
  These elements have the form
  $({\overline {a \cdot R_i}}|b)$ and $({\overline a}| L_i \cdot b)$,
  where $w_i(a)=w_i(b)=0$.  There are two subcases: in one subcase,
  only one of $a\cdot R_i$ or $L_i\cdot b$ is zero, so we have a
  single generator.  (This case occurs in the fifth column of
  Figure~\ref{fig:HomologyOfCwt1}.)  In the other case, when both
  $a\cdot R_i$ and $L_i \cdot b$ are non-zero (which occurs now only
  in the eighth column of Figure~\ref{fig:HomologyOfCwt1}), we have
  that
    $\partial_i({\overline {a \cdot R_i}}|b)=({\overline a}| L_i \cdot  b)$,
    so the homology is trivial.

  {\bf{Case~\ref{case:YesNo}, with $z_i\not\in\Z$, and 
      $z_i=\OneHalf+t>\OneHalf$.}}
  (These cases are illustrated in the fifth and  eighth columns of Figure~\ref{fig:HomologyOfCwt1}.)
  Suppose that $i\not\in\Upwards_2$.
  We can assume that
  all the generators a corresponding summand of $C(Z,\x,\y)$ have the following form
  $({\overline{a R_i U_i^j}}|b U_i^ k)$,
  $({\overline{a R_i U_i^j}}|b C_i U_i^{k-1})$,
  $({\overline{a U_i^j}}|L_i b U_i^ k)$, and
  $({\overline{a U_i^j}}| L_i b C_i U_i^{k-1})$,
  for fixed $a$ and $b$ with $w_i(a)=w_i(b)=0$, and $j+k+\OneHalf=z_i$.
  
  Note that at least one of $a R_i\neq 0$ or $L_i b\neq 0$, 
  otherwise the chain complex would be trivial.
  If $a R_i \neq 0$ and $L_i b=0$, then there are exactly two
  elements in the chain complex, $({\overline{a R_i U_i^t}}, b)$ and
  $({\overline{a U_i^{t-1}}}, C_i b)$, and the differential cancels them.
  (Again, cases of this kind can occur in either the fifth or the eighth columns of Figure~\ref{fig:HomologyOfCwt1}.)
  If $a R_i=0$ and $L_i b \neq 0$, there are once again two terms,
  $(\overline a, L_i b U_i^t)$
  and $(\overline a, L_i b C_i U_i^{t-1})$, and these two terms cancel in the differential.

  In the remaining case where $a R_i\neq 0$ and $L_i b \neq 0$, we show that
  $H(C(Z,\x,\y),\partial_i)=0$, using a further filtration on the complex, 
  induced by the function on pure algebra elements $a$ and $b$ that
  associates to $({\overline a}|b)$ the $i^{th}$ weight of $b$, $w_i(b)$.
  This induces a grading on the vector space underlying $C(Z,\x,\y)=\bigoplus_{k\in\frac{1}{2}\Z} \Filt_k$.
  Writing $\partial^i_0({\overline a}|b)=({\overline a}|d_i b)$, we have that
  $\partial_i = \partial^i_0 + L$,
  where $\partial^i_0\colon \Filt_k\to\Filt_k$, and $L\colon \Filt_k \to \bigoplus_{\ell>k}\Filt_\ell$.
  (We are using the form of $\partial_i$ from Equation~\eqref{eq:ExplicitPartialI}.) Thus, the associated graded complex on
  $(C(Z,\x,\y),\partial_i)$ is equipped with the differential $\partial^i_0$, and we will show that its homology vanishes.

  When $k\in \Z$, the complex
  $\Filt_k$ is spanned by
  $({\overline{a R_i U_i^j}}, b U_i^ k)$ and $({\overline{a R_i U_i^j}}, b C_i U_i^{k-1})$,
  and they are connected by a differential in $\partial^i_0$. Similarly,
  when $k\geq 1$, the complex $\Filt_{k+\OneHalf}$ contains the two elements when 
  $({\overline{a U_i^j}}, L_i b U_i^ k)$, and
  $({\overline{a U_i^j}}, L_i b C_i U_i^{k-1})$, and these elements are connected by a differential
  in $\partial^i_0$. Thus, $H(\Filt_k,\partial^i_0)=0$ except when
  $k=0$ and $1/2$; and $\Filt_0$ is generated by the single element $(\overline{a R_i U_i^t}|b)$
  and $\Filt_{1/2}$ is generated by the element $(\overline{a U_i^t}|L_i b)$. Thus,
  by a simple filtration argument, $H(C(Z,\x,\y),\partial^i)$ is computed as the homology of a two-dimensional
  vector space generated by these latter two generators
  $(\overline{a R_i U_i^t}|b)$ and $(\overline{a U_i^t}|L_i b)$; and these two elements are connected by a differential.
  This completes the verification
  that $H(C(Z,\x,\y),\partial_i)=0$ in this case. See Equation~\eqref{eq:WtThreeHalves} for an illustration when $z=3/2$
  (and observe that for the diagram, the horizontal coordinate in the plane measures the filtration considered here).
  The case where $i\in\Upwards_2$ works similarly: the homology of $\partial^i_0$ is now supported in $\Filt_{t}$ and $\Filt_{t+1/2}$,
  with generators $({\overline {a R_i}} | U_i^t b)$ and 
  $({\overline {a}} | L_i U_i^t b)$, with a cancelling differential.

  {\bf{Case~\ref{case:YesNo}, with $z_i=1$.}}
  Assume that $i\not\in\Upwards_2$. 
  The chain complex $C(Z,\x,\y)$ is spanned by four vectors
  elements of the form
  \[
  \begin{array}{llll}
    (\overline{a} | C_i b), &(\overline{a} | U_i b),&
    (\overline{a U_i} | b), &(\overline{a R_i} | R_i \cdot b),
  \end{array}
  \]
  where $a$ and $b$ are pure algebra elements with $w_i(a)=w_i(b)=0$.
  One of the following must hold:
  \begin{itemize}
    \item All four elements are non-zero 
      (which can occur in the ninth column of Figure~\ref{fig:HomologyOfCwt1}).
      In this case, 
      \[ 
      \begin{array}{lll}
        \partial_i ({\overline{a U_i}}|b)=({\overline{a R_i}}|R_i\cdot b)
        + ({\overline a}|C_i b) &
        \partial_i ({\overline a}|C_i b)=({\overline a}|U_i b) &
        \partial_i ({\overline{a R_i}}|R_i\cdot  b)=
        ({\overline a}|U_i b),
        \end{array}
\]
        and the complex has trivial homology.
        (See for example Equation~\eqref{eq:WtOne}.)
    \item 
      $a \cdot R_i=0$ and $R_i\cdot b=0$, in which case also
      $a U_i=bU_i =0$, so three of the four vectors are zero,
      and the remaining vector $(\overline{a} | C_i b)$ generates the homology.
    \item $a \cdot R_i=0$ but $R_i\cdot b\neq 0$. 
      In this case $U_i\cdot b\neq 0$ (by Lemma~\ref{lem:NontrivProd}), 
      so
      two of the above elements $(\overline{a} | C_i b)$
      and  $(\overline{a} | U_i b)$ are non-zero. 
      These two elements are connected by a differential, and the homology is trivial.
    \item $a\cdot R_i\neq 0$ and $R_i\cdot b=0$. 
      In this case, once
      again, two of the above elements are zero, and the remaining two
      are connected by a differential, so once again the homology is trivial.
    \end{itemize}      
    The case where $i\in\Upwards_2$ works similarly.

  {\bf{Case~\ref{case:YesNo}, with $z_i\in\Z$ and $z_i>1$.}}  Let $({\overline
    a}|b)$ be some non-trivial generator in the complex, where
  $i\in\w$, so $i+1\not\in\w$. If either $a\cdot R_i=0$ or $R_i\cdot b=0$, then
  the argument works as in earlier cases (i.e. we have a pair of generators that cancel in homology).
  Otherwise, as in the case where $z_i\not\in\Z$ and $z_i>\OneHalf$, we consider the filtration by the weight $w_i$
  of the $b$ component. If $i\not\in\Upwards_2$, then  for $\Filt_k$ with $k\geq 1$, $H(\Filt_k,\partial^i_0)=0$, since the complex 
  $\Filt_k$ has two terms in it that are connected by a differential in $\partial^i_0$.
  The remaining two terms are the generators of $\Filt_k$ with $k=0$ and $\OneHalf$, which have the form
  $(\overline{a U_i^t} | b)$ and
  $(\overline{a U_i^{t-1} R_i} | R_i\cdot b)$; and these two terms are connected by a differential.
  The case where $i\in\Upwards_2$ works similarly.
\end{proof}

\begin{proof}[Proof of Proposition~\ref{prop:HomologyTrivialSomewhere}]
  Let $f\colon \{1,\dots,m+1\}\to \Z$ be defined by
  \[ f(i) = \#\{j\big| j<i~\text{and}~j\in \x\} + 
  \#\{j\big| j<i~\text{and}~j\in \y\}\]
  It is easy to see that $i-1\leq f(i)\leq i+1$.

  Choose $n$ minimal so that the weight $z_n\neq 0$. 
  There are three cases:

  {\bf Case 1: $f(n)=n$.} (Note that in this case, $z_i\in\Z$.) We claim that
  $H_*(C(Z,\x,\y),\partial_n)=0$.  By Lemma~\ref{lem:HomologyOfCwt1},
  we need only consider cases where $z_n=1$.  

  See Figure~\ref{fig:fnnCases} when $n-1\in\x$; the cases where $n-1\in\y$ work similarly.
  \begin{figure}[ht]
    \input{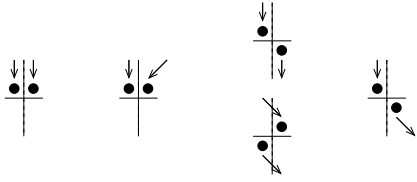}
    \caption{\label{fig:fnnCases} {\bf{Four cases where $f(n)=n$.}}
      This is Case~1 from the proof of Proposition~\ref{prop:HomologyTrivialSomewhere}.}
  \end{figure}

  We claim that in all the cases
  from the figure, we will have an element of the form
  $(\overline{a\cdot U_n}|b)$ (where $w_n(a)=0$) in the subcomplex.
  This is clear because in each case, either $n$ is not contained in 
  a generating interval, or the generating interval containing $n$ also
  contains $n-1$; but $w_{n-1}(a\cdot U_n)=0$. 
  Thus, Lemma~\ref{lem:HomologyOfCwt1}
  completes the case.
  
  {\bf Case 2: $f(n)=n+1$.} (See Figure~\ref{fig:fnnP1Case}.) Again, we claim that $H_*(C(Z,\x,\y),\partial_n)=0$. 
  Note that in this case, $z_n\in\OneHalf+\Z$; so
  by
  Lemma~\ref{lem:HomologyOfCwt1}, we can assume $z_n=\OneHalf$. The
  two possible chain complex generators are $({\overline a}| L_n \cdot b)$ and $({\overline{a \cdot R_n}}|b)$.
  Since by assumption $z_i=0$ for $i< n$, and 
  clearly $n$ is not the left endpoint of any generating interval (in either $\Blg_1$ or $\Blg_2$),
  it follows that both $a U_n\neq 0$ and $b U_n \neq 0$;
  and correspondingly $a\cdot R_n\neq 0$ and $L_n \cdot b\neq 0$; 
  i.e. both generators are non-zero, so they cancel in homology.

  \begin{figure}[ht]
    \input{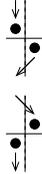}
    \caption{\label{fig:fnnP1Case} {\bf{$f(n)=n+1$.}}
     This is Case~2 from the proof.}
  \end{figure}

  {\bf Case 3: $f(n)=n-1$.}
  Since $f(m+1)=m+1$, we can find a minimal $j> n$ so that $f(j)=j$. Note that $z_{j-1}\equiv \OneHalf\pmod{1}$.
  We distinguish four further subcases. 

  {\bf Case 3a:  $z_{j-1}>\OneHalf$.}  By Lemma~\ref{lem:HomologyOfCwt1}, 
  $H_*(C(Z,\x,\y),\partial_{j-1})=0$.

  {\bf Case 3b: $j<m+1$ and $z_{j-1}=\OneHalf$ and $z_j=0$.}
  We claim that 
  \[H_*(C(Z,\x,\y),\partial_{j-1})=0.\]
   \begin{figure}[ht]
     \input{fnnM1andHalf.pstex_t}
     \caption{\label{fig:fnnM1andHalf} {\bf{$f(n)=n-1$ and $z_{j-1}=\OneHalf$.}}}
   \end{figure}
   The two possible generators are of the form 
   $({\overline a}|R_j\cdot b)$ and 
   $({\overline{a\cdot L_j}}| b)$.
   Since $z_j=0$, it follows that $a \cdot U_{j-1}\neq 0$. and $b\cdot U_{j-1}\neq 0$,
   so both generators are non-zero, and cancel in homology.

   {\bf Case 3c: $j=m+1$ and $z_{m}=\OneHalf$.}
 $H_*(C(Z,\x,\y),\partial_m)=0$, exactly as in Case 3b.

   {\bf Case 3d: $j<m+1$ and $z_{j-1}=\OneHalf$ and $z_j>0$.}
   We will show $H_*(C(Z,\x,\y),\partial_{j})=0$.
   \begin{figure}[ht]
     \input{fnnM1andOne.pstex_t}
     \caption{\label{fig:fnnM1andOne} {\bf{$f(n)=n-1$, $z_{j-1}=\OneHalf$, $z_j=1$.}}}
   \end{figure}
   We have illustrated five cases in Figure~\ref{fig:fnnM1andOne}.
   The remaining cases are symmetric, obtained by switching the roles
   of $a$ and $b$.

   In the first four of these five cases, observe that the
   corresponding chain complex contains a non-zero element of the form
   $(\overline{a \cdot U_j}| b)$ with $w_j(a)=w_j(b)=0$.  Thus, by
   Lemma~\ref{lem:HomologyOfCwt1}, the homology is trivial (since the generating intervals
   on $\Blg_2$ containing $j$ also contain $j-1$).  In the
   final case, the homology is also trivial by
   Lemma~\ref{lem:HomologyOfCwt1}, since $C_j$ is not present in the
   displayed generator.
\end{proof}

\begin{proof}[Proof of Proposition~\ref{prop:RankOneHomology}]
  Decompose $H(Y)$ into the summands $C(Z,\x,\y)$ as before.  Suppose
  that the total weight $Z$ is non-zero, and choose $i$ as in
  Proposition~\ref{prop:HomologyTrivialSomewhere}.  As in the proof of
  Lemma~\ref{lem:PlaceFiltrations}, there is a filtration on the
  complex $C(Z,\x,\y)$ of $Y$, whose associated graded object is
  $(C(Z,\x,\y),\partial_i)$.  It follows at once by an elementary
  spectral sequence argument that $H(C(Z,\x,\y),\partial)=0$ if $Z$ is a non-zero weight vector.

  It remains to consider summands where the weight is zero. These correspond to pairs of complementary 
  idempotents, equipped with a vanishing differential. There are, of course,
  $\binom{m+1}{k}$ such complementary pairs, as claimed.
\end{proof}

\begin{thm}
  \label{thm:DDisInvertible}
  The module $Y$ is a quasi-inverse of the type $DD$
  bimodule $\CanonDD$.
\end{thm}

\begin{proof}
  Note that $\Blg$ is positively graded over $\ground$, using the
  grading set $\Lambda=(\OneHalf \Z)^{m}$.
  Conditions~\ref{Koszul:Grading}-\ref{Koszul:Pos} of
  Definition~\ref{def:KoszulDual} are obvious from the construction of
  $\CanonDD$. The remaining condition of   Lemma~\ref{lem:CandidateIsInverse} 
  is supplied by Proposition~\ref{prop:RankOneHomology}; so the result
  follows from Lemma~\ref{lem:CandidateIsInverse}.
\end{proof}

\subsection{Grading sets associated to one-manifolds}
\label{subsec:OurGradingSets}

Our knot invariant will be constructed by tensoring together bimodules
with a particular kind of grading set. We formalize these grading sets
presently, and study the boundedness needed for forming the tensor product.

Let $W$ be an oriented disjoint union of finitely many intervals,
equipped with a partition of its boundary $\partial W = Y_1\cup Y_2$
into two sets of points. Let $Y_i$ consist of $m_i$ points. Let $s_i$
denote the number of intervals in $W$ that connect $Y_i$ to itself,
and $s_0$ denote the number of intervals that connect $Y_1$ to $Y_2$
in $W$.  Let $\Upwards_1$ be those points in $Y_1$ for which the oriented boundary of $W$
appears with positive multiplicity in the
oriented boundary of $W$, and let $\Upwards_2$ be those points in $Y_2$ for which 
the oriented boundary of $W$ appears with negative multiplicity in $W$.
Choose any integer $0\leq s\leq s_0+1$.
Let $\Blg_1=\Blg(m_1,s+s_1,\Upwards_1)$, $\Blg_2(m_2,s+s_2,\Upwards_2)$.

We can think of the Alexander multi-grading of
$\Blg_i$ as taking values in $H^0(Y_i;\Q)$: the weights of algebra elements
are functions on the points in $Y_i$. The sum of grading groups $H^0(Y_1;\Q)\oplus H^0(Y_2;\Q)=H^0(\partial W;\Q)$
act on $H^1(W,\partial W;\Q)$, via the coboundary map $d^0\colon H^0(\partial W)\to H^1(W,\partial W)$.

\begin{remark}
  In fact, the grading on the algebras is supported in
  $H^0(Y_i;\OneHalf\Z)\subset H^0(Y_i;\Q)$.  Also, the grading set for our modules is 
  contained in $H^1(W,\partial W;\frac{1}{4}\Z)$; compare
  Equation~\eqref{eq:GradeCrossing}.
\end{remark}

\begin{defn}
  \label{def:Adapted}
  Fix $W$ as above. 
  A type $DA$ bimodule $\lsup{\Blg_2}X_{\Blg_1}$ is called {\em adapted to $W$}
  if it has a $\Z$-grading (compatible with the Maslov grading on the algebra)
  and an Alexander multi-grading, with grading set $H^1(W,\partial W)$ as
  described above, and it is finite dimensional (as a vector space).
\end{defn}

\begin{prop}
  \label{prop:AdaptedTensorProducts}
  Let $W_1$ be a disjoint union of finitely many intervals joining
  $Y_1$ to $Y_2$; and let $W_2$ be a disjoint union of finitely many
  intervals joining $Y_2$ to $Y_3$.  Suppose moreover that $W_1\cup
  W_2$ has no closed components, i.e. it is a disjoint union of
  finitely many intervals joining $Y_1$ to $Y_3$.  Given any two
  bimodules $\lsup{\Blg_2}X^1_{\Blg_1}$ and
  $\lsup{\Blg_3}X^2_{\Blg_2}$ adapted to $W_1$ and $W_2$
  respectively, we can form their tensor product
  $\lsup{\Blg_3}X^2_{\Blg_2}\DT~\lsup{\Blg_2}X^1_{\Blg_1}$ (i.e. the infinite sums in its
  definition are finite); and moreover, it is
  a bimodule that is adapted to $W_1\cup W_2$.
\end{prop}

\begin{proof}
  Recall that the grading set of $X^2\DT X^1$ is 
  $H^1(W_2,Y_3\cup Y_2)\oplus H^1(W_1,Y_2\cup Y_1)$ modulo the coboundary of $H^0(Y_2)$, which is identified
  with $H^1(W_2\cup W_1,Y_3\cup Y_1)$.
  Clearly, the tensor product is finite-dimensional.

  Next, we argue the necessary finiteness.
  Fix algebra elements $(a_1,\dots,a_\ell)\in \Blg_1$.
  Fix generators $\x_1,\y_1$ for $X^1$ and $\x_2,\y_2$ for $X^2$.
  Suppose that $(b_1\otimes\dots \otimes b_j)\otimes \y_1$ appears in 
  $\delta^j_i(\x_1,a_1,\dots,a_\ell)$. 
  Since $X^1$ is graded, 
  for each point $i$ in $Y_2$ that is matched by $W_1$ to a point in $Y_1$, we have a constant $K$,
  depending only on the gradings of $(a_1,\dots,a_\ell)$ and $\x_1$ and $\y_1$, so that
  \[ |w_i(b_1\otimes\dots\otimes b_j)|\leq K.\]
  (In more detail, consider the component of $W_1$ that matched $i\in Y_2$ with some $i'\in Y_1$,
  and let $\xi_i$ and $\eta_i$ be the coefficient of $\gr(\x_1)$ and $\gr(\y_1)$ in that component.
  Since $X_1$ is graded, 
  \[ w_i(b_1\otimes\dots\otimes b_j)+\eta_i =
  \xi_i+w_i(a_1\otimes\dots \otimes a_{\ell}).\] Now let
  $K=|\xi_i+w_i(a_1\otimes\dots\otimes a_{\ell})-\eta_i|$.)  By the
  same reasonining, we can adjust $K$ so that, for any two points
  $i$ and $i'$ in $Y_2$ that are matched in $W_1$,
  \[ |w_{i}(b_1\otimes\dots\otimes b_j)-w_{i'}(b_1\otimes\dots\otimes b_j)|\leq K.\]
  Suppose that $c\otimes \y_2$ appears with non-zero multiplicity in 
  $\delta^1_{j+1}(\x_2,b_1,\dots,b_j)$, then for any point $i'$ in $Y_2$ that is matched by $W_2$ to $Y_3$,
  we can further adjust $K$ (depending now on $\x_2$ and $\y_2$) so that
  \[ |w_i(b_1\otimes\dots\otimes b_j)-w_{i'}(c)|\leq K.\] By further
  adjusting $K$ if necessary (depending only on $\x_2$ and $\y_2$), we can arrange that for any two points
  $i$ and $i'$ in $Y_2$ that are matched in $W_2$,
  \[ |w_i(b_1\otimes\dots\otimes b_j)-w_{i'}(b_1,\dots,b_j)|\leq K.\]

  Grading properties ensure that
  \begin{equation}
    \label{eq:MaslovBound}
    \Maslov(\x_1)+\Maslov(\x_2)-\Maslov(\y_1)-\Maslov(\y_2)+\sum_{i=1}^{\ell} \Maslov(a_i) =\Maslov(c)+\ell-1
  \end{equation}

  By the above considerations, for any point $i$ in $Y_2$ that is
  contained in a path connected component of $W_2\cup W_1$ that meets
  $Y_1$, there is a bound on $w_i(b_1\otimes\dots\otimes b_j)$
  depending only on $\x_1$, $\x_2$, $\y_1$, $\y_2$, and
  $(a_1,\dots,a_{\ell})$. 

  Suppose next that $i\in Y_2$ is contained in a component of $W_2
  \cup W_1$ that meets $Y_3$ but not $Y_1$.  Consider the inital point
  $p$ of that arc, with respect to the orientation it inherits from
  $W$, and observe that $p\in\Upwards_3$.  The above considerations
  give a bound on $|w_i(b_1\otimes\dots\otimes b_j)-w_p(c)|$, again
  depending only on $\x_1$, $\x_2$, $\y_1$, $\y_2$, and
  $(a_1,\dots,a_{\ell})$. Finally, observe that since $p\in
  \Upwards_3$, Equation~\eqref{eq:DefMaslov} shows that we can adjust
  $K$ so that (depending on $|\Upwards_2|$) with 
  $|\Maslov(c)+2 w_p(c)|\leq K$. An upper bound on $\Maslov(c)$ is
  provided by Equation~\eqref{eq:MaslovBound}.  Thus, we have an upper
  bound on $w_i(b_1\otimes \dots\otimes b_j)$ for any $i$ in $Y_2$
  that is contained in component of $W_2\cup W_1$ that meets $Y_3$ but
  not $Y_1$.

  Since $W_2\cup W_1$ has no closed components, we have obtained a
  universal bound on any weight $w_i(b_1\otimes\dots\otimes b_j)$ for
  $i\in Y_2$. 
  That upper bound implies also an upper bound on $j$; for if 
  $j$ could be arbitrarily large, we would be able to find
  arbitrarily large $k$ so that $\delta^k(\y')=b_{m}\otimes\dots
  \otimes b_{m+k}\otimes \y''$.  (Here, $\delta^k\colon X^1\to \Blg_2^{\otimes
    k}\otimes X^1$ is the map obtained by iterating $\delta^1$.)  But
  \[ \Maslov(\y')-k = \Maslov(b_m)+\dots+\Maslov(b_{m+k})+
  \Maslov(\y''),\] so since $X^1$ is finitely generated, we conclude
  that $\Maslov(b_m)+\dots+\Maslov(b_{m+k})$ can be arbitrarily large,
  contradicting the above obtained universal bound on the weight of
  $b_1\otimes\dots\otimes b_j$ at each point in $Y_2$, in view of
  Definition~\ref{eq:DefMaslov}.
  
  The resulting bound on $j$ shows that the coefficient of $\y_1\DT \y_2$ in
  $\partial(\x_1\DT \x_2)$ is a sum of finitely many terms.  The result now follows
  since 
  $X^2\DT X^1$ is finitely generated.
  (See Figure~\ref{fig:GradingSets}.)
\begin{figure}[ht]
\input{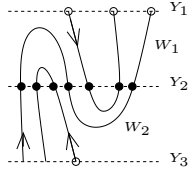}
\caption{\label{fig:GradingSets} 
  In the proof of Proposition~\ref{prop:AdaptedTensorProducts},
  the total weights of the algebra elements at the circled points in $Y_1$ and $Y_3$ 
  give bounds on the weights at points of $Y_2$ (drawn as dark dots) 
  of the algebra elements in $b_1\otimes\dots\otimes b_j$.}
\end{figure}
\end{proof}

\section{$DD$ bimodules for crossings}
\label{sec:CrossingDD}

Having defined the algebra associated to $y=t$ slice of the generic
diagram, we turn now to the definitions of the modules associated to
the partial knot diagrams. In Section~\ref{sec:CrossingDA}, we will construct
$DA$ bimodules associated to special partial knot diagrams, consisting
of a collection of vertical strands passing through the region $y_1\leq y \leq y_2$ in the plane,
containing exactly one crossing.  Before doing this, we
construct presently a simpler type $DD$ bimodule associated to such a
configuration.

\subsection{The $DD$ bimodule of a positive crossing}
We describe first the $DD$ bimodule $\Pos_i$ associated to a positive
crossing between the $i^{th}$ and $(i+1)^{st}$ strands.

Let $\tau\colon \{1,\dots,m\}\to  \{1,\dots,m\}$ 
be the transposition that 
switches $i$ and $i+1$. 
Let 
\begin{equation}
  \label{eq:DefB1B2}
  \Blg_1=\Blg(m,k_1,\Upwards_1)~\qquad{\text{and}}\qquad\Blg_2=\Blg(m,k_2,\Upwards_2),
\end{equation} where
$k_1+k_2=m+1$, $|\Upwards_1|+|\Upwards_2|=m$, and $\Upwards_1\cap \tau(\Upwards_2)=\emptyset$.
We think of the algebra $\Blg_1$ as coming from above the crossing  and algebra $\Blg_2$ as coming from below.

\begin{figure}[ht]
\input{PosCrossDD.pstex_t}
\caption{\label{fig:PosCrossDD} {\bf{Positive crossing $DD$ bimodule generators.}}  
The four generator types are pictured to the right.}
\end{figure}

As an $\IdempRing(m,k_2,\Upwards_2)$-$\IdempRing(m,k_1,\Upwards_1)$-bimodule, $\Pos_i$ is the submodule
of $\IdempRing(m,k_2,\Upwards_2)\otimes_{\Field}\IdempRing(m,k_1,\Upwards_1)$ generated by elements
$\Idemp{\y}\otimes \Idemp{\x}$ where one of the following two conditions holds:
\begin{itemize}
  \item $\x\cap\y=\emptyset$ or 
  \item $\x\cap\y=\{i\}$ and $\{0,\dots,m\}\setminus (\x\cup \y) =\{i-1\}~\text{or}~\{i+1\}$.
\end{itemize}

In a little more detail, generators correspond to certain pairs of
idempotent states $\x$ and $\y$, where $|\x|=k_1$ and $|\y|=k_2$. They
are further classified into four types, $\North$, $\South$, $\West$,
and $\East$. For generators of type $\North$ the subsets $\x$ and $\y$
are complementary subsets of $\{0,\dots,m\}$, with $i\in \y$.  For
generators of type $\South$, $\x$ and $\y$ are complementary subsets
of $\{0,\dots,m\}$ with $i\in \x$.  For generators of type $\West$,
$i-1\not\in \x$ and $i-1\not\in \y$, and $\x\cap\y=\{i\}$.  For
generators of type $\East$, $i+1\not\in \x$ and $i+1\not\in \y$, and
$\x\cap\y=\{i\}$. 

The differential has the following types of terms:
\begin{enumerate}[label=(P-\arabic*),ref=(P-\arabic*)]
\item 
  \label{type:OutsideLRP}
  $R_j\otimes L_j$
  and $L_j\otimes R_j$ for all $j\in \{1,\dots,m\}\setminus \{i,i+1\}$; these connect
  generators of the same type. 
\item
  \label{type:UCP}
  $U_{\tau(j)}\otimes C_j$ if $j\in \Upwards_1$
  and $C_{\tau(j)}\otimes U_j$ if $j\not\in \Upwards_1$; these connect generators of the same type.
\item 
  \label{type:InsideP}
  Terms in the diagram below connect  generators
  of different types:
  \begin{equation}
    \label{eq:PositiveCrossing}
    \begin{tikzpicture}[scale=1.7]
    \node at (0,3) (N) {$\North$} ;
    \node at (-2,2) (W) {$\West$} ;
    \node at (2,2) (E) {$\East$} ;
    \node at (0,1) (S) {$\South$} ;
    \draw[->] (S) [bend left=7] to node[below,sloped] {\tiny{$R_i\otimes U_{i+1}+L_{i+1}\otimes R_{i+1}R_i$}}  (W)  ;
    \draw[->] (W) [bend left=7] to node[above,sloped] {\tiny{$L_{i}\otimes 1$}}  (S)  ;
    \draw[->] (E)[bend right=7] to node[above,sloped] {\tiny{$R_{i+1}\otimes 1$}}  (S)  ;
    \draw[->] (S)[bend right=7] to node[below,sloped] {\tiny{$L_{i+1}\otimes U_i + R_i \otimes L_{i} L_{i+1}$}} (E) ;
    \draw[->] (W)[bend right=7] to node[below,sloped] {\tiny{$1\otimes L_i$}} (N) ;
    \draw[->] (N)[bend right=7] to node[above,sloped] {\tiny{$U_{i+1}\otimes R_i + R_{i+1} R_i \otimes L_{i+1}$}} (W) ;
    \draw[->] (E)[bend left=7] to node[below,sloped]{\tiny{$1\otimes R_{i+1}$}} (N) ;
    \draw[->] (N)[bend left=7] to node[above,sloped]{\tiny{$U_{i}\otimes L_{i+1} + L_{i} L_{i+1}\otimes R_i$}} (E) ;
  \end{tikzpicture}
\end{equation}
\end{enumerate}

Note that for a generator of type $\East$, the terms of
Type~\ref{type:OutsideLRP} with $j=i+2$ vanish; while for one of type
$\West$, the terms of Type~\ref{type:OutsideLRP} with $j=i-1$ vanish.

% \begin{figure}[ht]
% \input{PosCrossDDmod.pstex_t}
% \caption{\label{fig:PosCrossDDmod} {\bf{Positive crossing $DD$ bimodule arrows.}}  
% The four generator types are pictured to the right.}
% \end{figure}

\begin{prop}
  \label{prop:PosIsDD}
  The bimodule $\lsup{\Blg_2,\Blg_1}\Pos_i$ is a $DD$ bimodule.
\end{prop}

\begin{proof}
  The square of the differential, which we must verify vanishes, is obtained by either
  differentiating any of the terms of the Types~\ref{type:OutsideLR}-\ref{type:InsideP} above or,
  multiplying together two of them.

  We start by analyzing the terms of Types~\ref{type:InsideP}. Clearly, all of those terms have vanishing differential.
  The non-zero algebra elements obtained as products of pairs of such elements either connect generators of
  any type to itself, or it connects $\North$ and $\South$ or $\West$ and $\East$.
  
  As an example,
  consider products of terms of Type~\ref{type:InsideP} connecting $\West$
  to itself. Those products that factor through $\North$ give terms
  \[ (1\otimes L_i)\cdot (U_{i+1}\otimes R_i + R_{i+1} R_i \otimes L_{i+1}) 
  = U_{i+1}\otimes U_{i} + R_{i+1} R_{i}\otimes L_{i} L_{i+1}\]
  and those that factor through $\South$ give
  \[ (L_i \otimes 1) \cdot (R_i \otimes U_{i+1} + L_{i+1} \otimes R_{i+1} R_{i})
  = U_i\otimes U_{i+1} + L_{i} L_{i+1}\otimes R_{i+1} R_i.\] 
  Note that $(R_{i+1}R_i \otimes L_i L_{i+1} + L_{i} L_{i+1}\otimes R_{i+1} R_i)\otimes \West$
  vanishes for idempotent reasons: if $\x$ and $\y$ are pairs of idempotent states with
  $(\Idemp{\x}\otimes  \Idemp{\y})\cdot \West\neq 0$, then
  \[|\x \cap \{i-1,i,i+1\}| + |\y \cap \{i-1,i,i+1\}| = 3.\]
  The terms
  $U_i\otimes U_{i+1}$ and $U_{i+1}\otimes U_i$ cancel with the
  differentials of terms of Type~\ref{type:UC} with $j=i$ and $i+1$.

  Next consider products of terms of Type~\ref{type:InsideP} connecting $\North$ to $\South$.
  The products that factor through $\West$ give terms:
  \[(U_{i+1}\otimes R_i  + R_{i+1} R_i\otimes L_{i+1}) \cdot  (L_i\otimes 1)
  = L_{i} U_{i+1} \otimes R_i   + R_{i+1} U_i\otimes L_{i+1}.\]
  Similarly, the products that factor through $\East$ give
  \[ (U_i\otimes L_{i+1} + L_{i} L_{i+1} \otimes R_i)  \cdot (R_{i+1}\otimes 1)
  = R_{i+1} U_{i} \otimes L_{i+1}  + L_i U_{i+1}  \otimes R_i ; \]
  thus, these two factorizations give terms that cancel in pairs.

  The cancellation of terms of Type~\ref{type:InsideP} in pairs or with differentials
  of terms of Type~\ref{type:UC} with $j=i$ and $i+1$ proceeds to other types of generators similarly.

  Next, we consider products of pairs of terms of
  Type~\ref{type:OutsideLR}.  When $j\neq i, i+1$, these terms cancel
  differentials of terms  of Type~\ref{type:UC} exactly as in
  the proof Lemma~\ref{lem:CanonicalIsDD}. When the generators are of
  type $\East$ and $\West$, there is a possible complication in this argument,
  since in that case, the corresponding term with $j=i+2$ or $i-1$ of
  Type~\ref{type:OutsideLR} might vanish. For example, for terms of
  type $\East$, and $i+2\not\in\Upwards_1$, there is a possibly non-zero
  term $C_{i+2} \otimes U_{i+2}$ of Type~\ref{type:UC}, but the terms
  of Type~\ref{type:OutsideLR} of the form $L_{i+2}\otimes R_{i+2}$
  and $R_{i+2}\otimes L_{i+2}$ vanish. However, the differential of
  $C_{i+2}\otimes U_{i+2}$ vanishes in the idempotents of $\East$ (since
  if $\Idemp{\y}\otimes \Idemp{\x}$ is of type $\East$, then
  $i+1\not\in\x$ and $i+1\not\in\y$, and either $i+2\not\in\x$ or
  $i+2\not\in\y$; in either case, $(U_{i+2}\otimes U_{i+2})\cdot
  (\Idemp{\y}\otimes\Idemp{\x})=0$).

  Consider next terms that are products of terms of Type~\ref{type:InsideP} and those of
  Type~\ref{type:OutsideLR}.
  Typically, these are easily seen to cancel in pairs; 
  the case where support of the Type~\ref{type:OutsideLR} is immediately next to $\{i,i+1\}$ requires special care.
  Consider for example the terms of the form $L_{i+2}\otimes R_{i+2}$. Each term of Type~\ref{type:InsideP} commutes
  with this term; but both product might be zero. For example, $L_i\otimes 1$ commutes with 
  $(L_{i+2}\otimes R_{i+2})$. Also, if $a=L_i L_{i+1}\otimes R_i$, then 
  $(L_{i+2}\otimes R_{i+2})\cdot a = 0$ and $a\cdot (L_{i+2}\otimes R_{i+2})=0$.

  The remaining terms of Type~\ref{type:UC} and are easily seen to commute with each other and with terms
  of Type~\ref{type:InsideP}, giving the desired cancellation.
\end{proof}

It is interesting to note that the bimodule has some symmetries; for instance,
\begin{equation}
  \label{eq:SymmetryOfPos}
  \lsup{\Blg_2,\Blg_1}\Pos_i\cong \lsup{\Blg_1,\Blg_2}\Pos_i,
\end{equation}
by a symmetry which switches the roles of $\North$ and $\South$, and fixes $\West$ and $\East$.

\subsection{$DD$ bimodule for a negative crossing}

We can define a type $DD$ bimodule for a negative crossing. 
The generators are the same as for a positive crossing. 
Terms in the differential are also the same, except that those
of Type~\ref{type:InsideP} are replaced by the following:

\begin{equation}
  \label{eq:NegativeCrossing}
\begin{tikzpicture}[scale=1.7]
    \node at (0,3) (N) {$\North$} ;
    \node at (-2,2) (W) {$\West$} ;
    \node at (2,2) (E) {$\East$} ;
    \node at (0,1) (S) {$\South$} ;
    \draw[->] (W) [bend left=7] to node[above,sloped] {\tiny{${U_{i+1}}\otimes{L_i}+{L_{i} L_{i+1}}\otimes{R_{i+1}}$}}  (N)  ;
    \draw[->] (N) [bend left=7] to node[below,sloped] {\tiny{${1}\otimes{R_{i}}$}}  (W)  ;
    \draw[->] (N)[bend right=7] to node[below,sloped] {\tiny{${1}\otimes{L_{i+1}}$}}  (E)  ;
    \draw[->] (E)[bend right=7] to node[above,sloped] {\tiny{${U_i}\otimes{R_{i+1}} + {R_{i+1} R_{i}}\otimes{L_i}$}} (N) ;
    \draw[->] (S)[bend right=7] to node[above,sloped] {\tiny{${R_i}\otimes{1}$}} (W) ;
    \draw[->] (W)[bend right=7] to node[below,sloped] {\tiny{${L_i}\otimes{U_{i+1}} + {R_{i+1}}\otimes{L_{i} L_{i+1}}$}} (S) ;
    \draw[->] (S)[bend left=7] to node[above,sloped]{\tiny{${L_{i+1}}\otimes{1}$}} (E) ;
    \draw[->] (E)[bend left=7] to node[below,sloped]{\tiny{${R_{i+1}}\otimes{U_{i}} + {L_i}\otimes{R_{i+1} R_{i}}$}} (S) ;
  \end{tikzpicture}
\end{equation}

\begin{prop}
  The bimodule $\Neg_i$ is a $DD$ bimodule.
\end{prop}

\begin{proof}
  This follows exactly as in the proof of
  Proposition~\ref{prop:PosIsDD}.  Note that
  Diagram~\eqref{eq:NegativeCrossing} is obtained from
  Diagram~\ref{eq:PositiveCrossing} by reversing all the arrows, and
  switching the roles of $L_j$ and $R_j$.  More formally, for $\Blg_1$
  and $\Blg_2$ as in Equation~\eqref{eq:DefB1B2}, 
  $\lsup{\Blg_1,\Blg_2}\Neg_i$ is obtained from the opposite module of $\Pos_i$,
  ${\overline\Pos}_i^{\Blg_1,\Blg_2}=\lsup{{\Blg_1^{\op},\Blg_2^{\op}}}{\overline\Pos}_i$
  using the isomorphism
  $\Blg_i\cong\Blg_i^\op$ (denoted $\Opposite$ in Equation~\eqref{eq:OppositeIsomorphism}).
\end{proof}

\subsection{Gradings}

The bimodules $\Pos_i$ are graded by the set $\MGradingSet=\Q^m$ as follows
Let $e_1,\dots,e_m$ be the standard basis for $\Q^m$. 
Let 
\begin{equation}
  \label{eq:GradeCrossing}
  \begin{array}{llll}
  \gr(\North)=\frac{e_i + e_{i+1}}{4} & 
  \gr(\West)=\frac{e_{i}-e_{i+1}}{4} & 
  \gr(\East)=\frac{-e_{i}+e_{i+1}}{4} &
  \gr(\South)=\frac{-e_{i}-e_{i+1}}{4},
  \end{array}
\end{equation}
Now,
if $(a\otimes b)\otimes Y$ appears in $\partial X$, then 
\begin{equation}
  \label{eq:WeightGradingDD}
  \gr(X)=\gr(a)-\tau^{\gr}\gr(b)+\gr(Y),
\end{equation}
where $\tau_i^{\gr}$ is the linear transformation acting by $\tau_i$
on the standard basis vectors, and $\gr(a)=(w_1(a),\dots,w_m(a))$.
Verifying Equation~\eqref{eq:WeightGradingDD} is straightforward,
using Equation~\eqref{eq:PositiveCrossing}.

In the notation of Section~\ref{sec:PrelimGradings}, the grading set of $\Pos_i$ is
half-integral valued functions on the
the arcs in the diagram, thought of as an affine space for $\Lambda_{\Blg_1}\times\Lambda_{\Blg_2}$
in an obvious way.

Similarly, for $\Neg_i$, the gradings of the four generators $\North$, $\South$, $\East$, $\West$ are given by
\begin{equation}
  \label{eq:GradeNegCrossing}
  \begin{array}{llll}
  \gr(\North)=\frac{-e_i - e_{i+1}}{4} 
  &  \gr(\West)=\frac{-e_i+e_{i+1}}{4} 
  & \gr(\East)=\frac{e_i-e_{i+1}}{4} 
  & \gr(\South)=\frac{e_{i}+e_{i+1}}{4},
  \end{array}
\end{equation}

Note that the modules determine these gradings only up to an overall additive shift. 
The present gradings are consistent with conventions on the multivariable Alexander polynomial~\cite{Kauffman};
for further motivation, see Remark~\ref{rem:MotivateGradings}.

\subsection{Motivation}
\label{subsec:Motivation}

The modules $\lsup{\Blg_2,\Blg_1}\Pos_i$ and
$\lsup{\Blg_2,\Blg_1}\Neg_i$ came from considering a version of
Lagrangian Floer homology in the four-punctured sphere, shown in
Figure~\ref{fig:HeegaardPicture}, taken with respect to a single
$\beta$-circle and four $\alpha$-arcs that go out to the
punctures. Regions in the complement of this configuration are
labelled by algebra elements (as shown in
Figure~\ref{fig:HeegaardPicture}, for the negative crossing).

\begin{figure}
  \input{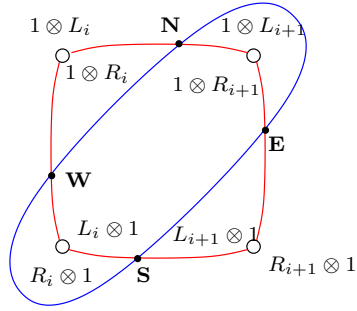}
  \caption{\label{fig:HeegaardPicture}
  {\bf Four-punctured sphere for a negative crossing.} 
  The four punctures are connected by four arcs as shown;  an auxiliary circle
  meets the four arcs in the points $\North$, $\West$, $\South$, and $\East$.
  Regions are labelled by algebra elements as shown.}
\end{figure}

The arrows in
Equation~\eqref{eq:NegativeCrossing} count rigid holomorphic bigons,
and their contribution is obtained by multiplying the algebra elements
specified by the regions; see Figure~\ref{fig:HeegaardDifferentials} for some examples.
This relationship is central to ~\cite{HolKnot}.

\begin{figure}
  \input{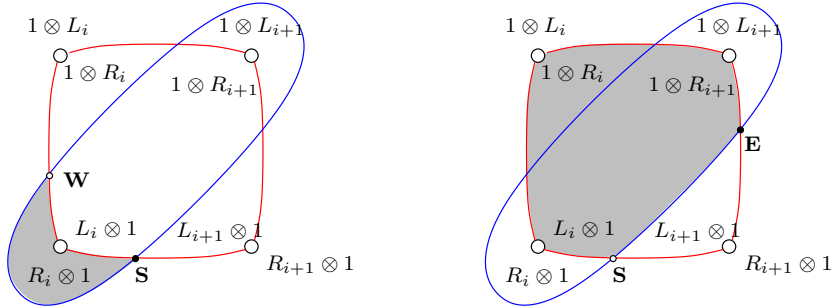}
  \caption{\label{fig:HeegaardDifferentials}
  {\bf{Terms in Equation~\eqref{eq:NegativeCrossing} as arising from disks.}}
  The shaded region on the left represents the term $(R_i\otimes 1)\otimes \West$
  in $\delta^1(\South)$; the one on the right represents the
  term $(L_i\otimes R_{i+1} R_i)\otimes \South$ in $\delta^1(\East)$.}
\end{figure}

\newcommand\LocPos{P}
\newcommand\Initial{I}
\newcommand\ed{\widetilde\delta}
\newcommand\Xelt{X}
\newcommand\Yelt{Y}
\newcommand\Zelt{Z}
\section{$DA$ bimodules associated to crossings}
\label{sec:CrossingDA}

Our goal here is to construct certain $DA$ bimodules associated to crossings,
which are in a suitable sense Koszul dual to the $DD$ bimodules 
constructed in Section~\ref{sec:CrossingDD}; see Lemma~\ref{lem:CrossingDADD}
for the precise duality statement.

For the construction,
choose integers $k$ and $m$ with $0\leq k\leq m+1$, and let
$\Upwards\subset \{1,\dots,m\}$ be arbitrary.  Let $\tau_i\colon
\{1,\dots,m\}\to \{1,\dots,m\}$ be the map that transposes $i$ and
$i+1$. 
Let $\Blg_1=\Blg(m,k,\Upwards)$ and $\Blg_2=\Blg(m,k,\tau(\Upwards))$.
(Note that $\Blg_1$ and $\Blg_2$ here denote different algebras than they did  in Section~\ref{sec:CrossingDD}.)
The aim of the present section is to construct a type $DA$
bimodule $\lsup{\Blg_2}\Pos^i_{\Blg_1}$,
which we think of as the bimodule associated to a region in the knot
diagram $t_2\leq y\leq t_1$ that contains exactly one ``positive''
crossing, and no local maxima or minima; and the crossing occurs
between the $i^{th}$ and $(i+1)^{st}$ strands, as shown on the left in
Figure~\ref{fig:PosCrossDA}. (Positivity here is meant with respect to
a braid orientation, e.g. where all the strands are oriented upwards,
which might differ from the orientation specified by $\Upwards$.)

\begin{figure}[ht]
\input{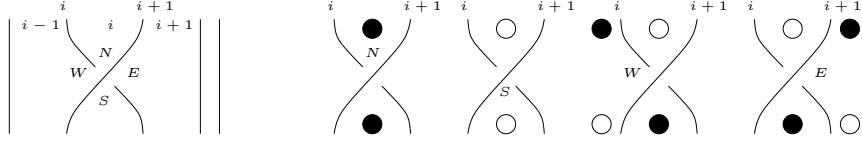}
\caption{\label{fig:PosCrossDA} {\bf{Positive crossing $DA$ bimodule generators.}}  
The four generator types are pictured to the right.}
\end{figure}

Consider the submodule $\Pos^i$ of $\IdempRing{(m,k)}\otimes_{\Field}\IdempRing{(m,k)}$,
consisting of $\Idemp{\x}\otimes \Idemp{\y}$ where either $\x=\y$ or
there is some $\w\subset \{1,\dots,i-1,i+1,\dots,m\}$ with
$\x=\w\cup\{i\}$ and $\y=\w\cup \{i-1\}$ or $\y=\w \cup\{i+1\}$.
Thus, there are once again four types of generators, of type $\North$, $\South$, $\West$, $\East$
as pictured in Figure~\ref{fig:PosCrossDA}; i.e.
\begin{align*}
  \sum_{i\in\x} \Idemp{\x}\cdot \North \cdot \Idemp{\x}= \North, \qquad &
  \sum_{i\not\in\x} \Idemp{\x}\cdot \South \cdot \Idemp{\x} = \South, \\
  \sum_{\begin{tiny}\begin{array}{c}
        i\not\in \x \\
        i-1\in\x
    \end{array}
  \end{tiny}} \Idemp{\{i\}\cup \x\setminus\{i-1\}}\cdot \West \cdot \Idemp{\x}
  =  \West, \qquad &
  \sum_{\begin{tiny}
      \begin{array}{c}
        i+1\in\x \\
        i\not\in\x
        \end{array}
        \end{tiny}}\Idemp{\{i\}\cup \x\setminus\{i+1\}}\cdot \East \cdot \Idemp{\x}=\East.
\end{align*}
  
This description has a geometric interpretation in terms of 
Kauffman states.

\begin{defn}
  \label{def:PartialKnotDiagram}
  A {\em partial knot diagram} is the portion of a knot diagram
  contained in the $(x,y)$ plane with $t_2\leq y\leq t_1$, so that the
  diagram meets the slices $y=t_1$ and $y=t_2$ generically.  The knot
  projection divides the partial knot diagram into regions.  A {\em
    partial Kauffman state} is a triple of data $(\state,\y,\x)$,
  where $\state$ is a map that associates to each crossing one of
  its four adjacent regions; $\x$ is a collection of intervals in the
  intersection of the diagram with the $y=t_1$ slice; and $\y$ is a
  collection of intervals in the intersection of the diagram with the
  $y=t_2$ slice.  The regions assigned to the crossings are called
  {\em occupied regions}; those that are not are called {\em
    unoccupied}. The data in a partial Kauffman state are further
  required to satisfy the following compatibility conditions:
  \begin{itemize}
  \item 
    No two crossings are assigned to the same region.
  \item 
    If a region $R$ is occupied, then $\x$ contains 
    all the intervals in $R\cap (y=t_1)$ and 
    $\y$ contains none of the regions in $R\cap (y=t_2)$.
   \item 
     If a region $R$ is unoccupied, then 
     either $\x$ contains all but one of the regions in $R\cap (y=t_1)$
     and $\y$ contains none of the regions in
     $R\cap (y=t_2)$; or
     $\x$ contains all of the regions in $R\cap (y=t_1)$ and
     the region $\y$ contains exactly one of the regions 
     of $R\cap (y=t_2)$.
  \end{itemize}
\end{defn}

The compatibility conditions can be formulated more succinctly
  as follows: each region contains exactly one of the following: a
  quadrant assigned by $\state$, an interval in $R\cap (y=t_1)$ not
  contained in $\x$, or an interval in $R\cap (y=t_2)$ contained
  in $\y$.

The picture is simplified considerably when the partial knot diagram
contains a single crossing. In that case, it is straightforward to see that
the generators of the bimodule $\Pos^i$ defined above correspond to partial Kauffman states; $\state$ is one of $\North$, $\South$, $\West$, or $\East$;
$\y$ specifies the right idempotent of the generator 
and $\x$ the left idempotent of the generator.

The bimodules $\Pos^i$ are graded by the set $\Q^m$ in the following sense.
Define gradings of the generators $\North$, $\South$, $\East$, and $\West$ as in Equation~\eqref{eq:GradeCrossing}
(only now thinking of these as generators of $\Pos^i$ rather than $\Pos_i$).
The actions $\delta^1_\ell$ respect this grading, in the following sense.
Suppose that $\Xelt$ is some homogeneous generator, and $a_1,\dots,a_{\ell}\in\Blg_1$ are homogeneous elements, then
$\delta^1_{\ell}(\Xelt,a_1,\dots,a_{\ell})$ can be written as a sum of elements of the form $b\otimes \Yelt$ where
$\Yelt$ are homogeneous generators and 
$b\in\Blg_2$ is are homogeneous elements of the algebra, with 
\begin{equation}
  \label{eq:WeightGradings}
  \gr(\Xelt) + \tau_i^{\gr}(\gr(a_1)+\dots+\gr(a_{\ell-1}))
= \gr(b)+\gr(\Yelt),
\end{equation}
where here $\gr(a_i)$ and $\gr(b)$ denote the weight gradings on $\Blg_1$ and $\Blg_2$.

There is also a $\Z$-grading, specified by
where the grading of a generator type coincides
with the local Maslov contribution of the corresponding Kauffman
state as in Figure~\ref{fig:LocalCrossing}. For instance, if
all crossings are oriented upwards, then
\begin{equation}
  \MasGr(\North)=-1 \qquad \MasGr(\South)=\MasGr(\West)=\MasGr(\East)=0.
\end{equation}
The modules respect this $\Z$-grading:
for $\Xelt$, $\Yelt$, $a_1,\dots,a_{\ell-1}$ and $b$ as above,
\begin{equation}
  \label{eq:MasGradedCrossing}
  \MasGr(\Xelt)+\MasGr(a_1)+\dots+\MasGr(a_{\ell-1})+\ell-1
  = \MasGr(b)+\MasGr(\Yelt).
\end{equation}

The partial knot diagram with a single crossing in it determines in an obvious way a collection of arcs.
Equation~\eqref{eq:WeightGradings} can be interpreted as saying that the bimodule $\Pos^i$ is adapted
to this one-manifold with boundary, in the sense of Definition~\ref{def:Adapted}.

We consider first the case where $\Upwards\cap \{i,i+1\}=\emptyset$,
we will specify these actions by localizing to the crossing region, defining first bimodules
over $\Blg(2)=\Blg(2,0)\oplus \Blg(2,1) \oplus \Blg(2,2)\oplus \Blg(2,3)$, and
then extending them to $\Blg(m,k,\Upwards)$ (with $\Upwards\cap\{i,i+1\}=\emptyset$)
in Subsection~\ref{subsec:Extension}. The local models when $\Upwards\cap\{i,i+1\}$ is non-empty is defined
in Section~\ref{sec:AddCs}, and it is extended in the general case in~\ref{sec:GenCrossing}.

\subsection{Local bimodule for a positive crossing}
\label{subsec:LocalPos}

We define first a type $DA$ bimodule $\LocPos$ over 
$\Blg(2)=\Blg(2,0)\oplus\Blg(2,1) \oplus \Blg(2,2)\oplus \Blg(2,3)$.  
Note that $\Blg(2)$ is
defined over
$\IdempRing(2)=\IdempRing(2,0)\oplus\IdempRing(2,1)\oplus\IdempRing(2,2)\oplus\IdempRing(2,3)\cong
\Field^8$. 

As an $\IdempRing(2)-\IdempRing(2)$-module, $\LocPos$ has four generators
$\North$, $\South$, $\West$, and $\East$. As an $\Field$-vector space, $\LocPos$ is $12$ dimensional,
and all $12$ basis vectors appear on the right hand sides of the following four expressions,
which determine the $\IdempRing(2)-\IdempRing(2)$-module structure:
\begin{align*}
  \North& =\Idemp{\{1\}} \cdot \North \cdot \Idemp{\{1\}} + 
\Idemp{\{0,1\}} \cdot \North \cdot \Idemp{\{0,1\}} + 
\Idemp{\{1,2\}} \cdot \North \cdot \Idemp{\{1,2\}} + \Idemp{\{1,2,3\}}\cdot \North\cdot \Idemp{\{1,2,3\}} \\
  \South& =\Idemp{\emptyset}\cdot \South\cdot\Idemp{\emptyset}
  + \Idemp{\{0\}} \cdot\South\cdot \Idemp{\{0\}}
  + \Idemp{\{2\}} \cdot\South\cdot \Idemp{\{2\}} 
  + \Idemp{\{0,2\}} \cdot\South\cdot \Idemp{\{0,2\}} \\
  \West& =\Idemp{\{1\}} \cdot\West\cdot\Idemp{\{0\}}  + 
  \Idemp{\{1,2\}} \cdot\West\cdot\Idemp{\{0,2\}} \\
  \East& =\Idemp{{\{1\}}} \cdot\East\cdot\Idemp{\{2\}}  + 
\Idemp{\{0,1\}} \cdot\East\cdot\Idemp{\{0,2\}}.
\end{align*}

Next, we define the $\delta^1_1$ and $\delta^1_2$ actions on $\LocPos$. Some of these actions are specified in the following diagram:
  \begin{equation}
    \label{eq:PositiveCrossingDA}
    \begin{tikzpicture}
    \node at (0,5) (N) {$\North$} ;
    \node at (-5,2) (W) {$\West$} ;
    \node at (5,2) (E) {$\East$} ;
    \node at (0,0) (S) {$\South$} ;
    \draw[->] (N) [loop above] to node[above,sloped]{\tiny{$1\otimes 1 + L_{1} L_{2}\otimes L_{1} L_{2}+ R_{2} R_{1} \otimes R_{2} R_{1}$}} (N);
    \draw[->] (W) [loop left] to node[above,sloped]{\tiny{$1\otimes 1$}} (W);
    \draw[->] (E) [loop right] to node[above,sloped]{\tiny{$1\otimes 1$}} (E);
    \draw[->] (S) [loop below] to node[below,sloped]{\tiny{$1\otimes 1$}} (S);
    \draw[->] (N) [bend right=7] to node[above,sloped]{\tiny {$U_{2}\otimes L_{1} + R_{2} R_{1} \otimes R_{2}$}}  (W)  ;
    \draw[->] (W) [bend right=7] to node[below,sloped,pos=.6]{\tiny {$1\otimes R_{1} + L_{1} L_{2} \otimes L_{2} U_{1}$}}  (N)  ;
    \draw[->] (E)[bend left=7] to node[below,sloped,pos=.6]{\tiny {$1 \otimes L_{2} + R_{2} R_{1} \otimes R_{1} U_{2}$}}  (N)  ;
    \draw[->] (N)[bend left=7] to node[above,sloped]{\tiny {$U_{1}\otimes R_{2} + L_{1} L_{2} \otimes L_{1}$}} (E) ;
    \draw[->,dashed] (W)to node[below,sloped]{\tiny {$L_{1}$}} (S) ;
    \draw[->,dashed] (E) to node[below,sloped]{\tiny{$R_{2}$}} (S) ;
    \draw[->] (W) [bend left=7] to node[above,sloped]{\tiny {$L_{1} L_{2} \otimes U_{1}$}} (E);
    \draw[->] (E) [bend left=7] to node[below,sloped]{\tiny {$R_{2} R_{1}\otimes U_{2}$}} (W);
  \end{tikzpicture}
\end{equation}
Here, dashed arrows indicate $\delta^1_1$ actions; while $\delta^1_k$ for $k>1$ are indicated by solid arrows. For example, $\delta^1_1(\West)= L_1\otimes \South$.

Further actions are obtained by the following {\em local extension rules}.  For any $X\in \{\North,\West,\East,\South\}$ and
any pure algebra element $a\in\Blg(2)$,
\begin{equation}
  \label{eq:U1U2}
  \delta^1_2(X, U_1 U_2\cdot a)=U_1 U_2 \cdot \delta^1_2(X,a).
\end{equation}
and also:
\begin{itemize}
\item If $b\otimes Y$ appears with non-zero coefficient in $\delta^1_2(\North,a)$,
then $(b\cdot U_2)\otimes Y$ appears with non-zero coefficient in $\delta^1_2(\North,a \cdot U_1)$
and $(b\cdot U_1)\otimes Y$ appears with non-zero coefficient in $\delta^1_2(\North,a \cdot U_2)$.
\item If $b\otimes Y$ appears with non-zero coefficient in 
$\delta^1_2(\West, a)$, then $(U_2\cdot b)\otimes \Yelt$ appears with non-zero coefficient
in $\delta^1_2(\West,U_1\cdot a)$.
\item If $b\otimes \Yelt$ appears with non-zero coefficient in 
$\delta^1_2(\East, a)$, then $(U_1\cdot b)\otimes \Yelt$ appears with non-zero coefficient
in $\delta^1_2(\East,U_2\cdot a)$.
\end{itemize}
For example, the action $\delta^1_2(\West,1)=\West$ combined with the local extension rules shows 
that $\delta^1_2(\West,U_{1})=U_{2}\otimes  \West + L_{1} L_{2} \otimes \East$.

The above rules uniquely specify the $\Field$-linear map
$\delta^1_2\colon \Pos \otimes \Blg(2) \to \Blg(2)\otimes \Pos(2)$.

We wish next to specify  $\delta^1_3$. As in the case for $\delta^1_2$, we define these on a basis.
More formally,
an algebra element is called {\em elementary} if it is of the form
$p \cdot e$, where $p$ is a monomial in $U_{1}$ and $U_{2}$, and 
\[ e\in\{ 1,\qquad L_{1},\qquad R_{1},\qquad L_{2},\qquad R_{2},\qquad L_{1} L_{2},\qquad R_{2} R_{1}\}. \]
Having specified  cations
$\delta^1_2(X,a)$, where
$X\in\{\North,\South,\West,\East\}$ and $a$ is elementary. we turn to defining $\delta^1_3(X,a_1,a_2)$ where
$a_1$ and $a_2$ are elementary. 

Suppose that $a_1$ and $a_2$ are elementary algebra elements with
$a_1\otimes a_2\neq 0$ (i.e. there is an idempotent state $\x$ so that 
$a_1\cdot\Idemp{\x}\neq 0$ and $\Idemp{\x}\cdot a_2\neq 0$); and suppose 
moreover that $U_1\cdot U_2$ does not divide either $a_1$ nor $a_2$. In this case,
$\delta^1_3(\South,a_1,a_2)$ is the sum of terms:
\begin{itemize}
  \item $R_{1} U_{1}^t \otimes \East$ if $(a_1,a_2)=(R_{1}, R_{2} U_{2}^t)$ and $t\geq 0$
  \item $L_2 U_1^t U_2^n\otimes \East$ if   $(a_1,a_2)\in$
\[
\begin{array}{llll}
 \{(U_1^{n+1},U_2^t), & (R_1U_1^n,L_1U_2^t),& (L_2U_1^{n+1}, R_2U_2^{t-1})\}
  &{\mbox{when $0\leq n<t$}} \\
 \{(U_2^t,U_1^{n+1}),& (R_1U_2^t, L_1U_1^n), &
(L_2U_2^{t-1}, R_2U_1^{n+1})\}& \mbox{when $1\leq t\leq n$}
\end{array}
\]
  \item $L_{2} U_{2}^n \otimes \West$ if $(a_1,a_2)=(L_{2},L_{1} U_{1}^n)$ and $n\geq 0$
\item $R_1 U_1^t U_2^n\otimes \West$ if 
$(a_1,a_2)\in $ 
\[ \begin{array}{llll}
\{(U_2^{t+1},U_1^n), &
(L_2U_2^t,R_2U_1^n), &
(R_1U_2^{t+1}, L_1U_1^{n-1})\} &\mbox{when $0\leq t<n$}\\
\{(U_1^n,U_2^{t+1}), &
(L_2U_1^n, R_2U_2^t), &
(R_1U_1^{n-1}, L_1U_2^{t+1})\}, &\mbox{when $1\leq n\leq t$}
\end{array}\]
\item $L_2 U_1^t U_2^n\otimes \North$  if
$(a_1,a_2)\in$
\[
\begin{array}{llll}
\{(U_1^{n+1},L_2U_2^t), &
(R_1U_1^n,L_1L_2U_2^t), &
(L_2U_1^{n+1}, U_2^{t})\} &\mbox{when $0\leq n<t$}\\
\{(L_2U_2^{t}, U_1^{n+1}), & (U_2^t,L_2U_1^{n+1})& 
(R_1U_2^t, L_1L_2U_1^n)\} &\mbox{when $1\leq t\leq n$}\\
\{(L_2, U_1^{n+1})\} &  & &\mbox{when $0=t\leq n$}
\end{array}
\]
\item 
$R_1 U_1^t U_2^n\otimes \North$ if 
$(a_1,a_2)$ is in the following list:
\[\begin{array}{llll}
\{(U_2^{t+1},R_1U_1^n), &
(L_2U_2^t,R_2R_1U_1^n), &
(R_1U_2^{t+1}, U_1^{n})\} &\mbox{when $0\leq t<n$ }\\
\{(R_1U_1^{n}, U_2^{t+1}) & 
(U_1^n,R_1U_2^{t+1}), &
(L_2U_1^n, R_2R_1U_2^t)\} &\mbox{when $1\leq n\leq t$}\\
\{(R_1, U_2^{t+1})\}& & & \mbox{when $0= n\leq t$.}
\end{array}\]
\end{itemize}
For example,
$\delta^1_3(\South,U_1,U_2^2)=(L_2 U_1^2)\otimes \East + (R_1 U_1  U_2) \otimes \West$.

Extend this to the case where $U_1 U_2$ divides $a_1$ or $a_2$ or both, by requiring
\begin{equation}
  \label{eq:ExtendM3}
  \delta^1_3(\Zelt,(U_1 U_2) \cdot a,b) =  \delta^1_3(\Zelt,a,(U_1 U_2)\cdot b)= 
(U_1 U_2) \cdot \delta^1_3(\Zelt, a, b).
\end{equation}

\begin{prop}
  \label{prop:LocalBimodule}
  The operations $\delta^1_\ell$ defined above give 
  $\LocPos$ the structure of a $DA$ bimodule over
  $\Blg(2)=\Blg(2,0)\oplus\Blg(2,1) \oplus \Blg(2,2)\oplus \Blg(2,3)$,
  which is graded as in Equations~\eqref{eq:WeightGradings}
  and~\eqref{eq:MasGradedCrossing}.
\end{prop}

We prove the above proposition after some lemmas.

\begin{lemma}
  The actions $\delta^1_{\ell}$ respect the gradings, as in Equation~\eqref{eq:WeightGradings}.
\end{lemma}

\begin{proof}
  In the present case, the gradings on $\Pos^i$ are specified by
\[\begin{array}{llll}
  \gr(\North)=\frac{e_i + e_{i+1}}{4} & 
  \gr(\West)=\frac{e_{i}-e_{i+1}}{4} & 
  \gr(\East)=\frac{-e_{i}+e_{i+1}}{4} &
  \gr(\South)=\frac{-e_{i}-e_{i+1}}{4},
  \end{array}
\]
Bearing this in mind, the lemma is a straighforward verification
using the above defined actions $\delta^1_\ell$.
\end{proof}

Observe that if $b\otimes \Yelt$ appears with non-zero multiplicity in
$\delta^1_{\ell}(\Xelt,a_1,\dots,a_{\ell-1})$, and $\Xelt$ is a fixed
generator, and the $a_i\in\Blg(2)$ are elementary for
$i=1,\dots,\ell-1$, then the output algebra $b$ is uniquely specified
in Equation~\eqref{eq:WeightGradings}.

\begin{lemma}
  \label{lem:UniqueInitial}
  Given a monomial $p$ in $U_{1}$ and $U_{2}$ and an algebra element
  \[ a\in \{p, \qquad R_{1}\cdot p, \qquad L_{2}\cdot p,\qquad  R_{2} R_{1}\cdot p,\qquad
  L_{1}L_{2}\cdot p\},\]
  there is a unique $X\in\{\North,\West,\East\}$ and $b$ so that 
  $b\otimes \North$ appears with non-zero multiplicity in $\delta^1_2(X,a)$.
  Similarly, given a monomial $p$ in $U_{1}$ and $U_{2}$
  an algebra element from
  \[ a\in\{p, \qquad L_{1}\cdot p, \qquad R_{2} \cdot p\},\]
  and a fixed state type $Y\in \{\West,\East\}$, there is a unique state type
  $X\in \{\North,\West,\East\}$ and $b$ so that $b\otimes Y$ appears with non-zero
  multiplicity in $\delta^1_2(X,a)$.
\end{lemma}

\begin{proof}
  This follows from a straightforward inspection
  of Diagram~\eqref{eq:PositiveCrossingDA} and the local extension rules.
  For example, if $Y=\North$, then if $a=p \cdot e$ where $p$ is a
  monomial in $U_{1}$ and $U_{2}$, and $e\in \{1, L_{1}L_{2}, R_{2}
  R_{1}\}$, then $X=\North$.  For $e=L_2$, write $p=U_{1}^x U_{2}^y$.
  If $x>y$, then $X=\West$, and if $x\leq y$, then $X=\East$.  The
  case where $e=R_1$ works similarly.
\end{proof}

For any elementary $a$ and $Y\in \{\North,\West,\East\}$ with $a\otimes Y\neq 0$, define
$\Initial(a,Y)$ to be the generator  $X\in \{\North,\West,\East\}$ as
defined in Lemma~\ref{lem:UniqueInitial}. Otherwise (i.e. if the pair $(a,Y)$
is not covered by the lists of Lemma~\ref{lem:UniqueInitial}), define
$\Initial(a,Y)=0$.

\begin{lemma}
  \label{lem:CharacterizeDelta3}
  Fix elementary algebra elements $a_1$, $a_2$ in $\Blg(2)$, with
  $\South\otimes a_1\otimes a_2\neq 0$, an elementary $b$ and
  $Y\in\{\North, \West, \East\}$ so that $b\otimes Y\neq 0$ and
  \begin{equation}
    \label{eq:GrCondition}
    \gr(\South\otimes a_1\otimes a_2)=\gr(b\otimes Y).
  \end{equation}
  Suppose moreover that $a_1\cdot a_2\neq 0$.
  Then the element
  $b\otimes Y$ appears with non-zero multiplicity in
  $\delta^1_3(\South,a_1,a_2)$ if 
  and only if
  $I(a_2,Y)\neq 0$ and 
  $\Initial(a_1, \Initial(a_2,Y))\neq \Initial(a_1\cdot a_2,Y)$;
  (i.e. one of them is $\East$ and the other is $\West$).
\end{lemma}

\begin{proof}
  This follows from a straightforward inspection of 
  the definition of $\delta^1_3$. 
\end{proof}

We introduce a notational shorthand. The maps 
$\delta^1_i\colon\LocPos\otimes \overbrace{\Blg\otimes\dots\otimes \Blg}^{i-1}\to \Blg\otimes\LocPos$
naturally extend to maps 
${\ed}^1_i\colon\Blg\otimes \LocPos\otimes \overbrace{\Blg\otimes\dots\otimes \Blg}^{i-1}\to \Blg\otimes \LocPos$,
by the rule 
\[ \de^1_i(b\otimes x\otimes a_1\otimes\dots\otimes a_{i-1})
= b\cdot \delta^1_i(x\otimes a_1 \dots\otimes \dots a_{i-1}).\]

Fix elementary algebra elements $a_1,a_2\in\Blg(2,k)$, so that
$\South\otimes a_1\otimes a_2\neq 0$. When $k=2$, $a_1\cdot a_2$ is
always non-zero. On the other hand, when $k=1$, the actions
$\delta^1_3(\South,a_1,a_2)$ are those for which $a_1\cdot a_2=0$ and
$b\otimes Y$ (as specified by Equation~\eqref{eq:GrCondition})
appears with non-zero multiplicity in
$\ed^1_2(\delta^1_2(\South,a_1),a_2)$.

\begin{lemma}
  \label{lem:CaseOfAinf}
  With the above definition, if $a_1$,  $a_2$, and $a_3$ are elementary, then 
  \begin{align}
    \ed^1_2 (\delta^1_3(\South,a_1,a_2),a_3)
    &+ 
    \ed^1_3(\delta^1_2(\South,a_1),a_2,a_3) \nonumber \\
    &+\delta^1_3(\South,a_1\cdot a_2, a_3) 
     + \delta^1_3(\South,a_1, a_2\cdot a_3) =0
    \label{eq:PartAinf}
    \end{align}
\end{lemma}

\begin{proof}
  The above actions $\delta^1_3$ defined above vanish on $\Blg(2,0)$
  and $\Blg(2,3)$, so Equation~\eqref{eq:PartAinf} is obvious.  On
  $\Blg(2,1)$, the non-trivial $\delta^1_3$ actions are
  \[\begin{array}{ll}
    \delta^1_3(\South,R_{1}, R_{2} U_{2}^t)= R_{1} U_{1}^t \otimes \East&
    \delta^1_3(\South,L_2,L_1 U_1^n)= L_2 U_2^n \otimes \West \\
    \delta^1_3(\South,L_2, U_1^{n+1})=L_2 U_2^n\otimes \North &
    \delta^1_3(\South,R_1, U_2^{t+1})=R_1 U_1^t\otimes \North
    \end{array}\]
  with $k,n\geq 0$.  For these, the verification of
  Equation~\eqref{eq:PartAinf} is straightforward. 

  In $\Blg(2,2)$, if
  $a_1\otimes a_2\otimes a_3\neq 0$, then in fact $a_1\cdot a_2\cdot
  a_3\neq 0$; so we wil henceforth  assume  that $a_1\cdot a_2\cdot a_3\neq 0$.
  We  treat two subcases separately, according to whether or
  not $a_1=(U_{1} U_{2})^\ell$.

  {\bf Case 1: $a_1=(U_{1} U_{2})^{\ell}$.}
  The first and the fourth terms in Equation~\eqref{eq:PartAinf}
  vanish, and the middle two agree by Equation~\eqref{eq:ExtendM3}.

  {\bf Case 2:  $a_1\neq (U_{1} U_{2})^\ell$.}
  In this case, $\delta^1_2(\South,a_1)=0$, so the second term in Equation~\eqref{eq:PartAinf} is missing.
  Given $Y\in \{\North,\West,\East\}$, let
  \begin{align*}
    \alpha=\Initial(a_1\cdot a_2,\Initial(a_3, Y)), & \qquad \beta=\Initial(a_1\cdot a_2\cdot a_3, Y), \\
    \qquad\gamma=\Initial(a_1,\Initial(a_2\cdot a_3,Y)), & \qquad
    \delta=\Initial(a_1,\Initial(a_2,\Initial(a_3,Y))).
  \end{align*}

  By Lemma~\ref{lem:CharacterizeDelta3}, $b\otimes Y$ 
  (where
    $\gr(\South\otimes a_1\otimes a_2\otimes a_3)=\gr(b\otimes Y)$)
    appears with non-zero multiplicity in:
  \begin{itemize}
  \item $\delta^1_3(\South,a_1\cdot a_2,a_3)$ if the set $\{\alpha,\beta\}$
    equals $\{\West,\East\}$
  \item $\delta^1_3(\South,a_1,a_2\cdot a_3)$ if  $\{\gamma, \beta\}=\{\West,\East\}$
  \item $\ed^1_2(\delta^1_3(\South,a_1,a_2),a_3)$ if $\{\alpha, \delta\}=\{\West,\East\}$
  \end{itemize}

  Equation~\ref{eq:PartAinf} states that either
  $\alpha=\beta=\gamma=\delta$ or exactly one of the three relations
  $(\alpha=\beta, \beta=\gamma,\alpha=\delta)$ holds.  To establish
  this, we will need the following two observations about
  $\Initial(a,Z)$:

  {\bf Observation 1.}  If $a$ is elementary and $Z$ is a generator type,
  then $\Initial(a,Z)$ is uniquely determined by
  $a$ and whether or not $Z=\North$, except in the special case where
  $a=(U_{1} U_{2})^\ell$ for some $\ell \geq 0$. For example,
  if $a= L_{1} \cdot p$ or $R_{2}\cdot p$, where $p$ is a monomial in $U_{1}$ and $U_{2}$, 
  then $\Initial(a,Z)=\North$; if $a= U_{1}^x \cdot U_{2}^y$ and $Z\in\{\West,\East\}$, then 
  $\Initial(a,Z)=\West$ if $x>y$ and $\Initial(a,Z)=\East$ if $x<y$.

  The generator types $\{\North, \West, \East\}$ are further grouped into two classes, $\{\North\}$ and $\{\West,\East\}$;
  in this language, the Observation 1 can be phrased as saying that if $a\neq (U_{1} U_{2})^\ell$,
  the class of $Z$ and the algebra element $a$
  uniquely determines $\Initial(a,Z)$.

  {\bf Observation 2.} If $a_1$ and $a_2$
  are elementary and $a_1\cdot a_2\neq 0$, 
  then the class of $\Initial(a_1,\Initial(a_2,Z))$
  agrees with the class of $\Initial(a_1\cdot a_2,Z)$.

  By the second observation, the class of $\Initial(a_2\cdot a_3,Y)$
  agrees with the class of $\Initial(a_2,\Initial(a_3,Y))$, so by the
  first observation, $\Initial(a_1,\Initial(a_2\cdot a_3,
  Y))=\Initial(a_1,\Initial(a_2,\Initial(a_3,Y)))$;
  i.e. $\gamma=\delta$, so if any two of the conditions
  $(\alpha=\beta, \beta=\gamma,\alpha=\delta)$ 
  holds, the third one holds, verifying Equation~\eqref{eq:PartAinf}
  when $a_1\neq (U_{1} U_{2})^{\ell}$.
\end{proof}

\begin{proof}[Proof of Proposition~\ref{prop:LocalBimodule}]
  Note that in $\Blg(2,3)$, there is only one non-trivial generator,
  and it is $\North$; so $\delta^1_\ell=0$ for all $\ell>2$. In fact, $\Pos^1$ in this case
  is simply the identity bimodule for $\Blg(2,3)$, and the $\Ainfty$ relation holds.
  A similar simplification occurs for $\Blg(2,0)$, where the algebra is one-dimensional,
  and $\South$ is the only module generator. 
  
  We turn to the more interesting cases, in the summands $\Blg(2,1)$
  and $\Blg(2,2)$.  Varying $n$, we verify the $DA$ bimodule relation
  with $n$ algebra inputs $a_1,\dots,a_n$ (and one module input), subdividing further
  each verification according to whether we are in $\Blg(2,1)$ or $\Blg(2,2)$.

  When $n=0$, the $\Ainfty$ relation
  is clear from the form of $\delta^1_1$. 

  When $n=1$, terms of the form $\ed^1_1\circ \delta^1_{2}$ cancel in pairs,
  except in the special case where $a_1=(U_1 U_2)^{\ell}$, in which case
  a term of the form $\ed^1_1\circ \delta^1_{2}$ cancels with another of the
  form $\ed^1_2\circ \delta^1_1$.
  
  When $n=2$, we find it convenient to label the terms in the $\Ainfty$ relation as follows:
  \begin{equation}
    \label{eq:LabelAinftyTerms}
  \begin{array}{lll}
  A=\mathcenter{\begin{tikzpicture}[scale=.8]
        \node at (0,1) (inD) { };
        \node at (0,0) (d) {$\delta^1_3$};
        \node at (-1,-1) (mu) {$\mu_1$};
        \node at (-1.9,-1.5) (outA) {};
        \node at (1,1) (inB1) {$a_1$};
        \node at (2,1) (inB2) {$a_2$};
        \node at (0,-1.5) (outD) {};
        \draw[modarrow] (inD) to (d); 
        \draw[algarrow] (inB1) to (d);
        \draw[algarrow] (inB2) to (d);
        \draw[algarrow] (d) to (mu);
        \draw[modarrow] (d) to (outD);
        \draw[algarrow] (mu) to (outA);
      \end{tikzpicture}
    } &
B=\mathcenter{\begin{tikzpicture}[scale=.8]
        \node at (-1,1) (inD) { };
        \node at (-1,-.7) (d) {$\delta^1_3$};
        \node at (-.4,0.2) (mu) {$\mu_1$};
        \node at (-2,-1.5) (outA) {};
        \node at (0,1) (inB1) {$a_1$};
        \node at (1.2,1) (inB2) {$a_2$};
        \node at (-1,-1.5) (outD) {};
        \draw[modarrow] (inD) to (d); 
        \draw[algarrow] (inB1) to (mu);
        \draw[algarrow] (inB2) to (d);
        \draw[algarrow] (mu) to (d);
        \draw[modarrow] (d) to (outD);
        \draw[algarrow] (d) to (outA);
      \end{tikzpicture}
    }
&    
C=\mathcenter{\begin{tikzpicture}[scale=.8]
        \node at (-1,1) (inD) { };
        \node at (-1,-.7) (d) {$\delta^1_3$};
        \node at (0,0) (mu) {$\mu_1$};
        \node at (-2,-1.5) (outA) {};
        \node at (-.5,1) (inB1) {$a_1$};
        \node at (1,1) (inB2) {$a_2$};
        \node at (-1,-1.5) (outD) {};
        \draw[modarrow] (inD) to (d); 
        \draw[algarrow] (inB1) to (d);
        \draw[algarrow] (inB2) to (mu);
        \draw[algarrow] (mu) to (d);
        \draw[modarrow] (d) to (outD);
        \draw[algarrow] (d) to (outA);
      \end{tikzpicture}
    } \\
D=\mathcenter{\begin{tikzpicture}[scale=.8]
        \node at (-1,1) (inD) { };
        \node at (-1,-.7) (d) {$\delta^1_2$};
        \node at (-.5,0.2) (mu) {$\mu_2$};
        \node at (-2,-1.5) (outA) {};
        \node at (0,1) (inB1) {$a_1$};
        \node at (1,1) (inB2) {$a_2$};
        \node at (-1,-1.5) (outD) {};
        \draw[modarrow] (inD) to (d); 
        \draw[algarrow] (inB1) to (mu);
        \draw[algarrow] (inB2) to (mu);
        \draw[algarrow] (mu) to (d);
        \draw[modarrow] (d) to (outD);
        \draw[algarrow] (d) to (outA);
      \end{tikzpicture}
    }
&
E=\mathcenter{\begin{tikzpicture}[scale=.8]
  \node at (0,1) (inD) { };
  \node at (1,1) (inB1) {$a_1$};
  \node at (2,1) (inB2) {$a_2$};
  \node at (0,.1) (d1) {$\delta^1_2$};
  \node at (0,-.8) (d2) {$\delta^1_2$};
  \node at (-1,-1) (mu) {$\mu_2$};
  \node at (0,-1.6) (outD) {};
  \node at (-1.9,-1.5) (outA) {};
  \draw[algarrow] (inB1) to (d1);
  \draw[algarrow] (inB2) to (d2);
  \draw[modarrow] (inD) to (d1); 
  \draw[algarrow] (d1) to (mu);
  \draw[modarrow] (d1) to (d2);
  \draw[algarrow] (d2) to (mu);
  \draw[algarrow] (mu) to (outA);
  \draw[modarrow] (d2) to (outD);
\end{tikzpicture}} & 
  F=\mathcenter{\begin{tikzpicture}[scale=.8]
        \node at (0,1) (inD) { };
        \node at (0,.1) (d1) {$\delta^1_3$};
        \node at (0,-.8) (d2) {$\delta^1_1$};
        \node at (-1,-1) (mu) {$\mu_2$};
        \node at (-1.9,-1.5) (outA) {};
        \node at (.5,1) (inB1) {$a_1$};
        \node at (1.5,1) (inB2) {$a_2$};
        \node at (0,-1.6) (outD) {};
        \draw[modarrow] (inD) to (d1); 
        \draw[algarrow] (inB1) to (d1);
        \draw[algarrow] (inB2) to (d1);
        \draw[algarrow] (d1) to (mu);
        \draw[algarrow] (d2) to (mu);
        \draw[modarrow] (d1) to (d2);
        \draw[modarrow] (d2) to (outD);
        \draw[algarrow] (mu) to (outA);
      \end{tikzpicture}} \\
  G=\mathcenter{\begin{tikzpicture}[scale=.8]
        \node at (0,1) (inD) { };
        \node at (0,.1) (d1) {$\delta^1_1$};
        \node at (0,-.8) (d2) {$\delta^1_3$};
        \node at (-1,-1) (mu) {$\mu_2$};
        \node at (-1.9,-1.5) (outA) {};
        \node at (1,1) (inB1) {$a_1$};
        \node at (2,1) (inB2) {$a_2$};
        \node at (0,-1.6) (outD) {};
        \draw[modarrow] (inD) to (d1); 
        \draw[algarrow] (inB1) to (d2);
        \draw[algarrow] (inB2) to (d2);
        \draw[algarrow] (d2) to (mu);
        \draw[algarrow] (d1) to (mu);
        \draw[modarrow] (d1) to (d2);
        \draw[modarrow] (d2) to (outD);
        \draw[algarrow] (mu) to (outA);
      \end{tikzpicture}}
  \end{array}
  \end{equation}
  So that the $\Ainfty$ relation on a $DA$ bimodule with two inputs reads
  \[ A+B+C+D+E+F+G=0.\] 
  On the algebras we are considering presently, $\mu_1=0$, so 
  $A=B=C=0$.

  When $n=2$, 
  since all actions are $U_1 U_2$-equivariant, the case
  where at least one of $a_1$ or $a_2$ is $(U_1 U_2)^{\ell}$ is
  straightforward: the only two possibly non-zero 
  terms are $D$ and $E$, and they cancel.  

  When $n=2$, and we are in the summand in $\Blg(2,2)$, and the
  starting and ending generator is in $\{\East, \West, \North\}$, the
  $\Ainfty$ relation follows from Lemma~\ref{lem:CharacterizeDelta3},
  and the observation that $a_1\cdot a_2\neq 0$.  Specifically, in
  this case, $A=B=C=0$, and the $\delta^1_3$ actions are defined so
  that the term in $G$ cancels $D$ and $E$. 
  
  When $n=2$ and we are in the summand in $\Blg(2,2)$, and the
  starting and ending generator is $\South$, then $A=B=C=G=0$, and we
  find that terms of type $F$ cancel in pairs, except in the special cases where
  $a_1\cdot a_2=(U_1 U_2)^\ell$, which are the cases where $D\neq 0$,
  in which case there is a single cancelling term of type $F$.  This
  is verified by looking at the formulas defining $\delta^1_3$.  For
  example, consider an $\Ainfty$ relation with generator of type
  $\South$, $a_1=U_1^n$, and $a_2=U_2^t$. If $1\leq n\leq t$, there is a
  term of type $F$, factoring through $\East$, and this cancels against another term of type $D$,
  factoring through $\West$, when $n<t$; otherwise, it cancels against 
  a contribution of type $D$.

  When $n=2$ and we are in the summand in $\Blg(2,1)$ and the initial
  generator is of type $\{\East,\West,\North\}$, terms of type $D$ and
    $E$ cancel except in special cases where $E$ contributes but
    $a_1\cdot a_2=0$, in which case there is a cancelling non-zero
    term of type $G$. When the initial generator is of type $\South$,
    and neither of $a_1$ nor $a_2$ equals $(U_1 U_2)^\ell$, the only
    possible non-zero term, which is of type $F$, vanishes thanks to
    the algebra; for example, in the $\Ainfty$ relation
    $\delta^1_3(\South,R_1,R_2 U_2^t)=R_1 U_1^t \otimes \East$,
    and $\delta^1_1(\East)=R_2\otimes \South$, but 
    $R_1\cdot R_2=0$ in $\Blg(2,1)$.
  
    The cases where $n=3$, the $\Ainfty$ relation
    trivially holds except in the special cases where the initial
    generator is of type $\South$. That case was covered by
    Lemma~\ref{lem:CaseOfAinf}.
  
  The case where $n>3$ is obvious.
\end{proof}

\subsection{Extension}
\label{subsec:Extension}

Fix integers $k$ and $m$ with $0\leq k \leq m+1$,
we extend the bimodule $\LocPos$ to a bimodule
$\lsup{\Blg(m,k)}\Pos^{i}_{\Blg(m,k)}$, as follows.  

Let $a\in\Blg_0(m,k)$ with $a=\Idemp{\x}\cdot a\cdot \Idemp{\y}$,
and suppose that $\x$ and $\y$ are close enough.
Suppose moreover that $a$ is pure, in the sense that it corresponds to some
monomial in $U_1,\dots,U_m$ under $\phi^{\x,\y}$.
We define the {\em type} of $a$, denoted $t(a)$, which is an expression in 
in $U_1$, $U_{2}$, $R_{1}$, $L_{1}$ $R_{2}$ and $L_{2}$, defined as follows:
\[ t(a)=\left\{
\begin{array}{ll}
R_{2}R_{1}U_{1}^{w_i(a)-\OneHalf} U_{2}^{w_{i+1}(a)-\OneHalf}
&{\text{if $w_i(a)\equiv w_{i+1}(a)\equiv \OneHalf\pmod{\Z}$}} \\
&{\text{and $v^{\x}_{i+1}<v^{\y}_{i+1}$}} \\
L_{1}L_{2}U_1^{w_i(a)-\OneHalf} U_{2}^{w_{i+1}(a)-\OneHalf}
&{\text{if $w_i(a)\equiv w_{i+1}(a)\equiv \OneHalf\pmod{\Z}$}} \\
&{\text{and $v^{\x}_{i+1}>v^{\y}_{i+1}$}} \\
R_{2}U_1^{w_i(a)} U_{2}^{w_{i+1}(a)-\OneHalf}
&{\text{if $w_i(a)\in\Z$ and $w_{i+1}(a)\equiv\OneHalf\pmod{\Z}$,}} \\
& {\text{and $v^{\x}_{i+1}<v^{\y}_{i+1}$}} \\
L_{2}U_1^{w_i(a)} U_{2}^{w_{i+1}(a)-\OneHalf}
&{\text{if $w_i(a)\in\Z$ and $w_{i+1}(a)\equiv\OneHalf\pmod{\Z}$,}} \\
& {\text{and $v^{\x}_{i+1}>v^{\y}_{i+1}$}} \\
R_{1}U_{1}^{w_i(a)-\OneHalf} U_{2}^{w_{i+1}(a)}
&{\text{if $w_i(a)\equiv \OneHalf\pmod{\Z}$ and 
    $w_{i+1}(a)\in\Z$,}} \\
& {\text{and $v^{\x}_{i}<v^{\y}_{i}$}} \\
L_{1}U_1^{w_i(a)-\OneHalf} U_{2}^{w_{i+1}(a)}
&{\text{if $w_i(a)\equiv \OneHalf\pmod{\Z}$ and 
    $w_{i+1}(a)\in\Z$,}} \\
& {\text{and $v^{\x}_{i}>v^{\y}_{i}$}} \\
 U_1^{w_i(a)} U_{2}^{w_{i+1}(a)}
&{\text{if $w_i(a)$ and $w_{i+1}(a)$ are integers.}}
\end{array}
\right.\]
Similarly, there is a map $t$ from generators of $\Pos^{i}$ to 
the four generators of $\LocPos$, that remembers only the type
($\North$, $\South$, $\West$, $\East$) of the generator of $\Pos^{i}$.

\begin{defn}
  \label{def:DefD}
  For $X\in\Pos^i$, an integer $\ell\geq 1$,
  and a sequence of algebra elements
  $a_1,\dots,a_{\ell-1}$ in $\Blg_0(m,k)$
  with specified weights, so that there exists a sequence of idempotent
  states $\x_0,\dots,\x_{\ell}$ with
  \begin{itemize}
    \item   $X=\Idemp{\x_0}\cdot X\cdot \Idemp{\x_1}$ \
    \item
      $a_t=\Idemp{\x_t}\cdot a_t\cdot \Idemp{\x_{t+1}}$ for
      $t=1,\dots,\ell-1$
      \item $\x_t$ and $\x_{t+1}$ are close enough (for $t=0,\dots,\ell-1$),
        \end{itemize}
  define
  $D_{\ell}(X,a_1,\dots,a_{\ell-1})\in \Blg_0(m,k)\otimes \Pos^i$ 
  as the sum of pairs $b\otimes Y$ where $b\in\Blg_0(m,k)$
  and $Y$ is a generator of $\Pos^i$, satisfying the following conditions:
  \begin{itemize}
  \item the weights of $b$ and $Y$ satisfy
    \begin{equation}
      \label{eq:GradedOperations}
      \gr(X)+\tau_i^{\gr}(\gr(a_1)+\dots+\gr(a_{\ell-1}))=\gr(b)+\gr(Y)
      \end{equation}
  \item
    There are generators $X_0$ and $Y_0$ with the same type
    (i.e. with the same label $\{\North,\South,\West,\East\}$)
    as $X$ and $Y$ respectively, so that 
    $t(b)\otimes Y_0$ appears with non-zero multiplicity in
    $\delta^1_{\ell}(X_0,t(a_1),\dots,t(a_{\ell-1}))$.
    \end{itemize}
\end{defn}

Clearly, $D_{\ell}=0$ for $\ell>3$.
Setting  $\delta^1_1=D_1$, we get
\[ 
\begin{array}{lll}
\delta^1_1(\West)=L_1\otimes \South, &
\delta^1_1(\East)=R_2\otimes \South, &
\delta^1_1(\North)=\delta^1_1(\South)=0.
\end{array}\]

\begin{lemma}
  \label{lem:DeltaWellDefined}
  Suppose that any $a_t\in\Ideal(\x_t,\x_{t+1})$
  (the ideal studied in
  Proposition~\ref{prop:IdentifyJ}), then
  the projection of $D_{\ell}(X,a_1,\dots,a_{\ell-1})$
  to $\Blg(m,k)\otimes \Pos^i$ vanishes; 
  i.e. the maps 
  $D_{\ell}$ induce well-defined maps 
  \[ \delta^1_\ell\colon \Pos^i\otimes 
  \overbrace{\Blg(m,k)\otimes \dots\otimes\Blg(m,k)}^{\ell-1}
  \to\Blg(m,k)\otimes \Pos^i,\]
  for all $\ell=1,2,3$.
\end{lemma}

\begin{proof}
  For the $\delta^1_1$ operations, the result is obvious.
  
  The operation $\delta^1_{2}$ can be written as a sum of $16$ terms,
  corresponding to the terms in the labels on
  Figure~\ref{eq:PositiveCrossingDA}.

  For example, consider $\delta^1_2(\East, a  \cdot U_{i+1}^\ell)$, 
  where $w_{i}(a)=w_{i+1}(a)=0$ and $\ell>0$. According to Equation~\ref{eq:PositiveCrossingDA},
  this can be written as a sum of two terms,
  \[ R_{i+1} R_{i} \cdot a \cdot U_{i}^{\ell-1} \otimes \West + a
  \cdot U_{i}^{\ell} \otimes \East.\] We check that 
  for any idempotent states $\x$ and $\y$
  with $\East\cdot \Idemp{\x}\neq 0$ and $\Idemp{\x}\cdot a \cdot U_{i+1}^{\ell}\cdot \Idemp{\x}\neq 0$,
  if $\Idemp{\x}\cdot a \cdot U_{i+1}^\ell\in{\mathcal J}$, then the two output terms also vanish.
  If $a$ is divisible by a monomial corresponding to a generating
  interval for $(\x,\y)$ that is disjoint from $\{i,i+1\}$, each output elements are also
  divisible by such monomials, and the claim is clear.  It remains to
  check the claim when $a \cdot U_{i+1}$
  is the monomial corresponding to a generating interval
  $U_{i+1}\cdots U_\beta$ for the incoming element.
  In this case, the output is divisible by $U_{i+2}\cdots U_\beta$,
  which corresponds to a generating interval in both output algebras.

  The other non-trivial checks include the following:
  \begin{align*}
    \delta^1_2(\West,U_{\alpha}\cdots U_{i})&= (U_{\alpha}\cdots U_{i-1} U_{i+1})\otimes \West
    + L_{i} L_{i+1} (U_{\alpha}\cdots U_{i-1})\otimes \East \\
    \delta^1_2(\North,U_{i+1}\cdots U_\beta\cdot L_{i})
    &= (U_{i}\cdots U_\beta)\otimes \West 
    + L_{i} L_{i+1} U_i (U_{i+2}\cdots U_\beta)\otimes \East \\
    \delta^1_2(\North,U_{\alpha}\cdots U_{i-1}\cdot L_{i})
    &= (U_{\alpha}\cdots U_i)\otimes \West 
    + L_{i} L_{i+1}(U_\alpha\cdots U_{i-1})\otimes \East \\
    \delta^1_2(\West,U_\alpha\cdots U_{i} \cdot L_{i+1} )
    &= U_\alpha\cdots U_{i-1}\cdot L_i \cdot L_{i+1} \otimes \North \\
    \delta^1_2(\West,U_\alpha\cdots U_{i-1} \cdot R_i )
    &= U_\alpha\cdots U_{i-1} \otimes \North
  \end{align*}
  where the monomials $U_{\alpha}\cdots U_{i}$ and $U_{i+1}\cdots
  U_\beta$ (with $\alpha\leq i$ and $\beta\geq i+2$) correspond to
  generating intervals for the input, for suitable choices of
  idempotents; it is then straightforward to check that the outputs
  are also zero. The remaining non-trivial checks are symmetric to the above ones (under the symmetry
  exchanging $\West$ and $\East$, and $L_i$ with $R_{i+1}$).

  For $\delta^1_3$, we must show that if $a_1$ or $a_2\in {\mathcal
    J}$, then the output $D_3(X,a_1,a_2)$ projects to zero in
  $\Blg(m,k)\otimes \Pos^i$.  In cases where
  $\min(w_i(a_1)+w_i(a_2),w_{i+1}(a_1)+w_{i+1}(a_2))=\OneHalf$,
  i.e. when $(t(a_1),t(a_2))\in \{(R_1,R_2 U_2^t),(L_2,L_1 U_1^n),
  (L_2,U_1^{n+1}),(R_1,U_2^{t+1})\}$, are considered separately. 
  In all these cases, the generating intervals for $a_1$ and $a_2$ do not contain
  $i$ or $i+1$, so the result is obvious.

  For the remaining cases,  the following stronger assertion holds:
  if $a_1$ and $a_2$ satisfy
  \begin{equation}
    \label{eq:WeightBound}
    \min(w_i(a_1)+w_i(a_2),w_{i+1}(a_1)+w_{i+1}(a_2))> \OneHalf,
  \end{equation}
  and $a_1\cdot a_2\in{\mathcal J}$
  (i.e. $a_1\cdot a_2$ projects to zero in $\Blg(m,k)$), 
  and with $\Idemp{\x}\cdot a_1=a_1$
  then $D_3(X,a_1,a_2)\in{\mathcal J}$.
  For
  instance, the terms in $\delta^1_3(\South, a_1,a_2)$ in $\Blg(2)$ that output
  $L_2 U_1^t U_2^n\otimes \East$ (all of which satisfy Equation~\eqref{eq:WeightBound})
  give rise to actions with
  $a_1, a_2\in\Blg(k,m)$ with $\Idemp{\x}\cdot a_1\cdot a_2=a_1\cdot a_2$
  with 
  \[ \begin{array}{lll}
    i\not\in{\mathbf p}, &
  w_{i}(a_1\cdot a_2)=n+1\geq 1, &w_{i+1}(a_1\cdot a_2)=t\geq 1
  \end{array}\] and
  output algebra element $b=\Idemp{\x}\cdot b \cdot \Idemp{\y}$ with
  \[ \begin{array}{lll}
    i\in\y, & i+1\not\in \y & w_{i}(b)=t\geq 1
  \end{array}
  \]
  It is easy to see that if $a_1\cdot a_2$ is divisible by a monomial
  corresponding to a generating interval of the form $U_{\alpha}\cdots
  U_i$, then so is $b$;   checking cases
  where $a_1\cdot a_2$ is divisible by other generating intervals is
  even more straightforward.  Remaining cases where the output
  contains $\West$ work similarly. Cases where the output contains
  $\North$ are easily verified, as well.  The stronger statement (following Equation~\eqref{eq:WeightBound}) now
  follows.
\end{proof}

For example,
$\delta^1_{2}(\North,R_{i+2}R_{i+1})=R_{i+2}R_{i+1}
R_{i}\otimes \West
+R_{i+2} U_{i}\otimes \East$.

As another example, choose $R_j$ with $j\neq i$ or $i+1$. Then, 
for all generators $X$, $\delta^1_{2}(X,R_j)=R_{j}\otimes X$.
(Note that in this example, $t(R_j)=1$.)

As another example, $\delta^1_{3}(\South,R_{i},R_{i+2}R_{i+1})=
\delta^1_{3}(\South,R_{i}R_{i+2},R_{i+1})=R_{i}R_{i+2}\otimes \East$.

Let $\delta^1_{\ell}=0$ for all $\ell\geq 4$.

It will be useful to have the following characterization of $\delta^1_3$:
\begin{lemma}
  \label{lem:CharacterizeD3Extended}
  The operation $\delta^1_3(X,a_1,a_2)$ contains only those terms $b\otimes Y$
  where both Equation~\eqref{eq:GradedOperations} holds, and 
  one of the following two conditions holds:
  \begin{enumerate}[label=(C-\arabic*),ref=(C-\arabic*)]
  \item \label{eq:ExceptionalCase}
    $(t(a_1),t(a_2))\in\{(R_1,R_2 U_2^t),(L_2, L_1 U_1^n),(L_2,U_1^{n+1}), (R_1,U_2^{t+1})\}$
    and $t(b)\otimes t(Y)$ appears with non-zero multiplicity in $\delta^1_3(t(X),t(a_1),t(a_2))$;
  \item $a_1 \cdot a_2\neq 0$,  $I(a_2,Y)\neq 0$ and 
    $\Initial(t(a_1), \Initial(t(a_2),t(Y)))\neq \Initial(t(a_1)\cdot t(a_2),t(Y))$.
  \end{enumerate}
\end{lemma}

\begin{proof}
  In the proof of Lemma~\ref{lem:DeltaWellDefined}, we showed that if
  we are not in Case~\ref{eq:ExceptionalCase}, then 
  Equation~\eqref{eq:WeightBound} holds, and so we conclude that
  $a_1\cdot a_2\neq 0$. Thus, our description of $\delta^1_3$ follows
  from Lemma~\ref{lem:CharacterizeDelta3}.
\end{proof}

\begin{prop}
  \label{prop:DAnoS}
  The above maps give $\Pos^i$ the structure of 
  a type $DA$ bimodule over $\Blg(m,k)$-$\Blg(m,k)$.
\end{prop}
\begin{proof}
  The proof of Proposition~\ref{prop:LocalBimodule} adapts, with a few
  remarks. In the verification of the 
  $\Ainfty$ relation with two algebra inputs, 
  that proof decomposed according to whether we
  were working in $\Blg(2,1)$ or $\Blg(2,2)$.  In the present case,
  when $(t(a_1),t(a_2))\in\{(R_1,R_2 U_2^t),(L_2,L_1 U_1^n),(L_2,U_1^{n+1}),($ the
  $\Ainfty$ relation holds exactly as it did in $\Blg(2,1)$. Otherwise , if
  $a_1\cdot a_2\neq 0$, the proof of the $\Ainfty$ relation for
  $\Blg(2,2)$ applies, using Lemma~\ref{lem:CharacterizeD3Extended}.
  Finally, if $a_1\cdot a_2=0$, then by
  Lemma~\ref{lem:CharacterizeD3Extended}, the term involving
  $\delta^1_3$ ($F$ or $G$) vanishes. The term $D$ vanishes by
  Lemma~\ref{lem:DeltaWellDefined}. The verification of the $\Ainfty$ relation
  in $\Blg(2,2)$ shows that the remaining possible non-zero term, which is of type $E$,
  has the same contribution as a term of type $F$, $G$, or $D$, all of which contribute $0$.
  
  Consider next the case of three algebra inputs $a_1$, $a_2$, and $a_3$.
  When $a_1\cdot a_2\cdot a_3=0$, the verification 
  (now in the proof of Lemma~\ref{lem:CaseOfAinf})
  works as it did in $\Blg(2,1)$. 
  In the remaining cases, the earlier proof of the $\Ainfty$ relation
  in $\Blg(2,2)$
  (contained in the proof of Lemma~\ref{lem:CaseOfAinf}) 
  still applies, as it hinges on the description of $\delta^1_3$
  (Lemma~\ref{lem:CharacterizeDelta3}) which still holds in this case
  (according to Lemma~\ref{lem:CharacterizeD3Extended}).
\end{proof}

\subsection{Adding $C_i$, locally}
\label{sec:AddCs}

We extend the local bimodule $\LocPos$ defined over $\Blg(2)$ to
a bimodule $\lsup{\Blg(2,\tau(\Upwards))}\LocPos_{\Blg(2,\Upwards)}$, 
where $\Upwards$ is a non-empty subset of
$\{1,2\}$.

In the local module $\LocPos$, there were $\delta^1_{2}$ actions
connecting generators of types $\{\North, \West, \East\}$ to 
$\{\North, \West, \East\}$ or $\South$ to $\South$. We extend these to
similar actions so that 
\begin{equation}
  \label{eq:CEquivariance}
  \begin{array}{ll}
\delta^1_{2}(X,C_1 \cdot a)=C_2\cdot \delta^1_{2}(X,a) &
\delta^1_{2}(X,C_2 \cdot a)=C_1\cdot \delta^1_{2}(X,a) 
\end{array}
\end{equation}
when the formulas make sense; i.e. the first holds when $1\in \Upwards$
and the second when $2\in\Upwards$.
Similarly, we extend the previous $\delta^1_{3}$ actions so that 
\[ 
\begin{array}{ll}
C_2\cdot \delta^1_{3}(\South,a_1,a_2) &=
\delta^1_{3}(\South,C_1 \cdot a_1,a_2)=
\delta^1_{3}(\South,a_1,C_1 \cdot a_2) \\
C_1\cdot \delta^1_{3}(\South,a_1,a_2) &=
\delta^1_{3}(\South,C_2 \cdot a_1,a_2)=
\delta^1_{3}(\South,a_1,C_2 \cdot a_2).
\end{array}
\]
For example, $\delta^1_{3}(C_1 R_1,C_2 R_2) = 
\delta^1_{3}(C_1 C_2 R_1,R_2) = 
R_1 C_1 C_2 \otimes \East$.

We specify further $\delta^1_2$ actions from $\South$ to $\{\North,\West,\East\}$:
\[\begin{array}{lll}
  \delta^1_2(\South,C_2)=R_1\otimes \West & \\
  \delta^1_2(\South,C_1)=L_2\otimes \East &
  \delta^1_2(\South,C_1 C_2)=C_2 R_1\otimes \West + C_1 L_2\otimes \East \\
  \delta^1_{2}(\South,U_1C_2)=U_1 L_2 \otimes \East &
  \delta^1_{2}(\South,U_1C_1 C_2)=U_1 C_2 L_2 \otimes \East \\
  \delta^1_{2}(\South,C_1 U_2)=R_1  U_2 \otimes \West &
    \delta^1_{2}(\South,C_1 C_2 U_2)=C_1 R_1  U_2 \otimes \West \\
  \delta^1_{2}(\South, R_1 C_2)= R_1\otimes \North &
  \delta^1_{2}(\South, R_1 C_1 C_2)= R_1 C_2\otimes \North \\
  \delta^1_{2}(\South, L_2 C_1)= L_2\otimes \North
  &
  \delta^1_{2}(\South, C_1 L_2 C_2)= C_1 L_2\otimes \North \\
  \delta^1_{2}(\South, R_1 C_1 U_2)= R_1 U_2\otimes \North 
  &
  \delta^1_{2}(\South, R_1 C_1 U_2 C_2)= R_1 C_1 U_2\otimes \North \\
  \delta^1_{2}(\South, U_1 L_2 C_2)= L_2 U_1\otimes \North & 
  \delta^1_{2}(\South, U_1 C_1 L_2 C_2)= U_1 L_2  C_2\otimes \North 
\end{array}\]

These are extended to commute with multiplication by $U_1U_2$ (as in
Equation~\eqref{eq:U1U2}).

By Equation~\eqref{eq:CEquivariance}, the actions
$\delta^1_2(\South,1)=\South$ gives rise to actions
$\delta^1_2(\South,C_2)=C_1\otimes \South$. The first action listed
above, i.e. the relation $\delta^1_2(\South,C_2)=R_1\otimes \West$, is
now a consequence of this fact, together with the $\Ainfty$ relation
with module element $\South$, the single algebra input $C_2$, and
output module element $\South$. The second action on the above list
follows symmetrically.  The actions on the left column of rows
three, five, and seven respectively are forced by the
following $\delta^1_3$ actions (with $\Upwards=\emptyset$) 
\begin{align*}
    \delta^1_3(\South,U_1,U_2)&=L_2 U_1\otimes \East \\
    \delta^1_3(\South,L_2,U_1)&=L_2 \otimes\North  \\
    \delta^1_3(\South,R_1 U_2,U_1)&= R_1 U_2 \otimes\North
\end{align*}
and the $\Ainfty$ relations with inputs $(\South,U_1,C_2)$, $(\South,L_2,C_1)$ and $(\South,R_1 U_2,C_1)$ respectively. The remaining actions on the left column follow symmetrically.
The actions in the second column follow from actions from
the first column column of the form $(\South,da)$
and $\Ainfty$ relations with inputs $(\South,a)$.

\begin{lemma}
  \label{lem:LocalBimoduleWithC}
  For any $\Upwards\subset \{1,2\}$, the above actions induce a 
  type $DA$ bimodule structure on
  $\lsup{\Blg(2,\tau(\Upwards))}\LocPos_{\Blg(2,\Upwards)}$.
\end{lemma}

\begin{proof}
  It suffices to prove the lemma in the case where $\Upwards=\{1,2\}$.
  For example, if $\Upwards=\{1\}$, the outputs of the $\delta^1_j$ actions
  lie in the subalgebra $\Blg(2,\{2\})$.
  
  We consider the $\Ainfty$ relation with $n$ incoming algebra
  elements $a_1,\dots,a_n$. 

  It is a straightforward verification to check that the actions defined above 
  are consistent with the $\Ainfty$ relation with $n=1$ input.

  When $n=2$, and the incoming generator is not $\South$, the
  $\Ainfty$ relation follows from Equation~\eqref{eq:CEquivariance}
  together with the $\Ainfty$ relation with $\Upwards=\emptyset$.

  When $n=2$ and the incoming generator is $\South$, label the terms
  in the $\Ainfty$ relation as in
  Equation~\eqref{eq:LabelAinftyTerms}.
  The key point to verifying this relation now is
  that when differentiating $a_1$ or $a_2$ (in terms $B$ or $C$
  above), the $U_1$ or $U_2$-power can change by at most one so the
  corresponding $\delta^1_3$ action with $a_1$ and $a_2$ is non-zero
  when the action with $(d a_1,a_2)$ and $(a_1,d a_2)$ is.  However,
  there are borderline cases where this change in the $U_1$ or the
  $U_2$ power is enough to to turn off one of those actions. These are
  precisely the cases either the product $a_1 a_2$ acts non-trivially
  (i.e. where $D$ contributes),
  or there is an iterated $\delta^1_2$ (i.e. $E$ contributes). 

  For example, consider the case where $a_1=U_1^n C_1$, and
  $a_2=U_2^t$ with $n,t\geq 0$, and consider the terms where the
  output generator has type $\East$.  We have the following non-zero
  terms in the $\Ainfty$ relation:
  \[\begin{array}{lll}
    A\neq 0 \Leftrightarrow 0\leq n-1< t, &
    B\neq 0 \Leftrightarrow 0\leq n< t, &
    D\neq 0 \Leftrightarrow  n=t\geq 0, \\
    & E\neq 0 \Leftrightarrow  n=0, t\geq 0; 
  \end{array}\]
  Thus, for all choices of $n,t\geq 0$, there are either no non-zero terms or exactly two, and so the $\Ainfty$ relation holds. Other cases where $n=2$
  and the initial generator is of type $\South$ work similarly.

  Consider the case where $n=3$. Recall that the algebra has a filtration
  in $\{0,1\}\times \{0,1\}$ given by the functions $(\MasFilt_1,\MasFilt_2)$,
  with the property that if $a$ is a pure algebra element not divisible by $C_j$, then
  $\MasFilt_j(a)=0$ and $\MasFilt_j(C_j a)=1$.
  Clearly, the operations $\delta^1_{\ell}$ respect this filtration,
  in the sense that if a $(\MasFilt_1,\MasFilt_2)$-homogeneous element
  $b\otimes Y$ appears with non-zero multiplicity
  in $\delta^1_{\ell}(a_1,\dots,a_{\ell-1})$, and each $a_i$ is homogeneous
  (with respect to $\mu$), then for $j=1,2$,
  \[ (\MasFilt_1(b),\MasFilt_2(b))\leq 
  \sum_{k=1}^{\ell-1}(\MasFilt_1(a_k),\MasFilt_2(a_k)).\]

  For the terms that preserve this $\{0,1\}\times \{0,1\}$-filtration, the $\Ainfty$
  relation with $3$ inputs is an immediate consequence of
  Proposition~\ref{prop:DAnoS}.  Consider next terms that drop
  filtration level by one.  Note that each $\delta^1_\ell$
  action of this type has $\ell=2$ and end in the span of
  $\{\North,\West,\East\}$, while all $\delta^1_3$ actions start from
  $\South$. It follows that there are no such terms that appear in the
  $\Ainfty$ relation with $3$ inputs. (It also follows that there are
  no terms of this kind in the $\Ainfty$ relation which drop by more
  than $1$.)

  For $n>3$, the $\Ainfty$ relation is easy.
\end{proof}

\subsection{The general case of $\Pos^i$}
\label{sec:GenCrossing}

Let $\Blg_1=\Blg(m,k,\Upwards)$, where $0\leq k \leq m+1$ and
$\Upwards\subset \{1,\dots,m\}$ is arbitrary; and let
$\Blg_2=\Blg(m,k,\tau(\Upwards))$.  In cases where
$\Upwards\cap\{i,i+1\}\neq \emptyset$, we must modify our earlier
constructions as follows.

Extend the type $t(a)$ to be a monomial in $U_1,U_2,L_1,L_2,C_1,C_2,R_1,R_2$,
so that 
\[\begin{array}{ll}
t(C_i\cdot a)=C_1\cdot t(a) &{\text{if  $C_i\cdot a\neq 0$}} \\
t(C_{i+1}\cdot a)=C_2\cdot t(a) &{\text{if $C_{i+1}\cdot a \neq 0$};}
\end{array}\]
and $t(a)$ is as defined before when $a$ is a pure algebra element
not divisible by $C_i$ or $C_{i+1}$.

Let $t(\Upwards)\subset \{1,2\}$ be the set with $1\in t(\Upwards)$ iff $i\in\Upwards$ and $2\in t(\Upwards)$ iff
$i+1\in\Upwards$.

\begin{defn}
  For $X\in\Pos^i$
  and a sequence of pure algebra elements $a_1,\dots,a_{\ell-1}$
  in $\Blg_0(m,k,\Upwards)$, so that there exist a sequence of idempotent states
  $\x_0,\dots,\x_{\ell}$ with
  \begin{itemize}
    \item   $X=\Idemp{\x_0}\cdot X\cdot \Idemp{\x_1}$ \
    \item
      $a_t=\Idemp{\x_t}\cdot a_t\cdot \Idemp{\x_{t+1}}$ for
      $t=1,\dots,\ell$
      \item $\x_t$ and $\x_{t+1}$ are close enough (for $k=0,\dots,\ell-1$),
      \end{itemize}
  define
  $D_{\ell}(X,a_1,\dots,a_{\ell-1})\in \Blg_0(m,k,\tau(\Upwards))\otimes \Pos^i$ 
  as the sum of pairs $b\otimes Y$ where $b\in\Blg_0(m,k,t(\tau(\Upwards)))$
  and $Y$ is a generator of $\Pos^i$, satisfying the
  the conditions
  of Definition~\ref{def:DefD}, with the understanding that now
  $\delta^1_{\ell}(X_0,t(a_1),\dots,t(a_{\ell-1})$ is computed using 
  the crossing over $\lsup{\Blg(2,t\circ\tau(\Upwards))}\LocPos_{\Blg(2,t(\Upwards))}$;
  and the additional condition that
  for any $k\neq i$ or $i+1$, and any $m=1,\dots,\ell-1$,
      \[ D_\ell(X,a_1,\dots, C_k\cdot  a_m,\dots,a_{\ell-1})=
      C_k\cdot D_\ell(X,a_1,\dots, a_m,\dots,a_{\ell-1}).\]
\end{defn}

\begin{lemma}
  If any $a_t\in\Ideal(\x_t,\x_{t+1})$, then
  the projection of $D_{\ell}(X,a_1,\dots,a_{\ell-1})$
  to $\Blg(m,k)\otimes \Pos^i$ vanishes; 
  i.e. the maps 
  $D_{\ell}$ induce well-defined maps 
  \[ \delta^1_\ell\colon \Pos^i\otimes 
  \overbrace{\Blg(m,k,\Upwards)\otimes \dots\otimes\Blg(m,k,\Upwards)}^{\ell-1}
  \to\Blg(m,k,\tau(\Upwards))\otimes \Pos^i,\]
  for all $\ell=2,3$.
\end{lemma}

\begin{proof}
  Lemma~\ref{lem:DeltaWellDefined} takes care of most of this; we must
  check further the additional $\delta^1_2$ actions from $\South$ to
  $\{\North,\West,\East\}$ listed in the beginning of this subsection;
  but this is straightforward.
\end{proof}

\begin{prop}
  \label{prop:CrossingGeneral}
  The above maps give $\Pos^i$ the structure of 
  a type $DA$ bimodule over $\Blg(m,k,\tau(\Upwards))$-$\Blg(m,k,\Upwards)$.
\end{prop}
\begin{proof}
  This follows easily from Proposition~\ref{prop:DAnoS}, handling terms with
  $C_i$ inputs as in Lemma~\ref{lem:LocalBimoduleWithC}.
\end{proof}

\subsection{The negative crossing}
\label{def:NegativeCrossing}

Consider the map $\VRot\colon \Blg(m,k,\Upwards)\to\Blg(m,k,\Upwards)$ 
from Section~\ref{subsec:Symmetry}.
Recall that $\Opposite(a\cdot b)=\Opposite(b)\cdot\Opposite(a)$,
$\Opposite(\Idemp{\x})=\Idemp{\x}$,
$\Opposite(L_t)=R_t$ $\Opposite(R_t)=L_{t}$
$\Opposite(U_t)=U_{t}$ and $\Opposite(C_j)=C_j$ for all $t=1,\dots,m$ and $j\in\Upwards$.

Let $\Neg^i$ be generated by the same generators $\{\North,\South,\West,\East\}$ as before. If 
$\delta^1_1(X)=b\otimes Y$ in $\Pos^i$, then
$\delta^1_1(Y)=\Opposite(b)\otimes X$ in $\Neg^i$
If $\delta^1_{2}(X,a)=\Opposite(b)\otimes Y$ in $\Pos^i$, then
$\delta^1_{2}(Y,\Opposite(a)))=\Opposite(b)\otimes X$ in $\Neg^i$
If $\delta^1_{3}(X,a_1,a_2)=b\otimes Y$ in $\Pos^i$, then
$\delta^1_{3}(Y,\Opposite(a_2),\Opposite(a_1))=\Opposite(b)\otimes X$.

More succinctly, the opposite module of $\Pos^i$ is identified with 
\[\lsub{\Blg_2}({\overline \Pos}^i)^{\Blg_1}
\cong \lsup{{\Blg_1^{\op}}}{\overline \Pos}^i_{\Blg_2^{\op}}
=\lsup{\Blg_1}{\Neg}^i_{\Blg_2},\]
under the identification of $\Blg_t^{\op}\cong \Blg_t$ 
for $t=1,2$
(Equation~\eqref{eq:OppositeIsomorphism}).

For example, in $\Neg^i$, we have actions
\begin{align*}
  \delta^1_1(\South)&=R_i\otimes \West + L_{i+1}\otimes \East \\
  \delta^1_2(\East,R_i) &= R_{i+1}R_{i}\otimes \North \\
  \delta^1_3(\West,R_{i} U_{i}^n,R_{i+1})&=R_{i+1} U_{i+1}^n \otimes \South.
\end{align*}

\begin{prop}
  The above maps give $\Neg^i$ the structure of 
  a type $DA$ bimodule over $\Blg(m,k,\tau(\Upwards))$-$\Blg(m,k,\Upwards)$.
\end{prop}
\begin{proof}
  This is a formal consequence of the above definition.
\end{proof}

\subsection{Gradings and partial Kauffman states}

As noted earlier, $\Pos^i$ is adapted to the underlying manifold. 
Thus, its grading set takes values in $H^1(W,\partial W)$. 
We can specialize to a simply $\Q$-graded-graded setting as follows.

Recall (Equation~\eqref{eq:DefAlex}) that the algebra has an Alexander
grading with values in $\Q$, obtained from the map $\phi\colon
\OneHalf \Z^{m}\to \OneHalf\Z$ defined by
\[\phi(e_i)=\left\{\begin{array}{ll}
    -1 &{\text{if $i\in\Upwards$}} \\
    1 &{\text{if $i\not\in\Upwards$.}}
\end{array}\right.\]
There is an induced grading on the bimodule, given by
\begin{equation}
  \label{eq:IntegerAlexanderGrading}
  A(X)=\phi(\gr(X));
\end{equation}
which we can think of, more abstractly, as the 
evaluation of the grading, thought of as an element of
$H^1(W,\partial W)$, against the element
$[W,\partial]\in H_1(W,\partial W)$ induced by the orientation of
$W$.)

It is now an immediate consequence of Equation~\eqref{eq:WeightGradings} that
\[ A(X,a_1,\dots,a_\ell)=A(X)+A(a_1)+\dots+A(a_{\ell})=A(b)+A(Y),\]
if $b\otimes Y$ appears with non-zero multiplicity in $\delta^1_{\ell+1}(X,a_1,\dots,a_{\ell})$.

\begin{prop}
  \label{prop:ComputeAlexander}
  The $\Q$-valued Alexander grading on generators from Equation~\eqref{eq:IntegerAlexanderGrading}
  for the crossing bimodules
  is computed by the local Kauffman contributions displayed in Figure~\ref{fig:LocalCrossing}.
\end{prop}

\begin{proof}
  This a straightforward check using
  Equation~\eqref{eq:IntegerAlexanderGrading},~\eqref{eq:GradeCrossing}, and~\eqref{eq:GradeNegCrossing},
  considering all the possible orientations of the braidlike positive
  and negative crossing.
\end{proof}

\newcommand\Canon{K}
\newcommand\gen{{\mathbf 1}}
\newcommand\NxN{\North\North}
\newcommand\NxW{\North\West}
\newcommand\NxE{\North\East}
\newcommand\ExS{\East\South}
\newcommand\SxS{\South\South}
\newcommand\WxS{\West\South}
\newcommand\ExW{\East\West}
\section{Braid relations}
\label{sec:Braid}

We prove that the $DA$ bimodules $\Pos^i$ and $\Neg^i$
satisfy the following braid relations:

\begin{thm}
  \label{thm:BraidRelation}
  Fix $m$, $k$ with $0\leq k\leq m+1$, 
  $i$ with $1\leq i\leq m-1$,
  and $\Upwards\subset \{1,\dots,m\}$.
  Let
  $\Blg_1=\Blg(m,k,\Upwards)$ and $\Blg_2=\Blg(m,k,\tau_i(\Upwards))$.
  Then,
  \begin{equation}
    \lsup{\Blg_1}\Pos^i_{\Blg_2}\DT~^{\Blg_2}\Neg^i_{\Blg_1}\simeq~^{\Blg_1}\Id_{\Blg_1}\!\simeq~
    \lsup{\Blg_1}\Neg^i_{\Blg_2}\DT~^{\Blg_2}\Pos^i_{\Blg_1}
    \label{eq:InvertPos} 
  \end{equation}
  Given $j\neq i$ with $1\leq j\leq m-1$, let
  $\Blg_3=\Blg(m,k,\tau_j \tau_i(\Upwards))$ and
  $\Blg_4=\Blg(m,k,\tau_j(\Upwards))$.
  If $|i-j|>1$
  \begin{equation}
    \lsup{\Blg_3}\Pos^j_{\Blg_2}\DT~^{\Blg_2}\Pos^i_{\Blg_1}
    \simeq~
    \lsup{\Blg_3}\Pos^i_{\Blg_4}\DT~^{\Blg_4}\Pos^j_{\Blg_1}
    \label{eq:FarBraids}.
    \end{equation}
    while if $j=i+1$,
    let 
    $\Blg_5=\Blg(m,k,\tau_i\tau_{i+1}\tau_i(\Upwards))$
    and $\Blg_6=\Blg(m,k,\tau_{i}\tau_{i+1}(\Upwards))$;
    then,    
    \begin{equation}
      \lsup{\Blg_5}\Pos^i_{\Blg_3}\DT
      \lsup{\Blg_3}\Pos^{i+1}_{\Blg_2}\DT
      \lsup{\Blg_2}\Pos^i_{\Blg_1}
      \simeq~
      \lsup{\Blg_5}\Pos^{i+1}_{\Blg_6}\DT~^{\Blg_6}\Pos^{i}_{\Blg_4}
      \DT~^{\Blg_4}\Pos^{i+1}_{\Blg_1}.
    \label{eq:NearBraids}
    \end{equation}
\end{thm}

\begin{figure}[ht]
\input{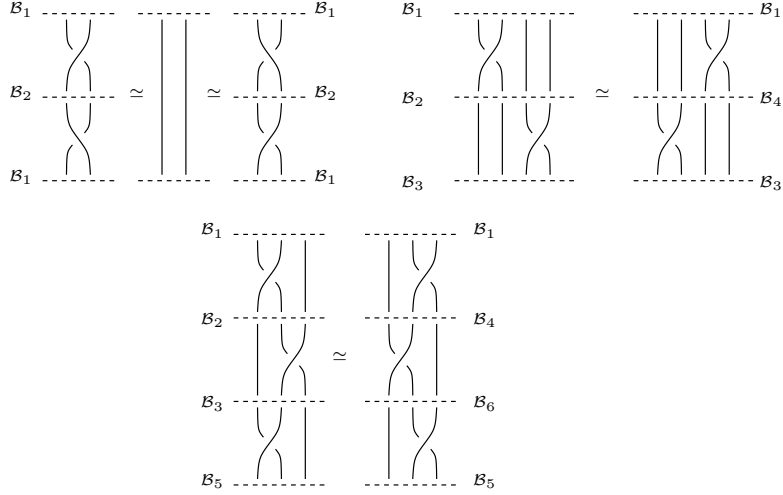}
\caption{\label{fig:BraidRelation} 
  Bimodules and algebras appearing in Theorem~\ref{thm:BraidRelation}.}
\end{figure}

We will prove the above with the help of the following:

\begin{lemma}
  \label{lem:CrossingDADD}
  Fix $m$, $k$ with $0\leq k\leq m+1$, $i$ with $1\leq i\leq m-1$, and
  $\Upwards\subset \{1,\dots,m\}$.  Let $\Blg_1=\Blg(m,k,\Upwards)$,
  $\Blg_2=\Blg(m,k,\tau_i(\Upwards))$,
  $\Blg_1'=\Blg(m,m+1-k,\{1,\dots,m\}\setminus\Upwards)$.
    Let
  $\lsup{\Blg_2}\Pos^i_{\Blg_1}$ be the $DA$ bimodule
  from Section~\ref{sec:CrossingDA}, and let $\lsup{\Blg_1,\Blg_1'}\CanonDD$ be the canonical type
  $DD$ bimodule from Section~\ref{subsec:CanonDD}. Then,
  \begin{align}
    \lsup{\Blg_2}\Pos^i_{\Blg_1} \DT
    \lsup{\Blg_1,\Blg_1'}\CanonDD\simeq~\lsup{\Blg_2,\Blg_1'}\Pos_i 
    \label{eq:PosDual} \\
    \lsup{\Blg_2}\Neg^i_{\Blg_1} \DT
    \lsup{\Blg_1,\Blg_1'}\CanonDD\simeq~\lsup{\Blg_2,\Blg_1'}\Neg_i 
    \label{eq:NegDual}
  \end{align}
  where the type $DD$ bimodule appearing on the right
  is the one defined in Section~\ref{sec:Crit}.
\end{lemma}

\begin{proof}
  We start by verifying Equation~\eqref{eq:PosDual}; and 
  for simplicity, we assume $i=1$.
  The computation depends on $\Upwards\cap \{1,2\}$.

  We start with the case where $\{1,2\}\subset\Upwards$.  In fact, in
  this case, $\Pos^i\DT \CanonDD= \Pos_i$.  The arrows connecting
  $\North$, $\West$, and $\East$ are all induced from $\delta_2$ actions on
  $\Pos^i$.  For example, the differential from $\North$ to $\East$
  labelled by $L_1 L_2\otimes R_1$ from Figure~\ref{eq:PositiveCrossing} is obtained
  by pairing the term $L_1\otimes R_1$ term in the canonical
  $DD$ bimodule pairs with the action $\delta^1_2(\North,L_1)=L_1 L_2\otimes
  \East$ in the $\Pos^1$,  as shown on the left diagram:
\[
\begin{array}{ll}
\mathcenter{\begin{tikzpicture}[scale=.8]
  \node at (0,1) (inX) {$\North$};
%  \node at (2,-1) (mult) {$\Pi$};
  \node at (2,-2) (outC) {};
  \node at (1,1) (inY) {};
  \node at (0,-2) (outX) {$\East$};
  \node at (1,-2) (outY) {};
  \node at (-1,-2) (outA) {};
  \node at (1,0) (diffY) {$\delta^1_{\CanonDD}$};
  \node at (0,-1) (diffX) {$\delta^1_{\Pos^1}$};
  \draw[modarrow] (inX) to (diffX);
  \draw[modarrow] (diffX) to (outX);
  \draw[modarrow] (inY) to (diffY);
  \draw[modarrow] (diffY) to (outY);
  \draw[algarrow] (diffY) to node[above,sloped] {\tiny{$L_1$}}(diffX);
  \draw[algarrow] (diffY) to node[above,sloped] {\tiny{$R_1$}}(outC);
  \draw[algarrow] (diffX) to node[above,sloped] {\tiny{$L_1 L_2$}}(outA);
\end{tikzpicture}}
&
\mathcenter{\begin{tikzpicture}[scale=.8]
  \node at (0,1) (inX) {$\North$};
%  \node at (2,-1) (mult) {$\Pi$};
  \node at (2,-2) (outC) {};
  \node at (1,1) (inY) {};
  \node at (0,-2) (outX) {$\North$};
  \node at (1,-2) (outY) {};
  \node at (-1,-2) (outA) {};
  \node at (1,0) (diffY) {$\delta^1_{\CanonDD}$};
  \node at (0,-1) (diffX) {$\delta^1_{\Pos^1}$};
  \draw[modarrow] (inX) to (diffX);
  \draw[modarrow] (diffX) to (outX);
  \draw[modarrow] (inY) to (diffY);
  \draw[modarrow] (diffY) to (outY);
  \draw[algarrow] (diffY) to node[above,sloped] {\tiny{$C_1$}}(diffX);
  \draw[algarrow] (diffY) to node[above,sloped] {\tiny{$U_1$}}(outC);
  \draw[algarrow] (diffX) to node[above,sloped] {\tiny{$C_2$}}(outA);
\end{tikzpicture}}
\end{array}
\]
On the right, we demonstrate how to construct 
the term $(C_2\otimes U_1)\otimes \North$ in $\delta^1(\North)$
(of Type~\ref{type:UC}).

The differentials into $\South$ come from the $\delta^1_1$ action on
$\Pos^i$, and the differentials out of $\South$ come from $\delta^1_2$
and $\delta^1_3$ actions on $\Pos^i$.  For example, the two terms out
of $\South$ involving the generator 
$\East$ arise as follows:
\[
\begin{array}{ll}
\mathcenter{\begin{tikzpicture}[scale=.8]
  \node at (0,1) (inX) {$\South$};
%  \node at (2,-1) (mult) {$\Pi$};
  \node at (2,-2) (outC) {};
  \node at (1,1) (inY) {};
  \node at (0,-2) (outX) {$\East$};
  \node at (1,-2) (outY) {};
  \node at (-1,-2) (outA) {};
  \node at (1,0) (diffY) {$\delta^1_{\CanonDD}$};
  \node at (0,-1) (diffX) {$\delta^1_{\Pos^1}$};
  \draw[modarrow] (inX) to (diffX);
  \draw[modarrow] (diffX) to (outX);
  \draw[modarrow] (inY) to (diffY);
  \draw[modarrow] (diffY) to (outY);
  \draw[algarrow] (diffY) to node[above,sloped] {\tiny{$C_1$}}(diffX);
  \draw[algarrow] (diffY) to node[above,sloped] {\tiny{$U_1$}}(outC);
  \draw[algarrow] (diffX) to node[above,sloped] {\tiny{$L_2$}}(outA);
\end{tikzpicture}}
&
\mathcenter{\begin{tikzpicture}[scale=.8]
  \node at (-.5,1) (inX) {$\South$};
  \node at (2.5,-2) (multC) {$\mu_2$};
  \node at (3,-3) (outC) {};
  \node at (1,1) (inY) {};
  \node at (-.5,-3) (outX) {$\East$};
  \node at (1,-3) (outY) {};
  \node at (-1,-3) (outA) {};
  \node at (1,0) (diffY) {$\delta^1_{\CanonDD}$};
  \node at (1,-1.5) (diffY2) {$\delta^1_{\CanonDD}$};
  \node at (-.5,-2) (diffX) {$\delta^1_{\Pos^1}$};
  \draw[modarrow] (inX) to (diffX);
  \draw[modarrow] (diffX) to (outX);
  \draw[modarrow] (inY) to (diffY);
  \draw[modarrow] (diffY) to (diffY2);
  \draw[modarrow] (diffY2) to (outY);
  \draw[algarrow] (diffY) to node[above,sloped] {\tiny{$R_1$}}(diffX);
  \draw[algarrow] (diffY) to node[above,sloped] {\tiny{$L_1$}}(multC);
  \draw[algarrow] (diffY2) to node[above,sloped] {\tiny{$L_2$}}(multC);
  \draw[algarrow] (diffY2) to node[above,sloped] {\tiny{$R_2$}}(diffX);
  \draw[algarrow] (diffX) to node[above,sloped] {\tiny{$R_1$}}(outA);
  \draw[algarrow] (multC) to node[above,sloped] {\tiny{$L_1 L_2$}} (outC);
\end{tikzpicture}}
\end{array}
\]
(Note that since $\CanonDD$ bimodule is a left/left bimodule, the multiplication appearing on the right
is taking place in $(\Blg_1')^{\op}$.)

Terms of Type~\ref{type:OutsideLR} are easily constructed from pairing
differentials in the identity $DD$ bimodule with the part of $\Pos^i$
that behaves like an identity bimodule.

  Consider next the case where $\Upwards\cap\{1,2\}=\{2\}$. Then,
  $\Pos^1\DT \CanonDD$ is given as in Figure~\ref{fig:PtimesK2},
  \begin{figure}
    \begin{tikzpicture}[scale=1.8]
    \node at (0,4) (N) {${\mathbf N}$};
    \node at (-1.5,2.5) (W) {${\mathbf W}$} ;
    \node at (1.5,2.5) (E) {${\mathbf E}$} ;
    \node at (0,1) (S) {${\mathbf S}$} ;
    \draw[->] (S) [bend left=12] to node[below,sloped] {\tiny{$R_1\otimes U_{2}+L_{2}\otimes R_{2}R_{1}$}}  (W)  ;
    \draw[->] (W) [bend left=12] to node[above,sloped] {\tiny{$L_{1}\otimes 1$}}  (S)  ;
    \draw[->] (E)[bend right=12] to node[above,sloped] {\tiny{$R_{2}\otimes 1$}}  (S)  ;
    \draw[->] (S)[bend right=12] to node[below,sloped] {\tiny{$R_{1} \otimes L_{1} L_{2}$}} (E) ;
    \draw[->] (W)[bend right=12] to node[below,sloped] {\tiny{$1\otimes L_{1}$}} (N) ;
    \draw[->] (N)[bend right=12] to node[above,sloped] {\tiny{$U_{2}\otimes R_{1} + R_{2} R_{1} \otimes L_{2}$}} (W) ;
    \draw[->] (E)[bend left=12] to node[below,sloped]{\tiny{$1\otimes R_{2}$}} (N) ;
    \draw[->] (N)[bend left=12] to node[above,sloped]{\tiny{$U_{1}\otimes L_{2} + L_{1} L_{2}\otimes R_{1}$}} (E) ;
    \draw[->] (N) [loop above] to node[above]{\tiny{$C_{1}\otimes U_{2} + U_2\otimes C_{1}$}} (N);
    \draw[->] (W) [loop left] to node[above,sloped]{\tiny{$C_{1}\otimes U_{2} + U_2\otimes C_{1}$}} (W);
    \draw[->] (E) [loop right] to node[above,sloped]{\tiny{$C_{1}\otimes U_{2}$}} (E);
    \draw[->] (S) [loop below] to node[below]{\tiny{$C_{1}\otimes U_{2}$}} (S);
    \draw[->] (W) to node[above,sloped,pos=.3] {\tiny{$L_1 L_2\otimes C_1$}} (E) ;
    \draw[->] (S) to node[below,sloped,pos=.3] {\tiny{$L_2\otimes C_1 R_2$}}  (N) ;
    \end{tikzpicture}
    \caption{\label{fig:PtimesK2}{\bf $\Pos^1\DT \CanonDD$ when $\Upwards\cap\{1,2\}=\{2\}$}}
  \end{figure}
    and outside actions (i.e. of Type~\ref{type:OutsideLRP} and~\ref{type:UCP} with $j\neq 1,2$).
    Consider the 
    map $h^1\colon \Pos^1\DT \CanonDD \to \Pos_1$
    \[ h^1(X) = \left\{\begin{array}{ll}
        \South+ (L_2\otimes C_1)\otimes \East & {\text{if $X=\South$}} \\
        X &{\text{otherwise.}}
      \end{array}
    \right.\]
    Let $g^1\colon \Pos_1\to \Pos^1\DT \CanonDD$ be given by the same formula.
    It is easy to verify that $h^1$ and $g^1$ are homomorphisms of type $DD$ structures, 
    $h^1\circ g^1 = \Id$, and $g^1\circ h^1=\Id$.
    
    The case where $\Upwards\cap\{1,2\}=\{1\}$ works similarly.
    
    When $\Upwards\cap\{1,2\}=\emptyset$,
    $\Pos^1\DT \CanonDD$ is given as in Figure~\ref{fig:PtimesKe}
    \begin{figure}
    \begin{tikzpicture}[scale=1.85]
    \node at (0,4) (N) {${\mathbf N}$} ;
    \node at (-2,2.5) (W) {${\mathbf W}$} ;
    \node at (2,2.5) (E) {${\mathbf E}$} ;
    \node at (0,.5) (S) {${\mathbf S}$} ;
    \draw[->] (S) [bend left=10] to node[below,sloped] {\tiny{$R_1 U_2 \otimes C_1 C_2+L_{2}\otimes R_{2}R_{1}$}}  (W)  ;
    \draw[->] (W) [bend left=10] to node[above,sloped] {\tiny{$L_{1}\otimes 1$}}  (S)  ;
    \draw[->] (E) [bend right=10] to node[above,sloped] {\tiny{$R_{2}\otimes 1$}}  (S)  ;
    \draw[->] (S)[bend right=10] to node[below,sloped] {\tiny{$R_{1} \otimes L_{1} L_{2} + L_2 U_1\otimes C_1 C_2$}} (E) ;
    \draw[->] (W)[bend right=10]to node[below,sloped] {\tiny{$1\otimes L_{1}$}} (N) ;
    \draw[->] (N)[bend right=10] to node[above,sloped] {\tiny{$U_{2}\otimes R_{1} + R_{2} R_{1} \otimes L_{2}$}} (W) ;
    \draw[->] (E)[bend left=10]to node[below,sloped]{\tiny{$1\otimes R_{2}$}} (N) ;
    \draw[->] (N)[bend left=10] to node[above,sloped]{\tiny{$U_{1}\otimes L_{2} + L_{1} L_{2}\otimes R_{1}$}} (E) ;
    \draw[->] (N) [loop above] to node[above]{\tiny{$U_1\otimes C_2 + U_2\otimes C_1$}} (N);
    \draw[->] (W) [loop left] to node[above,sloped]{\tiny{$U_2\otimes C_{1}$}} (W);
    \draw[->] (E) [loop right] to node[above,sloped]{\tiny{$U_1\otimes C_{2}$}} (E);
    \draw[->] (E) [bend right=5] to node[above,pos=.3] {\tiny{$R_2 R_1 \otimes C_2$}} (W) ;
    \draw[->] (W) [bend right=5] to node[below,pos=.3] {\tiny{$L_1 L_2\otimes C_1$}} (E) ;
    \draw[->] (S) to node[below,sloped,pos=.3] {\tiny{$L_2\otimes C_1 R_2 + R_1\otimes L_1 C_2$}} (N) ;
    \end{tikzpicture}
    \caption{\label{fig:PtimesKe} {\bf $\Pos^1\DT\CanonDD$ when $\Upwards\cap\{1,2\}=\emptyset$.}}
    \end{figure}

    Consider the 
    map $h^1\colon \Pos^1\DT \CanonDD \to \Pos_1$
    \[ h^1(X) = \left\{\begin{array}{ll}
        \South+ (L_2\otimes C_1)\cdot \East + (R_1\otimes C_2)\cdot 
        \West & {\text{if $X=\South$}} \\
        X &{\text{otherwise.}}
      \end{array}
    \right.\]
    Let $g^1\colon \Pos_1\to \Pos^1\DT \CanonDD$ be given by the same formula.
    It is easy to verify that $h^1$ and $g^1$ are homomorphisms of type $DD$ structures, 
    $h^1\circ g^1 = \Id$, and $g^1\circ h^1=\Id$.
    
    Equation~\eqref{eq:PosDual} in cases where $i\neq 1$ works the same way.
    Equation~\eqref{eq:NegDual} can be proved similarly. (Or alternatively,
    it can be seen as a consequence of Equation~\eqref{eq:PosDual} and the symmetry
    of the bimodules, phrased in terms of the opposite algebras.)
\end{proof}

\begin{lemma}
  \label{lem:InvertPos}
  Equation~\eqref{eq:InvertPos} holds.
\end{lemma}

\begin{proof}
  Since $(\Pos^i\DT\Neg^i)\DT\CanonDD \simeq 
  \Pos^i \DT(\Neg^i\DT \CanonDD) \simeq
  \Pos^i \DT\Neg_i$
  (by associativity of $\DT$ and Lemma~\ref{lem:CrossingDADD}),
  the verification that
  $\Pos^i\DT \Neg^i\simeq \Id_{\Blg_1}$, will follow from the identity
  \begin{equation}
    \label{eq:EasyNegPosInv}
    \Pos^i\DT\Neg_i\simeq  \CanonDD,
  \end{equation}
  which we verify presently. For notational simplicity, we assume that $i=1$.

  The verification of Equation~\eqref{eq:EasyNegPosInv} can be divided
  into cases, according to $\Upwards\cap \{1,2\}$.
  Consider first the case where  $\Upwards\cap\{1,2\}=\{1,2\}$.
  Classify the  generators of $\Pos^i\DT\Neg_i$ into types, labeled
  $XY$, where $X,Y\in\{\North,\South,\West,\East\}$, where the first
  symbol $X$ denotes the generator type in $\Pos^i$ and $Y$
  denotes the generator type in $\Neg_i$, as illustrated in Figure~\ref{fig:InvertPos}.
  \begin{figure}[ht]
    \input{InvertPos.pstex_t}
    \caption{\label{fig:InvertPos} Tensoring bimodules on the left; at the
      right, the generator type of $\NxE=\North\DT\East\DT \gen$.}
  \end{figure}

  By a straightforward computation,
  we find the following kinds of differential:
  the outside actions, $C_1\otimes U_1$,  $C_2\otimes U_2$, and additional arrows indicated on the diagram on the left
  of Figure~\ref{fig:NegPos}.
  \begin{figure}
  \[\mathcenter{
    \begin{tikzpicture}[scale=1.7]
    \node at (0,4) (NN) {$\NxN$} ;
    \node at (-2,3) (NW) {$\NxW$} ;
    \node at (2,3) (NE) {$\NxE$} ;
    \node at (-2,1) (WS) {$\WxS$} ;
    \node at (2,1) (ES) {$\ExS$} ;
    \node at (0,0) (SS) {$\SxS$} ;
    \draw[<-] (NN) [bend right=15] to node[above,sloped] {\tiny{$U_1\otimes L_1+L_1L_2\otimes R_2$}}  (NW)  ;
    \draw[<-] (NW) [bend right=15] to node[below,sloped] {\tiny{$1\otimes R_1$}}  (NN)  ;
    \draw[<-] (NN) [bend left=15] to node[above,sloped] {\tiny{$U_2\otimes R_2 + R_2 R_1\otimes L_1$}}  (NE)  ;
    \draw[<-] (NE) [bend left=15] to node[below,sloped] {\tiny{$1\otimes L_2$}}  (NN)  ;
    \draw[<-] (NW) [bend right=15] to node[below,sloped] {\tiny{$1\otimes 1$}}  (WS)  ;
    \draw[<-] (WS) [bend right=15] to node[below,sloped] {\tiny{$U_2\otimes U_2$}}  (NW)  ;
    \draw[<-] (NE) [bend left=15] to node[above,sloped] {\tiny{$1\otimes 1$}}  (ES)  ;
    \draw[<-] (ES) [bend left=15] to node[above,sloped] {\tiny{$U_1\otimes U_1$}}  (NE)  ;
    \draw[<-] (WS) to node[above, sloped,pos=.2] {\tiny{$R_2 R_1\otimes U_1 +  U_2\otimes R_2 R_1$}} (NE);
    \draw[<-] (ES) to node[above, sloped,pos=.2] {\tiny{$L_1 L_2\otimes U_2+U_1\otimes L_1 L_2$}} (NW);
    \draw[<-] (NN) to node[above,sloped,pos=.25] {\tiny{$R_1\otimes L_1 + L_2\otimes R_2$}}  (SS) ;
    \draw[<-] (WS) [bend right=15] to node[below,sloped] {\tiny{$R_1\otimes U_1+L_2\otimes R_2 R_1$}}  (SS)  ;
    \draw[<-] (SS) [bend right=15] to node[above,sloped] {\tiny{$L_1\otimes 1$}}  (WS)  ;
    \draw[<-] (ES) [bend left=15] to node[below,sloped] {\tiny{$L_2\otimes U_2+R_1\otimes L_1 L_2$}}  (SS)  ;
    \draw[<-] (SS) [bend left=15] to node[above,sloped] {\tiny{$R_2\otimes 1$}}  (ES)  ;
    \end{tikzpicture}}
  \qquad
    \Rightarrow
    \qquad
    \mathcenter{\begin{tikzpicture}[y=48pt,x=1in]
    \node at (0,3) (N) {$\North$} ;
    \node at (0,0) (S) {$\South$} ;
    \draw[->] (N) [bend right=20] to node[below,sloped] {\tiny{$L_1\otimes R_1 + R_2\otimes L_2$}}  (S)  ;
    \draw[->] (S) [bend right=20]to node[below,sloped] {\tiny{$R_1\otimes L_1 + L_2\otimes R_2$}} (N);
    \end{tikzpicture}}\]
    \caption{\label{fig:NegPos}{\bf Arrows in $\Pos^i\DT\Neg_i$ when $\Upwards\cap\{1,2\}=\{1,2\}$.}}
    \end{figure}
    Canceling arrows 
    (i.e. setting $\North=\NxN+(1\otimes R_1) \WxS +
     (1\otimes L_2)\ExS$ and $\South=\SxS$; and noting that the quotient complex is acyclic)
     gives the diagram on the right (in addition
    to the $C_i\otimes U_i$ arrows for $i=1,2$ and the outside
    arrows).  This is, in fact, the canonical $DD$ bimodule, verifying
    Equation~\eqref{eq:EasyNegPosInv}.

    Similarly, if $\Upwards\cap\{1,2\}=\{2\}$, we compute the bimodule to be as given in Figure~\ref{fig:NegPos2}.
    \begin{figure}\[
      \mathcenter{
    \begin{tikzpicture}[scale=1.7]
    \node at (0,4) (NN) {$\NxN$} ;
    \node at (-2,3) (NW) {$\NxW$} ;
    \node at (2,3) (NE) {$\NxE$} ;
    \node at (-2,1) (WS) {$\WxS$} ;
    \node at (2,1) (ES) {$\ExS$} ;
    \node at (0,0) (SS) {$\SxS$} ;
    \draw[->] (NN) [loop above] to node[above,sloped]{\tiny{$U_1 \otimes C_1$}} (NN);
    \draw[->] (NW) [loop left] to node[above,sloped]{\tiny{$U_1\otimes C_1$}} (NW);
    \draw[->] (NE) [loop right] to node[above,sloped]{\tiny{$U_1\otimes C_1$}} (NE);
    \draw[->] (ES) [loop right] to node[above,sloped]{\tiny{$U_1\otimes C_1$}} (ES);
    \draw[<-] (NN) [bend right=15] to node[above,sloped] {\tiny{$U_1\otimes L_1+L_1L_2\otimes R_2$}}  (NW)  ;
    \draw[<-] (NW) [bend right=15] to node[below,sloped] {\tiny{$1\otimes R_1$}}  (NN)  ;
    \draw[<-] (NN) [bend left=15] to node[above,sloped] {\tiny{$U_2\otimes R_2 + R_2 R_1\otimes L_1$}}  (NE)  ;
    \draw[<-] (NE) [bend left=15] to node[below,sloped] {\tiny{$1\otimes L_2$}}  (NN)  ;
    \draw[<-] (NW) [bend right=15] to node[below,sloped] {\tiny{$1\otimes 1$}}  (WS)  ;
    \draw[<-] (WS) [bend right=15] to node[below,sloped] {\tiny{$U_2\otimes U_2$}}  (NW)  ;
    \draw[<-] (NE) [bend left=15] to node[above,sloped] {\tiny{$1\otimes 1$}}  (ES)  ;
    \draw[<-] (ES) [bend left=15] to node[above,sloped] {\tiny{$U_1\otimes U_1$}}  (NE)  ;
    \draw[->] (ES) to  node[above,sloped,pos=.3] {\tiny{$R_2 R_1 \otimes C_1$}} (WS) ;
    \draw[->] (SS) to node[above,sloped,pos=.8] {\tiny{$R_1\otimes C_1$}} (NW) ;
    \draw[<-] (WS) to node[above, sloped,pos=.8] {\tiny{$R_2 R_1\otimes U_1 + U_2\otimes R_2 R_1$}} (NE);
    \draw[<-] (ES) to node[above, sloped,pos=.2] {\tiny{$L_1 L_2\otimes U_2+U_1\otimes L_1 L_2$}} (NW);
    \draw[<-] (NN) to node[above,sloped,pos=.25] {\tiny{$R_1\otimes L_1 + L_2\otimes R_2$}}  (SS) ;
    \draw[<-] (WS) [bend right=15] to node[below,sloped] {\tiny{$L_2\otimes R_2 R_1$}}  (SS)  ;
    \draw[<-] (SS) [bend right=15] to node[above,sloped] {\tiny{$L_1\otimes 1$}}  (WS)  ;
    \draw[<-] (ES) [bend left=15] to node[below,sloped] {\tiny{$L_2\otimes U_2+R_1\otimes L_1 L_2$}}  (SS)  ;
    \draw[<-] (SS) [bend left=15] to node[above,sloped] {\tiny{$R_2\otimes 1$}}  (ES)  ;
    \end{tikzpicture}}\]
    \caption{\label{fig:NegPos2}{\bf Arrows in $\Pos^i\DT\Neg_i$ when $\Upwards\cap\{1,2\}=\{2\}$.}}
      \end{figure}
  Again, we have suppressed the outside arrows, and additional arrows of the
  form $C_2\otimes U_2$ that connect every generator to themselves. By
  contrast, we did include self-arrows of the form $U_1\otimes C_1$,
  since not all generator types have these kinds of terms.  We can
  reduce to the previous case by changing basis; replacing $\SxS$ by
  $\SxS + (R_1\otimes C_1)\otimes \WxS$. A similar computation works for
  $\Upwards\cap\{1,2\}=\{1\}$, only now the reduction to the previous case
  involves the change of basis $\SxS + (L_2\otimes C_2)\otimes
  \ExS$. Finally, when $\Upwards\cap \{1,2\}=\emptyset$, we get a
  bimodule which can be reduced to the earlier case by the basis change
  $\SxS + (R_1\otimes C_1)\otimes \WxS + (L_2\otimes C_2)\otimes\ExS$.

  Once again, cancelling arrows we find that this gives the identity $DD$ bimdoule. 
  The other two computations needed to verify Equation~\eqref{eq:EasyNegPosInv} work similarly. 

  Having verified that $\Pos^i\DT\Neg^i\simeq~
  \lsup{\Blg_1}\Id_{\Blg_1}$, the verification that
  $\Neg^i\DT\Pos^i\simeq~ \lsup{\Blg_1}\Id_{\Blg_1}$ is a formality:
  \begin{align*}
    \Id&={\overline \Id} \simeq {\overline{\lsup{\Blg_1}{\Pos^i_{\Blg_2}}\DT \lsup{\Blg_2}{\Neg^i_{\Blg_1}}}}
    \simeq \lsub{\Blg_1}{{{\overline \Neg}^i}}^{\Blg_2}\DT \lsub{\Blg_2}{{\overline\Pos^i}}^{\Blg_1} \\
    &=  \lsup{{\Blg_1^{\op}}}{\overline\Pos^i}\lsub{\Blg_2^{\op}} \DT~
    \lsup{{\Blg_2^{\op}}}{\overline \Neg^i}_{\Blg_1^{\op}} 
    =\lsup{\Blg_1}{\Neg^i}_{\Blg_2} \DT
    \lsup{\Blg_2}{\Pos^i}_{\Blg_2} 
  \end{align*}
\end{proof}

\begin{lemma}
  \label{lem:FarBraids}
  Equation~\eqref{eq:FarBraids} holds.
\end{lemma}

\begin{proof}
  In view of Lemma~\ref{lem:CrossingDADD} and the invertibility of
  $\CanonDD$, it suffices
  to show that
  \[ \Pos^j\DT \Pos_i\sim \Pos^i \DT\Pos_j.\]

  Note that generators for both bimodules in Equation~\eqref{eq:FarBraids} correspond to partial Kauffman
  states in the partial knot diagrams contianing the two crossings. Clearly, these generators are independent of the order;
  to see that  the bimodules are independent of the order, it suffices to construct a third, more symmetric bimodule 
  (where we think of the two crossings as appearing in the same level), that is quasi-isomorphic to both.

  Generators for this bimodule will also correspond to partial Kauffman states.
  There will in general be $16$ generator types, labelled by pairs of letters among $\North$, $\South$, $\West$, and $\East$,
  corresponding to the four local choices at each of the two crossings.
  The computation is mostly straightforward.
  In the special case where the two crossings are adjacent,
  there are only $15$ generator types: there are no generators of type $\ExW$. This
  occurs, for example , when $i=1$ and $j=3$; see Figure~\ref{fig:FarBraids}. In this case, the $DD$
  bimodule is as specified in Figure~\ref{fig:FarBraidDD}
  (where,   as usual, we have suppressed here the distant arrows and the additional arrows of the form 
  $U_\tau(i)\otimes C_i$ or $C_\tau(i)\otimes U_i$).
\end{proof}
  \begin{figure}[ht]
    \input{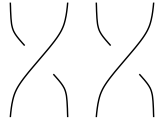}
    \caption{\label{fig:FarBraids} 
    When the two crossings appearing in Equation~\eqref{eq:FarBraids} are adjacent,
    there are only $15$ types of partial Kauffman states; a partial Kauffman state cannot
    associate $\East$ to the crossing on the left and 
    $\West$ to the one on the right.}
  \end{figure}

  \begin{figure}
  \begin{center}
  \begin{tikzpicture}[scale=.81]%[Scaleque=1.2]%[y=48pt,x=.95in]
    \node at (0,8) (NN) {${\mathbf {NN}}$} ;
    \node at (4,8) (NE) {${\mathbf {NE}}$} ;
    \node at (8,8) (NS) {${\mathbf {NS}}$} ;    
    \node at (0,4) (WN) {${\mathbf {WN}}$} ;
    \node at (4,4) (WE) {${\mathbf {WE}}$} ;
    \node at (8,4) (WS) {${\mathbf {WS}}$} ;    
    \node at (0,0) (SN) {${\mathbf {SN}}$} ;
    \node at (4,0) (SE) {${\mathbf {SE}}$} ;
    \node at (8,0) (SS) {${\mathbf {SS}}$} ;    
    \node at (-4,8) (NW1) {${\mathbf {NW}}$} ;
    \node at (-4,4) (WW1) {${\mathbf {WW}}$} ;
    \node at (-4,0) (SW1) {${\mathbf {SW}}$} ;
    \node at (12,8) (NW2) {$ $} ;
    \node at (12,4) (WW2) {$ $} ;
    \node at (12,0) (SW2) {$ $} ;
    \node at (0,-4) (EN1) {${\mathbf {EN}}$} ;
    \node at (4,-4) (EE1) {${\mathbf {EE}}$} ;
    \node at (8,-4) (ES1) {${\mathbf {ES}}$} ;
    \node at (0,12) (EN2) {$ $} ;
    \node at (4,12) (EE2) {$ $} ;
    \node at (8,12) (ES2) {$ $} ;
    \draw [->] [bend left=7] (NW1) to node[above,sloped] {\tiny{$1\!\otimes\! L_3$}} (NN);
    \draw [->] [bend left=7] (NN) to node[below,sloped] {\tiny{$U_4\!\otimes\! R_3+R_4 R_3\!\otimes\! L_4$}} (NW1);
    \draw [->] [bend left=7] (NN) to node[above,sloped] {\tiny{$U_3\otimes L_4+L_3 L_4 \otimes R_3$}} (NE) ;
    \draw [->] [bend left=7] (NE) to node[below,sloped] {\tiny{$1\!\otimes\! R_4$}} (NN) ;
    \draw [->] [bend left=7] (NE) to node[above,sloped] {\tiny{$R_4\!\otimes\! 1$}} (NS) ;
    \draw [->] [bend left=7] (NS) to node[below,sloped] {\tiny{$L_4\!\otimes\! U_3+R_3\!\otimes\! L_3 L_4$}} (NE) ;
    \draw [->] [bend left=7] (NS) to node[above,sloped] {\tiny{$R_3 \!\otimes\! U_4 + L_4\!\otimes\! R_4 R_3$}} (NW2) ;
    \draw [->] [bend left=7] (NW2) to node[below,sloped] {\tiny{$L_3\!\otimes\! 1$}} (NS) ;

    \draw [->] [bend left=7] (WW1) to node[above,sloped] {\tiny{$1\!\otimes\! L_3$}} (WN);
    \draw [->] [bend left=7] (WN) to node[below,sloped] {\tiny{$U_4\!\otimes\! R_3+R_4 R_3\!\otimes\! L_4$}} (WW1);
    \draw [->] [bend left=7] (WN) to node[above,sloped] {\tiny{$U_3\otimes L_4+L_3 L_4 \otimes R_3$}} (WE) ;
    \draw [->] [bend left=7] (WE) to node[below,sloped] {\tiny{$1\!\otimes\! R_4$}} (WN) ;
    \draw [->] [bend left=7] (WE) to node[above,sloped] {\tiny{$R_4\!\otimes\! 1$}} (WS) ;
    \draw [->] [bend left=7] (WS) to node[below,sloped] {\tiny{$L_4\!\otimes\! U_3+R_3\!\otimes\! L_3 L_4$}} (WE) ;
    \draw [->] [bend left=7] (WS) to node[above,sloped] {\tiny{$R_3 \!\otimes\! U_4 + L_4\!\otimes\! R_4 R_3$}} (WW2) ;
    \draw [->] [bend left=7] (WW2) to node[below,sloped] {\tiny{$L_3\!\otimes\! 1$}} (WS) ;

    \draw [->] [bend left=7] (SW1) to node[above,sloped] {\tiny{$1\!\otimes\! L_3$}} (SN);
    \draw [->] [bend left=7] (SN) to node[below,sloped] {\tiny{$U_4\!\otimes\! R_3+R_4 R_3\!\otimes\! L_4$}} (SW1);
    \draw [->] [bend left=7] (SN) to node[above,sloped] {\tiny{$U_3\otimes L_4+L_3 L_4 \otimes R_3$}} (SE) ;
    \draw [->] [bend left=7] (SE) to node[below,sloped] {\tiny{$1\!\otimes\! R_4$}} (SN) ;
    \draw [->] [bend left=7] (SE) to node[above,sloped] {\tiny{$R_4\!\otimes\! 1$}} (SS) ;
    \draw [->] [bend left=7] (SS) to node[below,sloped] {\tiny{$L_4\!\otimes\! U_3+R_3\!\otimes\! L_3 L_4$}} (SE) ;
    \draw [->] [bend left=7] (SS) to node[above,sloped] {\tiny{$R_3 \!\otimes\! U_4 + L_4\!\otimes\! R_4 R_3$}} (SW2) ;
    \draw [->] [bend left=7] (SW2) to node[below,sloped] {\tiny{$L_3\!\otimes\! 1$}} (SS) ;

    \draw [->] [bend left=7] (EN1) to node[above,sloped] {\tiny{$U_3\otimes L_4+L_3 L_4 \otimes R_3$}} (EE1) ;
    \draw [->] [bend left=7] (EE1) to node[below,sloped] {\tiny{$1\!\otimes\! R_4$}} (EN1) ;
    \draw [->] [bend left=7] (EE1) to node[above,sloped] {\tiny{$R_4\!\otimes\! 1$}} (ES1) ;
    \draw [->] [bend left=7] (ES1) to node[below,sloped] {\tiny{$L_4\!\otimes\! U_3+R_3\!\otimes\! L_3 L_4$}} (EE1) ;

    \draw[->] [bend left=7] (SN) to node[above,sloped] {\tiny{$R_1\!\otimes\! U_2+L_2\!\otimes\! R_2 R_1 $}} (WN) ;
    \draw[->] [bend left=7] (WN) to node[above,sloped] {\tiny{$L_1\!\otimes\! 1 $}} (SN) ;
    \draw[->] [bend left=7] (WN) to node[above,sloped] {\tiny{$1\!\otimes\! L_1$}} (NN) ;
    \draw[->] [bend left=7] (NN) to node[above,sloped] {\tiny{$U_2\!\otimes\! R_1+R_2 R_1\!\otimes\! L_2$}} (WN) ;
    \draw[->] [bend left=7] (EN1) to node[above,sloped] {\tiny{$R_2\!\otimes\! 1$}} (SN) ;
    \draw[->] [bend left=7] (SN) to node[above,sloped] {\tiny{$L_2\!\otimes\! U_1+ R_1\!\otimes\! L_1 L_2$}} (EN1) ;
    \draw[->] [bend left=7] (NN) to node[above,sloped] {\tiny{$U_1\!\otimes\! L_2+ L_1 L_2\!\otimes\! R_1$}} (EN2) ;
    \draw[->] [bend left=7] (EN2) to node[above,sloped] {\tiny{$1\!\otimes\! R_2$}} (NN) ;

    \draw[->] [bend left=7] (SE) to node[above,sloped] {\tiny{$R_1\!\otimes\! U_2+L_2\!\otimes\! R_2 R_1 $}} (WE) ;
    \draw[->] [bend left=7] (WE) to node[above,sloped] {\tiny{$L_1\!\otimes\! 1 $}} (SE) ;
    \draw[->] [bend left=7] (WE) to node[above,sloped] {\tiny{$1\!\otimes\! L_1$}} (NE) ;
    \draw[->] [bend left=7] (NE) to node[above,sloped] {\tiny{$U_2\!\otimes\! R_1+R_2 R_1\!\otimes\! L_2$}} (WE) ;
    \draw[->] [bend left=7] (EE1) to node[above,sloped] {\tiny{$R_2\!\otimes\! 1$}} (SE) ;
    \draw[->] [bend left=7] (SE) to node[above,sloped] {\tiny{$L_2\!\otimes\! U_1+ R_1\!\otimes\! L_1 L_2$}} (EE1) ;
    \draw[->] [bend left=7] (NE) to node[above,sloped] {\tiny{$U_1\!\otimes\! L_2+ L_1 L_2\!\otimes\! R_1$}} (EE2) ;
    \draw[->] [bend left=7] (EE2) to node[above,sloped] {\tiny{$1\!\otimes\! R_2$}} (NE) ;

    \draw[->] [bend left=7] (SS) to node[above,sloped] {\tiny{$R_1\!\otimes\! U_2+L_2\!\otimes\! R_2 R_1 $}} (WS) ;
    \draw[->] [bend left=7] (WS) to node[above,sloped] {\tiny{$L_1\!\otimes\! 1 $}} (SS) ;
    \draw[->] [bend left=7] (WS) to node[above,sloped] {\tiny{$1\!\otimes\! L_1$}} (NS) ;
    \draw[->] [bend left=7] (NS) to node[above,sloped] {\tiny{$U_2\!\otimes\! R_1+R_2 R_1\!\otimes\! L_2$}} (WS) ;
    \draw[->] [bend left=7] (ES1) to node[above,sloped] {\tiny{$R_2\!\otimes\! 1$}} (SS) ;
    \draw[->] [bend left=7] (SS) to node[above,sloped] {\tiny{$L_2\!\otimes\! U_1+ R_1\!\otimes\! L_1 L_2$}} (ES1) ;
    \draw[->] [bend left=7] (NS) to node[above,sloped] {\tiny{$U_1\!\otimes\! L_2+ L_1 L_2\!\otimes\! R_1$}} (ES2) ;
    \draw[->] [bend left=7] (ES2) to node[above,sloped] {\tiny{$1\!\otimes\! R_2$}} (NS) ;

    \draw[->] [bend left=7] (SW1) to node[above,sloped] {\tiny{$R_1\!\otimes\! U_2+L_2\!\otimes\! R_2 R_1 $}} (WW1) ;
    \draw[->] [bend left=7] (WW1) to node[above,sloped] {\tiny{$L_1\!\otimes\! 1 $}} (SW1) ;
    \draw[->] [bend left=7] (WW1) to node[above,sloped] {\tiny{$1\!\otimes\! L_1$}} (NW1) ;
    \draw[->] [bend left=7] (NW1) to node[above,sloped] {\tiny{$U_2\!\otimes\! R_1+R_2 R_1\!\otimes\! L_2$}} (WW1) ;
      \end{tikzpicture}
  \end{center}
  \caption{\label{fig:FarBraidDD}
  {\bf {$\Pos^i\DT\Pos_j$.}} This is  drawn on the torus; e.g. the arrows that point off
  to the right wrap around to the left.}
  \end{figure}
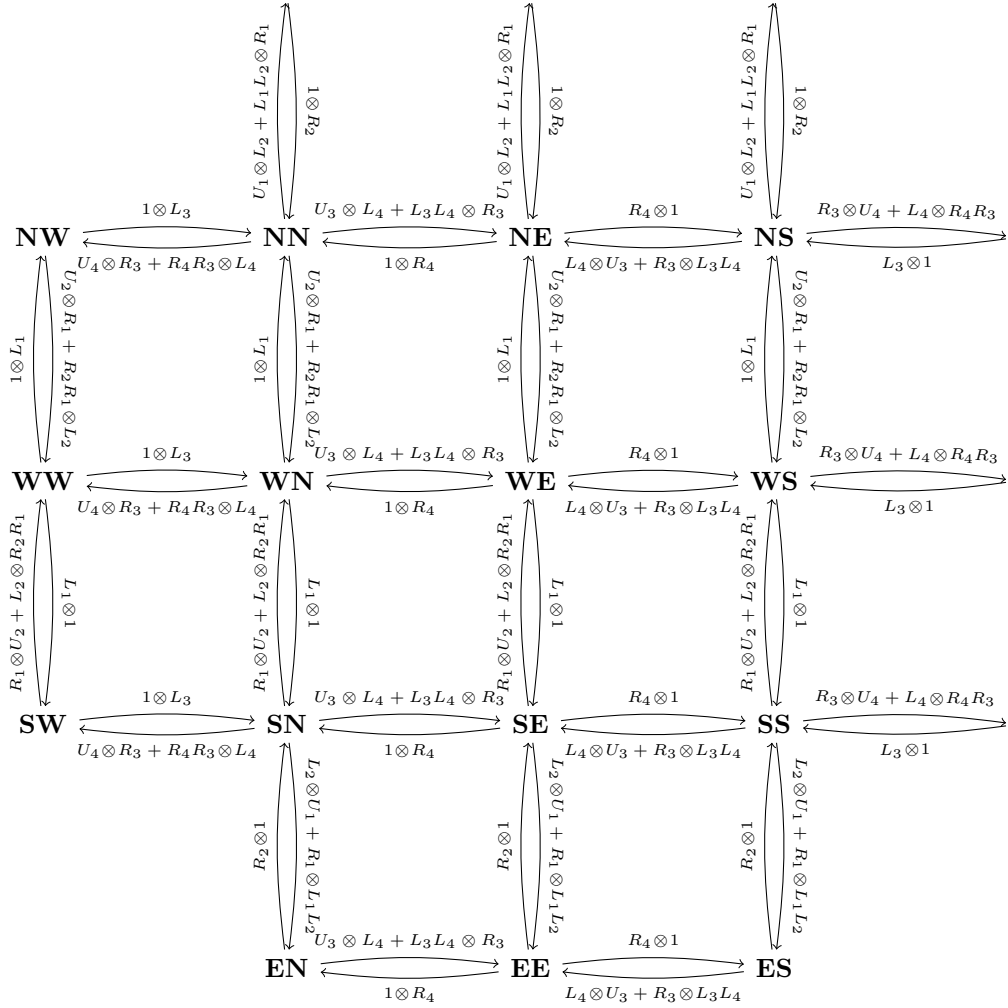

\begin{lemma}
  \label{lem:NearBraids}
  Equation~\eqref{eq:NearBraids} holds.
\end{lemma}

\begin{proof}
  We describe the case where $i=1$.
  We first compute 
  $\lsup{\Blg_3}\Pos^{2}_{\Blg_2} \DT \lsup{\Blg_2,\Blg_1'}\Pos_1$.
  Again, generators correspond to partial Kauffman states, which we now label as an ordered pair,
  showing which crossing is associated to which region; see Figure~\ref{fig:BraidGenNames}.
  \begin{figure}[ht]
    \input{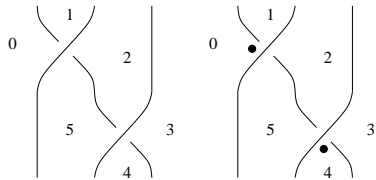}
    \caption{\label{fig:BraidGenNames} 
      At the right, we have shown a generator of type $SW$; or equivalently, $X_{04}$.}
  \end{figure}
  When $\{1,2,3\}\subset \Upwards$, after cancelling arrows, 
  we find that the bimodule is given as in Figure~\ref{fig:NearBraidDD},
  along with the usual outside arrows and self-arrows of the form $C_i
  \otimes U_{\tau_2\tau_1(i)}$
  or $U_i\otimes C_{\tau_2\tau_1(i)}$, depending on whether or not
  $i\in\Upwards$.   \begin{figure}
    \[\begin{tikzpicture}[scale=.77]%[scale=1.2]%[y=48pt,x=.95in]
    \node at (0,0) (X05) {$X_{05}$} ;
    \node at (4,0) (X15) {$X_{15}$} ;
    \node at (8,0) (X02) {$X_{02}$} ;    
    \node at (12,0) (X12) {$X_{12}$} ;
    \node at (0,4) (X04) {$X_{04}$} ;
    \node at (4,4) (X03) {$X_{03}$} ;
    \node at (8,4) (X14) {$X_{14}$} ;    
    \node at (12,4) (X13) {$X_{13}$} ;
    \node at (0,8) (X45) {$X_{45}$} ;
    \node at (4,8) (X35) {$X_{35}$} ;
    \node at (8,8) (X24) {$X_{24}$} ;    
    \node at (12,8) (X23) {$X_{23}$} ;
    \draw[->] (X35) [bend left=7] to node[above,sloped] {\tiny{$R_1\otimes U_2$}} (X03) ;
    \draw[->] (X03) [bend left=6] to node[above,sloped,pos=.2] {\tiny{$1\otimes R_3$}} (X02) ;
    \draw[->] (X02) [bend left=7] to node[above,sloped] {\tiny{$1\otimes L_1$}} (X12) ;
    \draw[->] (X12) [bend left=7] to node[above,sloped] {\tiny{$U_2\otimes L_{3}$}} (X13) ;
    \draw[->] (X13) [bend left=7] to node[above,sloped] {\tiny{$U_1\otimes L_2$}} (X23) ;
    \draw[->] (X45) [bend left=7] to node[above,sloped] {\tiny{$R_1\otimes U_2$}} (X04) ;
    \draw[->] (X04) [bend left=7] to node[above,sloped] {\tiny{$R_2\otimes U_3$}} (X05) ;
    \draw[->] (X05) [bend left=7] to node[above,sloped] {\tiny{$1 \otimes L_1$}} (X15) ;
    \draw[->] (X15) [bend left=6] to node[above,sloped,pos=.2] {\tiny{$L_2  \otimes 1$}} (X14) ;
    \draw[->] (X14) [bend left=7] to node[above,sloped] {\tiny{$U_1\otimes L_2$}} (X24) ;
    \draw[->] (X03) [bend left=7] to node[above,sloped] {\tiny{$L_1\otimes 1$}} (X35) ;
    \draw[->] (X02) [bend left=6] to node[below,sloped,pos=.2] {\tiny{$U_2\otimes L_3$}} (X03) ;
    \draw[->] (X12) [bend left=7] to node[below,sloped] {\tiny{$U_3\otimes R_1$}} (X02) ;
    \draw[->] (X13) [bend left=7] to node[above,sloped] {\tiny{$1\otimes R_{3}$}} (X12) ;
    \draw[->] (X23) [bend left=7] to node[above,sloped] {\tiny{$1\otimes R_2$}} (X13) ;
    \draw[->] (X04) [bend left=7] to node[above,sloped] {\tiny{$L_1\otimes 1$}} (X45) ;
    \draw[->] (X05) [bend left=7] to node[above,sloped] {\tiny{$L_2\otimes 1$}} (X04) ;
    \draw[->] (X15) [bend left=7] to node[below,sloped] {\tiny{$U_3 \otimes R_1$}} (X05) ;
    \draw[->] (X14) [bend left=6] to node[below,sloped,pos=.2] {\tiny{$R_2\otimes U_3$}} (X15);
    \draw[->] (X24) [bend left=7] to node[above,sloped] {\tiny{$1\otimes R_2$}} (X14) ;
    \draw[->] (X45) [bend left=7] to node[above,sloped] {\tiny{$L_3\otimes U_1$}} (X35) ;
    \draw[->] (X04) [bend left=7] to node[above,sloped] {\tiny{$L_3\otimes U_1$}} (X03) ;
    \draw[->] (X14) [bend left=7] to node[above,sloped] {\tiny{$L_3\otimes U_1$}} (X13) ;
    \draw[->] (X24) [bend left=7] to node[above,sloped] {\tiny{$L_3\otimes U_1$}} (X23) ;
    \draw[->] (X35) [bend left=7] to node[below,sloped, sloped] {\tiny{$R_3\otimes 1$}} (X45) ;
    \draw[->] (X03) [bend left=7] to node[below,sloped] {\tiny{$R_3\otimes 1$}} (X04) ;
    \draw[->] (X13) [bend left=7] to node[below,sloped] {\tiny{$R_3\otimes 1$}} (X14) ;
    \draw[->] (X23) [bend left=7] to node[below,sloped] {\tiny{$R_3\otimes 1$}} (X24) ;
    \draw[->] (X02) [bend left=24]to node[below,sloped] {\tiny{$R_3 R_2\otimes L_3$}} (X05) ;
    \draw[->] (X12) [bend left=24]to node[below,sloped] {{\tiny{$R_3 R_2\otimes L_3$}}} (X15) ;
    \draw[->] (X02) [bend left=7] to node[below,sloped] {\tiny{$L_1\otimes R_2$}} (X15) ;
    \draw[->] (X15) [bend left=7] to node[above,sloped] {\tiny{$R_1\otimes L_2$}} (X02) ;
    \draw[->] (X03) [bend left=24] to node[above,sloped,pos=.8] {\tiny{$1\otimes L_1$}} (X13) ;
    \draw[->] (X35) [bend left=24] to node[above,sloped] {\tiny{$R_1\otimes L_1 L_2$}} (X23) ;
    \draw[->] (X12) [bend right=24] to node[below,sloped] {\tiny{$L_1 L_2 L_3\otimes R_1$}} (X23) ;
    \draw[->] (X45) [bend left=24] to node[above,sloped] {\tiny{$R_1\otimes L_1L_2$}} (X24) ;
    \draw[->] (X04) [bend left=24] to node[above,sloped,pos=.2] {\tiny{$1\otimes L_1$}} (X14) ;
    \draw[->] (X45) [bend right=24] to node[below,sloped]{\tiny{$L_3\otimes R_3 R_2 R_1$}}  (X05) ;
    \draw[->] (X24) [bend left=7] to node[below,sloped] {\tiny{$R_2\otimes L_3$}} (X35) ;
    \draw[->] (X35) [bend left=7] to  node[above,sloped] {\tiny{$L_2\otimes R_3$}} (X24) ;
    \draw[->] (X35) [bend right=24] to  node[below,sloped,pos=.7] {\tiny{$1\otimes R_3 R_2$}} (X15) ;
    \draw[->] (X02) [bend right=24] to  node[below,sloped,pos=0.3] {\tiny{$L_1 L_2\otimes 1$}} (X24) ;
    \draw[->] (X14)  to  node[above,sloped] {\tiny{$R_2 R_1 \otimes L_2 L_3$}} (X03) ;
    \draw[->] (X15) [bend left=7] to  node[below,sloped,pos=0.51] {\tiny{$L_2 L_3 \otimes R_1$}} (X03) ;
    \draw[->] (X14) [bend left=7] to  node[above,sloped,pos=.4] {\tiny{$L_3 \otimes R_3 R_1$}} (X02) ;
  \end{tikzpicture}\]
  \caption{\label{fig:NearBraidDD}
   {\bf{${}^{\Blg_3}\Pos^{2}_{\Blg_2} \DT ^{\Blg_2,\Blg_1'}\Pos_1$ when $\{1,2,3\}\subset\Upwards$}}}
  \end{figure}
When $\{1,2,3\}\cap \Upwards$ is a proper subset
of $\{1,2,3\}$, we need to apply additional homotopies to obtain this bimodule,
as in the proof of Lemma~\ref{lem:CrossingDADD}.

\begin{figure}[ht]
  \input{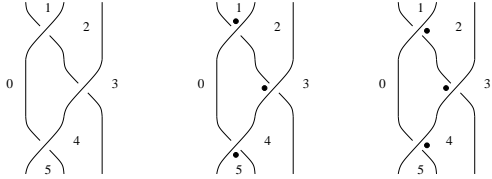}
\caption{\label{fig:R3GenNames} 
  We illustrate here
  generators of type $Y_{15}$ and $Y_{24}$ respectively.}
\end{figure}

Tensoring this on the left with $\lsup{\Blg_5}\Pos^i_{\Blg_3}$, and following the labelling conventions
from Figure~\ref{fig:R3GenNames}, we find that the bimodule is homotopic to one of  the following form
from Figure~\ref{fig:R3Module}, along with the usual outside arrows and self-arrows.
\begin{figure}
  \begin{tikzpicture}[scale=.77]%[Scaleque=1.2]%[y=48pt,x=.95in]
    \node at (0,0) (Y05) {$Y_{05}$} ;
    \node at (4,0) (Y15) {$Y_{15}$} ;
   \node at  (8,0) (Q03) {$Q_{03}$} ;    
    \node at (12,0) (Y13) {$Y_{13}$} ;
    \node at (0,4) (Y04) {$Y_{04}$} ;
    \node at (4,4) (P03) {$P_{03}$} ;
    \node at (8,4) (Y24) {$Y_{24}$} ;    
    \node at (12,4) (Y23) {$Y_{23}$} ;
    \node at (0,8) (Y45) {$Y_{45}$} ;
    \node at (4,8) (Y35) {$Y_{35}$} ;
    \node at (8,8) (Y34) {$Y_{34}$} ;    
    \node at (4,-4) (Y01) {$Y_{01}$} ;
    \node at (8,-4) (Y02) {$Y_{02}$} ;
    \node at (12,-4) (Y12) {$Y_{12}$} ;

    \draw[->] (Y35) [bend left=7] to node[above,sloped] {\tiny{$R_1\otimes U_3$}} (P03) ;
    \draw[->] (P03) [bend left=6] to node[above,sloped,pos=.2] {\tiny{$1\otimes U_2$}} (Q03) ;
    \draw[->] (Q03) [bend left=7] to node[above,sloped] {\tiny{$1\otimes L_1$}} (Y13) ;
    \draw[->] (Y13) [bend left=7] to node[above,sloped] {\tiny{$U_2\otimes L_{2}$}} (Y23) ;
    \draw[->] (Y45) [bend left=7] to node[above,sloped] {\tiny{$R_1\otimes U_3$}} (Y04) ;
    \draw[->] (Y04) [bend left=7] to node[above,sloped] {\tiny{$R_2\otimes U_2$}} (Y05) ;
    \draw[->] (Y05) [bend left=7] to node[above,sloped] {\tiny{$1 \otimes L_1$}} (Y15) ;
    \draw[->] (Y15) [bend left=6] to node[above,sloped,pos=.2] {\tiny{$L_2  \otimes L_2$}} (Y24) ;
    \draw[->] (Y24) [bend left=7] to node[above,sloped] {\tiny{$U_1\otimes L_3$}} (Y34) ;

    \draw[->] (P03) [bend left=7] to node[above,sloped] {\tiny{$L_1\otimes 1$}} (Y35) ;
    \draw[->] (Q03) [bend left=6] to node[below,sloped,pos=.2] {\tiny{$U_2\otimes 1$}} (P03) ;
    \draw[->] (Y13) [bend left=7] to node[below,sloped] {\tiny{$U_3\otimes R_1$}} (Q03) ;
    \draw[->] (Y23) [bend left=7] to node[above,sloped] {\tiny{$1\otimes R_{2}$}} (Y13) ;
    \draw[->] (Y04) [bend left=7] to node[above,sloped] {\tiny{$L_1\otimes 1$}} (Y45) ;
    \draw[->] (Y05) [bend left=7] to node[above,sloped] {\tiny{$L_2\otimes 1$}} (Y04) ;
    \draw[->] (Y15) [bend left=7] to node[below,sloped] {\tiny{$U_3 \otimes R_1$}} (Y05) ;
    \draw[->] (Y24) [bend left=6] to node[below,sloped,pos=.2] {\tiny{$R_2\otimes R_2$}} (Y15);
    \draw[->] (Y34) [bend left=7] to node[above,sloped] {\tiny{$1\otimes R_3$}} (Y24) ;

    \draw[->] (Y45) [bend left=7] to node[above,sloped] {\tiny{$L_3\otimes U_1$}} (Y35) ;
    \draw[->] (Y04) [bend left=7] to node[above,sloped] {\tiny{$L_3\otimes U_1$}} (P03) ;
    \draw[->] (Y24) [bend left=7] to node[above,sloped] {\tiny{$L_3\otimes U_1$}} (Y23) ;

    \draw[->] (Y35) [bend left=7] to node[below,sloped, sloped] {\tiny{$R_3\otimes 1$}} (Y45) ;
    \draw[->] (P03) [bend left=7] to node[below,sloped] {\tiny{$R_3\otimes 1$}} (Y04) ;
    \draw[->] (Y23) [bend left=7] to node[below,sloped] {\tiny{$R_3\otimes 1$}} (Y24) ;
    \draw[->] (Q03) [bend left=25]to node[below,sloped,pos=.8] {\tiny{$R_3 R_2\otimes 1$}} (Y05) ;
    \draw[->] (Y13) [bend left=25]to node[below,sloped,pos=.2] {{\tiny{$R_3 R_2\otimes 1$}}} (Y15) ;

    \draw[->] (Y15) to node[above,sloped] {\tiny{$R_1\otimes L_2L_3$}} (Q03) ;
    
    \draw[->] (P03) [bend left=25] to node[above,sloped,pos=.8] {\tiny{$1\otimes L_1 L_2$}} (Y23) ;

    \draw[->] (Y45) [bend left=25] to node[above,sloped] {\tiny{$R_1\otimes L_1L_2L_3$}} (Y34) ;
    \draw[->] (Y04) [bend left=25] to node[above,sloped,pos=.2] {\tiny{$1\otimes L_1L_2$}} (Y24) ;
    \draw[->] (Y45) [bend right=25] to node[below,sloped,pos=0.7]{\tiny{$L_3\otimes R_3 R_2 R_1$}}  (Y05) ;

    \draw[->] (Y34) [bend left=7] to node[below,sloped] {\tiny{$R_2\otimes 1$}} (Y35) ;
    \draw[->] (Y35) [bend left=7] to  node[above,sloped] {\tiny{$L_2\otimes U_2$}} (Y34) ;

    \draw[->] (Y35) [bend right=25] to  node[below,sloped,pos=.7] {\tiny{$1\otimes R_3 R_2$}} (Y15) ;
    \draw[->] (Q03) [bend right=25] to  node[below,sloped,pos=0.3] {\tiny{$L_1 L_2\otimes 1$}} (Y34) ;

    \draw[->] (Y24)  to  node[above,sloped] {\tiny{$R_2 R_1 \otimes L_3$}} (P03) ;

    \draw[->] (Y15) [bend left=7] to  node[below,sloped,pos=0.6] {\tiny{$L_2 L_3 \otimes R_1$}} (P03) ;
    \draw[->] (Y24) [bend right=7]to  node[above,sloped,pos=0.51] {\tiny{$L_3 \otimes R_2 R_1$}} (Q03) ;
    
    \draw[->] (Y01) [bend left=7] to node[above,sloped] {\tiny{$1 \otimes L_2$}} (Y02) ;
    \draw[->] (Y02) [bend left=7] to node[above,sloped] {\tiny{$1 \otimes L_1$}} (Y12) ;
    \draw[->] (Y02) [bend left=7] to node[below,sloped] {\tiny{$U_2 \otimes R_2$}} (Y01) ;
    \draw[->] (Y12) [bend left=7] to node[below,sloped] {\tiny{$U_3\otimes R_1$}} (Y02) ;

    \draw[->] (Y01) [bend left=7] to node[above,sloped] {\tiny{$L_1 \otimes 1$}} (Y15) ;
    \draw[->] (Y02) [bend left=7] to node[above,sloped,pos=0.51] {\tiny{$U_1\otimes L_3$}} (Q03) ;
    \draw[->] (Y12) [bend left=7] to node[above,sloped] {\tiny{$U_1\otimes L_3$}} (Y13) ;

    \draw[->] (Y15) [bend left=7] to node[above,sloped] {\tiny{$R_1 \otimes U_3$}} (Y01) ;
    \draw[->] (Q03) [bend left=7] to node[above,sloped,pos=0.49] {\tiny{$1\otimes R_3$}} (Y02) ;
    \draw[->] (Y13) [bend left=7] to node[above,sloped] {\tiny{$1\otimes R_3$}} (Y12) ;

    \draw[->] (P03) [bend right=25] to node[below,sloped,pos=.8] {\tiny{$1\otimes R_3 R_2$}} (Y01) ;
    \draw[->] (Y02) [bend right=25] to node[below,sloped,pos=.2] {\tiny{$L_1 L_2 \otimes 1$}} (Y24) ;
    \draw[->] (Y12) [bend right=25] to node[below,sloped,pos=0.3] {\tiny{$L_1 L_2 L_3 \otimes R_1$}} (Y23) ;
    \draw[->] (Y12) [bend left=25] to node[above,sloped] {\tiny{$R_3 R_2 R_1\otimes L_3$}} (Y01) ;
  \end{tikzpicture}
  \caption{
  \label{fig:R3Module}{\bf{${}^{\Blg_5}\Pos^i_{\Blg_3}\DT^{\Blg_3}\Pos^{2}_{\Blg_2} \DT ^{\Blg_2,\Blg_1'}\Pos_1$}}}
\end{figure}

Tensoring in the other order gives the same bimodule; this can be seen
quickly from a symmetry on the answer. The picture after the Reidemeister move can be realized
as a rotation by $180^\circ$ (i.e. rotate the leftmost picture in 
Figure~\ref{fig:R3GenNames}). There is
a corresponding action on the algebra, exchanging the two tensor
factors of the algebra, and further exchanging $R_i$ and $L_{4-i}$.
The corresponding symmetry of the bimodule can be realized by rotating
the description from Figure~\ref{fig:R3Module} by $180^\circ$. The fact that the bimodule is fixed by 
the symmetry implies the claimed invariance under Reidemeister $3$ moves.
\end{proof}

\subsection{Other symmetries in the crossing bimodule}

The canonical $DD$ bimodules commute with the action of the braid group, in the following sense.

\begin{lemma}
  \label{lem:CommuteDDBraid}
  Fix $0\leq k \leq m+1$ and an arbitrary subset $\Upwards\subset \{1,\dots,m\}$, and let 
  \[\begin{array}{ll}
    \Blg_1=\Blg(m,k,\Upwards), & 
    \Blg_2=\Blg(m,k,\tau_i(\Upwards))  \\
    \Blg_1'=\Blg(m,m+1-k,\{1,\dots,m\}\setminus \Upwards) &
    \Blg_2'=\Blg(m,m+1-k,\{1,\dots,m\}\setminus \tau_i(\Upwards)) 
  \end{array}\]
  There is an equivalence
  $\lsup{\Blg_2}\Pos^i_{\Blg_1}\DT~^{\Blg_1,\Blg_1'}\CanonDD\simeq~^{\Blg_1'}\Pos^i_{\Blg_2'}\DT~^{\Blg_2',\Blg_2}\CanonDD$.
\end{lemma}

\begin{proof}
  This is an immediate consequence of Lemma~\ref{lem:CrossingDADD}, and the symmetry
  of the $\Pos_i$ from Equation~\eqref{eq:SymmetryOfPos}.
\end{proof}

\newcommand\TridentDD{\mathcal T}
\section{The $DD$ bimodule of a critical point}
\label{sec:Crit}

In Sections~\ref{sec:Maximum} and~\ref{sec:Minimum}, we will construct
$DA$ bimodules $\lsup{\Blg_2}\Max_{\Blg_1'}$ and $\lsup{\Blg_1}\Min_{\Blg_2'}$ 
(where the algebras will be made precise shortly)
associated to a region in the knot diagram where there are no crossings and a
single critical point, which can be a maximum (as in
Section~\ref{sec:Maximum}) or a minimum (as in
Section~\ref{sec:Minimum}).  In the present section, we will construct
a type $DD$ bimodule, called {\em the type $DD$ bimodule for a
  critical point}, $\lsup{\Blg_1,\Blg_2}\Crit$ 
which is related to the aforementioned type $DA$ bimodules via equivalences
\[ 
\begin{array}{lll}
  \lsup{\Blg_2}\Max_{\Blg_1'}\DT~\lsup{{\Blg_1',\Blg_1}}\CanonDD \simeq~
  \lsup{\Blg_2,\Blg_1}\Crit &{\text{and}}&
  \lsup{\Blg_2}\Min_{\Blg_1'}\DT~\lsup{{\Blg_1',\Blg_1}}\CanonDD 
  \simeq~\lsup{\Blg_2,\Blg_1}\Crit. \\
\end{array}
\]
(See Propositions~\ref{prop:MaxDual} and~\ref{prop:MinDual} for precise statements.)

Fix integers $c$, $k$, and $m$ with $1\leq c\leq m+1$ and $0\leq k\leq m+1$. The  algebras appearing in the bimodule for a critical point are specified as follows. 
Let $\phi_c\colon \{1,\dots,m\}\to\{1,\dots,m+2\}$ be the function
\begin{equation}
\label{eq:DefInsert}
\phi_c(j)=\left\{\begin{array}{ll}
j &{\text{if $j< c$}} \\
j+2 &{\text{if $j\geq c$.}}
\end{array}\right.
\end{equation}
Let 
$\Blg_1=\Blg(m,k,\Upwards_1)$ and
$\Blg_2=\Blg(m+2,m+2-k,\Upwards_2)$,
where $\Upwards_1\subset
\{1,\dots,m\}$ and $\Upwards_2\subset \{1,\dots,m+2\}$, so that 
$\phi_c(\Upwards_1)\cap \Upwards_2=\emptyset$
and 
$|\Upwards_1|+|\Upwards_2|=m+1$
and $|\Upwards_2\cap \{c,c+1\}|=1$;
equivalently, 
$\Upwards_2=\phi_{c}(\{1,\dots,m\}\setminus \Upwards_1)\cup\{c\}$ or
$\Upwards_2=\phi_{c}(\{1,\dots,m\}\setminus\Upwards_1)\cup\{c+1\}$.

In this section, we 
construct the type $DD$ bimodule for a critical point, denoted
$\Crit_c= \lsup{\Blg_1,\Blg_2}\Crit_c$, where here $c$ (which we will sometimes drop from the notation) 
indicates where the critical point occurs.

Having specified the algebras, we now
describe the underlying vector space for $\Crit_c$.
We call an idempotent state $\y$ for $\Blg_2$ an {\em allowed idempotent state for $\Blg_2$} if
\[ |\y\cap\{c-1,c,c+1\}|\leq 2\qquad\text{and}\qquad c\in\y.\] 
There is a map $\psi$ from
allowed idempotent states $\y$ for $\Blg_2$ to idempotent states  for $\Blg_1$,
where $\x=\psi(\y)\subset \{0,\dots,m\}$ is characterized by
\begin{equation}
  \label{eq:SpecifyPsi}
  |\y\cap \{c-1,c,c+1\}| + 
  |\x\cap \{c-1\}| =2~\qquad{\text{and}}~\qquad \phi_c(\x)\cap \y=\emptyset.
\end{equation}

As a vector space, $\Crit_c$ is spanned by vectors that are in
one-to-one correspondence with allowed idempotent states for $\Blg_2$. 
The bimodule structure, over the rings of idempotents $\IdempRing(\Blg_1)$ and $\IdempRing(\Blg_2)$,
is specified as follows.
If ${\mathbf P}={\mathbf P}_{\y}$ is the generator associated to 
the idempotent state $\y$, then for idempotent states $\x$ and $\z$ for $\Blg_1$ and $\Blg_2$ respectively,
\[ (\Idemp{\x}\otimes \Idemp{\z})\cdot {\mathbf P}_{\y}= \left\{\begin{array}{ll}
{\mathbf P}_{\y} &{\text{if $\y=\z$ and $\x=\psi(\y)$}} \\
0 &{\text{otherwise.}}
\end{array}\right..\]

To specify the differential, 
consider the element 
$A\in \Blg_1\otimes\Blg_2$ 
\begin{align}
A&=(1\otimes L_{c} L_{c+1}) + 
(1\otimes R_{c+1} R_{c}) 
 + \sum_{j=1}^{m} R_j \otimes L_{\phi(j)} + L_j \otimes R_{\phi(j)} \label{eq:defA}\\
&
  + \left\{\begin{array}{ll}
        (1\otimes C_{c} U_{c+1}) &{\text{if $c\in\Upwards_2$}} \nonumber \\
        (1\otimes U_{c} C_{c+1}) &{\text{if $c+1\in\Upwards_2$,}}
      \end{array}\right\}
 + \sum_{j=1}^{m} \left\{\begin{array}{ll}
    C_j \otimes U_{\phi(j)} & {\text{if $j\in\Upwards_1$}} \nonumber \\
    U_j \otimes C_{\phi(j)} & {\text{if $j\not\in\Upwards_1$}}
    \end{array}\right\}. \nonumber
  \end{align}
where we have dropped the subscript $c$ from $\phi_c=\phi$.
Let 
\[ \delta^1({\mathbf P}_{\y}) = (\Idemp{\psi(\y)}\otimes \Idemp{\y})\cdot A \otimes
\sum_{\z} {\mathbf P}_{\z},\]
where the latter sum is taken over all allowed idempotent states $\z$ for $\Blg_2$.

\begin{lemma}
  The space $\lsup{\Blg_1,\Blg_2}\Crit_c$ defined above, and equipped with the map
  \[ \delta^1\colon \Crit_c \to (\Blg_1\otimes \Blg_2)\otimes \Crit_c,\]
  specified above, is a type $DD$ bimodule over $\Blg_1$ and $\Blg_2$.
\end{lemma}

\begin{proof}
  The proof is a straightforward adaptation of Lemma~\ref{lem:CanonicalIsDD}.
\end{proof}

It is helpful to understand $\Crit_c$ a little more explicitly.
To this end, we classify the allowed idempotents for $\Blg_2$ into three types,  labelled $\XX$, $\YY$, and $\ZZ$:
\begin{itemize}
  \item $\y$ is of type $\XX$ if $\y\cap \{c-1,c,c+1\}=\{c-1,c\}$,
  \item $\y$ is of type $\YY$ if $\y\cap \{c-1,c,c+1\}=\{c,c+1\}$,
  \item $\y$ is of type $\ZZ$ if $\y\cap \{c-1,c,c+1\}=\{c\}$. 
\end{itemize}
There is a corresponding classification of the generators ${\mathbf P}_{\y}$ into $\XX$, $\YY$, and $\ZZ$, according to the type of $\y$;
see Figure~\ref{fig:DDcap}.

\begin{figure}[ht]
\input{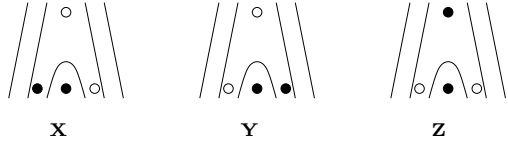}
\caption{\label{fig:DDcap} {\bf{$DD$ bimodule of a critical point.}}  
Three generator types are illustrated.}
\end{figure}

With respect to this decomposition, 
terms in  the differential are of the following four types:
\begin{enumerate}[label=(P-\arabic*),ref=(P-\arabic*)]
\item 
  \label{type:OutsideLR}
  $R_j\otimes L_{\phi(j)}$
  and $L_j\otimes R_{\phi(j)}$ for all $j\in \{1,\dots,m\}\setminus \{c-1,c\}$; these connect
  generators of the same type. 
\item
  \label{type:UC}
  $C_j\otimes U_{\phi(j)}$ if $j\in \Upwards_1$ 
  and $U_j\otimes C_{\phi(j)}$ if $j\in\{1,\dots,m\}\setminus \Upwards_1$
\item 
  \label{type:UC2}
  $1\otimes U_{c} C_{c+1}$ if $c+1\in \Upwards_2$ or 
  $1\otimes C_{c} U_{c+1}$ if $c\in \Upwards_2$.
\item 
  \label{type:InsideCup}
  Terms in the diagram below connect  generators
  of different types.
\begin{equation}
  \label{eq:CritDiag}
  \begin{tikzpicture}[scale=1.2]%[y=48pt,x=1in]
    \node at (-1.5,0) (X) {$\XX$} ;
    \node at (1.5,0) (Y) {$\YY$} ;
    \node at (0,-2) (Z) {$\ZZ$} ;
    \draw[->] (X) [bend right=7] to node[below,sloped] {\tiny{$1\otimes R_{c+1} R_{c}$}}  (Y)  ;
    \draw[->] (Y) [bend right=7] to node[above,sloped] {\tiny{$1\otimes L_{c} L_{c+1}$}}  (X)  ;
    \draw[->] (X) [bend right=7] to node[below,sloped] {\tiny{$R_{c-1}\otimes L_{c-1}$}}  (Z)  ;
    \draw[->] (Z) [bend right=7] to node[above,sloped] {\tiny{$L_{c-1}\otimes R_{c-1}$}}  (X)  ;
    \draw[->] (Z) [bend right=7] to node[below,sloped] {\tiny{$R_{c} \otimes L_{c+2}$}}  (Y)  ;
    \draw[->] (Y) [bend right=7] to node[above,sloped] {\tiny{$L_{c}\otimes R_{c+2}$}}  (Z)  ;
  \end{tikzpicture}
\end{equation}
\end{enumerate}

With the understanding that
if $c=1$, then the terms containing $L_{c-1}$ or $R_{c-1}$ are missing; 
similarly, if $c=m+1$, the terms containing $R_{c+2}$ and $L_{c+2}$ are missing.

\subsection{Grading sets}
\label{sec:CritGradingSet}

For any generator $P$ of $\Crit_c$, if $(a_2 \otimes
a_1) \otimes Q$ appears with non-zero multiplicity in $d^1(P)$, then
for all $i=1,\dots,m$, $w_i(a_1)=w_{\phi_c(i)}(a_2)$ and
$w_{c}(a_2)=w_{c+1}(a_2)$. This is obvious from the form of the
bimodule (Equation~\eqref{eq:defA}).
In the language of Section~\ref{subsec:OurGradingSets}, we can express this
by saying that the bimodule is supported in grading $0$ in $H_1(W,\partial W)$, where $W$
is the partial knot diagram with a single critical point in it.

\subsection{Critical points and crossings}

Later, we will study in detail how the bimodules of critical points
interact with the bimodules associated to crossings. For the time
being, we will content ourselves with the following special case,
where the critical point and the crossing are connected to each other.

To set notation, let
    $\psi_c\colon \{1,\dots,m\}\to\{1,\dots,m+2\}$ be the function
\[
\psi_c(j)=\left\{\begin{array}{ll}
j &{\text{if $j< c$}} \\
c+1 &{\text{if $j=c$}} \\
j+2 &{\text{if $j>c$;}}
\end{array}\right.
\]
i.e. $\psi_c=\tau_{c+1}\circ \phi_c=\tau_c\circ \phi_{c+1}$.

\begin{lemma}
  \label{lem:BasicTrident}
  Fix integers $0\leq k\leq m+1$, and an arbitrary subset $\Upwards_1\subset \{1,\dots,m\}$,
  and let $\Blg_1=\Blg(m,k,\Upwards_1)$, $\Blg_2=\Blg(m+2,m+2-k,\Upwards_2)$
  $\Blg_3=\Blg(m+2,m+2-k,\Upwards_3)$, $\Blg_4=\Blg(m+2,m+2-k,\Upwards_4)$
  where
  \[
  \Upwards_2=\phi_{c+1}(\{1,\dots,m\})\setminus \Upwards_1)\cup\{c+1\}
  \qquad{\text{or}}\qquad
  \Upwards_2=\phi_{c+1}(\{1,\dots,m\}\setminus\Upwards_1)\cup\{c+2\},
  \]
  $\Upwards_3=\tau_{c}(\Upwards_2)$ and
  $\Upwards_4=\tau_{c+1}(\Upwards_3)$.  
  There is homotopy equivalence of graded bimodules:
  \begin{equation}
    \label{eq:BasicTrident}
    \lsup{\Blg_3}\Pos^{c}_{\Blg_2}\DT~^{\Blg_2,\Blg_1}\Crit_{c+1}
    \simeq
    \lsup{\Blg_3}\Neg^{c+1}_{\Blg_4}\DT~^{\Blg_4,\Blg_1}\Crit_{c}
  \end{equation}
\end{lemma}

\begin{rem}
  \label{rem:MotivateGradings}
  The gradings on the crossing bimodules  were chosen so that Equation~\eqref{eq:BasicTrident} holds
  as a bigraded module map.
\end{rem}

The above lemma is a straightforward computation, using the following
{\em trident bimodule} $~^{\Blg_1,\Blg_3}\TridentDD$.

Generators correspond to pairs of idempotents $\x$ and $\y$ for $\Blg_1$ and $\Blg_3$
with the following properties:
\begin{itemize}
  \item Letting,
then
$\psi_c(\x)\cap\y=\emptyset$.
  \item $|\y|=|\x|+1$
  \item $|\y\cap\{c,c+1\}\geq 1$
  \item if $|\y\cap\{c,c+1\}|=2$ then either $c-1\not\in\x$ and $c-1\not\in\y$ or
    $c\not\in\x$ and $c+2\not\in\y$
\end{itemize}
We separate these pairs into four types:
\begin{itemize}
  \item $(\x,\y)$ is of Type~${\mathbf P}$ if $c-1\not\in \x$ and $c-1\not\in \y$
    (so $\{c,c+1\}\subset \y$)
  \item $(\x,\y)$ is of Type~${\mathbf Q}$ if $c\not\in \x$ and $c+2\not\in \y$
    (so $\{c,c+1\}\subset \y$)
  \item $(\x,\y)$ is of Type~$\XX$ if $c\not\in\y$
  \item $(\x,\y)$ is of Type~$\YY$ if $c+1\not\in\y$
\end{itemize}
See Figure~\ref{fig:TridentModule} for a picture.
(Note that these generator types correspond to the partial Kauffman states in the sense of Definition~\ref{def:PartialKnotDiagram}
of the diagram containing the maximum and the crossing, with the understanding that the incoming idempotent
is complementary to idempotent state coming from the top of the diagram.)

\begin{figure}[ht]
\input{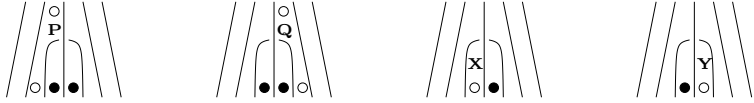}
\caption{\label{fig:TridentModule} {\bf{The four generator types of the trident bimodule.}}}
\end{figure}

Let ${\mathbf T}_{\x,\y}$ be the corresponding generator of $\TridentDD$.
As a left module over $\IdempRing(\Blg_1)\otimes \IdempRing(\Blg_3)$, 
the action is specified by $(\Idemp{\x}\otimes \Idemp{\y})\cdot {\mathbf T}_{\x,\y}={\mathbf T}_{\x,\y}$.
The differential has the following types of terms:
\begin{enumerate}
\item 
  \label{type:OutsideLRT}
  $R_j\otimes L_j$
  and $L_j\otimes R_j$ for all $j\in \{1,\dots,m\}\setminus \{c\}$; these connect
  generators of the same type. 
\item
  \label{type:UCT}
  $C_j\otimes U_{\psi_c(j)}$ if $j\in \Upwards_1\setminus \{c\}$
  and $U_j\otimes C_{\psi_c(j)}$ if $j\in\{1,\dots,m\}\setminus(\Upwards_1\cup\{c\})$; these connect generators of the same type.
\item $C_c\otimes U_{c+1}$ if $c\in \Upwards_1$ and
  $U_c\otimes C_{c+1}$ if $c\not\in\Upwards_1$
\item 
  \label{type:InsideT}
  Terms in the diagram below connect  generators
  of different types:
  \begin{equation}
    \label{eq:PositiveTrident}
    \begin{tikzpicture}[scale=1.6]
    \node at (-1,1) (P) {${\mathbf P}$} ;
    \node at (1,1) (Q) {${\mathbf Q}$} ;
    \node at (-1,-1) (X) {${\mathbf X}$} ;
    \node at (1,-1) (Y) {${\mathbf Y}$} ;
    \draw[->] (P)[bend left=7] to node[above,sloped]{\tiny{$L_{c}\otimes U_{c+1} + 1\otimes L_{c} L_{c+1} L_{c+2}$}} (Q) ;
    \draw[->] (Q)[bend left=7] to node[below,sloped]{\tiny{$R_{c}\otimes 1$}} (P);

    \draw[->] (P)[bend right=7] to node[below,sloped]{\tiny{$1 \otimes L_{c}U_{c+2} + L_{c}\otimes R_{c+2} R_{c+1}$}} (X) ;
    \draw[->] (X)[bend right=7] to node[below,sloped]{\tiny{$1 \otimes R_{c} $}} (P) ;

    \draw[->] (X)[bend left=7] to node[above,sloped]{\tiny{$1\otimes L_{c+1}$}} (Y) ;
    \draw[->] (Y)[bend left=7] to node[below,sloped]{\tiny{$U_{c}\otimes R_{c+1} +R_{c}\otimes L_{c}L_{c+2}$}} (X) ;

    \draw[->] (Y)[bend right=7] to node[below,sloped]{\tiny{$1\otimes U_{c} L_{c+2}+L_{c}\otimes R_{c+1}R_{c}$}} (Q) ;
    \draw[->] (Q)[bend right=7] to node[below,sloped]{\tiny{$1\otimes R_{c+2}$}} (Y) ;
  \end{tikzpicture}
\end{equation}
\end{enumerate}

\begin{proof}[Proof of Lemma~\ref{lem:BasicTrident}]
  A straightforward computation identifies both sides with the trident
  bimodule specified above, after a possible homotopy (as in the proof
  of Lemma~\ref{lem:CrossingDADD}).  The fact that the map respects
  gradings follows quickly from the grading conventions; see
  Equations~\eqref{eq:GradeCrossing} and\eqref{eq:GradeNegCrossing}.
\end{proof}

\section{The $DA$ bimodule associated to a maximum}
\label{sec:Maximum}

We will describe now the type $DA$ bimodule $\lsup{\Blg_2}\Max^c_{\Blg_1'}$
of a region in the knot
diagram where there are no crossings and a single local maximum,
which connects the $c^{th}$ and the $(c+1)^{st}$ outgoing strand.
The algebras are specified as follows.
Fix integers $c$, $k$, and $m$ with $1\leq c\leq m+1$ and $0\leq k\leq m+1$.,
and let
\[ \Blg_1'=\Blg(m,k,\Upwards_1)\qquad{\text{and}}\qquad
\Blg_2=\Blg(m+2,k+1,\Upwards_2),\]
where $\Upwards_1'\subset\{1,\dots,m\}$ is arbitrary,
and $\Upwards_2\subset \{1,\dots,m+2\}$ is given by
\[
\Upwards_2=\phi_c(\Upwards_1)\cup\{c\}
~\text{or}~\Upwards_2=\phi_c(\Upwards_1)\cup\{c+1\}.\]
The two cases as corresponding to the two possible orientations of the strand with the maximum;
see Figure~\ref{fig:Maximum}.
\begin{figure}[ht]
\input{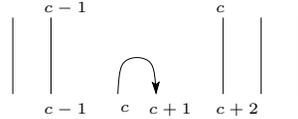}
\caption{\label{fig:Maximum} {\bf{Picture of maximum.}} Here the maximum is oriented from left to right,
so $\Upwards_2\cap \{c,c+1\}=c$.}
\end{figure}

As in Section~\ref{sec:Crit}, 
an {\em allowed idempotent state for $\Blg_2$} 
is an idempotent state $\y$ for $\Blg_2$ with 
$c\in\y$ and
$|\y\cap \{c-1,c+1\}|\leq 1$.
There is a map $\psi'$ from allowed idempotent states for 
$\Blg_2$ to idempotent states for $\Blg_1'$,
given by 
\[ \psi'(\x)=\left\{\begin{array}{ll}
\phi^{-1}(\y) & {\text{if $c+1\not\in \y$}} \\
\phi^{-1}(\y)\cup\{c-1\} & {\text{if $c+1\in \y$}}
\end{array}
\right.\]
Observe that $\psi'(\y)=\{0,\dots,m\}\setminus \psi(\y)$,
where $\psi$ is the map from Section~\ref{sec:Crit}
(c.f. Equation~\eqref{eq:SpecifyPsi}).

A basis for the underlying vector space of
$\lsup{\Blg_2}\Max^c_{\Blg_1'}$ is specified by the allowed idempotent
states for $\Blg_2$.  The bimodule structure, over the rings of
idempotents $\IdempRing(\Blg_1')$ and $\IdempRing(\Blg_2)$, is
specified as follows.  If ${\mathbf Q}_{\y}$ is the
generator associated to the allowed idempotent state $\y$, then for idempotent
states $\x$ and $\z$ for $\Blg_1'$ and $\Blg_2$ respectively,
\[ \Idemp{\z}\cdot {\mathbf Q}_{\y}\cdot \Idemp{\x}= \left\{\begin{array}{ll}
{\mathbf Q}_{\y} &{\text{if $\y=\z$ and $\x=\psi'(\y)$}} \\
0 &{\text{otherwise.}}
\end{array}\right.\]

The map
$\delta^1_1\colon \lsup{\Blg_2}\Max^c_{\Blg_1'}\to \Blg_2\otimes_{\IdempRing(\Blg_2)} {\lsup{\Blg_2}\Max^c_{\Blg_1'}}$
is given by
\[ {\mathbf Q}_{\y}\mapsto \Idemp{\y}\cdot \left(R_{c+1} R_{c} + L_{c} L_{c+1} + \left\{\begin{array}{ll}
U_{c} C_{c+1} & {\text{if $c+1\in\Upwards_2$}} \\
C_{c} U_{c+1} & {\text{if $c\in\Upwards_2$,}}
\end{array}
 \right\}\right)\otimes \sum_{\z} {\mathbf Q}_{\z}.\]
where the sum is taken over all allowed idempotents $\z$ for $\Blg_2$.

We split the bimodule
$\lsup{\Blg_2}\Max^c_{\Blg_1'}\cong \XX\oplus\YY\oplus \ZZ$
according to the types of the corresponding idempotents as defined in Section~\ref{sec:Crit};
i.e.
\begin{align*}
  \XX = \bigoplus_{\{\y\big|\y\cap\{c-1,c,c+1\}=\{c-1,c\}\}} &{\mathbf Q}_\y \hskip1cm
  \YY = \bigoplus_{\{\y\big|\y\cap\{c-1,c,c+1\}=\{c,c+1\}\}} {\mathbf Q}_\y \\
  \ZZ = &\bigoplus_{\{\y\big|\y\cap\{c-1,c,c+1\}=\{c\}\}} {\mathbf Q}_\y.
\end{align*}

With respect to this splitting, $\delta^1_1$ can be expressed as follows.
If $c\in\Upwards_2$, then 
\begin{align*}
\delta^1_1(\XX)=& C_{c}U_{c+1}\otimes \XX + R_{c+1}R_{c}\otimes \YY\\
\delta^1_1(\YY)=& C_{c}U_{c+1}\otimes \YY + L_{c}L_{c+1}\otimes \XX\\
\delta^1_1(\ZZ)= & C_{c}U_{c+1}\otimes \ZZ;
\end{align*}
whereas if $c+1\in \Upwards_2$, 
\begin{align*}
\delta^1_1(\XX)=& U_{c}C_{c+1}\otimes \XX + R_{c+1}R_{c}\otimes \YY\\
\delta^1_1(\YY)=& U_{c}C_{c+1}\otimes \YY + L_{c}L_{c+1}\otimes \XX\\
\delta^1_1(\ZZ)= & U_{c}C_{c+1}\otimes \ZZ;
\end{align*}

To define $\delta^1_2$, it is helpful to have the following:

\begin{lemma}
  \label{lem:ConstructDeltaTwo}
  Let $\x'$ and $\y'$ be two idempotents for $\Blg_1'$ that are close
  enough, and $\x$ be an allowed idempotent state for $\Blg_2$ with
  $\psi'(\x)=\x'$.  Then, there is a uniquely associated allowed
  idempotent state $\y$ with $\psi'(\y)=\y'$ so that there is a
  surjective map
  $\Phi_{\x}\colon \Idemp{\x'}\cdot \Blg'_1\cdot \Idemp{\y'}\to
  \Idemp{\x}\cdot \Blg_2\cdot \Idemp{\y}$
  that maps the portion of $\Idemp{\x'}\cdot \Blg_1'\cdot \Idemp{\y'}$ with weights
  $(w_1',\dots,w_m')$ surjectively onto the portion of 
  $\Idemp{\x}\cdot \Blg_2\cdot\Idemp{\y}$ with 
  $w_{\phi_c(i)}=w'_i$ and $w_{c}=w_{c+1}=0$, and that satisfies 
  the relations
  $\Phi_{\x}(U_i\cdot a)=U_{\phi_c(i)} \cdot \Phi_{\x}(a)$
  and $\Phi_{\x}(C_j\cdot a)=C_{\phi_c(j)}\cdot \Phi_{\x}(a)$
  for any $i\in 1,\dots,m$, $j\in\Upwards_1'$, and $a\in \Idemp{\x'}\cdot \Blg_1'\cdot \Idemp{\y'}$.
\end{lemma}
  
\begin{proof}
  Recall the weight $v^{\x}$ of idempotents, as defined in Equation~\eqref{eq:DefOfV}.
  Since $\x'$ and $\y'$ are close enough,
  then there is some integer 
  $j\in\Z$ so that exactly one of the following holds:
  \begin{enumerate}
  \item $(v_{c-1}^{\x'},v_{c}^{\x'})=(j,j)$ and $(v_{c-1}^{\y'},v_{c}^{\y'})=(j,j-1)$ 
  \item $(v_{c-1}^{\x'},v_{c}^{\x'})=(j,j)$ and $(v_{c-1}^{\y'},v_{c}^{\y'})=(j,j)$ 
  \item $(v_{c-1}^{\x'},v_{c}^{\x'})=(j,j)$ and $(v_{c-1}^{\y'},v_{c}^{\y'})=(j+1,j)$ 
  \item $(v_{c-1}^{\x'},v_{c}^{\x'})=(j,j-1)$ and $(v_{c-1}^{\y'},v_{c}^{\y'})=(j-1,j-2)$
  \item $(v_{c-1}^{\x'},v_{c}^{\x'})=(j,j-1)$ and $(v_{c-1}^{\y'},v_{c}^{\y'})=(j-1,j-1)$
  \item $(v_{c-1}^{\x'},v_{c}^{\x'})=(j,j-1)$ and $(v_{c-1}^{\y'},v_{c}^{\y'})=(j,j-1)$
  \item $(v_{c-1}^{\x'},v_{c}^{\x'})=(j,j-1)$ and $(v_{c-1}^{\y'},v_{c}^{\y'})=(j,j)$
  \item $(v_{c-1}^{\x'},v_{c}^{\x'})=(j,j-1)$ and $(v_{c-1}^{\y'},v_{c}^{\y'})=(j+1,j)$
  \end{enumerate}
  In each of these cases, we claim that 
  there is a unique allowed idempotent state $\y$ for $\Blg_2$ with $\psi'(\y)=\y'$ 
  so that
  \begin{equation}
    \label{eq:ZeroWeight}
    (v_{c-1}^{\y},v_{c}^{\y})=(v_{c-1}^{\x},v_{c}^{\x}).
  \end{equation}
  To specify $\y$, it suffices to specify its type, as we do in the eight cases listed above:
  \begin{enumerate}
  \item $\x$ is of type $\ZZ$;
    $\y$ is of type $\YY$.
  \item 
    $\x$ is of type $\ZZ$ and $\y$ is of type $\ZZ$
  \item 
    $\x$ is of type $\ZZ$; $\y$ is of type $\XX$.
  \item 
    $\x$ is of type $\XX$ and $\y$ is of type $\YY$.
  \item 
    $\x$ is of type $\XX$ and $\y$ is of type $\ZZ$
  \item 
    \label{case:Special}
    $\x$ and $\y$ are both of type $\XX$ or both of type $\YY$
  \item $\x$ is of type $\YY$ and $\y$ is of type $\ZZ$
  \item $\x$ is of type $\YY$ and $\y$ is of type $\XX$.
  \end{enumerate}

  Recall that there are graded identifications
  \begin{align*}
    \phi^{\x',\y'}\colon \Field[U_1,\dots,U_m] &\to \Idemp{\x'}\cdot \AlgBZ(m,k)\cdot \Idemp{\y'} \\
    \phi^{\x,\y}\colon \Field[U_1,\dots,U_{m+2}] &\to \Idemp{\x}\cdot \AlgBZ(m+2,k+1)\cdot\Idemp{\y}.
  \end{align*}
  The ring map 
  \[ \varphi\colon \Field[U_1,\dots,U_m]\to \Field[U_1,\dots,U_{m+2}]\]
  with $\varphi(1)=1$ and $\varphi(U_i)=U_{\phi(i)}$
  induces a map 
  \[ \Idemp{\x'}\cdot \AlgBZ(m,k)\cdot \Idemp{\y'} \to \Idemp{\x}\cdot \AlgBZ(m+2,k+1)\cdot\Idemp{\y}\]
  that maps the portion of
  $\Idemp{\x'}\cdot \AlgBZ(m,k)\cdot \Idemp{\y'}$ 
  with weights fixed at $(w_1',\dots,w_m')$ onto the portion of 
  $\Idemp{\x}\cdot \AlgBZ(m+2,k+1)\cdot\Idemp{\y}$
  with weights fixed at $w_i=w'_{\phi(i)}$ and $w_{c}=w_{c+1}=0$.

  In the eight above cases, we check that $\varphi(\Ideal(\x',\y'))$
  is mapped into $\Ideal(\x,\y)$.  In all cases other than
  Case~\ref{case:Special}, there are no generating intervals for
  $(\x',\y')$ that contain $c-1$ in their interior. Moreover, the images
  of these generating intervals under $\phi_c$ are generating
  intervals for $(\x,\y)$.  In Case~\ref{case:Special}, there is a
  generating interval $[p,q]$ with $p<c-1<q$. If, furthermore, $\x$ and $\y$ are of
  type $\XX$, then $[c+2,\dots, q+2]$ is a generating interval for $(\x,\y)$, so
  $\Ideal(\x',\y')$ is still mapped into $\Ideal(\x,\y)$.
  Similarly, if $\x$ and $\y$ are of type $\YY$,
  $[p,\dots,c-1]$ is a generating interval for $(\x,\y)$, so once again $\Ideal(\x',\y')$ is 
  mapped into $\Ideal(\x,\y)$.

  It follows that $\varphi$ induces the map $\Phi_\x$ on $\Blg(m,k)\subset \Blg_1'$, which we can extend so that
  $\Phi_{x}(a\cdot C_j)=C_{\phi(j)} \cdot \Phi_{\x}(a)$ for all $j\in \Upwards_1'$ to get the map required by the lemma.
\end{proof}

\begin{lemma}
  \label{lem:RingMap}
  Suppose that $\x_1'$, $\x_2'$, and $\x_3'$ are three idempotent
  states for $\Blg_1'$, so that $\x_i'$ is close enough to $\x_{i+1}'$
  for $i=1,2$; and choose any $\x_1$ so that $\psi(\x_1)=\x_1'$.  Let
  $\x_2$ be the idempotent state associated to $(\x_1,\x_1',\x_2')$ as
  in Lemma~\ref{lem:ConstructDeltaTwo}, and then $\x_3$ be associated
  to $(\x_2,\x_2',\x_3')$ by that lemma. Then, there is a commutative
  diagram:
\[ \begin{CD}
  (\Idemp{\x_1'} \cdot \Blg_1'\cdot \Idemp{\x_2'})\otimes (\Idemp{\x_2'}\cdot \Blg_1'\cdot\Idemp{\x_3'}) @>{\mu_2'}>> \Idemp{\x_1'}\cdot \Blg_1'\cdot \Idemp{\x_3'} \\
@V{\Phi_{\x_1}\otimes \Phi_{\x_2}}VV @VV{\Phi_{\x_1}}V \\
  (\Idemp{\x_1}\cdot  \Blg_2\cdot \Idemp{\x_2})\otimes (\Idemp{\x_2}\cdot \Blg_2\cdot\Idemp{\x_3}) @>{\mu_2}>> 
  \Idemp{\x_1}\cdot \Blg_2\cdot\Idemp{\x_3},
\end{CD}\]
where $\mu_2'$ and $\mu_2$ are multiplications on $\Blg_1'$ and $\Blg_2$ respectively.
\end{lemma}
\begin{proof}
  Commutativity of this diagram is an immediate consequence of the fact that $\Phi$ 
  preserves the weights, which in turn determine the multiplication on the algebras.
\end{proof}

If $a=\Idemp{\x'}\cdot a\cdot \Idemp{\y'}\in\Blg(m,k)\subset \Blg_1'$ is a non-zero algebra element,
and $\x$ is any allowed idempotent with $\psi'(\x)=\x'$, let
$\delta^1_2({\mathbf Q}_{\x},a)=\Phi_{\x}(a)\otimes {\mathbf
  Q}_{\y}$, where $\y$ is the idempotent from Lemma~\ref{lem:ConstructDeltaTwo} 
associated to $(\x,\x',\y')$; i.e.
$\Phi_{\x}(a)\cdot \Idemp{\y}=\Phi_{\x}(a)$.

\begin{thm}
  \label{thm:MaxDA}
  These actions give $\lsup{\Blg_2}\Max^c_{\Blg_1'}$ the structure of a  $DA$ bimodule
  that is adapted to the corresponding partial knot diagram (containing a single maximum).
\end{thm}

\begin{proof}
  Consider the case where $c+1\in\Upwards_2$; the case where
  $c\in\Upwards_2$ works similarly.

  The $\Ainfty$ relation with no
  incoming algebra elements has the form
\begin{equation}
    \label{eq:D1Squared}
\mathcenter{
  \begin{tikzpicture}[scale=.8]
    \node at (0,0) (cin) {}; 
    \node at (0,-1) (d1) {$\delta^1_1$}; 
    \node at (0,-2) (d2) {$\delta^1_1$};
    \node at (-1.5,-2.5) (muA) {$\mu_2^{\Blg_2}$}; 
    \node at (-2.5,-3) (aout) {}; 
    \node at (0,-3) (cout) {};
    \draw[modarrow] (cin) to (d1); 
    \draw[modarrow] (d1) to (d2);
    \draw[modarrow] (d2) to (cout);
    \draw[algarrow] (d1) to (muA); 
    \draw[algarrow] (d2) to (muA); 
    \draw[algarrow] (muA) to (aout);
  \end{tikzpicture}
} +
 \mathcenter{
  \begin{tikzpicture}[scale=.8]
    \node at (0,0) (cin) {}; 
    \node at (0,-1) (d) {$\delta^1_1$};
    \node at (-1,-2.1) (mu) {$\mu_1^{\Blg_2}$}; 
    \node at (-2,-3) (aout) {}; 
    \draw[modarrow] (cin) to (d); 
    \draw[modarrow] (d) to (cout); 
    \draw[blgarrow] (d) to (mu); 
    \draw[blgarrow] (mu) to (aout) ;
  \end{tikzpicture}
}=0.
\end{equation}
 The terms that change type from $\XX$ to $\YY$ and back to $\XX$ (contributing on the left to Equation~\eqref{eq:D1Squared})
 contribute 
 $(R_{c+1}\cdot R_{c} \cdot L_{c} \cdot L_{c+1})\otimes \XX =
 U_{c}\cdot U_{c+1}\otimes \XX$; which cancels with $d(C_{c+1}\cdot
 U_{c}) \otimes \XX$ (contributing on the right to Equation~\eqref{eq:D1Squared}).  A similar argument works for
 generators of type $\YY$.  
 Since $U_{c} U_{c+1}\otimes \ZZ=0$, 
 Equation~\eqref{eq:D1Squared} now follows.

 The $\Ainfty$ relation with one algebra input follows from the equations
\begin{equation}
    \label{eq:D1Squared2}
\mathcenter{
  \begin{tikzpicture}[scale=.7]
    \node at (0,0) (cin) {}; 
    \node at (0,-1) (d1) {$\delta^1_2$}; 
    \node at (0,-2) (d2) {$\delta^1_1$};
    \node at (1,0) (ain) {};
    \node at (-1.5,-2.5) (muA) {$\mu_2^{\Blg_2}$}; 
    \node at (-2,-3.5) (aout) {}; 
    \node at (0,-3.5) (cout) {};
    \draw[algarrow] (ain) to (d1) ;
    \draw[modarrow] (cin) to (d1); 
    \draw[modarrow] (d1) to (d2);
    \draw[modarrow] (d2) to (cout);
    \draw[algarrow] (d1) to (muA); 
    \draw[algarrow] (d2) to (muA); 
    \draw[algarrow] (muA) to (aout);
  \end{tikzpicture}
} +
\mathcenter{
  \begin{tikzpicture}[scale=.7]
    \node at (0,0) (cin) {}; 
    \node at (0,-1) (d1) {$\delta^1_1$}; 
    \node at (0,-2) (d2) {$\delta^1_2$};
    \node at (1,0) (ain) {};
    \node at (-1.5,-2.5) (muA) {$\mu_2^{\Blg_2}$}; 
    \node at (-2,-3.5) (aout) {}; 
    \node at (0,-3.5) (cout) {};
    \draw[algarrow,bend left=10] (ain) to (d2) ;
    \draw[modarrow] (cin) to (d1); 
    \draw[modarrow] (d1) to (d2);
    \draw[modarrow] (d2) to (cout);
    \draw[algarrow] (d1) to (muA); 
    \draw[algarrow] (d2) to (muA); 
    \draw[algarrow] (muA) to (aout);
  \end{tikzpicture}
}=0
\end{equation}
and 
\begin{equation}
    \label{eq:D1Squared3}
    \mathcenter{
  \begin{tikzpicture}[scale=.8]
    \node at (0,0) (cin) {}; 
    \node at (0,-1) (d1) {$\delta^1_2$}; 
    \node at (1,0) (ain) {};
    \node at (-1,-2) (muA) {$\mu_1^{\Blg_2}$}; 
    \node at (-2,-3) (aout) {}; 
    \node at (0,-3) (cout) {};
    \draw[algarrow] (ain) to (d1) ;
    \draw[modarrow] (cin) to (d1); 
    \draw[modarrow] (d1) to (cout);
    \draw[algarrow] (d1) to (muA); 
    \draw[algarrow] (muA) to (aout);
  \end{tikzpicture}
} +
\mathcenter{
  \begin{tikzpicture}[scale=.8]
    \node at (-1,0) (cin) {}; 
    \node at (-1,-2) (d2) {$\delta^1_2$};
    \node at (1,0) (ain) {};
    \node at (0,-1) (muA) {$\mu_1^{\Blg_1'}$}; 
    \node at (-2,-3) (aout) {}; 
    \node at (-1,-3) (cout) {};
    \draw[algarrow] (ain) to (muA) ;
    \draw[modarrow] (cin) to (d2); 
    \draw[modarrow] (d2) to (cout);
    \draw[algarrow] (d2) to (aout); 
    \draw[algarrow] (muA) to (d2);
  \end{tikzpicture}
}=0.
\end{equation}
We verify Equation~\eqref{eq:D1Squared2} for incoming algebra elements
$a=\Idemp{\x'}\cdot a \cdot \Idemp{\y'}$. 
For the terms in $\delta^1_1$ that preserve type, $U_{c} C_{c+1}$, both have the same contribution:
multiplication by  $U_{c} \cdot C_{c+1}$ induces an endomorphism of 
$\Idemp{\x}\cdot \Blg_2 \cdot \Idemp{\y}$  that maps the portion with $w_{c}=w_{c+1}=0$ injectively into
$w_{c}=w_{c+1}=1$ (since the generating interval for $(\x,\y)$ containing $c$ or $c+1$ contains both $c$ and $c+1$).

Consider next the terms ($L_{c} L_{c+1}$ and $R_{c+1} R_{c}$)
where $\delta^1_1$ changes the type of the idempotent. When $\delta^1_{2}$ 
changes the type of the idempotent, as well, both terms in
Equation~\eqref{eq:D1Squared2} vanish, since in that case the outgoing algebra element moves too far.
Finally, consider the the case where $\delta^1_2$ preserves the idempotent, and suppose the incoming 
generator is of the form $Q_{\x_1}$ where $\x_1$ is of type $\XX$ and $\psi'(\x_1)=\x$. 
Let $\x_2$ be the allowed algebra element of type $\YY$ with $\psi'(\x_2)=\y$, so that
$\Idemp{\x_1} \cdot L_{c} L_{c+1}= L_{c}L_{c+1}  \cdot \Idemp{\x_2}$.
The cancellation of the corresponding terms in Equation~\eqref{eq:D1Squared2} now follows
from the easily verified identity 
$\Phi_{\x_1}(a)\cdot L_{c} L_{c+1} = L_{c} L_{c+1} \cdot \Phi_{\x_2}(a)$.
The case where the incoming generator is of type $\YY$ follows from 
the similar equation: $\Phi_{\x_2}(a)\cdot R_{c+1}R_{c} = R_{c+1}R_{c} \cdot \Phi_{\x_1}(a)$.

Equation~\eqref{eq:D1Squared3} follows from the fact that $C_{\phi_c(j)}\cdot \Phi_{\x}=\Phi_{\x}\cdot C_j$ for
all $j\in\Upwards_1$.

The relation
\[
\mathcenter{
  \begin{tikzpicture}[scale=.8]
    \node at (0,0) (cin) {}; 
    \node at (2,0) (a1in) {};
    \node at (3,0) (a2in) {};
    \node at (1,-1) (m2) {$\mu^{\Blg_1'}_{2}$} ;
    \node at (0,-2) (d2) {$\delta^1_{2}$} ;
    \node at (0,-3) (cout) {} ;
    \node at (-1,-3) (aout) {} ;
    \draw[algarrow] (a1in) to (m2) ;
    \draw[algarrow] (a2in) to (m2) ;
    \draw[algarrow] (m2) to (d2) ;
    \draw[algarrow] (d2) to (aout) ;
    \draw[modarrow] (cin) to (d2); 
    \draw[modarrow] (d2) to (cout); 
  \end{tikzpicture}
}
+
\mathcenter{
  \begin{tikzpicture}[scale=.8]
    \node at (0,0) (cin) {}; 
    \node at (0,-1) (d1) {$\delta^1_2$}; 
    \node at (0,-2) (d2) {$\delta^1_2$};
    \node at (1,0) (a1in) {};
    \node at (2,0) (a2in) {};
    \node at (-1.5,-2.5) (muA) {$\mu_2^{\Blg_2}$}; 
    \node at (-2,-3.5) (aout) {}; 
    \node at (0,-3.5) (cout) {};
    \draw[algarrow] (a1in) to (d1) ;
    \draw[algarrow] (a2in) to (d2); 
    \draw[modarrow] (d1) to (d2);
    \draw[modarrow] (cin) to (d1);
    \draw[modarrow] (d2) to (cout);
    \draw[algarrow] (d1) to (muA) ;
    \draw[algarrow] (d2) to (muA) ;
    \draw[algarrow] (muA) to (aout);
  \end{tikzpicture}
}
=0.
\]
is equivalent to Lemma~\ref{lem:RingMap}.
Since $\delta^1_{\ell}=0$ for $\ell>2$, the $\Ainfty$ relations now follow.

To verify the grading, note that all the algebra output in the bimodule satisfy $w_c(b)=w_{c+1}(b)$.
Combined with the grading properties from Lemma~\ref{lem:ConstructDeltaTwo}, it follows
that $\Max^c$ is graded by $H^1(W,\partial W)$, where $W$ is specified by the partial knot diagram,
and the module is thought of as supported in grading $0$.
To verify the Maslov grading, note that all the algebra outputs 
appearing in $\delta^1_1$ have Maslov grading $-1$. We can think of $\Max^c$ as supported in Maslov grading $0$.
It follows readily that $\Max^c$ is adapted to $W$, as in Definition~\ref{def:Adapted}.
\end{proof}

It is convenient to summarize this as follows: when
$c+1\in\Upwards_2$, the $\delta^1_1$ actions are specified by the following diagram:
\[  \begin{tikzpicture}[scale=1.2]
    \node at (-1.5,0) (X) {$\XX$} ;
    \node at (1.5,0) (Y) {$\YY$} ;
    \node at (0,-1.5) (Z) {$\ZZ$} ;
    \draw[->] (X) [bend right=7] to node[below,sloped] {\tiny{$R_{c+1} R_{c}$}}  (Y)  ;
    \draw[->] (Y) [bend right=7] to node[above,sloped] {\tiny{$L_{c} L_{c+1}$}}  (X)  ;
    \draw[->]  (X) [loop above] to node[above,sloped] {\tiny{$C_{c} U_{c+1}$}}  (X)  ;
    \draw[->]  (Y) [loop above] to node[above,sloped] {\tiny{$C_{c} U_{c+1}$}}  (Y)  ;
    \draw[->]  (Z) [loop above] to node[above,sloped] {\tiny{$C_{c} U_{c+1}$}}  (Z)  ;
  \end{tikzpicture}\]
and the $\delta^1_2$ actions are specified by the diagram
\[  \begin{tikzpicture}[scale=1.3]
    \node at (-1.5,0) (X) {$\XX$} ;
    \node at (1.5,0) (Y) {$\YY$} ;
    \node at (0,-2) (Z) {$\ZZ$} ;
    \draw[->] (Y) [bend right=5] to node[above,sloped] {\tiny{$R_{c+2} R_{c-1}\otimes R_{c} R_{c-1}$}}  (X)  ;
    \draw[->] (X) [bend right=5] to node[below,sloped] {\tiny{$L_{c-1} L_{c+2}\otimes L_{c-1} L_{c}$}}  (Y)  ;
    \draw[->] (X) [bend right=5] to node[below,sloped] {\tiny{$L_{c-1}\otimes L_{c-1}$}}  (Z)  ;
    \draw[->] (Z) [bend right=5] to node[above,sloped] {\tiny{$R_{c-1}\otimes R_{c-1}$}}  (X)  ;
    \draw[->] (Z) [bend right=5] to node[below,sloped] {\tiny{$L_{c+2} \otimes L_{c}$}}  (Y)  ;
    \draw[->] (Y) [bend right=5] to node[above,sloped] {\tiny{$R_{c+2}\otimes R_{c}$}}  (Z)  ;
  \end{tikzpicture}
\]
once we include outside actions $L_i\otimes L_{\phi_c(i)}$, $R_i \otimes R_{\phi(i)}$, $U_i\otimes U_{\phi(i)}$ and 
$C_j\otimes C_{\phi_c(j)}$ for $i\in\{1,\dots,m\}$ and $j\in\Upwards_1$, with the further understanding that
$\delta^1_2$ is extended to be multiplicative in the incoming algebra elements.

\begin{prop}
  \label{prop:MaxDual}
  $\Max^c$ is dual to the module $\Crit_c$ from Section~\ref{sec:Crit}, in the following sense.
  Let $\Blg_1=\Blg(m,k,\Upwards_1)$, where $\Upwards_1\subset \{1,\dots,m\}$;
  $\Blg_1'=\Blg(m,m+1-k,\Upwards_1')$, where $\Upwards_1'=\{1,\dots,m\}\setminus\Upwards_1$;
  and
  $\Blg_2=\Blg(m,m+1-k,\Upwards_2)$, where $\Upwards_2=\Upwards_1'\cup\{c\}$ or $\Upwards_1'\cup\{c+1\}$.
  Then, $\lsup{\Blg_2}\GenMax^c_{\Blg_1'}\DT~\lsup{{\Blg_1',\Blg_1}}\CanonDD
  \simeq ~~\lsup{\Blg_2,\Blg_1}\Crit_c$.
\end{prop}

\begin{proof}
  This is straightforward to check using the definitions.
\end{proof}

\subsection{A special case}
\label{subsec:SmallMaximum}

When there is a single maximum, and no other strands, we define a type
$D$ structure over the algebra $\Blg_2=\Blg(2,1,\{1\})$ or
$\Blg(2,1,\{2\})$, $\lsup{\Blg_2}\Max$. This type $D$ structure has 
one generator $\ZZ$ with $\Idemp{\{1\}}\cdot \ZZ= \ZZ$, and
$\delta^1(\ZZ)= C_1 U_2 \otimes \ZZ$ or  $\delta^1(\ZZ)= U_1 C_2 \otimes \ZZ$,
according to how the strand is oriented.
Obviously, this can be thought of as a degenerate case of the earlier construction,
where the incoming algebra $\Blg_1'=\Blg(0,0,\emptyset)\cong \Field$.

\subsection{Partial Kauffman states and grading sets}

Formally, the generators can be thought of as corresponding to partial Kauffman states. There are no crossings, and it has the following types of regions:
\begin{itemize}
\item one of these regions does not meet the top slice, and so it must be ``unoccupied'', and so its only intersection
  with the bottom slice must be occupied,
\item one of these regions meets the top slice in one interval and the bottom in two: thinking of this region as ``occupied''
  gives the generator of type $\ZZ$; thinking of it as unoccupied gives the other two generator types $\XX$ and $\YY$,
\item all other regions meet both the top and the bottom slice in one interval apiece; they can be either occupied or unoccupied.
\end{itemize}

The bimodules can be graded by $\OneHalf \Z^m$ exactly as in the case
of Section~\ref{sec:Crit}. The fact that $\delta^1_2$ respects this
grading is contained in Lemma~\ref{lem:ConstructDeltaTwo}; the fact
that $\delta^1_1$ respects it is clear.  As in Section~\ref{sec:CritGradingSet},
the grading set of can be thought of as $\OneHalf\Z$-valued functions
on the arcs in the partial knot diagram. This grading is consistent
with the grading of $\CanonDD$ by $\OneHalf \Z^m$, combined with the
induced grading on the tensor product, Proposition~\ref{prop:MaxDual}.

\section{The minimum}
\label{sec:Minimum}

Fix integers $0\leq k\leq m+1$ and some $1\leq c \leq m+1$, and 
let 
\begin{equation}
  \label{eq:InitialAlgebras}
  \Blg_1=\Blg(m+2,k+1,\Upwards_1)\qquad{\text{and}}\qquad\Blg_2=\Blg(m,k,\Upwards_2)
\end{equation}
and 
where $\Upwards_2\subset \{1,\dots,m\}$ is arbitrary and
$\Upwards_1=\phi_c(\Upwards_2)\cup\{c\}$ or
$\Upwards_1=\phi_c(\Upwards_2)\cup\{c+1\}$,
where $\phi_c$ is the function from Equation~\eqref{eq:DefInsert}.
We will describe a bimodule
$\lsup{\Blg_2}\GenMin^c_{\Blg_1}$ that will correspond to introducing a
new minimum (or cup) in the diagram that connects the incoming strands $c$ and $c+1$.
 (We will be primarily interested in the
case where $m=2n$, $k=n=|\Upwards|$.)  Note that we follow the convention
that $\Blg_1$ is the incoming algebra and $\Blg_2$ is the outgoing
algebra; thus the notation for $\Blg_1$ and $\Blg_2$ is opposite to
the one used in Section~\ref{sec:Maximum}.

\subsection{Description of the bimodule when $c=1$}
\label{subsec:cEqualsOne}

We start by describing
$\lsup{\Blg_2}\GenMin^c_{\Blg_1}$ when 
$c=1$ and $2\in \Upwards_1$ (and so $1\not\in\Upwards_1$).

A {\em preferred idempotent state} for $\Blg_1=\Blg(m+2,k+1,\Upwards_2)$ is
an idempotent state $\x$ with $\x\cap\{0,1,2\}\in \{\{0\}, \{2\}, \{0,2\}\}$.
We define a map $\psi$ from preferred idempotent states of $\Blg_1$ to idempotent states of $\Blg_2$, as follows.
Given preferred idempotent state $\x$ for $\Blg_1$, order the components
$\x=\{x_1,\dots,x_{k+1}\}$ so that
$x_1<\dots<x_{k+1}$. Define
\[
\psi(\x) = \left\{
\begin{array}{ll}
\{0,x_3-2,\dots,x_{k+1}-2\} &{\text{if $|\x\cap\{0,1,2\}|=2$}} \\
\{x_2-2,\dots,x_{k+1}-2\} &{\text{if $|\x\cap\{0,1,2\}|=1$}} \\
\end{array}\right.
\]

Generators of the $DA$ bimodule
$\GenMin^1=\lsup{\Blg_2}\GenMin^1_{\Blg_1}$ correspond to preferred
idempotent states, and the bimodule structure over the idempotent algebras is as follows. 
If $\x$ is a preferred idempotent state, let $\MinGen_\x$ be its corresponding generator.
Then,
\[\Idemp{\y}\cdot {\mathbf T}_\x \cdot \Idemp{\z}=
\left\{\begin{array}{ll}
    \MinGen_{\x} & {\text{if $\x=\z$ and $\y=\psi(\x)$}} \\
    0 &{\text{otherwise.}}
  \end{array}\right.\]

The bimodule structure is expressed in terms of the following oriented graph
$\Gamma$. The vertices of $\Gamma$
correspond to words
$\{ C_2, L_1 C_2, R_2, U_1^t, L_1 U_1^t\}_{t\geq 0}$,
with the understanding that $U_1^0=1$. The graph
comes equipped with the following oriented edges, labelled by words of
the form $\{1, L_1, U_1^n, U_1^t C_2, R_1 U_1^t C_2\}_{t\geq 0, n>0}$:
\begin{itemize}
  \item An edge labelled $1$ from $R_2$ to $R_2$.
  \item An edge labelled $1$ from $L_1 C_2$ to $L_1 C_2$.
    \item An edge labelled $L_1 C_2$ from $1$ to $L_1 C_2$.
  \item  An edge labelled $L_1$ from $C_2$ to $L_1 C_2$.
  \item  An edge labelled $R_2$ from from $1$ to $R_2$.
  \item An edge labelled $U_2$ from $1$ to $C_2$.
    \item An edge labelled $L_2$ from $R_2$ to $C_2$.
  \item  For each $n>0$, an edge labelled $U_1^n$ from $L_1 C_2$ to $L_1 U^{n-1}$.
  \item  For each $t\geq 0$, an edge labelled $R_1 U_1^t$ from $L_1 C_2$ to $U_1^t$.
  \item  For each $s\geq 0$ and $t\geq 0$ so that $s+t>0$,
    an edge labelled $C_2 U_1^s$ 
    from $U_1^t$ to $U_1^{s+t-1}$; and another with the same label from $U_1^t L_1$ to $U_1^{s+t-1} L_1 $.
  \item  For each $s\geq 0$ and $t\geq 0$ with $s+t>0$
    an edge labelled $C_2 U_1^s L_1 $ 
    from $U_1^t$ to $U_1^{s+t-1} L_1$.
  \item  For each $s\geq 0$ and $t\geq 0$, an edge labelled $C_2  U_1^s R_1$ from $U_1^t L_1$ to $U_1^{s+t}$.
  \item For each $n\geq 1$, an edge labelled $U_1^n$ from $C_2$ to $U_1^{n-1}$.
  \item For each $n\geq 1$, an edge labelled $U_1^n L_1$ from $C_2$ to $U_1^n L_1$.
\end{itemize}

See Figure~\ref{eq:MinimumOperations} for an illustration.

\begin{figure}
\begin{tikzpicture}[scale=1.6]
  \node at (-2.5,0) (p0) {$L_1 C_2$} ;
  \node at (0,-1) (px1) {$C_2$} ;
  \node at (0,1) (p1) {$1$} ;
  \node at (0,2.5) (p2) {$L_1$} ;
  \node at (0,4) (p3) {$U_1$} ;
  \node at (2.5,0) (y) {$R_2$} ;
  \draw[<-] (px1) [bend left=30] to node[above,sloped]{\tiny{$U_2$}} (p1) ;
  \draw[->] (y) to node[above,sloped]{\tiny{$L_2$}} (px1) ;
  \draw[->] (p1) to node[above,sloped,pos=.7]{\tiny{$R_2$}} (y) ;
  \draw[->] (p0) [bend left=5] to node[above,sloped]{\tiny{$R_1$}} (p1) ;
  \draw[->] (px1) to node[above,sloped]{\tiny{$L_1$}} (p0) ;
  \draw[<-] (p0) [bend right=5] to node[below,sloped,pos=.55]{\tiny{$C_2$}} (p2) ;
  \draw[->] (p1) [bend left=5] to node[below,sloped]{\tiny{$L_1 C_2$}} (p0) ;
  \draw[<-] (p1) [bend left=30] to node[above,sloped,pos=.2]{\tiny{$U_1$}} (px1) ;
  \draw[<-] (p2) [bend left=50] to node[above,sloped,pos=.4]{\tiny{$U_1 L_1$}} (px1) ;
  \draw[<-] (p3) [bend left=60] to node[above,sloped,pos=.5]{\tiny{$U_1^2$}} (px1) ;
  \draw[<-] (p1) [bend left=50] to node[above,sloped,pos=.55]{\tiny{$C_2$}} (p3) ;
  \draw[<-] (p2) [bend left=30] to node[below,sloped,pos=.55]{\tiny{$L_1 U_1 C_2$}} (p1) ;
  \draw[<-] (p3) [bend left=50] to node[below,sloped,pos=.6]{\tiny{$U_1^2 C_2$}} (p1) ;
  \draw[<-] (p1) [bend left=30] to node[above,sloped,pos=.55]{\tiny{$R_1 C_2$}} (p2) ;
  \draw[<-] (p2) [bend right=10] to node[above,sloped]{\tiny{$U_1$}} (p0) ;  
  \draw[<-] (p3) [bend right=25] to node[above,sloped]{\tiny{$U_1 R_1$}} (p0) ;  
  \draw[->] (p1) [loop above] to node[above,sloped]{\tiny{$U_1 C_2$}} (p1) ;
  \draw[->] (p2) [loop above] to node[above,sloped]{\tiny{$U_1 C_2$}} (p2) ;
  \draw[->] (p3) [loop above] to node[right,sloped]{\tiny{$U_1 C_2$}} (p3) ;
  \draw[<-] (p2) [bend left=30] to node[below,sloped]{\tiny{$L_1 C_2$}} (p3) ;
  \draw[<-] (p3) [bend left=30] to node[below,sloped]{\tiny{$U_1 R_1 C_2$}} (p2) ;
  \draw[->] (p0) [loop below] to node[below]{\tiny{$1$}} (p0) ;
  \draw[->] (y) [loop below] to node[below]{\tiny{$1$}} (y) ;
\end{tikzpicture}
\caption{{\bf Operation graph for a minimum when $2\in\Upwards_1$.}
\label{eq:MinimumOperations}
}
\end{figure}

Fix a sequence $a_1,\dots,a_{\ell-1}$ of pure algebra elements  in $\Blg(m+2,k+1,\Upwards_1)$. 
We call the sequence an {\em preferred sequence} if the following conditions are satisfied:
\begin{enumerate}[label=({$\GenMin$}-\arabic*),ref=({$\GenMin$}-\arabic*)]
\item there are idempotent states $\x_1,\dots,\x_\ell$ so that
  $\Idemp{\x_{i}}\cdot a_i\cdot \Idemp{\x_{i+1}}=a_i$ for $i=1,\dots,\ell-1$,
\item The idempotent states $\x_1$ and $\x_{\ell}$ are  preferred idempotent states.
\item Each $C_j$ with $j\in \Upwards_1\setminus\{2\}$ divides at most one of the $a_k$.
\item There is an oriented path $e_1,\dots,e_{\ell}$ in $\Gamma$ 
  beginning at 
  $C_2 L_1$ if $0\in\x_1$ and at $R_2$ if $0\not\in\x_1$.
  The label on the edge $e_i$, thought of as an algebra element in $\Blg(m+2,k+1,\Upwards_1)$,
  has $w_1(a_i)=w_1(e_i)$ and $w_2(a_i)=w_2(e_i)$;
  and $a_i$ is divisible by $C_2$ if and only if $e_i$ is divisible by $C_2$, 
  none of the internal vertices are at $R_2$ or $L_1 C_2$,
  and the terminal vertex is
  at $L_1 C_2$ or $R_2$.
\end{enumerate}

For a preferred sequence, there is at most one pure non-zero element $b\in\Blg_2$ characterized by the following properties
\begin{enumerate}[label=(PS-\arabic*),ref=(PS-\arabic*)]
  \item $b=\Idemp{\psi(\x_1)}\cdot b$
  \item 
    \label{eq:MasGradeMinimum} For $j\in \Upwards_2$, $C_j$ divides $b$ if and only if $C_{j+2}$ divides some $a_k$
  \item 
    \label{eq:GradeMinimum}
    For $i=1,\dots,m$,
      $w_i(b)=\sum_{j=1}^{\ell-1} w_{i+2}(a_j)$.
\end{enumerate}

Define maps
$\delta^1_{\ell}\colon \lsup{\Blg_2}\GenMin^1_{\Blg_1}\otimes \overbrace{\Blg_1 \otimes \dots \otimes \Blg_1}^{\ell-1} \to 
\Blg_2\otimes~^{\Blg_2}\GenMin^1_{\Blg_1}$
by specifying them on sequences
$a_1,\dots,a_{\ell-1}\in\Blg(m+2,k+1,\Upwards_1)$
for which there are   idempotent states $\x_1,\dots,\x_{\ell}$ 
with 
$\Idemp{\x_i}\cdot a_i \cdot \Idemp{\x_{i+1}}=a_i$ and the
$a_i$ are pure. For such a sequence, the operation
$\delta^1_{\ell}(\MinGen_{\x_1},a_1,\dots,a_{\ell-1})$ is non-zero
only if the sequence is a preferred sequence, and we define
$\delta^1_{\ell}(\MinGen_{\x_1},a_1,\dots,a_{\ell-1})=b\otimes \MinGen_{\x_{\ell}}$,
where $b$ is the algebra element specified by the sequence.
For $\ell=1$, we define $\delta^1_1=0$.

For example, sequences $(a_1,\dots,a_{\ell-1})$ for which
$\delta^1_\ell(\MinGen,a_1,\dots,a_\ell)$ is non-zero (for suitably chosen $\MinGen$) include the sequences
\begin{equation}
  \label{eq:SampleSequences}
  \begin{array}{llll}
    (L_2,L_1), & (R_1,R_2), & (L_2,U_1,R_2), & (U_1,C_2);
  \end{array}
\end{equation}
the
first comes from a path from $R_2$ to $L_1 C_2$; the second comes from a path from
$L_1 C_2$ to $R_2$,
 the third goes from $R_2$ to itself, and the last  goes from
$L_1 C_2$ to itself).
The loops labelled $1$ also give rise to the following actions: 
if $\x_1$ and $\x_2$ are preferred idempotent states, then:
\[ \begin{array}{ll}
  \delta^1_2(\MinGen_{\x_1},R_i)=R_{i-2}\otimes \MinGen_{\x_2} & {\text{if $i>2$ and $\Idemp{\x_1}\cdot R_i = R_i\cdot \Idemp{\x_2}$}} \\
  \delta^1_2(\MinGen_{\x_1},L_i)=L_{i-2}\otimes \MinGen_{\x_2} &{\text{if $i>2$ and if $\Idemp{\x_1}\cdot L_i = L_i\cdot \Idemp{\x_2}$}} \\
  \delta^1_2(\MinGen_{\x_1},U_i)=U_{i-2}\otimes \MinGen_{\x_1} &{\text{if $i>2$}} \\
  \delta^1_2(\MinGen_{\x_1},C_j)=C_{j-2}\otimes \MinGen_{\x_1} &{\text{if $2<j\in\Upwards_1$.}}
  \end{array}
\]

As an illustration, we list some other preferred sequences:
\[ (U_1 R_1, L_1 C_2, C_2), \qquad (L_1, \overbrace{U_1, U_2, \dots, U_1, U_2 }^n, L_2),
\qquad (U_1^n, \overbrace{C_2,\dots, C_2}^n).\]

So far, we described the case where $2\in\Upwards_1$.  In cases where
$1\in\Upwards_1$, we modify the earlier construction slightly as
follows: in the description of the graph $\Gamma$, switch the roles of
$U_1$ and $U_2$, $L_1$ and $R_2$, $R_1$ and $L_2$, $C_2$ and $C_1$.
To read off the actions, allowed idempotent states with $0\in\x$ correspond to the starting vertex $L_1$,
and those with $0\not\in\x$ correspond to $R_2 C_1$.
With this adjustment, the bimodule is defined as before.
See Figure~\ref{eq:MinimumOperations2} for a picture.

\begin{figure}
\begin{tikzpicture}[scale=1.6]
  \node at (-2.5,0) (p0) {$R_2 C_1$} ;
  \node at (0,-1) (px1) {$C_1$} ;
  \node at (0,1) (p1) {$1$} ;
  \node at (0,2.5) (p2) {$R_2$} ;
  \node at (0,4) (p3) {$U_2$} ;
  \node at (2.5,0) (y) {$L_1$} ;
  \draw[<-] (px1) [bend left=30] to node[above,sloped]{\tiny{$U_1$}} (p1) ;
  \draw[->] (y) to node[above,sloped]{\tiny{$R_1$}} (px1) ;
  \draw[->] (p1) to node[above,sloped,pos=.7]{\tiny{$L_1$}} (y) ;
  \draw[->] (p0) [bend left=5] to node[above,sloped]{\tiny{$L_2$}} (p1) ;
  \draw[->] (px1) to node[above,sloped]{\tiny{$R_2$}} (p0) ;
  \draw[<-] (p0) [bend right=5] to node[below,sloped,pos=.55]{\tiny{$C_1$}} (p2) ;
  \draw[->] (p1) [bend left=5] to node[below,sloped]{\tiny{$R_2 C_1$}} (p0) ;
  \draw[<-] (p1) [bend left=30] to node[above,sloped,pos=.2]{\tiny{$U_2$}} (px1) ;
  \draw[<-] (p2) [bend left=50] to node[above,sloped,pos=.4]{\tiny{$U_2 R_2$}} (px1) ;
  \draw[<-] (p3) [bend left=60] to node[above,sloped,pos=.5]{\tiny{$U_2^2$}} (px1) ;
  \draw[<-] (p1) [bend left=50] to node[above,sloped,pos=.55]{\tiny{$C_1$}} (p3) ;
  \draw[<-] (p2) [bend left=30] to node[below,sloped,pos=.55]{\tiny{$R_2 U_2 C_1$}} (p1) ;
  \draw[<-] (p3) [bend left=50] to node[below,sloped,pos=.6]{\tiny{$U_2^2 C_1$}} (p1) ;
  \draw[<-] (p1) [bend left=30] to node[above,sloped,pos=.55]{\tiny{$L_2 C_1$}} (p2) ;
  \draw[<-] (p2) [bend right=10] to node[above,sloped]{\tiny{$U_2$}} (p0) ;  
  \draw[<-] (p3) [bend right=25] to node[above,sloped]{\tiny{$U_2 L_2$}} (p0) ;  
  \draw[->] (p1) [loop above] to node[above,sloped]{\tiny{$U_2 C_1$}} (p1) ;
  \draw[->] (p2) [loop above] to node[above,sloped]{\tiny{$U_2 C_1$}} (p2) ;
  \draw[->] (p3) [loop above] to node[right,sloped]{\tiny{$U_2 C_1$}} (p3) ;
  \draw[<-] (p2) [bend left=30] to node[below,sloped]{\tiny{$R_2 C_1$}} (p3) ;
  \draw[<-] (p3) [bend left=30] to node[below,sloped]{\tiny{$U_2 L_2 C_1$}} (p2) ;
  \draw[->] (p0) [loop below] to node[below]{\tiny{$1$}} (p0) ;
  \draw[->] (y) [loop below] to node[below]{\tiny{$1$}} (y) ;
\end{tikzpicture}
\caption{{\bf Operation graph for a minimum when $1\in\Upwards_1$.}
\label{eq:MinimumOperations2}}
\end{figure}

\begin{prop}
  \label{prop:ConstructMinimum}
  The above maps $\delta^1_{\ell}$ on $\lsup{\Blg_2}\GenMin^1_{\Blg_1}$ satisfy the $DA$ bimodule relations;
  and  $\GenMin^1$ is adapted to the one-manifold underlying
  its knot diagram (Definition~\ref{def:Adapted}).
\end{prop}

Proposition~\ref{prop:ConstructMinimum} will be proved after we give
an alternative construction of $\GenMin^1$, later in this section.
The significance of $\GenMin^1$ is formulated in the following:

\begin{lemma}
  \label{lemma:MinDual}
  $\GenMin^1$ is dual to $\Crit_1$, in the following sense.
  Fix arbitrary integers $0\leq k<m$, a
  subsequence $\Upwards_2=\{u_1,\dots,u_\ell\} \subset \{1,\dots,m\}$, and let
  \[
    \Upwards_1=\{1,u_1+2,\dots,u_\ell+2\}~\qquad~\text{or}~\qquad
      \Upwards_1=\{2,u_1+2,\dots,u_\ell+2\},
      \]
      $\Upwards_3=\{1,\dots,m+2\}\setminus \Upwards_1$;
      $\Blg_1=\Blg(m+2,k+1,\Upwards_1)$,
      $\Blg_2=\Blg(m,k,\Upwards_2)$, and
      $\Blg_3=\Blg(m+2,m+1-k,\Upwards_3)$. Then, there is an equivalence
  $\lsup{\Blg_2} \GenMin_{\Blg_1} \DT~\lsup{\Blg_1,\Blg_3}\CanonDD 
  \simeq ~\lsup{\Blg_2,\Blg_3}\Crit_1$.
\end{lemma}

\begin{proof}
  The generators of $\GenMin^1$ corresponding to idempotent states
  with $\x\cap\{0,1,2\}=\{2\}$, $\{0\}$, and $\{0,2\}$ respectively
  induce generators in 
  $\lsup{\Blg_2} \GenMin_{\Blg_1} \DT~\lsup{\Blg_1,\Blg_3}\CanonDD $
  corresponding to generators of
  $\GenMax$ of type 
  $\XX$, $\YY$, and $\ZZ$, in the notation of Section~\ref{sec:Maximum}.

  If  $2\in\Upwards_1$, we can 
  pair the actions from Equation~\eqref{eq:SampleSequences}
  (counting the last one twice) with the differential from $\CanonDD$
  to give the  differentials 
  \[ 
  \begin{array}{lll}
    \XX \to (1\otimes L_1 L_2) \otimes \YY, &  \YY\to
    (1\otimes R_2 R_1)\otimes \XX, &
    \YY\to (1\otimes C_1 U_2)\otimes \YY \\
    \XX\to 
    (1\otimes C_1 U_2)\otimes \XX, 
    & \ZZ\to (1\otimes C_1 U_2)\otimes \ZZ
  \end{array}\]
  appearing in the description of $^{\Blg_2,\Blg_3}\Crit_1$.
  The loops labelled $1$ induce $\delta^1_2$ actions by the part of the algebra with $w_1=w_2=0$.
  These actions give rise to the remaining terms for the differential in $\Crit_1$.
  The case where $1\in\Upwards_1$ works similarly.
\end{proof}

\subsection{An alternative construction}

We describe a $\Blg_2-\Blg_1$ DG bimodule $M$ that is quasi-isomorphic
to the bimodule $\lsup{\Blg_2}\GenMin^c_{\Blg_1}$ with
$c=1$. Proposition~\ref{prop:ConstructMinimum} will be an immediate
consequence of this construction. We continue with the hypothesis that
$2\in\Upwards_1$, returning to the other case at the end of the subsection.

Let ${\mathbf I}=\sum_{\{\x\big|x_1=1\}} \Idemp{\x}\in\Blg_1$.
There is a natural inclusion of $\phi\colon \Blg_2\to {\mathbf I}\cdot 
\Blg_1 \cdot {\mathbf I}$, whose image consists of the portion of ${\mathbf
  I}\cdot \Blg_1\cdot {\mathbf I}$ with $w_1=w_2=0$. 
In particular, $\phi(R_i)={\mathbf I}\cdot R_{i+2}$
and $\phi(L_i)={\mathbf I}\cdot L_{i+2}$.
Thus, we can think of ${\mathbf I}\cdot \Blg_1$ as a left module for
$\Blg_2$; 
 it is also a right module for $\Blg_1$. 
Consider the bimodule
\[ \lsub{\Blg_2}M_{\Blg_1}= {\mathbf I}\cdot \Blg_1/L_1 L_2
\cdot {\Blg_1},\] 
thought of as a $\Blg_2$-$\Blg_1$-module, i.e., with 
$m_{1|1|0}(b,a)=\phi(b)\cdot a$ and $m_{0|1|1}(a,a')=a\cdot a'$; and equipped with the 
equipped with the endomorphism
$\partial=d+U_1 C_2$.

In the next lemma, we describe the left $\Blg_2$-module
  strcture on $M$, using the following notation. Fix an algebra
  element $a\in\Blg_1$, and consider ${\mathbf I}\cdot a$. This
  generates a left $\IdempRing(\Blg_2)$-module, under the identification
  $\IdempRing(\Blg_2)\cong {\mathbf I}\cdot \IdempRing(\Blg_2)\cdot {\mathbf I}$. Thus, we can form  a new
  left $\Blg_2$-module $\Blg_2\otimes_{\IdempRing(\Blg_2)} \IdempRing({\Blg_2})\cdot {\mathbf I} a$,
  which we abbreviate $\Blg_2\otimes_{\IdempRing(\Blg_2)}a$. In general,
  there is a quotient map of left $\Blg_2$-modules
  $\Blg_2\otimes_{\IdempRing(\Blg_2)}a\to \Blg_2\cdot a$.

\begin{lemma}
  \label{lem:BimoduleSplitting}
  Consider $\lsub{\Blg_2}M_{\Blg_1}$ equipped 
  with endomorphism $m_{0|1|0}=\partial$
  as a left module over $\Blg_2$. There is an isomorphism
  $\lsub{\Blg_2}M \cong \bigoplus_{a_i} \Blg_2\otimes_{\IdempRing(\Blg_2)} a_i$, where the $a_i$ are chosen from the generating set
  \begin{equation}
    \label{eq:IdealGenerators}
    \begin{array}{lllllllll}
    \{ R_2 U_2^t,& U_2^n, &
    C_2 R_2 U_2^t,& C_2 U_2^t, &
    & L_1 U_1^t,& U_1^t,& 
    C_2 U_1^t\}_{t\geq 0,n>0}.
    \end{array}
  \end{equation}
\end{lemma}

\begin{proof}
  Consider any $a=\Idemp{\x}\cdot a\cdot \Idemp{\y}$ with $x_1=1$,
  and with non-trivial projection to  ${\mathbf I}\cdot \Blg_1/L_1 L_2\cdot \Blg_1$.
  If $y_1=0$, then $a=L_1 U_1^t \cdot a_2$ with $w_1(a_2)=w_2(a_2)=0$; indeed, in this case,
  $a=a_2' L_1 U_1^t$ (where $a_2'$ and $a_1'$ differ in their initial idempotent).
  If $y_1=2$, then $a=b\cdot R_2 U_2^t$ or $b\cdot R_2 U_2^t C_2$.
  Finally, if $y_1=1$, then $a=b\cdot U_2^t$ or $a=b\cdot U_1^t$.
  This shows that the generating sets are as enumerated above.

  So far, we have identified the $\Blg_2$-orbits $\lsub{\Blg_2}M \cong
  \bigoplus_{a_i}\Blg_2\cdot a_i$, where the $a_i$ are as listed
  above.  For each of the above summands, the identification
  $\Blg_2\cdot a_i\cong \Blg_2\otimes_{\IdempRing(\Blg_2)} a_i$ is
  equivalent to the statement that if $\x$ is an idempotent state with
  $x_1=1$, $b={\mathbf I}\cdot b\cdot {\mathbf I}$ is an element of
  $\Blg_1$ with $w_1(b)=w_2(b)=0$, and $\Idemp{\x}\cdot b \cdot
  a_i=0$, then $\Idemp{\x}\cdot b=0$. But this follows from
  Proposition~\ref{prop:IdentifyJ}, together with the fact that that
  if $\x$ is an idempotent state for $\Blg_1$ with $x_1=1$ and $\y$ is
  another idempotent state, then no generating interval (in the sense
  of Definition~\ref{def:GeneratingInterval}) can contain exactly one
  of $1$ or $2$, whereas none of the $a_i$ enumerated above is
  divisible by $U_1 U_2$.
\end{proof}

According to the above lemma, $\lsub{\Blg_2}M_{\Blg_1}$ can be thought
of as a type $DA$ bimodule with generating set specified in the lemma
(see Equation~\eqref{eq:TypeDModuleSplitting});
i.e. there is a type $DA$ bimodule $\lsup{\Blg_2}X_{\Blg_1}$
so that $\lsub{\Blg_2}M_{\Blg_1}=\Blg_2\DT \lsup{\Blg_2}X_{\Blg_1}$.
There is a much smaller model for this type $DA$ bimodule using the following:

\begin{lemma}
  \label{lem:SmallerModelMinimum}
  The inclusion map $\left(\Blg_2 \cdot L_1 C_2\right) \oplus \left(\Blg_2 \cdot R_2\right)\subset \lsub{\Blg_2}M_{\Blg_1}$
  induces an isomorphism in homology.
\end{lemma}

\begin{proof}
  In view of Lemma~\ref{lem:BimoduleSplitting},
  we can define $H\colon M \to M$ by specifying
  \begin{align*}
    H(b_2\cdot R_2 U_2^n) &= b_2 R_2 U_2^{n-1} C_2 \\
    H(b_2\cdot U_2^n) &= b_2 U_2^{n-1} C_2  \\
    H(b_2\cdot C_2 U_1^n) &= b_2 U_1^{n-1}\\
    H(b_2)&= H(b_2 \cdot R_2)=H(b_2 C_2 R_2 U_2^t)=
    H(b_2 \cdot L_1 U_1^t)  \\
    &=H(b_2 \cdot C_2 R_2 U_2^t)  =
    H(b_2 \cdot L_1 C_2 U_1^t) =0 
  \end{align*}
  for all $n>0$, $t\geq 0$, and $b_2\in \Blg_2$ (thought of as the subalgebra of ${\mathbf I}\cdot \Blg_1 \cdot{\mathbf I}$ with $w_1=w_2=0$).
  It is straightforward to verify that 
  $\Id + \partial \circ H + H \circ \partial$
  is projection onto the image of $\left(\Blg_2 \cdot L_1 C_2\right) \oplus \left(\Blg_2 \cdot R_2\right)\subset \lsub{\Blg_2}M_{\Blg_1}$.
  For example, if $a=b_2\cdot C_2 U_1^n$ with $n>0$, then $\partial \circ H(a)=a+(d b_2) U_2^{n-1}$, and $H\circ \partial(a)=(d b_2) U_1^{n-1}$;
  so $(\Id + \partial\circ H + H\circ \partial)(a)=0$. 
\end{proof}

\begin{proof}[Proof of Proposition~\ref{prop:ConstructMinimum}]
  We continue with the hypothesis that $2\in\Upwards_1$.
  By Lemma~\ref{lem:SmallerModelMinimum} (and the homological
  perturbation lemma, in the form of
  Lemma~\ref{lem:HomologicalPerturbation2}) $\left(\Blg_2 \cdot C_2
    L_1\right) \oplus \left(\Blg_2 \cdot R_2\right)$ inherits an
  $\Ainfty$ bimodule structure quasi-isomorphic to $M$. To this end,
  note that the maps $H$ appearing in
  Lemma~\ref{lem:SmallerModelMinimum} are $\Blg_2$-linear, and so they
  can be thought of as morphisms of type $D$ structures as in
  Lemma~\ref{lem:HomologicalPerturbation2}.
  
  The induced $DA$ bimodule structure on $\left(\Blg_2 \cdot C_2
    L_1\right) \oplus \left(\Blg_2 \cdot R_2\right)$ coincides with
  the bimodule structure on $\lsup{\Blg_2}\GenMin^1_{\Blg_1}$. In a
  little more detail, the labels on the edges $e_1,\dots,e_{\ell-1}$,
  associated to the algebra elements $a_1,\dots,a_{\ell-1}$ are chosen
  according to the following rules:
  \begin{itemize}
    \item 
      labels for the possible edges $e_1$ out of 
      $v_1=R_1$ or $L_1 C_2$ are chosen
      so that, according to Lemma~\ref{lem:BimoduleSplitting} we have
      $b_1\in\Blg_2$ with $v_1 \cdot a_1 = b_1\cdot v_2$ (thinking of
      $v_i$ as certain left $\Blg_2$-module generators of $M$ from
      Equation~\eqref{eq:IdealGenerators}); 
    \item labels for the edges $e_i$ with $1<i<\ell-1$ from $v_i$ to
      $v_{i+1}$ (thought of as certain elements of $\Blg_1$) are
      chosen so that, according to Lemma~\ref{lem:BimoduleSplitting},
      there is  $b_i\in\Blg_2$ with $v_i\cdot a_i = b_i \cdot w_i$,
      and 
      $v_{i+1}=H(w_i)$
    \item For the choices of final edges $e_{\ell-1}$ into
      (with $v_{\ell}=R_1$ or $L_1 C_2$), once again the labels were chosen
      so that for some $b_{\ell-1}\in\Blg_2$,  $v_{\ell-1}\cdot a_{\ell-1}= b_{\ell-1}\cdot v_{\ell}$.
    \end{itemize}
    Thus, as in Lemma~\ref{lem:HomologicalPerturbation2}, the
    induced $DA$ bimodule structure associates to the sequence
    $a_1,\dots,a_{\ell-1}$, is the product $b=b_1\cdots
    b_{\ell-1}$, which in turn is the output
    element associated to the pure sequence in our definition of
    $\delta^1_\ell$ from Section~\ref{subsec:cEqualsOne}.
    
    Observe that for all sequences $a_1\otimes\dots\otimes a_{\ell-1}$ appearing in the $\Ainfty$ operations,
    $\sum_{i=1}^{\ell-1} w_1(a_i)= \sum_{i=1}^{\ell-1} w_2(a_i)$. This,
    together with Property~\ref{eq:GradeMinimum} amounts to the
    statement that $\GenMin^i$ is graded  by the set
    $H^1(W,\partial W)$. Observe also that
    for each preferred sequence $a_1,\dots,a_{\ell-1}$, we have that
    $\sum_{i=1}^{\ell}\MasFilt_2(a_i)-2 w_2(a_i)=1-\ell$.
    (To see this, note that the change in the vertical coordinate of an arrow labelled by 
    an algebra element $a$ in Figure~\ref{eq:MinimumOperations} is given by
    $\MasFilt_2(a)-2 w_2(a)-1$.)
    This, together with 
    Property~\ref{eq:MasGradeMinimum} ensures
    that the operations respect a $\Z$-valued Maslov grading.  Obviously, 
    $\GenMin^i$ is finite dimensional. We have thus verified that $\GenMin^i$ is adapted
    to the one-manifold underlying the partial knot diagram with a left-most minimum.

    When $1\in\Upwards_1$, the above discussion applies with
    straightforward notational changes, switching the roles of $U_1$
    and $U_2$, $L_1$ and $R_2$, $R_1$ and $L_2$, $C_2$ and $C_1$.  For
    example, the analogue of
    Lemma~\ref{lem:BimoduleSplitting} holds, where now $a_i$ are chosen from
    \[ 
      \begin{array}{lllllllll}
        \{ L_1 U_1^t,& U_1^n, &
        C_1 L_1 U_1^t,& C_1 U_1^t, &
        & R_2 U_2^t,& U_2^t,& 
        C_1 U_2^t\}_{t\geq 0,n>0}.
      \end{array}
      \]
      The analogue of Lemma~\ref{lem:SmallerModelMinimum} gives a
      model generated by $(\Blg_2\cdot R_2 C_1) \oplus(\Blg_2 \cdot L_1)$, with  $\Ainfty$ structure obtained from 
      the graph $\Gamma$ with the corresponding modification.
  \end{proof}

\subsection{The bimodule of a minimum with arbitrary $c$}

Having defined $\lsup{\Blg_2}\GenMin^c_{\Blg_1}$ for $c=1$, we can
inductively define it for arbitrary $c$ by the relation 
\begin{equation}
  \label{eq:GenMinDef}
\lsup{\Blg_2}\GenMin^{c}_{\Blg_1}=\lsup{\Blg_2}\GenMin^{c-1}_{\Blg_4}\DT
\lsup{\Blg_4}\Pos^{c}_{\Blg_3}\DT~^{\Blg_3}\Pos^{c-1}_{\Blg_1};
\end{equation}
with
algebras $\Blg_1$ and $\Blg_2$ chosen as in the beginning of the
section; and
$\Blg_3=\Blg(m+2,k+1,\tau_{c-1}(\Upwards_1))$ and
$\Blg_4=\Blg(m+2,k+1,\tau_c\circ\tau_{c-1}(\Upwards_1))$.
See Figure~\ref{fig:GenMinDef}.

\begin{figure}[ht]
\input{GenMinDef.pstex_t}
\caption{\label{fig:GenMinDef} To define $\GenMin^{c}$, tensor $\GenMin^{c-1}$
  with $\Pos^{c}$ and $\Pos^{c-1}$, as shown.}
\end{figure}

\begin{prop}
  \label{prop:MinDual}
  Choose $\Blg_1$ and $\Blg_2$ as in Equation~\eqref{eq:InitialAlgebras},
  and let $\Blg_1'=\Blg(m+2,m+2-k,\{1,\dots,m+2\}\setminus\Upwards_1)$.
  Then
  $\lsup{\Blg_2}\GenMin^c_{\Blg_1}$  is a type $DA$ bimodule that is adapted to the one-manifold underlying
  its partial knot diagram, and it is dual to $\Crit_c$, in the sense that
  \begin{equation}
    \label{eq:MinDual}
    \lsup{\Blg_2} \GenMin^c_{\Blg_1} \DT~\lsup{\Blg_1,\Blg_1'}\CanonDD 
    \simeq ~\lsup{\Blg_2,\Blg_1'}\Crit_c.
  \end{equation}
\end{prop}

\begin{proof}
  The tensor product defining $\GenMin^c$ is well defined thanks to 
  Proposition~\ref{prop:AdaptedTensorProducts}, and induction on $c$.

  Equation~\eqref{eq:MinDual} is verified by induction on $c$, and the basic case $c=1$ 
  is Lemma~\ref{lemma:MinDual}. For the inductive step, we compute:
  \begin{align*}
    ^{\Blg_2}\GenMin^c_{\Blg_1}\DT~^{\Blg_1,\Blg_1'}\CanonDD &\simeq
    ~^{\Blg_2}\GenMin^{c-1}_{\Blg_4} \DT 
    ~^{\Blg_4}\Pos^c_{\Blg_3}\DT
    \Big({}^{\Blg_3}\Pos^{c-1}_{\Blg_1}\DT
    ~^{\Blg_1,\Blg_1'}\CanonDD\Big) \\
    &\simeq 
    ~^{\Blg_2}\GenMin^{c-1}_{\Blg_4} \DT 
    ~^{\Blg_4}\Pos^c_{\Blg_3}\DT
    \Big({}^{\Blg_1'}\Pos^{c-1}_{\Blg_3'}\DT
    ~^{\Blg_3',\Blg_3}\CanonDD\Big) \\
    &\simeq 
    ~^{\Blg_2}\GenMin^{c-1}_{\Blg_4} \DT 
    ~\left(^{\Blg_1'}\Pos^{c-1}_{\Blg_3'}\DT
    ~\Big({}^{\Blg_4}\Pos^c_{\Blg_3}\DT
    ~^{\Blg_3,\Blg_3'}\CanonDD\Big)\right) \\
    &\simeq 
    ~^{\Blg_2}\GenMin^{c-1}_{\Blg_4} \DT 
    ~\left(^{\Blg_1'}\Pos^{c-1}_{\Blg_3'}\DT
    ~\Big({}^{\Blg_3'}\Pos^c_{\Blg_4'}\DT
    ~^{\Blg_4',\Blg_4}\CanonDD\Big)\right) \\
    &\simeq 
    ~^{\Blg_1'}\Pos^{c-1}_{\Blg_3'}\DT
   ~ ^{\Blg_3'}\Pos^c_{\Blg_4'}\DT~\Big({}^{\Blg_2}\GenMin^{c-1}_{\Blg_4} \DT
    ~^{\Blg_4,\Blg_4'}\CanonDD\Big) \\
    &\simeq 
    ~^{\Blg_1'}\Pos^{c-1}_{\Blg_3'}\DT
   ~ ^{\Blg_3'}\Pos^c_{\Blg_4'}\DT~^{\Blg_4',\Blg_2}\Crit_{c-1} \\
    &\simeq 
    ~^{\Blg_1'}\Pos^{c-1}_{\Blg_3'}\DT
   ~ ^{\Blg_3'}\Neg^{c-1}_{\Blg_1'}\DT~^{\Blg_1',\Blg_2}\Crit_{c} \\
   &\simeq ~^{\Blg_1',\Blg_2}\Crit_{c},
  \end{align*}
  using associativity of $\DT$ (Lemmas~\ref{lem:AssociateDA} and~\ref{lem:AssociateDD}), Lemma~\ref{lem:CommuteDDBraid},
  the trident relation (Lemma~\ref{lem:BasicTrident}), the inductive hypothesis,
  and the fact that $\Pos$ and $\Neg$ are inverses (Equation~\eqref{eq:InvertPos}).
  Lemma~\ref{lem:AssociateDD} applies, since the bimodules 
  $\Pos^{k}$ for any $k$ have $\delta^1_j=0$ for all $j>3$.
  The steps (skipping associativity) are illustrated in Figure~\ref{fig:TridentEquations}.
\begin{figure}[ht]
\input{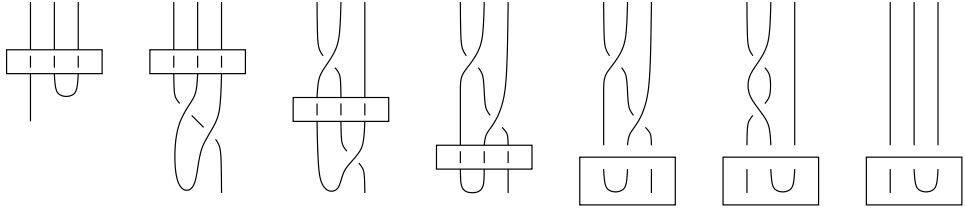}
\caption{\label{fig:TridentEquations} 
  Pictures of bimodules appearing in the  inductive step of Proposition~\ref{prop:ConstructMinimum}.
  Boxed components correspond to type $DD$ bimodules, the rest correspond to type $DA$ bimodules.}
\end{figure}
  \end{proof}

\subsection{The terminal Type $A$ module}
\label{subsec:TerminalMinimum}

We have described the $DA$ bimodule associated to a local (but not
global) minimum. For the global minimum, we give the following different
construction. 
Note that the algebra immediately above the global minimum
has $\Blg=\Blg(2,1,\{2\})$ or $\Blg(2,1,\{1\})$, depending on the orientation on the
global minimum (left to right or right to left).
Consider first the case where $\Blg=\Blg(2,1,\{2\})$.
Since the outgoing algebra has no strands (and so it can be thought of as $\Field$), 
the we construct a type $A$ module over the incoming algebra.

As a vector space, $\TerMin_{\Blg(2,1,\{2\})}$ is a two-dimensional, with generators $\XX$ and $\YY$.
The right $\IdempRing(\Blg(2,1,\{2\}))$-module structure is determined by
\begin{equation}
  \label{eq:TerMinIdemp}
  \XX =\XX\cdot \Idemp{\{1\}}~\qquad \YY=\YY\cdot \Idemp{\{0\}}.
\end{equation}
The right module structure is determined by the formulas:
\begin{equation}
  \label{eq:ModuleStructure}
\XX \cdot L_1 = \YY\qquad \YY\cdot R_1 = \XX \qquad \XX\cdot C_2=\YY\cdot C_2=0.
\end{equation}
The idempotent relations imply that $\XX\cdot R_2=0$. It follows that
\[
\XX \cdot L_1 U_1^t=\YY \qquad\YY \cdot  R_1 U_1^t=\XX \qquad 
\XX \cdot U_1^t = \XX \qquad \YY\cdot U_1^t = \YY,\]
for all $t\geq 0$; and $U_2$ acts trivially.

Defining
$\MasGr(\XX)=\MasGr(\YY)=0$,
it follows that $\TerMin$ has a $\Z$-valued Maslov grading.

Unlike the modules encountered thus far, $\TerMin$ is not graded by Alexander set, which in this case is $\OneHalf \Z$.
However, it is filtered by it, in the sense that if $a$ is a homogeneous element with $m_2(X,a)\neq 0$, then 
$w_2(a)-w_1(a)\leq 0$.

An analogous construction works for $\Blg=\Alg(1,\{1\})$. In that case,
$\XX =\XX\cdot \Idemp{\{1\}}$, $\YY=\YY\cdot \Idemp{\{2\}}$; 
and 
actions are determined by 
$\XX \cdot R_2 = \YY$,
$\YY\cdot L_2 = \XX$,
$\XX\cdot C_1=\YY\cdot C_1=0$.
Again, this module is filtered.

Both versions of $\TerMin$ have an associated graded object, with two generators $\XX$ and $\YY$ satisfying
Equation~\eqref{eq:TerMinIdemp}. The actions by $L_1$ $L_2$, $R_1$, $R_2$, $U_1$, $U_2$ are all $0$;
the action by $C_1$ or $C_2$ (whichever is in the algbera) is also $0$.

\begin{prop}
  \label{prop:FilteredTensorProduct}
  Let $Y$ be a type $D$ structure over $\Blg=\Blg(2,1,\{1\})$ or $\Blg(2,1,\{2\})$, with equipped with
  a $\Z$-valued Maslov grading and a $\OneHalf\Z$-valued Alexander grading.
  Then, the tensor product $\TerMin\DT Y$ is naturally $\Z$-graded (by the Maslov grading)
  and it is filtered by the Alexander grading. Its associated graded object coincides with
  the tensor product $\TerMina\DT Y$.
\end{prop}

\begin{proof}
  The sums in the tensor product are finite since $\TerMin$ is bounded. 
  The Maslov gradings on $Y$ and $\TerMin$ induce Maslov gradings on the tensor product as usual,
  and the  Alexander grading on $Y$ induces an Alexander function on $\TerMin\DT Y$.
  Since $\TerMina$ is the associated graded object for $\TerMin$, it follows that the
  associated graded object for $\TerMin\DT Y$ is $\TerMina\DT Y$.
\end{proof}

\newcommand\Mirr{\mathcal M}
\section{Symmetries}
\label{sec:Symmetries}

We note some symmetries in the algebras and bimodules described above.

One symmetry, which can be visualized as rotation through a vertical
axis, is induced by an isomorphism between the algebras.  That is,
consider the isomorphism $\VRot\colon \Blg(m,k,\Upwards)\to
\Blg(m,k,\rho(\Upwards))$ defined in Equation~\eqref{eq:DefVRot}.  As
in Example~\ref{ex:MorphismBimodule}, this map gives a $DA$ bimodule,
denoted $[\VRot]=\lsup{\Blg(m,k,\rho(\Upwards))}[\VRot]_{\Blg(m,k,\Upwards)}$.

The bimodules for crossings and critical points are symmetric under
this vertical rotation, according to the following:

\begin{lemma}
\label{lem:Rotate}
Fix integers $c$, $k$, and $m$, with $1\leq c \leq m+1$ and $0\leq k\leq m+1$; and fix
$\Upwards_1\subset \{1,\dots,m\}$, $\Upwards_2\subset\{1,\dots,m+2\}$ so that
$\Upwards_2=\phi_c(\Upwards_1)\cup\{c\}$ or $\Upwards_2=\phi_c(\Upwards_1)\cup\{c+1\}$
(with $\phi_c$ as in Equation~\eqref{eq:DefInsert}).
Let $\Blg_1=\Blg(m,k,\Upwards_1)$, $\Blg_2=\Blg(m,k,\Upwards_2)$,
$\Blg_1'=\Blg(m,k,\rho'_m(\Upwards_1))$, and
$\Blg_2'=\Blg(m,k,\rho'_{m+2}(\Upwards_2))$
(using notation from Section~\ref{subsec:Symmetry}). 
The following identities hold:
\[
  \lsup{{\Blg_2'}}[\VRot]_{\Blg_2}\DT \lsup{\Blg_2}\Max^c_{\Blg_1} \simeq \lsup{{\Blg'}}\Max^{m-c}_{\Blg_1'}\DT 
  \lsup{{\Blg_1'}}[\VRot]_{\Blg_1} \qquad
  \lsup{{\Blg_1'}}[\VRot]_{\Blg_1}\DT \lsup{\Blg_1}\Min^c_{\Blg_2} \simeq 
  \lsup{{\Blg_1'}}\Min^{m-c}_{\Blg_2'}\DT \lsup{{\Blg_2'}}[\VRot]_{\Blg'}.
\]
Also, for any $i=1,\dots,m-1$, let $\Blg_3=\Blg(m,k,\tau(\Upwards_1))$
and $\Blg_3'=\Blg(m,k,\rho'_m(\tau(\Upwards_1))$ we have identities
\[  \lsup{\Blg_3'}[\VRot]_{\Blg_3}\DT \lsup{\Blg_3}\Pos^i_{\Blg_1} \simeq 
  \lsup{\Blg_3'}\Pos^{m-i}_{\Blg_1'}\DT \lsup{{\Blg_1'}}[\VRot]_{\Blg_1} \qquad
  \lsup{\Blg_3'}[\VRot]_{\Blg_3}\DT \lsup{\Blg_3}\Neg^i_{\Blg_1} \simeq 
  \lsup{\Blg_3'}\Neg^{m-i}_{\Blg_1'}\DT \lsup{{\Blg_1'}}[\VRot]_{\Blg_1} \]
\end{lemma}

\begin{proof}
  It is easy to see that $[\VRot]\DT\CanonDD\cong \CanonDD\DT[\VRot]$.

  The above results then follow from the invertibility of the $DD$ bimodule, and the
  easily verified identities
  $[\VRot]\DT\Crit_c \cong \Crit_{m+2-c}\DT[\VRot]$
  and 
  $[\VRot]\DT\Pos_i\cong \Pos_{m-i}\DT[\VRot]$,
  where the latter isomorphism preserves 
  $\North$ and $\South$ and switches $\East$ and $\West$.
\end{proof}

Recall that there is an algebra isomorphism 
$\Blg(m,k,\Upwards)\cong\Blg(m,k,\Upwards)^{\op}$
(Equation~\eqref{eq:OppositeIsomorphism}).
An arbitrary $DA$ bimodule $\lsup{\Blg_2}X_{\Blg_1}$ has an {\em opposite module} ${\overline X}$, of the form
$\lsub{\Blg_1}X^{\Blg_2}\cong \lsup{\Blg_2^{\op}}{\overline X}_{\Blg_1^{\op}}$. 
If $\Blg_i=\Blg(m_i,k_i,\Upwards_i)$ for $i=1,2$, we can view the opposite module $X^{\op}$ as a bimodule
 over the same two algebras,
$X^{\opp}=[\Opposite]\DT {\overline X}\DT [\Opposite]$.

% If $\lsup{\Alg_2}X_{\Alg_1}$ is a bimodule, then its {\em opposite module}
% $X^{\op}$ is the bimodule $X^{\op}$ is a type $DA$ bimodule
% $\lsub{\Alg_1}X^{\Alg_2}$. This could alternatively be viewed
% as a type $DA$ bimodule
% $\lsup{\Alg_1^{\op}}X^{\op}_{\Alg_2^{\op}}$.
% 
% 
% \begin{lemma}
%   \label{lem:OppositeDT}
% Let $\Alg_1$, $\Alg_2$, and $\Alg_3$ be three $DG$ modules, and let
% $\lsup{\Alg_1}X_{\Alg_2}$ and $\lsup{\Alg_2}Y_{\Alg_3}$ be two type $DA$ bimodules.
% Then, there is an isomorphism
% \[ (\lsup{\Alg_1}X_{\Alg_2}\DT \lsup{\Alg_2}Y_{\Alg_3})^{\op}
% \cong
% \lsub{\Alg_3}Y_{\op}^{\Alg_2}\DT
% \lsub{\Alg_2}X_{\op}^{\Alg_1}\]
% \end{lemma}

\begin{prop}
  \label{lem:OppositeBimodules}
  Under the identification $\Blg(m,k,\Upwards)\cong\Blg(m,k,\Upwards)^{\op}$ from Equation~\eqref{eq:OppositeIsomorphism},
  there are identifications of bimodules
  $\Max_c^{\op} \simeq \Max_{c}$,
  $\Min_c^{\op} \simeq \Min_c$,
  $(\Pos^i)^{\op} \simeq \Neg_i$, and
  $(\Neg^i)^{\op} \simeq \Pos_i$.
\end{prop}

\begin{proof}
  We claim that $\Crit_c^{\op}\simeq \Crit_c$.  The identification
  switches the roles of $\XX$ and $\YY$, and fixes $\ZZ$.  For each
  pair of generators in Equation~\eqref{eq:CritDiag}, there are two
  arrows, and the symmetry switches those two arrows; observe that if
  one arrow is labelled by $a\otimes b$, then the other is labelled by
  $\Opposite(a)\otimes \Opposite(b)$.
  
  The identity $(\Pos_i)^{\op}\simeq \Neg_i$ follows from the definition 
  of  $\Neg_i$ given
  in Section~\ref{def:NegativeCrossing}.
\end{proof}

\newcommand\orK{\vec{K}}
\section{Construction and invariance of the invariant}
\label{sec:ConstructionAndInvariance}

\subsection{Topological preliminaries}

\begin{defn}
  A {\em pointed knot diagram}
  is a projection of an oriented knot diagram in $S^2$,
  with a marked point on it.
  A {\em pointed Reidemeister move} is a Reidemeister move supported
  in a complement of the marked point.
\end{defn}

The following result is well known; see~\cite[Section~3]{KhovanovPatterns}.

\begin{prop}
  \label{prop:ConnectPointedKnotDiagrams}
  Any two pointed knot diagrams for the same knot can be connected by a sequence of isotopies (in $S^2$) and pointed Reidemeister moves (in $S^2$).
\end{prop}

\begin{proof}
  By  Reidemeister's theorem~\cite{Reidemeister}, any two
  knot diagrams for the same knot can be connected by a sequence of
  Reidemeister moves. To see that we can arrange to make those moves
  disjoint from the marked point $p$, it suffices to show
  that  $p$ can be
  moved through a crossing by a sequence of pointed Reidemeister
  moves. Suppose we are attempting to move a
  crossing across $p$, situated immediately south of the crossing.  
  Viewing the knot
  projection as supported in $S^2$, we can equivalently move the
  strand north across the knot diagram, bringing it back north 
  through the point at infinity, as shown in
  Figure~\ref{fig:MoveBasepoint}. This move can be realized as a
  sequence of Reidemeister $2$ and $3$ moves disjoint from $p$.
\begin{figure}[ht]
\input{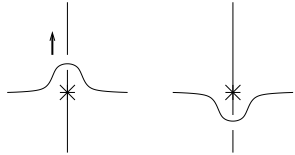}
\caption{\label{fig:MoveBasepoint} {Moving an overpass across the basepoint
      is equivalent to a sequence of Reidemeister moves away from the basepoint.}}
\end{figure}
\end{proof}

Consider a pointed knot diagram in the plane, and consider the projection
to the $y$ axis.
We say that the diagram is in {\em bridge position} if the following properties hold:
\begin{itemize}
  \item All critical points are either minima or maxima.
  \item The minima, maxima, and crossings project to distinct 
    points in the $y$ axis.
  \item The global minimum is the marked point.
\end{itemize}

There are two kinds of diagrams in bridge position, depending on the
orientation of the global minimum; i.e. 
left to right (LR) or  right to left (RL).

\begin{defn}
  A {\em bridge move} connecting two diagrams ${\mathcal D}_1$ and ${\mathcal D}_2$ in bridge position is any of the following types of moves:
  \begin{itemize}
    \item Creation of a single local maximum and local minimum; or the cancellation of such a pair.
      This is called a {\em pair creation} or {\em pair annihilation}.
    \item Sliding a maximum or non-global minimum through a crossing, called a {\em trident move}
    \item Commutations between distant pairs of  special points (each of which is a maxmimum, non-global minimum, or a crossing), called a {\em  commutation move}.
    \end{itemize}
\end{defn}

\begin{figure}[ht]
\input{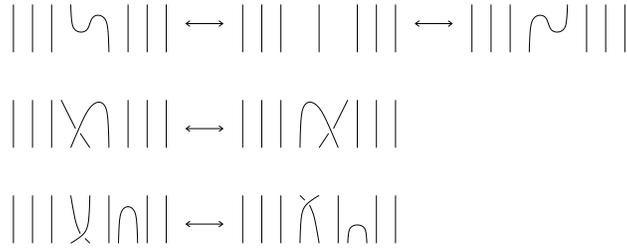}
\caption{\label{fig:PlanarMoves} {\bf{Bridge moves.}}
  Top row: annihilation and creation of a pair of local
  minimum and a maximum; second row: a trident move.
  (The three other trident moves are obtained by changing the two crossings, 
  reflecting through a horizontal line, or both.)
  Third row: commutations of a maximum and a distant crossing.}
\end{figure}

In the above description, ``distant'' refers to special points not on
the same strand, which are covered by the other two kinds of moves.
All of the above moves are pointed isotopies; they leave the
global minimum in place.

If  ${\mathcal D}$ is in bridge position,
there is a pointed diagram ${\mathcal D}'$
in the sphere, by adding a  point at infinity adjacent to the
marked point. We say that ${\mathcal D}$ {\em induces} ${\mathcal D'}$.

\begin{lemma}
  \label{lem:BraidMoves}
  Any pointed knot diagram in the sphere is induced from a planar knot
  diagram in bridge position.  Moreover if two planar knot diagrams
  ${\mathcal D}_1$ and ${\mathcal D}_2$ of the same type (i.e. $LR$ or
  $RL$) induce (planar) isotopic knot diagrams in the sphere, then they can be
  connected by a sequence of bridge moves as above.
\end{lemma}

\begin{proof}
  Consider the basepoint on the knot projection, and place a second
  point $p\in S^2$ so that the point on the knot with minimal distance
  to $p$ is the basepoint.  There is a choice here: the orientation on
  the knot distiguishes the two  sides where $p$ can be
  placed; e.g. thinking of the knot locally as the $x$ axis, oriented
  from left to right, $p$ could either be slightly above or
  slightly below the $x$ axis.  Let $a$ be a minimal arc from $p$ to
  the basepoint in $K$.  Stereographic projection from $p$ gives a
  planar knot diagram, which we rotate so that $a$ is taken to the
  portion of the $y$ axis with $y\leq t$.  If $p$ was slightly below
  resp. above the distinguished point, the resulting diagram is of
  type $LR$ resp. $RL$.  Perturbing this map slightly (e.g. by moving
  $p$ slightly) we can arrange that the resulting diagram is in bridge
  position.
  
  Performing this construction in a one-parameter family, we get a
  one-parameter family of planar diagrams, for which the global
  minimum is the marked point. In a generic  one-parameter family,
  the following  can occur: two critical points create or
  cancel, a crossing crosses a critical point, or any two
  distant special points project to the same $y$ coordinate.  Crossing
  this codimension one locus, the diagram undergoes a pair creation
  or annihilation, a trident move, or a commutation move.
\end{proof}

\subsection{Constructing the invariant}

Choose a knot planar knot diagram ${\mathcal D}$ in bridge position, and
slice it up  $t_1<\dots<t_k$ so that the following conditions hold:
\begin{itemize}
\item for $i=1,\dots,k-1$, the interval $[t_i,t_{i+1}]$ contains the
  projection onto the $y$ axis of exactly one crossing or critical
  point
\item for $i=1,\dots,k$, $t_i$ is not the projection of any crossing
  or critical point 
\item there are no crossings or critical points whose $y$ value is
  greater than $t_k$ (and so $[t_{k-1},t_k]$ contains the global
  maximum)
\item there are no crossings below $t_1$, 
  and the only critical point whose $y$ value is smaller than $t_1$
  is the global minium.
\end{itemize}
  
For $i=1,\dots,k-1$, we have seen how to associate a type $DA$
bimodule to the portion of the diagram that projects into
$[t_i,t_{i+1}]$; it is either of the form $\Pos^j$, $\Neg^j$, $\Max^c$
or $\Min^c$.  The top portion $[t_{k-1},t_k]$ contains the global
maximum, where the incoming algebra has no strands, so the $DA$
bimodule is, in fact, simply a type $D$ module (see
Section~\ref{subsec:SmallMaximum}). The bottom portion $[-\infty,t_1]$
has only the global minimum, so the outgoing algebra has no strands;
there we attach the terminal type $A$ module $\TerMin$ from
Section~\ref{subsec:TerminalMinimum}.  Note that for each partial knot
diagram, the assoicated bimodule is adapted to the one-manifold
underlying the partial knot diagram; and the hypothesis that our
diagram is a indeed one for a knot ensures that the tensor products
can always be taken; see Proposition~\ref{prop:AdaptedTensorProducts}.
It is also worth noting that the generating sets for all the partial knot diagrams
correspond to partial Kauffman states; this property is evidently preserved by $\DT$.

The chain complex $\KC(\Diag)$ is now obtained as an iterated tensor
product (via $\DT$) of these pieces. It inherits a filtration, whose associated graded object,
$\KCa(\Diag)$, can be computed by exchanging $\TerMin$ with the simpler terminal module $\TerMina$.

\begin{thm}
  \label{thm:KnotInvariance}
  The filtered homotopy type of $\KC(\Diag)$ is an oriented knot invariant.
\end{thm}

The quasi-isomorphism type of the chain complex $\KC(\Diag)$ does not
depend on the order in which the type $DA$ bimodules are tensored
together (Lemma~\ref{lem:AssociateDA}). Thus, to check the invariance of
$\KC(\Diag)$, it suffices to check identities between $DA$ bimodules
associated to the bridge moves, which we do in the next section. The proof of Theorem~\ref{thm:KnotInvariance}
is then given in Subsection~\ref{sec:InvarianceProof}.

\subsection{Invariance under bridge moves}
\label{subsec:InvarainceUnderBridgeMoves}

We order bridge moves as follows:
\begin{enumerate}
\item Commutations of distant crossings
\item Trident moves
\item Critical points commute with distant crossings
\item Commuting distant critical points
\item Pair creation and annihilation.
\end{enumerate}
We verify the invariance of our invariant under bridge moves in the
above order. 

Bridge moves involve the interactions of two consecutive pieces in the
chopped up knot diagram, The above procedure associates bimodules to
each of those consecutive pieces which we tensor together, and the
bridge moves are verified by verifying identities between the bimodule
associated to the two bimodules tensored together and the tensor of
the two bimodules after the bridge move (except in pair creation or
annihilation, which states a relation between the tensor of two
bimodules and a third bimodule, the identity bimodue).
Thus, for example, commutations of distant positive crossings asserts the relation
$\Pos^i\DT \Pos^j\sim \Pos^j\DT \Pos^i$
when $|i-j|>1$, which was already verified in the
verification of the braid relations (Theorem~\ref{thm:BraidRelation}). 

Consider next trident moves. There are four types of trident
moves: the one illustrated in the second row of
Figure~\ref{fig:PlanarMoves}, the one obtained by changing the
crossings in the picture, the one obtained by mirroring the picture
through a horizontal line, and the one obtained by changing the
crossings in the horizontally reflected picture.  

\begin{lemma}
  \label{lem:TridentMoves}
  The $DA$ bimodules associated to the two pictures before and after a trident move
  are quasi-isomorphic; i.e. the following four identities hold, corresponding to the four kinds of trident moves:
  \begin{align}
    \label{eq:Trident1}
    \lsup{\Blg_3}{\Pos^{c}_{\Blg_4}}\DT~\lsup{\Blg_4}{\Max^{c+1}_{\Blg_1}}
    &\simeq~\lsup{\Blg_3}{\Neg^{c+1}_{\Blg_2}}\DT~\lsup{\Blg_2}\Max^{c}_{\Blg_1} \\
    \lsup{\Blg_3}{\Pos^{c+1}_{\Blg_2}}\DT~\lsup{\Blg_2}\Max^{c}_{\Blg_1}
    &\simeq~\lsup{\Blg_3}{\Neg^{c}_{\Blg_4}}\DT~\lsup{\Blg_4}\Max^{c+1}_{\Blg_1} 
    \label{eq:Trident2} \\
    \lsup{\Blg_3}{\Min^{c+1}_{\Blg_2}}\DT \lsup{\Blg_2}{\Pos^{c}_{\Blg_1}}
    &\simeq\lsup{\Blg_3}\Min^{c}_{\Blg_4}\DT~\lsup{\Blg_4}\Neg^{c+1}_{\Blg_1}
    \label{eq:Trident3}  \\
    \lsup{\Blg_3}{\Min^{c}_{\Blg_2}}\DT \lsup{\Blg_2}\Pos^{c+1}_{\Blg_1}
    &\simeq~\lsup{\Blg_3}{\Min^{c+1}_{\Blg_4}}\DT~\lsup{\Blg_4}{\Neg^{c}_{\Blg_1}}
    \label{eq:Trident4}
  \end{align}
\end{lemma}

\begin{proof}
  All four identities follow from
  Lemma~\ref{lem:BasicTrident}, together with symmetries of the
  bimodules, as follows.
  
  Since the $\CanonDD$ is invertible
  (Theorem~\ref{thm:DDisInvertible}), 
  Equation~\eqref{eq:Trident1}
  follows from 
  $\Pos^{c}\DT \Max^{c+1}\DT \CanonDD
  \simeq 
  \Neg^{c+1}\DT \Max^{c}\DT\CanonDD$;
  which follows from  Proposition~\ref{prop:MaxDual} and Lemma~\ref{lem:BasicTrident}.

  Applying horizontal rotation to Equation~\eqref{eq:Trident1}
  (with $m-c-1$ in place of $c$), together with
  Lemma~\ref{lem:Rotate}, Equation~\eqref{eq:Trident2} follows from Equation~\eqref{eq:Trident1}

  Similarly, 
  observe that
  \begin{align*}
    \lsup{\Blg_3}{\Min^{c+1}_{\Blg_2}}\DT~\lsup{\Blg_2}{\Pos^{c}_{\Blg_1}}
    \DT \lsup{\Blg_1,\Blg_1'}{\CanonDD}
        &\simeq\lsup{\Blg_3}{\Min^{c+1}}_{\Blg_2}\DT 
                    (\lsup{{\Blg_1'}}\Pos^{c}_{\Blg_2'}\DT 
                        \lsup{{\Blg_2',\Blg_2}}\CanonDD) \\
        &\simeq~\lsup{{\Blg_1'}}\Pos^{c}_{\Blg_2'}
        \DT
        (\lsup{\Blg_3}\Min^{c+1}_{\Blg_2} \DT \lsup{{\Blg_2,\Blg_2'}}\CanonDD)  \\
        &\simeq~\lsup{{\Blg_1'}}\Pos^{c}_{\Blg_2'} 
        \DT~ \lsup{{\Blg_2',\Blg_3}}\Crit_{c+1}
        \end{align*}
    (using Lemma~\ref{lem:CommuteDDBraid} twice, 
    Lemma~\ref{lem:AssociateDD}, 
    Proposition~\ref{prop:MinDual});
    and by the same logic,
    \[
    \lsup{\Blg_3}\Min^{c}_{\Blg_4}\DT~^{\Blg_4}\Neg^{c+1}_{\Blg_1}\DT~^{\Blg_1,\Blg_1'}\CanonDD
    \simeq~^{\Blg_1'}\Neg^{c+1}_{\Blg_4'}\DT~^{\Blg_4',\Blg_3}\Crit_{c}.\]
    Thus, Equation~\eqref{eq:Trident3} follows from Lemma~\ref{lem:BasicTrident}.

    Equation~\eqref{eq:Trident4} follows from Equation~\eqref{eq:Trident3}
    by horizontal rotation.
\end{proof}

\begin{lemma}
  \label{lem:CommuteMaxPos}
  The $DA$ bimodules for positive crossings commute with those for local maxima and minima, in the sense that
  if $|c+1-i|>0$, then 
  (with $\phi_c$ as in Equation~\eqref{eq:DefInsert})
  \begin{align}
    \label{eq:CommuteMaxCross}
    \Pos^{\phi_c(i)} \DT \Max^c&\simeq 
    \Max^c\DT \Pos^i \\
    \label{eq:CommuteMinCross}
    \Pos^i \DT \Min^c&\simeq 
    \Min^c\DT \Pos^{\phi_c(i)}
  \end{align}
\end{lemma}

\begin{proof}
  Direct computation, as in the proof of Lemma~\ref{lem:FarBraids},
  shows that
  \begin{equation}
    \label{eq:DistantCommuteCritCross}
    \Pos^{\phi_c(i)}\DT \Crit_c \simeq \Max^c\DT \Pos_i.
  \end{equation}
  Both stated equations now follow from
  Equation~\eqref{eq:DistantCommuteCritCross}, the invertibility of
  $\CanonDD$, and
  Propositions~\ref{prop:MaxDual} and~\ref{prop:MinDual}; compare the
  proof of Lemma~\ref{lem:TridentMoves}.
\end{proof}

\begin{lemma}
  \label{lem:CommuteCrossCritPoints}
  The $DA$ bimodules for crossings commute with those for distant critical points.
\end{lemma}

\begin{proof}
  Lemma~\ref{lem:CommuteMaxPos} handles positive crossings.
  Multiplying Equation~\eqref{eq:CommuteMaxCross} on both
  sides by $\Neg^i$ and using the fact that $\Neg^i$ and $\Pos^i$ are
  inverses of one another (Equation~\eqref{eq:InvertPos}), it follows
  that negative crossings also commute with local maxima; Equation~\eqref{eq:CommuteMinCross} shows that
  negative crossings commute with local minima.
\end{proof}

\begin{lemma}
  \label{lem:CommuteCritCrit}
  The $DA$ bimodules for pairs of distant critical points commute.
\end{lemma}

\begin{proof}
  Fix $i<j$. The lemma states identifications
\[\begin{array}{ll}
    \Max^{i}\DT\Max^{j-1}\simeq \Max^{j+1}\DT \Max^{i} &
    \Min^{i}\DT\Min^{j+1}\simeq \Min^{j-1}\DT \Min^{i} \\
    \Max^{j-1}\DT\Min^{i}\simeq \Min^{i}\DT \Max^{j+1} &
    \Min^{j+1}\DT\Max^{i}\simeq \Max^{i}\DT \Min^{j-1}.
    \end{array}\]
  To verify the first identity, we tensor on the right with $\CanonDD$,
  to reduce to the identity between two type $DD$ bimodules
  (using Proposition~\ref{prop:MaxDual}):
  \begin{equation}
    \label{eq:FarCrits}
      \Omega^i\DT \Crit_{j-1}\simeq \Max^{j+1}\DT \Crit_i,
  \end{equation}
  which is easy to verify in the spirit of the proof of 
  Lemma~\ref{lem:FarBraids}.
  Tensoring on the right with $\CanonDD$, the second on the right,
  \begin{align*}
    \lsup{\Blg_3}\Min^i_{\Blg_2}\DT \lsup{\Blg_2}\Min^{j+1}_{\Blg_1}
    \DT\lsup{\Blg_1,\Blg_1'}\CanonDD 
    &\simeq 
    \lsup{\Blg_3}\Min^i_{\Blg_2}\DT \lsup{\Blg_2,\Blg_1'}\Crit_{j+1} \\
    &\simeq
    \lsup{\Blg_3}\Min^i_{\Blg_2}\DT \left(\lsup{\Blg_1'}\Max^{j+1}_{\Blg_2'} \DT 
    \lsup{\Blg_2,\Blg_2'}\CanonDD\right) \\
  &\simeq 
    \lsup{\Blg_1'}\Max^{j+1}_{\Blg_2'}\DT \left(\lsup{\Blg_3}\Min^i_{\Blg_2'} \DT 
    \lsup{\Blg_2,\Blg_2'}\CanonDD\right) \\
  &\simeq 
    \lsup{\Blg_1'}\Max^{j+1}_{\Blg_2'}\DT \lsup{\Blg_2',\Blg_3}\Crit_i 
  \end{align*}
  Similarly,
  $\Min^{j-1}\DT \Min^i\DT \CanonDD\simeq \Max^{i}\DT \Crit_{j-1}$,
  so the second identity also follows from Equation~\eqref{eq:FarCrits}.

  Consider the third identity, now with $i=1$. Tensoring on the right
  with $\CanonDD$, we reduce to the easily verified identity
  $\Max^{j-1}\DT \Crit_1 \simeq \Min^1\DT \Crit_{j+1}$.
  (Compare Equation~\eqref{eq:FarCrits}.) 
  The cases where $i>1$ now follow from the case where $i=1$,
  the inductive definition of $\GenMin^i$ (Equation~\eqref{eq:GenMinDef}),
  and Lemma~\ref{lem:CommuteMaxPos}.

  The fourth identity follows from the third and Lemma~\ref{lem:Rotate}.
\end{proof}

\begin{lemma}
  \label{lem:PairCreation}
  The $DA$ bimodules are invariant under pair creation and
  annihilations, in the sense that $\Min^{c}\DT \Max^{c+1}\simeq
  \Id\simeq \Min^{c+1}\DT\Max^{c}$
\end{lemma}

\begin{proof}
  Rotating through a vertical axis, i.e. using Lemma~\ref{lem:Rotate},
  we see that $\Min^{c}\DT \Max^{c+1}\simeq \Id$ implies also that
  $\Id\simeq \Min^{c+1}\DT\Max^{c}$.  The verification of $\Min^{c}\DT
  \Max^{c+1}\simeq \Id$ can be reduced to the case where $c=1$ using
  Reidemeister $2$ moves, as follows. First, introduce a sequence of
  $c-1$ positive crossings that carry the $c^{th}$ strand to the far
  left, and let $P$ be the corresponding bimodule; then let $N$ be its
  inverse, i.e. $N\DT P \simeq \Id$. (Clearly, $N$ is obtained from
  $P$ by reversing all the crossings and taking them in reverse
  order.)  (cf. Equation~\eqref{eq:InvertPos}).  In view of
  Lemma~\ref{lem:TridentMoves}, a sequence of trident moves
  identifies $P\DT \Min^{c}\DT \Max^{c+1} \simeq \Min^{1}\DT
  \Max^{2}\DT P$. 
  Thus, if $\Min^{1}\DT\Max^{2}\simeq\Id$, we can conclude that 
  $\Min^{c}\DT \Max^{c+1}\simeq N\DT P\DT \Min^{c}\DT \Max^{c+1} \simeq N\DT \Min^{1}\DT
  \Max^{2}\DT P\simeq N\DT P \simeq \Id$.

  It remains now to verify that  $\Min^{1}\DT \Max^{2}\simeq \Id$.
  We claim that the generators of
  $\lsup{\Blg_1}\Min^{1}_{\Blg_2}\DT^{\Blg_2}\Max^{2}_{\Blg_1}$
  correspond to the idempotents in the algebra $\Blg_1$; see
  Figure~\ref{fig:CapCup}.
  \begin{figure}[ht]
   \input{CapCup.pstex_t}
    \caption{\label{fig:CapCup} 
      {\bf{Generators of ${}^{\Blg_1}\Min^{1}_{\Blg_2}\DT^{\Blg_2}{\Max^{2}_{\Blg_1}}$.}}
      Generators correspond to the idempotents in the incoming algebra
      $\Blg_1$; the constraints on the intermediate generators are as
      shown.}
  \end{figure}
  Specifically, $\Max^2_\Blg$, the generator is constrained to be
  either of type $\YY$ or $\ZZ$; while in $\Min^{1}$,
  the generator is either of type $\XX$ or $\ZZ$. Thus, in the tensor 
  product, we divide the generators into four types
  $\XX \DT \YY$, $\XX \DT \ZZ$, $\ZZ \DT\YY$, and $\ZZ \DT \ZZ$.
  \begin{figure}[ht]
    \input{CapCupGen.pstex_t}
    \caption{\label{fig:CapCupGen} {\bf{Generator types for
          ${}^{\Blg_1}\Min^{1}_{\Blg_2}\DT^{\Blg_2}\Max^{2}_{\Blg_1}$.}}}
  \end{figure}
  We claim that
  \begin{equation}
    \label{eq:SomewhatId}
    \begin{array}{ll}
    \delta^1_2(\XX\DT\YY,L_1)=L_1\otimes (\ZZ\DT \ZZ) &
  \delta^1_2(\ZZ\DT\ZZ,R_1)=R_1\otimes (\XX\DT\YY),
  \end{array}
  \end{equation}
  as can be seen from the following diagrams:
\[
\mathcenter{
  \begin{tikzpicture}[scale=.7]
    \node at (0,1) (cin) {$\Max^{2}_{\Blg_1}$}; 
    \node at (0,-1) (d1) {$\delta^1_1$}; 
    \node at (0,-2.5) (d2) {$\delta^1_2$};
    \node at (2.5,1) (ain) {$\Blg_1$};
    \node at (-2,1) (Ain) {$\Min^{1}_{\Blg_2}$};
    \node at (-2,-4) (Aout) {};
    \node at (-2,-3) (muA) {$\delta^1_3$}; 
    \node at (-3.5,-4) (aout) {}; 
    \node at (0,-4) (cout) {};
    \draw[algarrow,bend left=10] (ain) to node[above,sloped]{\tiny{$L_1$}}(d2) ;
    \draw[modarrow] (Ain) to node[left]{\tiny{$\XX$}}(muA) ;
    \draw[modarrow] (muA) to node[left]{\tiny{$\ZZ$}}(Aout) ;
    \draw[modarrow] (cin) to node[right]{\tiny{$\YY$}}(d1); 
    \draw[modarrow] (d1) to node[right]{\tiny{$\XX$}}(d2);
    \draw[modarrow] (d2) to node[right]{\tiny{$\ZZ$}}(cout);
    \draw[algarrow] (d1) to node[above,sloped] {\tiny{$L_2 L_3$}}(muA); 
    \draw[algarrow] (d2) to node[above,sloped,pos=.3] {\tiny{$L_1$}} (muA); 
    \draw[algarrow] (muA) to node[above,sloped]{\tiny{$L_1$}} (aout);
  \end{tikzpicture}
}
\mathcenter{
  \begin{tikzpicture}[scale=.7]
    \node at (0,1) (cin) {$\Max^{2}_{\Blg_1}$}; 
    \node at (0,-1) (d1) {$\delta^1_2$}; 
    \node at (0,-2.5) (d2) {$\delta^1_1$};
    \node at (2.5,1) (ain) {$\Blg_1$};
    \node at (-2,1) (Ain) {$\Min^{1}_{\Blg_2}$};
    \node at (-2,-4) (Aout) {};
    \node at (-2,-3) (muA) {$\delta^1_3$}; 
    \node at (-3.5,-4) (aout) {}; 
    \node at (0,-4) (cout) {};
    \draw[algarrow] (ain) to node[above,sloped]{\tiny{$R_1$}}(d1) ;
    \draw[modarrow] (Ain) to node[left]{\tiny{$\ZZ$}}(muA) ;
    \draw[modarrow] (muA) to node[left]{\tiny{$\XX$}}(Aout) ;
    \draw[modarrow] (cin) to node[right]{\tiny{$\ZZ$}}(d1); 
    \draw[modarrow] (d1) to node[right]{\tiny{$\XX$}}(d2);
    \draw[modarrow] (d2) to node[right]{\tiny{$\YY$}}(cout);
    \draw[algarrow] (d1) to node[above,sloped] {\tiny{$R_1$}}(muA); 
    \draw[algarrow,pos=.3] (d2) to node[above,sloped] {\tiny{$R_3 R_2$}} (muA); 
    \draw[algarrow] (muA) to node[above,sloped]{\tiny{$R_1$}} (aout);
  \end{tikzpicture}
} \]
If the leftmost strand is downwards, then furthermore 
\begin{equation}
  \label{eq:LeftDown}
  \delta^1_2({\mathbf p}\DT{\mathbf q},U_1)=U_1\otimes ({\mathbf p}\DT{\mathbf q})
\end{equation}
for all choices of ${\mathbf p}\in\{\XX,\ZZ\}$ and ${\mathbf q}\in\{\YY,\ZZ\}$:
when ${\mathbf p}\DT{\mathbf q}=\XX\DT\ZZ$, both sides vanish; for the other cases:
\[
\mathcenter{
  \begin{tikzpicture}[scale=.7]
    \node at (0,.5) (cin) {$\Max^{2}_{\Blg_1}$}; 
    \node at (0,-1) (d1) {$\delta^1_2$}; 
    \node at (0,-2.5) (d2) {$\delta^1_1$};
    \node at (2,.5) (ain) {$\Blg_1$};
    \node at (-2,.5) (Ain) {$\Min^{1}_{\Blg_2}$};
    \node at (-2,-4) (Aout) {};
    \node at (-2,-3) (muA) {$\delta^1_3$}; 
    \node at (-3,-4) (aout) {}; 
    \node at (0,-4) (cout) {};
    \draw[algarrow] (ain) to node[above,sloped]{\tiny{$U_1$}}(d1) ;
    \draw[modarrow] (Ain) to node[left]{\tiny{$\ZZ$}}(muA) ;
    \draw[modarrow] (muA) to node[left]{\tiny{$\ZZ$}}(Aout) ;
    \draw[modarrow] (cin) to node[right]{\tiny{$\ZZ$}}(d1); 
    \draw[modarrow] (d1) to node[right]{\tiny{$\ZZ$}}(d2);
    \draw[modarrow] (d2) to node[right]{\tiny{$\ZZ$}}(cout);
    \draw[algarrow] (d1) to node[above,sloped] {\tiny{$U_1$}}(muA); 
    \draw[algarrow,pos=.3] (d2) to node[above,sloped] {\tiny{$C_2U_3$}} (muA); 
    \draw[algarrow] (muA) to node[above,sloped]{\tiny{$U_1$}} (aout);
  \end{tikzpicture}}
\mathcenter{
  \begin{tikzpicture}[scale=.7]
    \node at (0,.5) (cin) {$\Max^{2}_{\Blg_1}$}; 
    \node at (0,-.6) (d1) {$\delta^1_1$}; 
    \node at (0,-1.6) (d2) {$\delta^1_2$};
    \node at (0,-2.6) (d3) {$\delta^1_1$};
    \node at (2,.5) (ain) {$\Blg_1$};
    \node at (-2,.5) (Ain) {$\Min^{1}_{\Blg_2}$};
    \node at (-2,-4) (Aout) {};
    \node at (-2,-2.8) (muA) {$\delta^1_4$}; 
    \node at (-3,-4) (aout) {}; 
    \node at (0,-4) (cout) {};
    \draw[algarrow,bend left=20] (ain) to node[above,sloped]{\tiny{$U_1$}}(d2) ;
    \draw[modarrow] (Ain) to node[left]{\tiny{$\XX$}}(muA) ;
    \draw[modarrow] (muA) to node[left]{\tiny{$\XX$}}(Aout) ;
    \draw[modarrow] (cin) to node[right]{\tiny{$\YY$}}(d1); 
    \draw[modarrow] (d1) to node[right]{\tiny{$\XX$}}(d2);
    \draw[modarrow] (d2) to node[right]{\tiny{$\XX$}}(d3);
    \draw[modarrow] (d3) to node[right]{\tiny{$\YY$}}(cout);
    \draw[algarrow,pos=.3] (d1) to node[above,sloped,pos=.4] {\tiny{$L_2 L_3$}}(muA); 
    \draw[algarrow] (d2) to node[above,sloped,pos=.1] {\tiny{$U_1$}} (muA); 
    \draw[algarrow] (d3) to node[above,sloped,pos=.1] {\tiny{$R_3R_2$}} (muA); 
    \draw[algarrow] (muA) to node[above,sloped]{\tiny{$U_1$}} (aout);
\end{tikzpicture}}
\mathcenter{
  \begin{tikzpicture}[scale=.7]
    \node at (0,.5) (cin) {$\Max^{2}_{\Blg_1}$}; 
    \node at (0,-1) (d1) {$\delta^1_2$}; 
    \node at (0,-2.5) (d2) {$\delta^1_1$};
    \node at (2,.5) (ain) {$\Blg_1$};
    \node at (-2,.5) (Ain) {$\Min^{1}_{\Blg_2}$};
    \node at (-2,-4) (Aout) {};
    \node at (-2,-3) (muA) {$\delta^1_3$}; 
    \node at (-3,-4) (aout) {}; 
    \node at (0,-4) (cout) {};
    \draw[algarrow] (ain) to node[above,sloped]{\tiny{$U_1$}}(d1) ;
    \draw[modarrow] (Ain) to node[left]{\tiny{$\ZZ$}}(muA) ;
    \draw[modarrow] (muA) to node[left]{\tiny{$\ZZ$}}(Aout) ;
    \draw[modarrow] (cin) to node[right]{\tiny{$\YY$}}(d1); 
    \draw[modarrow] (d1) to node[right]{\tiny{$\YY$}}(d2);
    \draw[modarrow] (d2) to node[right]{\tiny{$\YY$}}(cout);
    \draw[algarrow] (d1) to node[above,sloped] {\tiny{$U_1$}}(muA); 
    \draw[algarrow] (d2) to node[above,sloped,pos=.2] {\tiny{$C_2U_3$}} (muA); 
    \draw[algarrow] (muA) to node[above,sloped]{\tiny{$U_1$}} (aout);
  \end{tikzpicture}}\]
With these computations in hand, it follows that
$\Min^{1}\DT \Max^{2}\DT\CanonDD=\CanonDD$,  so 
$\Min^{1}\DT\Max^{2}\simeq \Id$.

If the leftmost strand is upwards, then  instead of 
Equation~\eqref{eq:LeftDown}, we claim that
$\delta^1_2(\XX\DT{\mathbf q},C_1)=C_1\otimes (\XX\DT{\mathbf q})$
for ${\mathbf q}\in\{\YY,\ZZ\}$; for example,
\[\mathcenter{
  \begin{tikzpicture}[scale=.7]
    \node at (0,1) (cin) {$\Max^{2}_{\Blg_1}$}; 
    \node at (0,-.8) (d1) {$\delta^1_1$}; 
    \node at (0,-2.5) (d2) {$\delta^1_2$};
    \node at (2.5,1) (ain) {$\Blg_1$};
    \node at (-2,1) (Ain) {$\Min^{1}_{\Blg_2}$};
    \node at (-2,-4) (Aout) {};
    \node at (-2,-2.8) (muA) {$\delta^1_3$}; 
    \node at (-3.5,-4) (aout) {}; 
    \node at (0,-4) (cout) {};
    \draw[algarrow,bend left=10] (ain) to node[above,sloped]{\tiny{$C_1$}}(d2) ;
    \draw[modarrow] (Ain) to node[left]{\tiny{$\XX$}}(muA) ;
    \draw[modarrow] (muA) to node[left]{\tiny{$\XX$}}(Aout) ;
    \draw[modarrow] (cin) to node[right]{\tiny{$\YY$}}(d1); 
    \draw[modarrow] (d1) to node[right]{\tiny{$\YY$}}(d2);
    \draw[modarrow] (d2) to node[right]{\tiny{$\YY$}}(cout);
    \draw[algarrow] (d1) to node[above,sloped] {\tiny{$U_2 C_3$}}(muA); 
    \draw[algarrow] (d2) to node[above,sloped] {\tiny{$C_1$}} (muA); 
    \draw[algarrow] (muA) to node[above,sloped]{\tiny{$C_1$}} (aout);
  \end{tikzpicture}
}\]
We also claim that $\delta^1_3(\ZZ\DT\ZZ,R_1,L_1)=C_1\otimes \ZZ\DT \ZZ$,
since
\[
\mathcenter{
  \begin{tikzpicture}[scale=.7]
    \node at (0,1.5) (cin) {$\Max^{2}_{\Blg_1}$}; 
    \node at (0,0) (d1) {$\delta^1_2$}; 
    \node at (0,-1.5) (d2) {$\delta^1_1$};
    \node at (0,-2.5) (d3) {$\delta^1_2$};
    \node at (2,1.5) (ain) {$\Blg_1$};
    \node at (3.5,1.5) (a2in) {$\Blg_1$};
    \node at (-2,1.5) (Ain) {$\Min^{1}_{\Blg_2}$};
    \node at (-2,-4) (Aout) {};
    \node at (-2,-2.5) (muA) {$\delta^1_4$}; 
    \node at (-3.5,-4) (aout) {}; 
    \node at (0,-4) (cout) {};
    \draw[algarrow,bend left=10] (ain) to node[above,sloped]{\tiny{$R_1$}}(d1) ;
    \draw[algarrow,bend left=20] (a2in) to node[above,sloped]{\tiny{$L_1$}}(d3) ;
    \draw[modarrow] (Ain) to node[left]{\tiny{$\ZZ$}}(muA) ;
    \draw[modarrow] (muA) to node[left]{\tiny{$\ZZ$}}(Aout) ;
    \draw[modarrow] (cin) to node[right]{\tiny{$\ZZ$}}(d1); 
    \draw[modarrow] (d1) to node[right]{\tiny{$\XX$}}(d2);
    \draw[modarrow] (d2) to node[right,pos=.3]{\tiny{$\XX$}}(d3);
    \draw[modarrow] (d3) to node[right]{\tiny{$\ZZ$}}(cout);
    \draw[algarrow,bend right=10] (d1) to node[above,sloped,pos=.4] {\tiny{$R_1$}}(muA); 
    \draw[algarrow] (d2) to node[above,sloped,pos=.3] {\tiny{$U_2 C_3$}} (muA); 
    \draw[algarrow] (d3) to node[above,sloped,pos=.3] {\tiny{$L_1$}} (muA); 
    \draw[algarrow] (muA) to node[above,sloped]{\tiny{$C_1$}} (aout);
\end{tikzpicture}}
\]
Again, it is straightforward to verify now that 
$\Min^{1}\DT\Max^{2}\simeq \Id$.
\end{proof}

\subsection{The invariance proof}
\label{sec:InvarianceProof}

We can now assemble the pieces to prove Theorem~\ref{thm:KnotInvariance}; and tie it with the discussion from the introduction.

\begin{proof}[Proof of Theorem~\ref{thm:KnotInvariance}]
  The tensor product description gives a chain complex thanks to a
  combination of Proposition~\ref{prop:AdaptedTensorProducts}, and
  finally Proposition~\ref{prop:FilteredTensorProduct} when attaching
  the final stage, to get a filtered complex.  Combining
  Lemma~\ref{lem:BraidMoves} with the
  invariance of the bimodules under bridge moves
  (Theorem~\ref{thm:BraidRelation},
  Lemmas~\ref{lem:TridentMoves},~\ref{lem:CommuteMaxPos},~\ref{lem:CommuteCrossCritPoints},~\ref{lem:PairCreation})
  shows that (Alexander-)filtered chain homotopy type of $\KC(K)$
  depends only on the pointed knot diagram for $K$, and the orientation
  type of the global minimum, i.e. $LR$ or $RL$ of Lemma~\ref{lem:BraidMoves}. 
  (Keep here in mind that
  homotopy equivalences of $DA$ bimodules induce homotopy equivalences
  under  tensor product.) 
  Choose for definiteness the orientation type $LR$.

  Next, we appeal to Proposition~\ref{prop:ConnectPointedKnotDiagrams} to see that we have a knot invariant.

  Invariance under Reidemeister (1)
  moves follows from the easily checked relation
  $\Pos_c\DT \Max^c_r\simeq \Max^c_{\ell}$. (Here $\Max^c_r$ and
  $\Max^c_{\ell}$ are bimodules associated to maxima, with opposite
  orientations on the new strand.)  Invariance under Reidemeister (2)
  and (3) moves now follows from the braid relations,
  Theorem~\ref{thm:BraidRelation}.

  Arguing in the same manner for $RL$, we obtain another conceivably
  different knot invariant.  A straightforward computation shows that
  \[\TerMin_{\Blg(2,1,\{2\})}=\TerMin_{\Blg(2,1,\{1\})}\DT \lsup{\Blg(2,1,\{1\})}\Pos^1_{\Blg(2,1,\{2\})}.\]
    It follows that both $RL$ and $LR$ give the same knot invariant.

  The Alexander filtration, takes values in the $H^1$ of the knot
  modulo a neighborhood of the global minimum, can be turned into a
  rational number by evaluating against the orientation class of the
  knot. This evaluation is computed by the local formula from
  Figure~\ref{fig:LocalCrossing} according to
  Proposition~\ref{prop:ComputeAlexander}. Since it is the
  exponent of $t$ in the contribution of the corresponding Kauffman
  state to the Alexander polynomial (compare~\cite{Kauffman}), it
  follows that the Alexander filtration takes values in $\Z$.
\end{proof}

\begin{cor}
  \label{cor:KnotInvarianceA}
  The homology of $\KCa(\Diag)$ is a bigraded invariant
  of the underlying oriented knot, whose
  Euler characteristic agrees with the Alexander polynomial.
\end{cor}

\begin{proof}
  Invariance follows from Theorem~\ref{thm:KnotInvariance}, since the
  homology of the associated graded object is invariant under filtered
  homotopy equivalences. The Euler characteristic computation follows
  from the correspondence between the generators and Kauffman states,
  and the observation that that $(-1)^{M(\x)} t^{A(\x)}$ (see
  Figure~\ref{fig:LocalCrossing}) is the monomial associated to a
  Kauffman state in the computation of the Alexander polynomial
  from~\cite{Kauffman}.
\end{proof}

It is a straightforward matter to go from the filtered homotopy type
to the invariant over a polynomial algebra described in the
introduction; see~\cite[Chapter~14]{GridBook}. 
Explicitly, replace the filtered 
module $\TerMin$ with an Alexander graded module $\TerMinM$
over a polynomial algebra $\Field[v]$. For $\Blg(2,1\{2\})$, 
replace Equation~\eqref{eq:ModuleStructure} with
\[\XX \cdot L_1 = v\cdot \YY\qquad \YY\cdot R_1 = v\cdot \XX \qquad \XX\cdot C_2=\YY\cdot C_2=0\]
(and making the analogous construction for $\Blg(2,1,\{1\})$).  In
this construction the base algebra now is $\Field[v]$, and $v$ is
given gradings $A(v)=1/2$ and $M(v)=0$. Replacing
$\TerMin$ with $\TerMinM$ and forming the same tensor product as
before, we arrive at a bigraded complex $\KCm$ over $\Field[v]$. Since
the Alexander grading of generators is integral, we can restrict this
to a complex over $\Field[U]$ with $U=v^2$. According to
Theorem~\ref{thm:KnotInvariance}, the homology of the associated
graded object, thought of as a bigraded module $H'(\orK)$ over $\Field[U]$, is an oriented
knot invariant. Letting $\KHm_{d}(\orK,s)=H'_{d-2s}(\orK,-s)$,
we arrive at the bigraded invariant from the introduction.
Setting $v=0$ (and so $U=0$) clearly recaptures
$\KCa$. Since $U$ drops Alexander grading by one and Maslov grading by
$2$ on $\KHm(K)$, Equation~\eqref{eq:GradedEulerHFm} follows from the fact that the
graded Euler characteristic of $\KCa$ is the Alexander polynomial.

\begin{rem}
  \label{rem:CanUseSubalgebra}
  We have described $\TerMin$ as having two generators.  In fact, to
  define our knot invariants, it suffices to work with the submodule
  $\TerMin \cdot \Idemp{\{1\}}$, since the outputs of our $D$ modules
  are all contained in the subalgebras of $\Blg(m,k,\Upwards)$
  $\left(\sum_{\x\big| 0,m\not\in \x}\Idemp{\x}\right)\cdot 
  \Blg(m,k,\Upwards) \cdot \left(\sum_{\x\big| 0,m\not\in \x}\Idemp{\x}\right)$.
\end{rem}

\newcommand\PartInv{\mathcal C}
\section{First properties}
\label{sec:Properties}

We will discuss efficient computations of these invariants in~\cite{HFKa2};
but
there are some easy computations that can be done readily by hand.  
As
a trivial example, the unknot has one-dimensional $\KHa$,
supported in bigrading $(0,0)$, since it has a diagram with a single
Kauffman state in it.  More generally:

\begin{prop}
  \label{prop:AlternatingKnots}
  If $K$ is an alternating knot, then $\KHa(K)$
  is determined by the signature
  $\sigma=\sigma (K)$ of $K$
  and
  the symmetrized Alexander polynomial
  $\Delta_K(t)=\sum_{i} a_i \cdot t^i$ 
  by $\KHa_{d}(K,s)= \Field^{|a_s|}$
  if $d=s+\frac{\sigma}{2}$, and $0$ otherwise.
\end{prop}

\begin{proof}
  As in~\cite{AltKnots}, this result is simply a consequence of
  the local contributions to the Alexander and Maslov gradings
  (pictured in Figure~\ref{fig:LocalCrossing}).
  Those formulas show that for an alternating diagram 
  $M-A=\frac{\sigma}{2}$. The rest now follows from the 
  interpretation of the Alexander polynomial in terms of Kauffman states~\cite{Kauffman}.
\end{proof}

We list a few of the properties that follow immediately from the
constructions from this paper. The methods of this paper are best suited 
for $\KHa$; analogous results for $\KHm$ will be established in
a follow-up paper~\cite{HFKa2}.

\begin{prop}
  If $K$ and $K'$ are mirror knots, then
  $\KHa_d(K,s)\cong \KHa_{-d}(K',-s)$.
\end{prop}

\begin{proof}
  According to Lemma~\ref{lem:OppositeBimodules}, the complex of
  $\KCa(K')$ is computed by tensoring together the opposites of the
  various bimodules used to compute $\KCa(K)$. It follows that
  $\KCa(K)$ and $\KCa(K')$ are dual complexes. Since we are working
  over the field $\Field$, the proposition now follows from the
  universal coefficient theorem.
\end{proof}

\begin{prop}
  \label{prop:ConnSum}
  Let $K_1$ and $K_2$ be two knots. Then,
  $\KHa(K_1\# K_2)$ is the bigraded tensor product of $\KHa(K_1)$
  and $\KHa(K_2)$.
\end{prop}

Before turning to the proof, we start with some more general properties.

Let $\Diag$ be a knot diagram with a single global minimum, oriented
up and to the left. Let $\Blg=\Blg(2,1\{1\})$ and let
$\lsup{\Blg}\PartInv(\Diag)$ denote
the type $D$ structure associated to the diagram with minimum removed.

\begin{lemma}
  \label{lem:Subalgebra}
  Let $\Blg=\Blg(2,1,\{1\})$
  and $\Blg'=\Idemp{\{1\}}\cdot\Blg\cdot\Idemp{\{1\}}$.
  The output algebra for $\lsup{\Blg}\PartInv(\Diag)$ is
  contained in $\Blg'$;
  i.e. if $\lsup{\Blg}[i]_{\Blg'}$ denotes the bimodule associated to 
  the inclusion $i\colon \Blg'\to \Blg$, then we have a type $D$ structure 
  $\lsup{\Blg'}\PartInv(\Diag)$ so that
  $\lsup{\Blg}[i]_{\Blg'}\DT~\lsup{\Blg'}\PartInv(\Diag)
  =\lsup{\Blg}\PartInv(\Diag)$.
\end{lemma}

\begin{proof}
  For  $m,k$ with $0\leq k\leq m+1$
  and $\Upwards\subset\{1,\dots,m\}$, 
  let 
  $\Blg'(m,k,\Upwards)\subset \Blg(m,k,\Upwards)$, be the subalgebra
  $\left(\sum_{\x\big| 0,m\not\in \x}\Idemp{\x}\right)\cdot
  \Blg(m,k,\Upwards) \cdot \left(\sum_{\x\big| 0,m\not\in
      \x}\Idemp{\x}\right)$. We claim that all of the $DA$ bimodules
  $\Max^c$, $\Min^c$, $\Pos^i$, $\Neg^i$ have the property that their
  restriction to $\Blg'\subset \Blg$ (with appropriate decorations) have their
  output algebras in the corresponding $\Blg'\subset\Blg$.
  Moreover, the global maximum has output algebra contained in $\Blg'$
  (c.f. Subsection~\ref{subsec:SmallMaximum}). 
\end{proof}

Recall that $\Blg'=\Idemp{\{1\}}\cdot\Blg(2,1,\{1\}) \cdot\Idemp{\{1\}}$
is $\Field[C_1,U_1,U_2]/C_1^2=0$, with $d C_1=U_1$, with gradings
\[\begin{array}{lll}(w_1(U_1),w_2(U_1))=(1,0),&
(w_1(C_1),w_2(C_1))=(1,0),&
(w_1(U_2),w_2(U_2))=(0,1).
\end{array}\]
It has a subalgebra $\Blg''\subset\Blg'$ isomorphic to 
$\Field[E_1,U_2]/E_1^2=0$ (with vanishing differential), and 
$(w_1(E_1),w_2(E_1))=(1,1)$ and $(w_1(U_2),w_2(U_2))=(0,1)$.
The subalgebra is specified by $E_1=C_1 U_2$.

\begin{lemma}
  \label{lem:DoverHomology}
  With $\Blg=\Blg(2,1,\{1\})$,
  the type $D$ module $\lsup{\Blg}\PartInv(\Diag)$ is homotopy equivalent
  to a type $D$ module over $\Blg''\subset\Blg$.
\end{lemma}

\begin{proof}
  Use Lemma~\ref{lem:Subalgebra} to restrict to $\Blg'$.
  This result then follows from homological perturbation theory, since
  the inclusion $\Blg''\subset\Blg'$ is a homotopy equivalence.
\end{proof}

\begin{proof} [Proof of Proposition~\ref{prop:ConnSum}]
  Consider a connected sum diagram for $K_1$ and $K_2$, where the
  connected sum region is taken to be the global minimum and the next
  minimum above it, as pictured in Figure~\ref{fig:ConnSum}.

\begin{figure}[ht]
\input{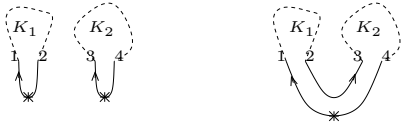}
\caption{\label{fig:ConnSum} {\bf{Connected sums.}}}
\end{figure}

The invariant associated to the
disjoint union of $K_1$ and $K_2$ (missing the two minima)
is a  type $D$ structure $\PartInv(K_1\cup K_2)$
over $\Blg(4,2,\{1,3\})$. Arguing as in Lemma~\ref{lem:Subalgebra},
the output algebra of this $D$ module is contained in 
$\Idemp{\{1,3\}}\cdot \Blg(4,2,\{1,3\})
\cdot \Idemp{\{1,3\}}$. 
Note that there is a natural identification
\[ \left(\Idemp{\{1\}}\cdot \Blg(2,1,\{1\})\cdot \Idemp{\{1\}}\right)
\otimes\left(\Idemp{\{1\}}\cdot \Blg(2,1,\{1\})\cdot \Idemp{\{1\}}\right)
\cong 
\Idemp{\{1,3\}}\cdot \Blg(4,2,\{1,3\})
\cdot \Idemp{\{1,3\}};
\]
and under this identification, $\PartInv(K_1\cup K_2)=\PartInv(K_1)\otimes\PartInv(K_2)$.

Consider the type $A$ structure $N$ over
$\Blg(4,2,\{1,3\})$ with one generator $\XX$ satisfying
$\XX\cdot \Idemp{\{1,3\}}=\XX$,
and trivial action by all other pure algebra elements.
Clearly, 
\begin{equation}
  \label{eq:Clearly}
  N\DT (\PartInv(K_1\cup K_2))\cong \KCa(K_1)\otimes \KCa(K_2)
\end{equation}

Consider next the type $A$ structure $P$ over $\Blg(4,2,\{1,3\})$
obtained by tensoring together a minimum with a global minimum to obtain the type $A$ structure pictured on the right in Figure~\ref{fig:ConnSum}.
We restrict to the idempotent type $P\cdot \Idemp{\{1,3\}}$.

It is easy to see that $P$ has actions $m_2(\XX,\Idemp{\{1,3\}})=\XX$.
Also,  if 
$m_{\ell+1}(\XX,a_1,\dots,a_{\ell})\neq 0$, then
$\sum_{i=1}^{\ell} w_1(a_i)=0=\sum_{i=1}^{\ell} w_4(a_i)$
and
\begin{equation}
  \label{eq:w2w3}
  \sum_{i=1}^{\ell} w_{2}(a_i)=\sum_{i=1}^{\ell} w_3(a_i).
\end{equation}

Let 
$\Blg'=\Idemp{\{1,3\}}\cdot \Blg(4,2,\{1,3\})\cdot\Idemp{\{1,3\}}$,
and $\Blg''$ be the subalgebra isomorphic to 
$\Field[E_1,E_3,U_2,U_4]/E_1^2=E_2^2=0$,
were $E_1=C_1 U_2$ and $E_2 =C_3 U_4$.
Lemma~\ref{lem:DoverHomology} gives 
type $D$ structures 
$\lsup{\Blg''(2,1,\{1\})}Q_1$ and 
$\lsup{\Blg''(2,1,\{1\})}Q_2$
so that
\[ \lsup{\Blg(2,1,\{1\})}\PartInv(K_1)\simeq Q_1\qquad
\text{and}\qquad 
\lsup{\Blg(2,1,\{1\})}\PartInv(K_2)\simeq Q_2;\]
and so 
$\lsup{\Blg''}(Q_1\otimes Q_2)
\simeq~\lsup{\Blg'}\PartInv(K_1\cup K_2)$.

  Since $w_1(E_1)>0$, we see that the action by $E_1$ on $P$ is trivial.
  Next, we claim that  any output algebra
  element in $Q_1$ of the form $U_2^t$ with $t\geq 0$
  must pair with some action on $P$ with $w_2>0$;
  but such an action must also involve an output from $Q_2$ with $w_3>0$
  (by Equation~\eqref{eq:w2w3}),
  and so it must come from some multiple of  $E_3$.
  But any sequence containing a non-trivial factor of
  $E_3$ in it acts trivially on $P$, since  $w_4(E_3)>0$.
  Similarly, outputs from $Q_2$ containing a non-trivial multiple of
  $E_3$ or $U_4$ act trivially since $w_4$ of both of those algebra elements
  are trivial.

  We have shown that the outputs from $Q_1\otimes Q_2$ giving non-zero
  differentials when paired with $P$ consist only of the trivial
  output ($1$), showing that 
  \[N\DT (Q_1\otimes Q_2)\cong P\DT (Q_1\otimes Q_2)=\KCa(K_1\#K_2).\]
  Combining this with Equation~\eqref{eq:Clearly}, we get an isomorphism
  of chain complexes $\KCa(K_1\# K_2)\cong \KCa(K_1)\otimes \KCa(K_2)$;
  and the proposition follows readily.
\end{proof}

% WONT BE IN FINAL VERSION 
% \include{kalgebra}
% \include{notes}
% \include{singular}
% \include{minimum}
% \include{skein}
% \include{crossingchange}
% \include{extensions}

\bibliographystyle{plain}
\bibliography{biblio}

\end{document}